\documentclass[namedate,webpdf,imamat]{ima-authoring-template}

\graphicspath{{Fig/}}

\usepackage{enumerate}
\usepackage{bm}
\usepackage{wrapfig}
\usepackage{multicol}
\usepackage{verbatim}
\usepackage{color}
\usepackage{exscale}
\usepackage{relsize}
\usepackage{epstopdf}
\usepackage{booktabs}

\usepackage{dutchcal}

\theoremstyle{thmstyletwo}%
\newtheorem{theorem}{Theorem}%
\newtheorem{lemma}{Lemma}%
\newtheorem{rmk}{Remark}%
\newtheorem{assumption}{Assumption}%
\numberwithin{equation}{section}

\begin{document}

\DOI{DOI HERE}
\copyrightyear{2026}
\vol{00}
\pubyear{2026}
\access{Advance Access Publication Date: Day Month Year}
\appnotes{Paper}
\copyrightstatement{Published by Oxford University Press on behalf of the Institute of Mathematics and its Applications. All rights reserved.}
\firstpage{1}


\title[Efficient higher-order multi-scale method]{Efficient higher-order multi-scale method and its convergence estimate for dynamic nonlinear hygro-thermo-mechanical coupling problems of heterogeneous structures}

\author{Yifei Ding
	\address{\orgdiv{School of Mathematics and Statistics}, \orgname{Xidian University}, \orgaddress{\postcode{710071}, \state{Shaanxi}, \country{PR China}}}}
\author{Hao Dong*
	\address{\orgdiv{School of Mathematics and Statistics}, \orgname{Xidian University}, \orgaddress{\postcode{710071}, \state{Shaanxi}, \country{PR China}}}}
\author{Jiale Linghu
	\address{\orgdiv{School of Mathematics and Statistics}, \orgname{Xidian University}, \orgaddress{\postcode{710071}, \state{Shaanxi}, \country{PR China}}}}
\author{Yangshuai Wang
	\address{\orgdiv{Department of Mathematics}, \orgname{National University of Singapore}, \orgaddress{\street{10 Lower Kent Ridge Road}, \postcode{119076}, \country{Singapore}}}}

\authormark{Ding et al.}

\corresp[*]{\href{mailto:donghao@mail.nwpu.edu.cn}{donghao@mail.nwpu.edu.cn}}

\received{Date}{0}{Year}
\revised{Date}{0}{Year}
\accepted{Date}{0}{Year}


\abstract{This paper presents a novel higher-order multi-scale (HOMS) computational framework for efficient, high-accuracy, and low-cost simulation of nonlinear hygro-thermo-mechanical (H-T-M) coupling problems in heterogeneous structures. The inherent nonlinearity in the investigated model stems primarily from temperature- or moisture-dependent material properties, and this model also accounts for temperature-dependent internal heat source and moisture sink terms induced by exothermic, moisture-consuming chemical reactions (e.g., hydration). The main contributions of this work are as follows. First, a high-accuracy multi-scale asymptotic model incorporating higher-order correction terms is constructed for nonlinear H-T-M coupling problems in heterogeneous structures with highly spatial inhomogeneity, using the multi-scale asymptotic approach together with Taylor series expansions. Second, rigorous error estimates in both point-wise and integral senses are derived for the multi-scale asymptotic solutions, which theoretically demonstrate the necessity and superiority of the proposed HOMS method. Third, an efficient two-stage numerical algorithm with off-line and on-line stages is developed, based on finite difference and finite element methods, and its convergence is also proved rigorously. Finally, two- and three-dimensional numerical experiments are performed to assess the computational performance of the proposed HOMS approach, showing excellent numerical accuracy and robustness with low computational overhead.}


\maketitle

\section{Introduction}
	Composite materials, by virtue of the matching and synergy of their constituents, offer high specific stiffness, high specific strength, light weight, and good fatigue resistance. Through rational design, they can also achieve tailorable properties such as corrosion resistance and thermal stability. These exceptional properties and remarkable designability have driven their extensive utilization in aerospace, marine, construction, mechanical manufacturing, electronic packaging, and other cutting-edge engineering fields \citep{R1,R2}. However, for many materials, the physical and mechanical properties typically vary with temperature and moisture, causing the material parameters in the governing equations of heat conduction, moisture diffusion, and elastic-dynamics to become functions of temperature and humidity, thereby rendering the entire system of equations nonlinear \citep{R3,R6,R7,R8}. Owing to the significant practical value and theoretical importance, the development of effective analytical models and efficient numerical methods has become a research focus in the field of nonlinear dynamic hygro-thermo-mechanical coupling problems of heterogeneous structures. In the field of electronic packaging, plastic-encapsulated devices are subjected to thermal, moisture, and mechanical loads during manufacturing (e.g., reflow soldering) and in service. The hygrothermal environment is widely recognized as one of the most critical factors affecting device reliability \citep{R56,R57}. At the microscopic scale, thermal stress and moisture-induced degradation are the two primary mechanisms of interface failure in plastic-encapsulated devices \citep{R56}. Hygroscopic swelling, adhesion degradation, and high-temperature vapor pressure often lead to interfacial delamination, cracking, and even the "popcorn" effect \citep{R6,R57}. Therefore, accurate prediction of nonlinear hygro-thermo-mechanical coupling behavior in heterogeneous structures is essential for improving electronic packaging reliability. These coupled problems have also attracted considerable attention in other fields, including progressive damage of composites under hygrothermal aging in marine and wind energy applications \citep{R8}, cracking risk of early-age concrete bridge piers on plateaus \citep{R3}, hygrothermal debonding of polymeric interfaces in photovoltaic laminates \citep{R7}, flexural fatigue life of GFRP-balsa sandwich bridge decks \citep{R10}, and fatigue life of CFRP/Al hybrid joints in automotive lightweight structures \citep{R9}, etc. However, the presence of temperature- and moisture-dependent coefficients, together with the multi-scale nature of the underlying heterogeneous structures, makes the numerical solution of such coupled problems particularly challenging.
	
	The physical and mechanical properties of composites depend on their constituent materials and microstructural characteristics. To investigate these properties, engineers and mathematicians have developed numerous analytical and numerical methods from different perspectives. In mechanics, two main approaches exist for analyzing composite behavior: macromechanics and micromechanics \citep{R13}. Mathematically, predicting the physical and mechanical response of heterogeneous structures under given loads amounts to solving initial-boundary value problems of partial differential equations (PDEs) with coefficients that are periodic, quasi-periodic, or random oscillatory functions. However, conventional numerical methods such as the finite difference method (FDM), finite element method (FEM), and finite volume method (FVM) face two major challenges in this context: (1) The macroscopic structural size of composites is much larger than the characteristic size of their microstructure. Direct numerical simulation thus requires a sufficiently fine mesh to resolve the microscale information, leading to a sharp increase in mesh generation effort and computational cost, which severely limits solution efficiency. (2) The coexistence of multi-physics coupling (e.g., hygro-thermal-mechanical) and macro-micro coupling introduces strong nonlinearity, high coupling, sharply discontinuous and highly oscillatory coefficients, spatial multi-scale features, and time-dependent behavior, which poses significant challenges to the accuracy and convergence of these classical methods. Therefore, the development of computable models and multi-scale methods for multi-physics coupled and cross-scale problems has become a frontier research need in modern science and technology. In response, mathematicians and mechanicians have proposed various multi-scale methods, including the asymptotic homogenization method (AHM) \citep{R14,R15}, heterogeneous multi-scale method (HMM) \citep{R16}, variational multi-scale method (VMS) \citep{R17}, multi-scale finite element method (MsFEM) \citep{R18}, and more recent approaches such as the local orthogonal decomposition (LOD) \citep{R19}, generalized multi-scale finite element method (GMsFEM) \citep{R20}, and constraint energy minimizing generalized multi-scale finite element method (CEM-GMsFEM) \citep{R21}. Over the past three decades, Cui and his team have systematically established a family of higher-order multi-scale methods for the accurate simulation of thermal, mechanical, and multiphysics behaviors of composites \citep{R22,R23,R24,R25,R26}. Nevertheless, these methods are often restricted to linear problems. With the demand for high-accuracy simulations, attention has turned to nonlinear problems of heterogeneous solids. In references \citep{R27,R28,R29}, the asymptotic homogenization method was employed to handle nonlinear heat conduction with temperature-dependent coefficients. Multi-scale finite element method for nonlinear elliptic and parabolic equations was proposed by Efendiev et al. \citep{R30}. Dong et al. pioneered a higher-order multi-scale approach, accompanied by rigorous convergence analysis, for nonlinear thermo-mechanical coupling problems \citep{R31,R32}.
	
	However, research on nonlinear multiphysics problems in multi-scale modeling of heterogeneous structures remains scarce, particularly for nonlinear hygro-thermo-mechanical (H-T-M) coupling. To date, very few studies have addressed this class of problems using multi-scale methods. Gholami et al. developed a coupled HTM-multi-scale-APFEA algorithm for failure analysis of thick composites under hygrothermal conditions \citep{R33}; however, their approach is based on a first-order computational homogenization ($\text{FE}^2$) framework and lacks a systematic higher-order theoretical framework with rigorous convergence estimates. More recently, Dong et al. proposed a higher-order multi-scale computational method for quasi-periodic composites, incorporating higher-order correction terms with rigorous error analyses \citep{R34}; nevertheless, this work focuses on quasi-periodic heterogeneous structures and does not account for temperature- and moisture-dependent material nonlinearities. Consequently, a systematic higher-order multi-scale method specifically for nonlinear H-T-M coupling problems of heterogeneous structures is still lacking, despite its practical importance and mathematical challenge. This motivates the present work.
	
	To effectively address nonlinear hygro-thermo-mechanical coupling problems in heterogeneous structures with highly discontinuous and oscillatory parameters, this paper proposes a higher-order multi-scale computational model and a corresponding numerical algorithm. The model maintains local equilibrium of the relevant physical quantities, thereby enabling high-fidelity multi-scale simulations. The paper is organized as follows. Section \ref{sec:2} defines the mathematical governing equations for nonlinear dynamic hygro-thermo-mechanical coupling problems of heterogeneous structures with temperature- and moisture-dependent material properties. Using multi-scale asymptotic analysis and Taylor series expansion, a higher-order multi-scale computational model is established for such problems. Section \ref{sec:3} performs a pointwise numerical accuracy comparison between the traditional lower-order multi-scale (LOMS) solution and the proposed higher-order multi-scale (HOMS) solution. Through this comparison, the necessity and importance of the HOMS model in capturing microscopic oscillatory information are theoretically demonstrated. Moreover, under certain simplifications and assumptions, rigorous error estimates with explicit convergence rates are derived for the HOMS solution, validating its approximation performance and convergence. Section \ref{sec:4} develops a two-stage multi-scale numerical algorithm, which consists of an off-line stage for computing microscopic cell problems and an on-line stage for computing the macroscopic homogenized problem and HOMS solutions, and is designed to efficiently simulate nonlinear hygro-thermo-mechanical coupling problems of heterogeneous structures with temperature- and moisture-dependent properties. Section \ref{sec:4.4} presents the corresponding error analysis of the two-stage algorithm. Section \ref{sec:5} presents two and three dimensional numerical examples to verify the computational performance of the proposed HOMS model and its associated numerical algorithm. Section \ref{sec:6} concludes remarks and discusses future research directions.
	
	For notational brevity, we apply the Einstein summation convention to repeated indices throughout this work.
	
\section{Novel higher-order multi-scale computational model}
\label{sec:2}
	\subsection{Problem setting and governing equations}
	\label{sec:21}
	Inspired by the H-T-M coupling model for cementitious materials proposed by Zhang et al. \citep{R3}, we formulate the following governing equations for nonlinear H-T-M coupling problems of heterogeneous structures. The domain $\Omega\in\mathbb{R}^n(n=2,3)$ is a bounded convex domain with Lipschitz continuous boundary $\partial\Omega=\partial\Omega_{T}\cup\partial\Omega_{q}\cup\partial\Omega_{\omega}\cup\partial\Omega_{d}\cup\partial\Omega_{u}\cup\partial\Omega_{\sigma}$, where these boundary parts are pairwise disjoint. The domain is formed by the repetition of periodic unit cell (PUC) $Y$.
	\begin{equation}
		\label{eq:2.1}
		\begin{cases}
			\begin{aligned}
				&\rho^\epsilon(\mathbf{x},T^\epsilon)c^\epsilon(\mathbf{x},T^\epsilon)\frac{\partial T^\epsilon(\mathbf{x},t)}{\partial t} - \frac{\partial}{\partial x_i}\Bigl(k_{ij}^\epsilon(\mathbf{x},T^\epsilon)\frac{\partial T^\epsilon(\mathbf{x},t)}{\partial x_j}\Bigr) \\
				& \quad = h(\mathbf{x},t) + Q_{hyd}^\epsilon(\mathbf{x},T^\epsilon),  \text{in } \Omega \times (0,T^*), \\
				&\frac{\partial \omega^\epsilon(\mathbf{x},t)}{\partial t} - \frac{\partial}{\partial x_i}\Bigl(g_{ij}^\epsilon(\mathbf{x},\omega^\epsilon)\frac{\partial \omega^\epsilon(\mathbf{x},t)}{\partial x_j}\Bigr) = m(\mathbf{x},t) - S_{hyd}^\epsilon(\mathbf{x},T^\epsilon),  \text{in } \Omega \times (0,T^*), \\
				& - \frac{\partial}{\partial x_j}\Bigl(C_{ijkl}^\epsilon(\mathbf{x},T^\epsilon)\frac{\partial u_k^\epsilon(\mathbf{x},t)}{\partial x_l} - \alpha_{ij}^\epsilon(\mathbf{x},T^\epsilon)(T^\epsilon(\mathbf{x},t) - \tilde{T}) - \beta_{ij}^\epsilon(\mathbf{x},T^\epsilon)(\omega^\epsilon(\mathbf{x},t) - \tilde{\omega})\Bigr) \\
				& \quad= f_i(\mathbf{x},t),  \text{in } \Omega \times (0,T^*), \\
				&T^\epsilon(\mathbf{x},t) = \hat{T}(\mathbf{x},t),  \text{on } \partial\Omega_T \times (0,T^*), \\
				&k_{ij}^\epsilon(\mathbf{x},T^\epsilon)\frac{\partial T^\epsilon(\mathbf{x},t)}{\partial x_j}n_i = \bar{q}(\mathbf{x},t), \text{on } \partial\Omega_q \times (0,T^*), \\
				&\omega^\epsilon(\mathbf{x},t) = \hat{\omega}(\mathbf{x},t), \text{on } \partial\Omega_\omega \times (0,T^*), \\
				&g_{ij}^\epsilon(\mathbf{x},\omega^\epsilon)\frac{\partial \omega^\epsilon(\mathbf{x},t)}{\partial x_j}n_i = \bar{d}(\mathbf{x},t), \text{on } \partial\Omega_d \times (0,T^*), \\
				&\bm{u}^\epsilon(\mathbf{x},t) = \hat{\bm{u}}(\mathbf{x},t), \text{on } \partial\Omega_u \times (0,T^*), \\
				&\Bigl(C_{ijkl}^\epsilon(\mathbf{x},T^\epsilon)\frac{\partial u_k^\epsilon(\mathbf{x},t)}{\partial x_l} - \alpha_{ij}^\epsilon(\mathbf{x},T^\epsilon)(T^\epsilon(\mathbf{x},t) - \tilde{T}) - \beta_{ij}^\epsilon(\mathbf{x},T^\epsilon)(\omega^\epsilon(\mathbf{x},t) - \tilde{\omega})\Bigr)n_j \\
				& \quad= \bar{\sigma}_i(\mathbf{x},t), \text{on } \partial\Omega_\sigma \times (0,T^*), \\
				&T^\epsilon(\mathbf{x},0) = \tilde{T}, \omega^\epsilon(\mathbf{x},0) = \tilde{\omega}, \bm{u}^\epsilon(\mathbf{x},0) = \tilde{\bm{u}}, \text{in } \Omega.
			\end{aligned}
		\end{cases}
	\end{equation}
	where $\epsilon$ represents the characteristic size of microscopic cell $Y$; $T^{\epsilon}(\mathbf{x},t)$, $\omega^{\epsilon}(\mathbf{x},t)$ and $\bm{u}^{\epsilon}(\mathbf{x},t)$ are the undetermined temperature, moisture and displacement fields; $\hat{T}(\mathbf{x},t)$, $\hat{\omega}(\mathbf{x},t)$ and $\hat{\bm{u}}(\mathbf{x},t)$ are the prescribed temperature, moisture and displacement on the domain boundaries $\partial\Omega_{T}\cup\partial\Omega_{\omega}\cup\partial\Omega_{u}$ with $meas(\partial\Omega_{T})>0$, $meas(\partial\Omega_{\omega})>0$ and $meas(\partial\Omega_{u})>0$; $\overline{q}(\mathbf{x},t)$, $\overline{d}(\mathbf{x},t)$ and $\overline{\sigma}_{i}(\mathbf{x},t)$ are the prescribed heat flux, moisture flux and traction on the domain boundaries $\partial\Omega_{q}\cup\partial\Omega_{d}\cup\partial\Omega_{\sigma}$, where $n_j$ denotes the $j$-th component of the unit normal vector at a given point on the domain boundaries. $\tilde{T}$, $\tilde{\omega}$, and $\tilde{\bm{u}}$ represent the initial temperature, moisture and displacement of domain $\Omega$. $\rho^\epsilon(\mathbf{x},T^\epsilon)$ and $c^\epsilon(\mathbf{x},T^\epsilon)$ are the mass density and specific heat;  $\displaystyle\{k_{ij}^\epsilon(\mathbf{x},T^\epsilon)\}$ and $\displaystyle\{g_{ij}^\epsilon(\mathbf{x},\omega^\epsilon)\}$ are the second order thermal conductivity tensor and moisture diffusion tensor; $\displaystyle\{C_{ijkl}^\epsilon(\mathbf{x},T^\epsilon)\}$ is the fourth order elastic tensor; $\displaystyle\{\alpha_{ij}^\epsilon(\mathbf{x},T^\epsilon)\}$ and  $\displaystyle\{\beta_{ij}^\epsilon(\mathbf{x},T^\epsilon)\}$ are the second order thermal stress coefficient tensor and moisture stress coefficient tensor. Furthermore, the internal source terms $Q_{hyd}^\epsilon(\mathbf{x},T^\epsilon)$ and $S_{hyd}^\epsilon(\mathbf{x},T^\epsilon)$ represent heat release and moisture consumption due to temperature-dependent chemical reactions, while $h(\mathbf{x},t)$ and $m(\mathbf{x},t)$ denote external heat and moisture sources, respectively. This formulation is motivated by similar reactive source structures observed in cement hydration \citep{R3} and commonly found in thermoset curing processes. $f_{i}(\mathbf{x},t)$ is the body force.
	
	To begin with, following the theoretical framework of AHM, let $\mathbf{x}$ denote the macroscopic coordinates and $\displaystyle \mathbf{y} = {\mathbf{x}}/{\epsilon}$ denote the corresponding microscopic coordinates of the PUC $Y=[0,1]^n$. With this notation, the material coefficients $\rho^\epsilon(\mathbf{x},T^\epsilon)$, $c^\epsilon(\mathbf{x},T^\epsilon)$, $k_{ij}^\epsilon(\mathbf{x},T^\epsilon)$,  $g_{ij}^\epsilon(\mathbf{x},\omega^\epsilon)$, $C_{ijkl}^\epsilon(\mathbf{x},T^\epsilon)$, $\alpha_{ij}^\epsilon(\mathbf{x},T^\epsilon)$, $\beta_{ij}^\epsilon(\mathbf{x},T^\epsilon)$, and the internal source terms $Q_{hyd}^\epsilon(\mathbf{x},T^\epsilon)$ and $S_{hyd}^\epsilon(\mathbf{x},T^\epsilon)$ can be rewritten as $\rho(\mathbf{y},T^\epsilon)$, $c(\mathbf{y},T^\epsilon)$, $k_{ij}(\mathbf{y},T^\epsilon)$,  $g_{ij}(\mathbf{y},\omega^\epsilon)$, $C_{ijkl}(\mathbf{y},T^\epsilon)$, $\alpha_{ij}(\mathbf{y},T^\epsilon)$, $\beta_{ij}(\mathbf{y},T^\epsilon)$, $Q_{hyd}(\mathbf{y},T^\epsilon)$, and $S_{hyd}(\mathbf{y},T^\epsilon)$, which necessitates that these quantities be $1$-periodic functions of the micro-scale variable $\mathbf{y}$. Furthermore, the chain rule for the performed spatial scales is given by
	\begin{equation}
		\label{eq:2.2}
		\frac{\partial \Phi^{\epsilon} (\mathbf{x},t)}{\partial x_i}=\frac{\partial \Phi(\mathbf{x},\mathbf{y},t)}{\partial x_i}+\frac{1}{\epsilon }\frac{\partial \Phi(\mathbf{x},\mathbf{y},t)}{\partial y_i},
	\end{equation}
	which will be employed frequently in the subsequent analysis.
	
	Moreover, similar to the assumptions adopted in \citep{R25,R22, R39, R40, R28, R42}, we make the following assumptions for the governing equations \eqref{eq:2.1}.
	\begin{enumerate}
		\item[(A)] $k_{ij}^\epsilon(\mathbf{x},T^\epsilon)$,  $g_{ij}^\epsilon(\mathbf{x},\omega^\epsilon)$, $C_{ijkl}^\epsilon(\mathbf{x},T^\epsilon)$, $\alpha_{ij}^\epsilon(\mathbf{x},T^\epsilon)$,  $\beta_{ij}^\epsilon(\mathbf{x},T^\epsilon)$ are symmetric, and there exist two positive constants $\underline{\gamma}$ and $\overline{\gamma}$ independent of $\epsilon$ such that
		\begin{displaymath}
			\begin{aligned}
				&k_{ij}^\epsilon = k_{ji}^\epsilon,\;\underline{\gamma} | \bm{\xi} |^2 \leq k_{ij}^\epsilon(\mathbf{x},T^\epsilon)\xi_i\xi_j \leq \overline{\gamma} | \bm{\xi} |^2,\\
				&g_{ij}^\epsilon=g_{ji}^\epsilon,\; \underline{\gamma} | \bm{\xi} |^2 \leq g_{ij}^\epsilon(\mathbf{x},\omega^\epsilon) \xi_i \xi_j \leq \overline{\gamma} | \bm{\xi} |^2,\\		
				&C_{ijkl}^\epsilon = C_{ijlk}^\epsilon = C_{klij}^\epsilon, \;\underline{\gamma} \eta_{ij} \eta_{ij} \leq C_{ijkl}^\epsilon(\mathbf{x},T^\epsilon)  \eta_{ij} \eta_{kl} \leq \overline{\gamma} \eta_{ij} \eta_{ij},\\	
				&\alpha_{ij}^\epsilon=\alpha_{ji}^\epsilon,\;\underline{\gamma}| \bm{\xi} |^2 \leq \alpha_{ij}^\epsilon(\mathbf{x},T^\epsilon) \xi_i \xi_j \leq \overline{\gamma} | \bm{\xi} |^2,\\
				&\beta_{ij}^\epsilon=\beta_{ji}^\epsilon,\;\underline{\gamma}| \bm{\xi} |^2 \leq \beta_{ij}^\epsilon(\mathbf{x},T^\epsilon) \xi_i \xi_j \leq \overline{\gamma} | \bm{\xi} |^2,\\
			\end{aligned}
		\end{displaymath}
		where $\{\eta_{ij}\}$ is an arbitrary symmetric matrix in $\mathbb{R}^{n \times n}$, $\bm{\xi}=(\xi_1, \xi_2,\cdots,\xi_n)$ is an arbitrary vector with real elements in $ \mathbb{R}^n $.
		\item[(B)] $\rho^\epsilon(\mathbf{x},T^\epsilon)$, $c^\epsilon(\mathbf{x},T^\epsilon)$, $k_{ij}^\epsilon(\mathbf{x},T^\epsilon)$,  $g_{ij}^\epsilon(\mathbf{x},\omega^\epsilon)$, $C_{ijkl}^\epsilon(\mathbf{x},T^\epsilon)$, $\alpha_{ij}^\epsilon(\mathbf{x},T^\epsilon)$,  $\beta_{ij}^\epsilon(\mathbf{x},T^\epsilon)$, $Q_{hyd}^\epsilon(\mathbf{x},T^\epsilon)$, and  $S_{hyd}^\epsilon(\mathbf{x},T^\epsilon)$ are scalar functions belonging to $L^\infty(\Omega)$, and there are constants $\rho ^0$, $c^0$, $Q_{hyd}^0$ and $S_{hyd}^0$ such that
		\begin{displaymath}
			0 < \rho^0 \le \rho^\epsilon(\mathbf{x},T^\epsilon), 0 < c^0 \le c^\epsilon(\mathbf{x},T^\epsilon), Q_{hyd}^0 \le Q_{hyd}^\epsilon(\mathbf{x},T^\epsilon), S_{hyd}^0 \le S_{hyd}^\epsilon(\mathbf{x},T^\epsilon).
		\end{displaymath}
		where $\rho^0$, $c^0$, $Q_{hyd}^0$, and $S_{hyd}^0$ are constants irrespective of $\epsilon$.
		\item[(C)] $h(\mathbf{x},t)$, $m(\mathbf{x},t)$ and $f_i(\mathbf{x},t) \in L^2(\Omega \times (0,T^*))$. $\hat{T}(\mathbf{x},t)$ and $\hat{\omega}(\mathbf{x},t) \in L^2(0,T^*;H^1(\Omega))$, and $\hat{\bm{u}}(\mathbf{x},t) \in L^2(0,T^*;(H^1(\Omega))^n)$. $\overline{q}(\mathbf{x},t)$, $\overline{d}(\mathbf{x},t)$ and $\overline{\sigma}_{i}(\mathbf{x},t)\in L^2(\Omega \times (0,T^*))$.
	\end{enumerate}
	
	\subsection{HOMS computational model of governing equations}
	\label{sec:22}
	In this subsection, the HOMS computational model is constructed. For the multi-scale problem \eqref{eq:2.1}, we assume that
	$T^{\epsilon}(\mathbf{x},t)$, $\omega^{\epsilon}(\mathbf{x},t)$ and $u_i^{\epsilon}(\mathbf{x},t)$ can be expanded as the following power series representations.
	\begin{equation}
		\label{eq:2.3}
		\begin{cases}
			\begin{aligned}
				&T^{\epsilon}(\mathbf{x},t) = T^{(0)}(\mathbf{x},\mathbf{y},t) + \epsilon T^{(1)}(\mathbf{x},\mathbf{y},t) + \epsilon^{2} T^{(2)}(\mathbf{x},\mathbf{y},t) + O(\epsilon^{3}), \\
				&\omega^{\epsilon}(\mathbf{x},t) = \omega^{(0)}(\mathbf{x},\mathbf{y},t) + \epsilon \omega^{(1)}(\mathbf{x},\mathbf{y},t) + \epsilon^{2} \omega^{(2)}(\mathbf{x},\mathbf{y},t) + O(\epsilon^{3}), \\
				&u_{i}^{\epsilon}(\mathbf{x},t) = u_{i}^{(0)}(\mathbf{x},\mathbf{y},t) + \epsilon u_{i}^{(1)}(\mathbf{x},\mathbf{y},t) + \epsilon^{2} u_{i}^{(2)}(\mathbf{x},\mathbf{y},t) + O(\epsilon^{3}).
			\end{aligned}
		\end{cases}
	\end{equation}
	where $T^{(0)}$, $\omega^{(0)}$, $u_{i}^{(0)}$ are defined as the zeroth-order expansion terms; $T^{(1)}$, $\omega^{(1)}$, $u_{i}^{(1)}$ as the first-order (lower-order) asymptotic terms; and $T^{(2)}$, $\omega^{(2)}$, $u_{i}^{(2)}$ as the second-order (higher-order) asymptotic terms.
	
	Next, the key idea for handling temperature- or moisture-dependent material coefficients, and temperature-dependent internal source terms is introduced. Specifically, using the Taylor's formula with multi-index notation $\displaystyle f({x_0},{y_0} + h) = f({x_0},{y_0}) + {f_y}({x_0},{y_0})h + \frac{1}{2}{f_{yy}}({x_0},{y_0}){h^2} + O({h^3}) = f(x_0, y_0) + \mathbf{D}^{(0,1)}f(x_0, y_0)h + \frac{1}{2}\mathbf{D}^{(0,2)}f(x_0, y_0)h^2 + O(h^3)$ in \citep{R39}, all such material coefficients and internal source terms in the multi-scale problem \eqref{eq:2.1} can be expanded in the same manner. Taking the temperature-dependent elastic stiffness coefficients as an example, they can be expanded as follows \citep{R39,R40}.
	\begin{equation}
		\label{eq:2.6}
		\begin{aligned}
			& C_{ijkl}^\epsilon(\mathbf{x},T^\epsilon) = C_{ijkl}(\mathbf{y},T^\epsilon) = C_{ijkl}\bigl(\mathbf{y},T^{(0)} + \epsilon T^{(1)} + \epsilon^2 T^{(2)} + \mathrm{O}(\epsilon^3)\bigr)\\
			& = C_{ijkl}(\mathbf{y},T^{(0)}) + \mathbf{D}^{(0,1)}C_{ijkl}(\mathbf{y},T^{(0)})\bigl[\epsilon T^{(1)} + \epsilon^2 T^{(2)} + \mathrm{O}(\epsilon^3)\bigr]\\
			& + \frac{1}{2}\mathbf{D}^{(0,2)}C_{ijkl}(\mathbf{y},T^{(0)})\bigl[\epsilon T^{(1)} + \epsilon^2 T^{(2)} + \mathrm{O}(\epsilon^3)\bigr]^2 + O\Bigl(\bigl[\epsilon T^{(1)} + \epsilon^2 T^{(2)} + \mathrm{O}(\epsilon^3)\bigr]^3\Bigr)\\
			& = C_{ijkl}(\mathbf{y},T^{(0)}) + \epsilon T^{(1)}\mathbf{D}^{(0,1)}C_{ijkl}(\mathbf{y},T^{(0)})\\
			& + \epsilon^2\bigl[T^{(2)}\mathbf{D}^{(0,1)}C_{ijkl}(\mathbf{y},T^{(0)}) + \frac{1}{2}(T^{(1)})^2\mathbf{D}^{(0,2)}C_{ijkl}(\mathbf{y},T^{(0)})\bigr] + \mathrm{O}(\epsilon^3)\\
			& =: C_{ijkl}^{(0)}(\mathbf{y}, T^{(0)}) + \epsilon C_{ijkl}^{(1)}(\mathbf{x}, \mathbf{y}, T^{(0)}) + \epsilon^2 C_{ijkl}^{(2)}(\mathbf{x}, \mathbf{y}, T^{(0)}) + \mathrm{O}(\epsilon^3).
		\end{aligned}
	\end{equation}
	Following the same expansion method as in \eqref{eq:2.6}, the other material parameters $\rho^\epsilon(\mathbf{x},T^\epsilon)$, $c^\epsilon(\mathbf{x},T^\epsilon)$, $k_{ij}^\epsilon(\mathbf{x},T^\epsilon)$,  $g_{ij}^\epsilon(\mathbf{x},\omega^\epsilon)$, $\alpha_{ij}^\epsilon(\mathbf{x},T^\epsilon)$,  $\beta_{ij}^\epsilon(\mathbf{x},T^\epsilon)$, and the internal source terms $Q_{hyd}^\epsilon(\mathbf{x},T^\epsilon)$, $S_{hyd}^\epsilon(\mathbf{x},T^\epsilon)$ are expanded as follows.
	\begin{equation}
		\label{eq:2.7}
		\begin{aligned}
			\rho^\epsilon(\mathbf{x},T^\epsilon)
			&= \rho^{(0)}(\mathbf{y},T^{(0)}) + \epsilon \rho^{(1)}(\mathbf{x}, \mathbf{y}, T^{(0)}) + \epsilon^2 \rho^{(2)}(\mathbf{x}, \mathbf{y}, T^{(0)}) + \mathrm{O}(\epsilon^3).
		\end{aligned}
	\end{equation}
	\begin{equation}
		\label{eq:2.8}
		\begin{aligned}
			c^\epsilon(\mathbf{x},T^\epsilon)
			&= c^{(0)}(\mathbf{y},T^{(0)}) + \epsilon c^{(1)}(\mathbf{x}, \mathbf{y}, T^{(0)}) + \epsilon^2 c^{(2)}(\mathbf{x}, \mathbf{y}, T^{(0)}) + \mathrm{O}(\epsilon^3).
		\end{aligned}
	\end{equation}
	\begin{equation}
		\label{eq:2.9}
		\begin{aligned}
			k_{ij}^\epsilon(\mathbf{x},T^\epsilon)
			&= k_{ij}^{(0)}(\mathbf{y},T^{(0)}) + \epsilon k_{ij}^{(1)}(\mathbf{x}, \mathbf{y}, T^{(0)}) + \epsilon^2 k_{ij}^{(2)}(\mathbf{x}, \mathbf{y}, T^{(0)}) + \mathrm{O}(\epsilon^3).
		\end{aligned}
	\end{equation}
	\begin{equation}
		\label{eq:2.10}
		\begin{aligned}
			g_{ij}^\epsilon(\mathbf{x},\omega^\epsilon)
			& = g_{ij}^{(0)}(\mathbf{y},\omega^{(0)}) + \epsilon g_{ij}^{(1)}(\mathbf{x}, \mathbf{y}, \omega^{(0)}) + \epsilon^2 g_{ij}^{(2)}(\mathbf{x}, \mathbf{y}, \omega^{(0)}) + \mathrm{O}(\epsilon^3).
		\end{aligned}
	\end{equation}
	\begin{equation}
		\label{eq:2.11}
		\begin{aligned}
			\alpha_{ij}^\epsilon(\mathbf{x},T^\epsilon)
			&= \alpha_{ij}^{(0)}(\mathbf{y},T^{(0)}) + \epsilon \alpha_{ij}^{(1)}(\mathbf{x}, \mathbf{y}, T^{(0)}) + \epsilon^2 \alpha_{ij}^{(2)}(\mathbf{x}, \mathbf{y}, T^{(0)}) + \mathrm{O}(\epsilon^3).
		\end{aligned}
	\end{equation}
	\begin{equation}
		\label{eq:2.12}
		\begin{aligned}
			\beta_{ij}^\epsilon(\mathbf{x},T^\epsilon)
			&= \beta_{ij}^{(0)}(\mathbf{y},T^{(0)}) + \epsilon \beta_{ij}^{(1)}(\mathbf{x}, \mathbf{y}, T^{(0)}) + \epsilon^2 \beta_{ij}^{(2)}(\mathbf{x}, \mathbf{y}, T^{(0)}) + \mathrm{O}(\epsilon^3).
		\end{aligned}
	\end{equation}
	\begin{equation}
		\label{eq:2.13}
		\begin{aligned}
			Q_{hyd}^\epsilon(\mathbf{x},T^\epsilon)
			& = Q_{hyd}^{(0)}(\mathbf{y},T^{(0)}) + \epsilon Q_{hyd}^{(1)}(\mathbf{x}, \mathbf{y}, T^{(0)}) + \epsilon^2 Q_{hyd}^{(2)}(\mathbf{x}, \mathbf{y}, T^{(0)}) + \mathrm{O}(\epsilon^3).
		\end{aligned}
	\end{equation}
	\begin{equation}
		\label{eq:2.14}
		\begin{aligned}
			S_{hyd}^\epsilon(\mathbf{x},T^\epsilon)
			& = S_{hyd}^{(0)}(\mathbf{y},T^{(0)}) + \epsilon S_{hyd}^{(1)}(\mathbf{x}, \mathbf{y}, T^{(0)}) + \epsilon^2 S_{hyd}^{(2)}(\mathbf{x}, \mathbf{y}, T^{(0)}) + \mathrm{O}(\epsilon^3).
		\end{aligned}
	\end{equation}
	
	Then substituting \eqref{eq:2.3}-\eqref{eq:2.14} into the multi-scale problem \eqref{eq:2.1} and expanding the derivatives using the chain rule \eqref{eq:2.2}, we obtain a series of equations by equating terms of the same order in the small periodic parameter $\epsilon$.
	\begin{equation}
		\label{eq:2.16}
		O(\epsilon^{-2}):
		\begin{cases}
			\begin{aligned}
				& \frac{\partial }{\partial y_i}\Bigl( k_{ij}^{(0)}\frac{\partial T^{(0)}}{\partial y_j} \Bigr) = 0,\\
				& \frac{\partial }{\partial y_i}\Bigl( g_{ij}^{(0)}\frac{\partial \omega^{(0)}}{\partial y_j} \Bigr) = 0, \\
				& \frac{\partial }{\partial y_j}\Bigl( C_{ijkl}^{(0)}\frac{\partial u_k^{(0)}}{\partial y_l} \Bigr) = 0.
			\end{aligned}
		\end{cases}
	\end{equation}
	\begin{equation}
		\label{eq:2.17}
		O(\epsilon^{-1}):
		\begin{cases}
			\begin{aligned}
				&\frac{\partial }{\partial x_i}\Bigl( k_{ij}^{(0)}\frac{\partial T^{(0)}}{\partial y_j} \Bigr) + \frac{\partial }{\partial y_i}\Bigl[ k_{ij}^{(0)}(\mathbf{y}, T^{(0)})\Bigl( \frac{\partial T^{(0)}}{\partial x_j} + \frac{\partial T^{(1)}}{\partial y_j} \Bigr)\Bigl]+ \frac{\partial }{\partial y_i}\Bigl( k_{ij}^{(1)}\frac{\partial T^{(0)}}{\partial y_j} \Bigr) = 0, \\
				&\frac{\partial }{\partial x_i}\Bigl( g_{ij}^{(0)}\frac{\partial \omega^{(0)}}{\partial y_j} \Bigr) + \frac{\partial }{\partial y_i}\Bigl[ g_{ij}^{(0)}\Bigl( \frac{\partial \omega^{(0)}}{\partial x_j} + \frac{\partial \omega^{(1)}}{\partial y_j} \Bigr)\Bigl] + \frac{\partial }{\partial y_i}\Bigl( g_{ij}^{(1)}\frac{\partial \omega^{(0)}}{\partial y_j} \Bigr) = 0, \\
				&\frac{\partial }{\partial x_j}\Bigl( C_{ijkl}^{(0)}\frac{\partial u_k^{(0)}}{\partial y_l} \Bigr) + \frac{\partial }{\partial y_j}\Bigl[ C_{ijkl}^{(0)}\Bigl( \frac{\partial u_k^{(0)}}{\partial x_l} + \frac{\partial u_k^{(1)}}{\partial y_l} \Bigr)\Bigr] + \frac{\partial }{\partial y_j}\Bigl( C_{ijkl}^{(1)}\frac{\partial u_k^{(0)}}{\partial y_l} \Bigr)\\
				& - \frac{\partial }{\partial y_j}\Bigl( \alpha_{ij}^{(0)}\bigl(T^{(0)} - \tilde{T}\bigr)\Bigr) - \frac{\partial }{\partial y_j}\Bigl( \beta_{ij}^{(0)}\bigl(\omega^{(0)} - \tilde{\omega}\bigr) \Bigr) = 0.
			\end{aligned}
		\end{cases}
	\end{equation}
	\begin{equation}
		\label{eq:2.18}
		O(\epsilon^0):
		\begin{cases}
			\begin{aligned}
				&\rho^{(0)} c^{(0)} \frac{\partial T^{(0)}}{\partial t} = \frac{\partial }{\partial x_i}\Bigl[ k_{ij}^{(0)}\Bigl( \frac{\partial T^{(0)}}{\partial x_j} + \frac{\partial T^{(1)}}{\partial y_j} \Bigr)\Bigr] + \frac{\partial }{\partial x_i}\Bigl( k_{ij}^{(1)}\frac{\partial T^{(0)}}{\partial y_j} \Bigr) + h + Q_{hyd}^{(0)} \\
				&+ \frac{\partial }{\partial y_i}\Bigl[ k_{ij}^{(0)}\Bigl( \frac{\partial T^{(1)}}{\partial x_j} + \frac{\partial T^{(2)}}{\partial y_j} \Bigr)\Bigl] + \frac{\partial }{\partial y_i}\Bigl[ k_{ij}^{(1)}\Bigl( \frac{\partial T^{(0)}}{\partial x_j} + \frac{\partial T^{(1)}}{\partial y_j} \Bigr) + k_{ij}^{(2)}\frac{\partial T^{(0)}}{\partial y_j} \Bigr], \\
				&\frac{\partial \omega^{(0)}}{\partial t} = \frac{\partial }{\partial x_i}\Bigl[ g_{ij}^{(0)}\Bigl( \frac{\partial \omega^{(0)}}{\partial x_j} + \frac{\partial \omega^{(1)}}{\partial y_j} \Bigr) \Bigr] +  \frac{\partial }{\partial y_i}\Bigl[ g_{ij}^{(0)}\Bigl( \frac{\partial \omega^{(1)}}{\partial x_j} + \frac{\partial \omega^{(2)}}{\partial y_j} \Bigr)\Bigr] \\
				&+  \frac{\partial }{\partial x_i} \Bigl( g_{ij}^{(1)}\frac{\partial \omega^{(0)}}{\partial y_j} \Bigr) + \frac{\partial }{\partial y_i}\Bigl[ g_{ij}^{(1)}\Bigl( \frac{\partial \omega^{(0)}}{\partial x_j} + \frac{\partial \omega^{(1)}}{\partial y_j} \Bigr)  + g_{ij}^{(2)}\frac{\partial \omega^{(0)}}{\partial y_j} \Bigr] + m - S_{hyd}^{(0)}, \\
				&\frac{\partial }{\partial x_j}\Bigl[ C_{ijkl}^{(0)}\Bigl( \frac{\partial u_k^{(0)}}{\partial x_l} + \frac{\partial u_k^{(1)}}{\partial y_l} \Bigr)+ C_{ijkl}^{(1)}\frac{\partial u_k^{(0)}}{\partial y_l} \Bigr] \!-\! \frac{\partial }{\partial x_j}\bigl[ \alpha_{ij}^{(0)}\bigl( T^{(0)} \!-\! \tilde{T} \bigr) \!+\! \beta_{ij}^{(0)}\bigl(\omega^{(0)} \!-\! \tilde{\omega}\bigr) \bigr] \\
				&+ \frac{\partial }{\partial y_j}\Bigl[ C_{ijkl}^{(0)}\Bigl( \frac{\partial u_k^{(1)}}{\partial x_l} + \frac{\partial u_k^{(2)}}{\partial y_l} \Bigr)\Bigr] \!+\! \frac{\partial }{\partial y_j}\Bigl[ C_{ijkl}^{(1)}\Bigl( \frac{\partial u_k^{(0)}}{\partial x_l} + \frac{\partial u_k^{(1)}}{\partial y_l} \Bigr)\Bigr] \!+\! \frac{\partial }{\partial y_j}\Bigl( C_{ijkl}^{(2)}\frac{\partial u_k^{(0)}}{\partial y_l} \Bigr) \\
				&- \frac{\partial }{\partial y_j}\bigl[\! \alpha_{ij}^{(0)} T^{(1)} + \alpha_{ij}^{(1)}\bigl( T^{(0)} - \tilde{T}\bigr)\bigl] - \frac{\partial }{\partial y_j}\bigl[ \beta_{ij}^{(0)} \omega^{(1)} + \beta_{ij}^{(1)}\bigl(\omega^{(0)} - \tilde{\omega}\bigr) \bigr] + f_i = 0.
			\end{aligned}
		\end{cases}
	\end{equation}
	
	The $O(\epsilon^{-2})$-order equations \eqref{eq:2.16} imply that $T^{(0)}$, $\omega^{(0)}$, and $u_k^{(0)}$ do not depend on the microscopic variable $\mathbf{y}$, namely
	\begin{equation}
		\label{eq:2.19}
		T^{(0)}(\mathbf{x},\mathbf{y},t) = T^{(0)}(\mathbf{x},t), \quad \omega^{(0)}(\mathbf{x},\mathbf{y},t) = \omega^{(0)}(\mathbf{x},t), \quad u_k^{(0)}(\mathbf{x},\mathbf{y},t) = u_k^{(0)}(\mathbf{x},t).
	\end{equation}
	
	Then, using \eqref{eq:2.19}, the terms $\frac{\partial T^{(0)}}{\partial y_j}$, $\frac{\partial \omega^{(0)}}{\partial y_j}$ and $\frac{\partial u_k^{(0)}}{\partial y_l}$ are all equal to zero. Substituting these into the $O(\epsilon^{-1})$-order equations \eqref{eq:2.17},  the first-order correction terms $T^{(1)}$, $\omega^{(1)}$, and $u_i^{(1)}$ can be expressed as follows.
	\begin{equation}
		\label{eq:2.21}
		\begin{cases}
			\begin{aligned}
				T^{(1)}(\mathbf{x},\mathbf{y},t) &= \mathcal{H}_{\alpha_1}(\mathbf{y}, T^{(0)}) \frac{\partial T^{(0)}(\mathbf{x},t)}{\partial x_{\alpha_1}}, \\
				\omega^{(1)}(\mathbf{x},\mathbf{y},t) &= \mathcal{J}_{\alpha_1}(\mathbf{y}, \omega^{(0)}) \frac{\partial \omega^{(0)}(\mathbf{x},t)}{\partial x_{\alpha_1}}, \\
				u_i^{(1)}(\mathbf{x},\mathbf{y},t) &= \mathcal{X}_{ih}^{\alpha_1}(\mathbf{y}, T^{(0)}) \frac{\partial u_h^{(0)}(\mathbf{x},t)}{\partial x_{\alpha_1}} - \mathcal{M}_i(\mathbf{y}, T^{(0)})\bigl(T^{(0)}(\mathbf{x},t) - \tilde{T}\bigr) \\
				& - \mathcal{N}_i(\mathbf{y}, T^{(0)})\bigl(\omega^{(0)}(\mathbf{x},t) - \tilde{\omega}\bigr).
			\end{aligned}
		\end{cases}
	\end{equation}
	where $\mathcal{H}_{\alpha_1}$, $\mathcal{J}_{\alpha_1}$, $\mathcal{X}_{ih}^{\alpha_1}$, $\mathcal{M}_i$ and $\mathcal{N}_i$ are $1$-periodic functions on the PUC $Y$, known as the first-order cell functions.
	
	After combining \eqref{eq:2.19} and \eqref{eq:2.21} with the $O(\epsilon^{-1})$-order equations \eqref{eq:2.17}, simplification and calculation yield the following equations subject to homogeneous Dirichlet boundary conditions, which are referred to as the first-order cell problems.
	\begin{equation}
		\label{eq:2.22}
		\begin{cases}
			\begin{aligned}
				& \frac{\partial }{\partial y_i}\Bigl( k_{ij}^{(0)}\frac{\partial \mathcal{H}_{\alpha_1}}{\partial y_j} \Bigr) = -\frac{\partial k_{i\alpha_1}^{(0)}}{\partial y_i}, \quad \mathbf{y} \in Y, \\
				& \mathcal{H}_{\alpha_1}(\mathbf{y}, T^{(0)}) = 0, \quad \mathbf{y} \in \partial Y.
			\end{aligned}
		\end{cases}
	\end{equation}
	\begin{equation}
		\label{eq:2.23}
		\begin{cases}
			\begin{aligned}
				& \frac{\partial }{\partial y_i}\Bigl( g_{ij}^{(0)}\frac{\partial \mathcal{J}_{\alpha_1}}{\partial y_j} \Bigr) = -\frac{\partial g_{i\alpha_1}^{(0)}}{\partial y_i}, \quad \mathbf{y} \in Y, \\
				& \mathcal{J}_{\alpha_1}(\mathbf{y}, \omega^{(0)}) = 0, \quad \mathbf{y} \in \partial Y.
			\end{aligned}
		\end{cases}
	\end{equation}
	\begin{equation}
		\label{eq:2.24}
		\begin{cases}
			\begin{aligned}
				& \frac{\partial }{\partial y_j}\Bigl( C_{ijkl}^{(0)}\frac{\partial \mathcal{X}_{kh}^{\alpha_1}}{\partial y_l} \Bigr) = -\frac{\partial C_{ijh\alpha_1}^{(0)}}{\partial y_j}, \quad \mathbf{y} \in Y, \\
				& \mathcal{X}_{kh}^{\alpha_1}(\mathbf{y}, T^{(0)}) = 0, \quad \mathbf{y} \in \partial Y.
			\end{aligned}
		\end{cases}
	\end{equation}
	\begin{equation}
		\label{eq:2.25}
		\begin{cases}
			\begin{aligned}
				& \frac{\partial }{\partial y_j}\Bigl( C_{ijkl}^{(0)}\frac{\partial \mathcal{M}_k}{\partial y_l} \Bigr) = -\frac{\partial \alpha_{ij}^{(0)}}{\partial y_j}, \quad \mathbf{y} \in Y, \\
				& \mathcal{M}_k(\mathbf{y}, T^{(0)}) = 0, \quad \mathbf{y} \in \partial Y.
			\end{aligned}
		\end{cases}
	\end{equation}
	\begin{equation}
		\label{eq:2.26}
		\begin{cases}
			\begin{aligned}
				& \frac{\partial }{\partial y_j}\Bigl( C_{ijkl}^{(0)}\frac{\partial \mathcal{N}_k}{\partial y_l} \Bigr) = -\frac{\partial \beta_{ij}^{(0)}}{\partial y_j}, \quad \mathbf{y} \in Y, \\
				& \mathcal{N}_k(\mathbf{y}, T^{(0)}) = 0, \quad \mathbf{y} \in \partial Y.
			\end{aligned}
		\end{cases}
	\end{equation}
	
	\begin{rmk}
		It should be noted that, unlike the linear periodic setting, the first-order cell functions exhibit quasi-periodicity and explicitly depend on $T^{(0)}$ or $\omega^{(0)}$ as varying parameters.
	\end{rmk}
	
	Subsequently, integrating both sides of the $O(\epsilon^{0})$-order equations \eqref{eq:2.18} over the PUC $Y$ and applying the Gauss theorem to \eqref{eq:2.18} yields the macroscopic homogenized equations associated with the multi-scale problem \eqref{eq:2.1}, given below.
	\begin{equation}
		\label{eq:2.27}
		\begin{cases}
			\begin{aligned}
				& \hat{S}(T^{(0)})\frac{\partial T^{(0)}}{\partial t} - \frac{\partial }{\partial x_i}\Bigl( \hat{k}_{ij}(T^{(0)})\frac{\partial T^{(0)}}{\partial x_j} \Bigr) = h + \hat{Q}_{hyd}(T^{(0)}), \text{in } \Omega \times (0, T^*), \\
				& \frac{\partial \omega^{(0)}}{\partial t} - \frac{\partial }{\partial x_i}\Bigl( \hat{g}_{ij}(\omega^{(0)})\frac{\partial \omega^{(0)}}{\partial x_j} \Bigr) = m - \hat{S}_{hyd}(T^{(0)}), \text{in } \Omega \times (0, T^*), \\
				& \!-\! \frac{\partial }{\partial x_j}\Bigl( \hat{C}_{ijkl}(T^{(0)})\frac{\partial u_k^{(0)}}{\partial x_l} \!-\! \hat{\alpha}_{ij}(T^{(0)})\bigl(T^{(0)} \!-\! \tilde{T}\bigr) \!-\! \hat{\beta}_{ij}(T^{(0)})\bigl(\omega^{(0)} \!-\! \tilde{\omega}\bigr) \Bigr) \!=\! f_i, \text{in } \Omega \times (0, T^*), \\
				& T^{(0)}(\mathbf{x},t) = \hat{T}(\mathbf{x},t), \text{on } \partial \Omega_T \times (0, T^*), \\
				& \hat{k}_{ij}(T^{(0)})\frac{\partial T^{(0)}}{\partial x_j}n_i = \bar{q}(\mathbf{x},t), \text{on } \partial \Omega_q \times (0, T^*), \\
				& \omega^{(0)}(\mathbf{x},t) = \hat{\omega}(\mathbf{x},t),  \text{on } \partial \Omega_\omega \times (0, T^*), \\
				& \hat{g}_{ij}(\omega^{(0)})\frac{\partial \omega^{(0)}}{\partial x_j}n_i = \bar{d}(\mathbf{x},t), \text{on } \partial \Omega_d \times (0, T^*), \\
				& \bm{u}^{(0)}(\mathbf{x},t) = \bm{\hat{u}}(\mathbf{x},t), \text{on } \partial \Omega_u \times (0, T^*), \\
				&\! \!\Bigl(\!\! \hat{C}_{ijkl}(T^{(0)}\!)\frac{\partial u_k^{(0)}}{\partial x_l} \!-\! \hat{\alpha}_{ij}(T^{(0)}\!)\!\bigl(T^{(0)} \!-\! \tilde{T}\bigr)\! \!-\! \hat{\beta}_{ij}(T^{(0)}\!)\!\bigl(\omega^{(0)} \!-\! \tilde{\omega}\bigr)\! \!\!\Bigr)\! n_j \!=\! \bar{\sigma}_i(\mathbf{x},t), \text{on } \partial \Omega_\sigma \!\times \! (0, T^*\!), \\
				& T^{(0)}(\mathbf{x},0) = \tilde{T},  \omega^{(0)}(\mathbf{x},0) = \tilde{\omega},  \bm{u}^{(0)}(\mathbf{x},0) = \bm{\tilde{u}},  \text{in } \Omega.
			\end{aligned}
		\end{cases}
	\end{equation}
	The macroscopic homogenized material parameters and homogenized internal source terms in \eqref{eq:2.27} are defined as follows.
	\begin{equation}
		\label{eq:2.28}
		\begin{aligned}
			& \hat{S}(T^{(0)}) = \frac{1}{|Y|}\int_Y \rho^{(0)}c^{(0)} dY,\\
			& \hat{k}_{ij}(T^{(0)}) = \frac{1}{|Y|}\int_Y \Bigl( k_{ij}^{(0)} + k_{ik}^{(0)}\frac{\partial \mathcal{H}_j}{\partial y_k} \Bigr) dY,\quad \hat{Q}_{hyd}(T^{(0)}) = \frac{1}{|Y|}\int_Y Q_{hyd}^{(0)} dY, \\
			& \hat{g}_{ij}(\omega^{(0)}) = \frac{1}{|Y|}\int_Y \Bigl( g_{ij}^{(0)} + g_{ik}^{(0)}\frac{\partial \mathcal{J}_j}{\partial y_k} \Bigr) dY,\quad \hat{S}_{hyd}(T^{(0)}) = \frac{1}{|Y|}\int_Y S_{hyd}^{(0)} dY, \\
			& \hat{C}_{ijkl}(T^{(0)}) = \frac{1}{|Y|}\int_Y \Bigl( C_{ijkl}^{(0)} + C_{ijmn}^{(0)}\frac{\partial \mathcal{X}_{mk}^l}{\partial y_n} \Bigr) dY, \\
			& \hat{\alpha}_{ij}(T^{(0)}) = \frac{1}{|Y|}\int_Y \Bigl( \alpha_{ij}^{(0)} + C_{ijkl}^{(0)}\frac{\partial \mathcal{M}_k}{\partial y_l} \Bigr) dY, \\
			& \hat{\beta}_{ij}(T^{(0)}) = \frac{1}{|Y|}\int_Y \Bigl( \beta_{ij}^{(0)} + C_{ijkl}^{(0)}\frac{\partial \mathcal{N}_k}{\partial y_l} \Bigr) dY.
		\end{aligned}
	\end{equation}
	
	\begin{rmk}
		All macroscopic homogenized material parameters and homogenized internal source terms depend on the macroscopic homogenized solution $T^{(0)}$ or $\omega^{(0)}$. This dependence has two origins: (1) the temperature- or moisture-dependent material parameters and the temperature-dependent internal source terms in the integrands; (2) the implicit dependence of the first-order cell functions on $T^{(0)}$ or $\omega^{(0)}$. This differs significantly from the case of linear periodic composites.
	\end{rmk}
	
	\begin{rmk}
		Stemming from the approach as outlined in \citep{R50, R60}, it can be proved that $\underline{\varsigma} | \bm{\xi} |^2\leq \hat{k}_{ij}(T^{(0)}) \xi _i\xi_j \leq\overline{\varsigma} | \bm{\xi} |^2$, $\underline{\varsigma} | \bm{\xi} |^2\leq \hat{g}_{ij}(\omega^{(0)}) \xi _i\xi_j \leq\overline{\varsigma} | \bm{\xi} |^2$, $\underline{\varsigma} \eta_{ij} \eta_{ij} \leq \hat{C}_{ijkl} (T^{(0)}) \eta_{ij} \eta_{kl} \leq \overline{\varsigma} \eta_{ij} \eta_{ij}$, $\underline{\varsigma} | \bm{\xi} |^2\leq \hat{\alpha}_{ij}(T^{(0)}) \xi _i\xi_j \leq\overline{\varsigma} | \bm{\xi} |^2$ and $\underline{\varsigma} | \bm{\xi} |^2\leq \hat{\beta}_{ij}(T^{(0)}) \xi _i\xi_j \leq\overline{\varsigma} | \bm{\xi} |^2$, where $\underline{\varsigma}$ and $\overline{\varsigma}$ are two positive constants independent of $\epsilon$.
	\end{rmk}
	
	We now derive the crucial second-order correctors $T^{(2)}$, $\omega^{(2)}$ and $u_i^{(2)}$. By substituting \eqref{eq:2.19} and \eqref{eq:2.21} into \eqref{eq:2.18}, subtracting \eqref{eq:2.18} from \eqref{eq:2.27}, and simplifying, we obtain the following equations.
	\begin{equation}
		\label{eq:2.29a}
		\begin{aligned}
			& \frac{\partial }{\partial y_i}\Bigl( k_{ij}^{(0)}\frac{\partial T^{(2)}}{\partial y_j} \Bigr) = \Bigl[ \frac{\partial \hat{k}_{i\alpha_1}}{\partial x_i} - \frac{\partial k_{i\alpha_1}^{(0)}}{\partial x_i} - \frac{\partial }{\partial x_i}\Bigl( k_{ij}^{(0)}\frac{\partial \mathcal{H}_{\alpha_1}}{\partial y_j} \Bigr) - \frac{\partial }{\partial y_i}\Bigl( k_{ij}^{(0)}\frac{\partial \mathcal{H}_{\alpha_1}}{\partial x_j} \Bigr) \Bigr] \frac{\partial T^{(0)}}{\partial x_{\alpha_1}}\\
			& + \bigl( \rho^{(0)}c^{(0)} - \hat{S} \bigr)\frac{\partial T^{(0)}}{\partial t} + \Bigl[ \hat{k}_{\alpha_1 \alpha_2} - k_{\alpha_1 \alpha_2}^{(0)} - k_{\alpha_1 j}^{(0)}\frac{\partial \mathcal{H}_{\alpha_2}}{\partial y_j} - \frac{\partial }{\partial y_i}\bigl( k_{i\alpha_1}^{(0)}\mathcal{H}_{\alpha_2} \bigr) \Bigr]\frac{\partial^2 T^{(0)}}{\partial x_{\alpha_1} \partial x_{\alpha_2}} \\
			& - \frac{\partial }{\partial y_i}\Bigl[ \mathcal{H}_{\alpha_1}\mathbf{D}^{(0,1)}k_{i\alpha_2}^{(0)} + \mathcal{H}_{\alpha_1}\mathbf{D}^{(0,1)}k_{ij}^{(0)}\frac{\partial \mathcal{H}_{\alpha_2}}{\partial y_j} \Bigr]\frac{\partial T^{(0)}}{\partial x_{\alpha_1}}\frac{\partial T^{(0)}}{\partial x_{\alpha_2}} + \hat{Q}_{hyd} - Q_{hyd}^{(0)}.
		\end{aligned}
	\end{equation}
	\begin{equation}
		\label{eq:2.29b}
		\begin{aligned}
			& \frac{\partial }{\partial y_i}\Bigl( g_{ij}^{(0)}\frac{\partial \omega^{(2)}}{\partial y_j} \Bigr) = \Bigl[ \frac{\partial \hat{g}_{i\alpha_1}}{\partial x_i} - \frac{\partial g_{i\alpha_1}^{(0)}}{\partial x_i} - \frac{\partial }{\partial x_i}\Bigl( g_{ij}^{(0)}\frac{\partial \mathcal{J}_{\alpha_1}}{\partial y_j} \Bigr) - \frac{\partial }{\partial y_i}\Bigl( g_{ij}^{(0)}\frac{\partial \mathcal{J}_{\alpha_1}}{\partial x_j} \Bigr) \Bigr]\frac{\partial \omega^{(0)}}{\partial x_{\alpha_1}} \\
			& + \Bigl[ \hat{g}_{\alpha_1 \alpha_2} - g_{\alpha_1 \alpha_2}^{(0)} - g_{\alpha_1 j}^{(0)}\frac{\partial \mathcal{J}_{\alpha_2}}{\partial y_j} - \frac{\partial }{\partial y_i}\bigl( g_{i\alpha_1}^{(0)}\mathcal{J}_{\alpha_2} \bigr) \Bigr]\frac{\partial^2 \omega^{(0)}}{\partial x_{\alpha_1} \partial x_{\alpha_2}} \\
			& - \frac{\partial }{\partial y_i}\Bigl[ \mathcal{J}_{\alpha_1}\mathbf{D}^{(0,1)}g_{i\alpha_2}^{(0)} + \mathcal{J}_{\alpha_1}\mathbf{D}^{(0,1)}g_{ij}^{(0)}\frac{\partial \mathcal{J}_{\alpha_2}}{\partial y_j} \Bigr]\frac{\partial \omega^{(0)}}{\partial x_{\alpha_1}}\frac{\partial \omega^{(0)}}{\partial x_{\alpha_2}} + S_{hyd}^{(0)} - \hat{S}_{hyd}.
		\end{aligned}
	\end{equation}
	\begin{equation}
		\label{eq:2.29c}
		\begin{aligned}
			& \frac{\partial }{\partial y_j}\Bigl( C_{ijkl}^{(0)}\frac{\partial u_k^{(2)}}{\partial y_l} \Bigr) = \Bigl[ \frac{\partial \hat{C}_{ijm\alpha_1}}{\partial x_j} - \frac{\partial C_{ijm\alpha_1}^{(0)}}{\partial x_j} - \frac{\partial }{\partial x_j}\Bigl( C_{ijkl}^{(0)}\frac{\partial \mathcal{X}_{km}^{\alpha_1}}{\partial y_l} \Bigr) - \frac{\partial }{\partial y_j}\Bigl( C_{ijkl}^{(0)}\frac{\partial \mathcal{X}_{km}^{\alpha_1}}{\partial x_l} \Bigr) \Bigr]\frac{\partial u_m^{(0)}}{\partial x_{\alpha_1}}\\
			& + \Bigl[ \hat{C}_{i\alpha_1 m\alpha_2} - C_{i\alpha_1 m\alpha_2}^{(0)} - C_{i\alpha_1 kl}^{(0)}\frac{\partial \mathcal{X}_{km}^{\alpha_2}}{\partial y_l} - \frac{\partial }{\partial y_j}\bigl( C_{ijk\alpha_1}^{(0)}\mathcal{X}_{km}^{\alpha_2} \bigr) \Bigr]\frac{\partial^2 u_m^{(0)}}{\partial x_{\alpha_1} \partial x_{\alpha_2}} \\
			& - \frac{\partial }{\partial y_j}\Bigl[ \mathcal{H}_{\alpha_1}\mathbf{D}^{(0,1)}C_{ijm\alpha_2}^{(0)} + \mathcal{H}_{\alpha_1}\mathbf{D}^{(0,1)}C_{ijkl}^{(0)}\frac{\partial \mathcal{X}_{km}^{\alpha_2}}{\partial y_l} \Bigr]\frac{\partial T^{(0)}}{\partial x_{\alpha_1}}\frac{\partial u_m^{(0)}}{\partial x_{\alpha_2}} \\
			& - \Bigl[ \frac{\partial \hat{\alpha}_{ij}}{\partial x_j} - \frac{\partial \alpha_{ij}^{(0)}}{\partial x_j} - \frac{\partial }{\partial x_j}\Bigl( C_{ijkl}^{(0)}\frac{\partial \mathcal{M}_k}{\partial y_l} \Bigr) - \frac{\partial }{\partial y_j}\Bigl( C_{ijkl}^{(0)}\frac{\partial \mathcal{M}_k}{\partial x_l} \Bigr) \Bigr]\bigl( T^{(0)} - \tilde{T} \bigr) \\
			& - \Bigl[ \hat{\alpha}_{i\alpha_1} - \alpha_{i\alpha_1}^{(0)} - C_{i\alpha_1 kl}^{(0)}\frac{\partial \mathcal{M}_k}{\partial y_l} - \frac{\partial }{\partial y_j}\bigl( \alpha_{ij}^{(0)}\mathcal{H}_{\alpha_1} \bigr) - \frac{\partial }{\partial y_j}\bigl( C_{ijk\alpha_1}^{(0)}\mathcal{M}_k \bigr) \Bigr]\frac{\partial T^{(0)}}{\partial x_{\alpha_1}} \\
			& + \frac{\partial }{\partial y_j}\Bigl[ \mathcal{H}_{\alpha_1}\mathbf{D}^{(0,1)}C_{ijkl}^{(0)}\frac{\partial \mathcal{M}_k}{\partial y_l} + \mathcal{H}_{\alpha_1}\mathbf{D}^{(0,1)}\alpha_{ij}^{(0)} \Bigr]\frac{\partial T^{(0)}}{\partial x_{\alpha_1}}\bigl( T^{(0)} - \tilde{T} \bigr) \\
			& + \frac{\partial }{\partial y_j}\Bigl[ \mathcal{H}_{\alpha_1}\mathbf{D}^{(0,1)}C_{ijkl}^{(0)}\frac{\partial \mathcal{N}_k}{\partial y_l} + \mathcal{H}_{\alpha_1}\mathbf{D}^{(0,1)}\beta_{ij}^{(0)} \Bigr]\frac{\partial T^{(0)}}{\partial x_{\alpha_1}}\bigl( \omega^{(0)} - \tilde{\omega} \bigr) \\
			& - \Bigl[ \frac{\partial \hat{\beta}_{ij}}{\partial x_j} - \frac{\partial \beta_{ij}^{(0)}}{\partial x_j} - \frac{\partial }{\partial x_j}\Bigl( C_{ijkl}^{(0)}\frac{\partial \mathcal{N}_k}{\partial y_l} \Bigr) - \frac{\partial }{\partial y_j}\Bigl( C_{ijkl}^{(0)}\frac{\partial \mathcal{N}_k}{\partial x_l} \Bigr) \Bigr]\bigl( \omega^{(0)} - \tilde{\omega} \bigr) \\
			& - \Bigl[ \hat{\beta}_{i\alpha_1} - \beta_{i\alpha_1}^{(0)} - C_{i\alpha_1 kl}^{(0)}\frac{\partial \mathcal{N}_k}{\partial y_l} - \frac{\partial }{\partial y_j}\bigl( \beta_{ij}^{(0)}\mathcal{J}_{\alpha_1} \bigr) - \frac{\partial }{\partial y_j}\bigl( C_{ijk\alpha_1}^{(0)}\mathcal{N}_k \bigr) \Bigl]\frac{\partial \omega^{(0)}}{\partial x_{\alpha_1}}.
		\end{aligned}
	\end{equation}
	
	Using \eqref{eq:2.29a}-\eqref{eq:2.29c}, we obtain the explicit expressions for the second-order correctors $T^{(2)}$, $\omega^{(2)}$ and $u_i^{(2)}$ as
	\begin{equation}
		\label{eq:2.30}
		\begin{cases}
			\begin{aligned}
				& T^{(2)}(\mathbf{x}, \mathbf{y}, t) = \mathcal{S}(\mathbf{y}, T^{(0)})\frac{\partial T^{(0)}(\mathbf{x}, t)}{\partial t} + \mathcal{H}_{\alpha_1 \alpha_2}(\mathbf{y}, T^{(0)})\frac{\partial^2 T^{(0)}(\mathbf{x}, t)}{\partial x_{\alpha_1} \partial x_{\alpha_2}} \\
				& + \mathcal{R}_{\alpha_1}(\mathbf{y}, T^{(0)})\frac{\partial T^{(0)}(\mathbf{x}, t)}{\partial x_{\alpha_1}} - \mathcal{E}_{\alpha_1 \alpha_2}(\mathbf{y}, T^{(0)})\frac{\partial T^{(0)}(\mathbf{x}, t)}{\partial x_{\alpha_1}}\frac{\partial T^{(0)}(\mathbf{x}, t)}{\partial x_{\alpha_2}} + \mathbb{Q} (\mathbf{y}, T^{(0)}), \\
				& \omega^{(2)}(\mathbf{x}, \mathbf{y}, t) = \mathcal{J}_{\alpha_1 \alpha_2}(\mathbf{y}, \omega^{(0)})\frac{\partial^2 \omega^{(0)}(\mathbf{x}, t)}{\partial x_{\alpha_1} \partial x_{\alpha_2}} + \mathcal{I}_{\alpha_1}(\mathbf{y}, \omega^{(0)})\frac{\partial \omega^{(0)}(\mathbf{x}, t)}{\partial x_{\alpha_1}} \\
				& - \mathcal{F}_{\alpha_1\alpha_2}(\mathbf{y}, \omega^{(0)})\frac{\partial \omega^{(0)}(\mathbf{x}, t)}{\partial x_{\alpha_1}}\frac{\partial \omega^{(0)}(\mathbf{x}, t)}{\partial x_{\alpha_2}} -\mathbb{S} (\mathbf{y}, T^{(0)}), \\
				& u_i^{(2)}(\mathbf{x}, \mathbf{y}, t) = \mathcal{X}_{ih}^{\alpha_1 \alpha_2}(\mathbf{y}, T^{(0)})\frac{\partial^2 u_h^{(0)}(\mathbf{x}, t)}{\partial x_{\alpha_1} \partial x_{\alpha_2}} + \mathcal{Q}_{ih}^{\alpha_1}(\mathbf{y}, T^{(0)})\frac{\partial u_h^{(0)}(\mathbf{x}, t)}{\partial x_{\alpha_1}} \\
				& - \mathcal{P}_{ih}^{\alpha_1 \alpha_2}(\mathbf{y}, T^{(0)})\frac{\partial T^{(0)}(\mathbf{x}, t)}{\partial x_{\alpha_1}}\frac{\partial u_h^{(0)}(\mathbf{x}, t)}{\partial x_{\alpha_2}} - \mathcal{W}_i(\mathbf{y}, T^{(0)})(T^{(0)}(\mathbf{x}, t) - \tilde{T}) \\
				& - \mathcal{Z}_i^{\alpha_1}(\mathbf{y}, T^{(0)})\frac{\partial T^{(0)}(\mathbf{x}, t)}{\partial x_{\alpha_1}} + \mathcal{A}_i^{\alpha_1}(\mathbf{y}, T^{(0)})\frac{\partial T^{(0)}(\mathbf{x}, t)}{\partial x_{\alpha_1}}(T^{(0)}(\mathbf{x}, t) - \tilde{T}) \\
				& - \mathcal{V}_i(\mathbf{y}, T^{(0)})(\omega^{(0)}(\mathbf{x}, t) - \tilde{\omega}) - \mathcal{G}_i^{\alpha_1}(\mathbf{y}, T^{(0)}, \omega^{(0)})\frac{\partial \omega^{(0)}(\mathbf{x}, t)}{\partial x_{\alpha_1}} \\
				& + \mathcal{B}_i^{\alpha_1}(\mathbf{y}, T^{(0)})\frac{\partial T^{(0)}(\mathbf{x}, t)}{\partial x_{\alpha_1}}(\omega^{(0)}(\mathbf{x}, t) - \tilde{\omega}).
			\end{aligned}
		\end{cases}
	\end{equation}
	where $\mathcal{S}$, $\mathcal{H}_{\alpha_1 \alpha_2}$, $\mathcal{R}_{\alpha_1}$, $\mathcal{E}_{\alpha_1 \alpha_2}$, $\mathbb{Q}$, $\mathcal{J}_{\alpha_1 \alpha_2}$, $\mathcal{I}_{\alpha_1}$, $\mathcal{F}_{\alpha_1\alpha_2}$, $\mathbb{S}$, $\mathcal{X}_{ih}^{\alpha_1 \alpha_2}$, $\mathcal{Q}_{ih}^{\alpha_1}$, $\mathcal{P}_{ih}^{\alpha_1 \alpha_2}$, $\mathcal{W}_i$, $\mathcal{Z}_i^{\alpha_1}$, $\mathcal{A}_i^{\alpha_1}$, $\mathcal{V}_i$, $\mathcal{G}_i^{\alpha_1}$ and $\mathcal{B}_i^{\alpha_1}$ are $1$-periodic functions defined in PUC $Y$, which are defined as second-order auxiliary cell functions.
	
	Substituting \eqref{eq:2.30} into \eqref{eq:2.29a}-\eqref{eq:2.29c}, we derive a series of equations subject to the homogeneous Dirichlet boundary conditions, as follows.
	\begin{equation}
		\label{eq:2.31}
		\begin{cases}
			\begin{aligned}
				& \frac{\partial }{\partial y_i}\Bigl( k_{ij}^{(0)}\frac{\partial \mathcal{S}}{\partial y_j} \Bigr) = \rho^{(0)}c^{(0)} - \hat{S}, \quad \mathbf{y} \in Y, \\
				& \mathcal{S}(\mathbf{y}, T^{(0)}) = 0, \quad \mathbf{y} \in \partial Y.
			\end{aligned}
		\end{cases}
	\end{equation}
	\begin{equation}
		\label{eq:2.32}
		\begin{cases}
			\begin{aligned}
				& \frac{\partial }{\partial y_i}\Bigl( k_{ij}^{(0)}\frac{\partial \mathcal{H}_{\alpha_1 \alpha_2}}{\partial y_j} \Bigr) = \hat{k}_{\alpha_1\alpha_2} - k_{\alpha_1\alpha_2}^{(0)} - k_{\alpha_1 j}^{(0)}\frac{\partial \mathcal{H}_{\alpha_2}}{\partial y_j} - \frac{\partial }{\partial y_i}\bigl( k_{i\alpha_1}^{(0)} \mathcal{H}_{\alpha_2} \bigr), \quad \mathbf{y} \in Y, \\
				& \mathcal{H}_{\alpha_1 \alpha_2}(\mathbf{y}, T^{(0)}) = 0, \quad \mathbf{y} \in \partial Y.
			\end{aligned}
		\end{cases}
	\end{equation}
	\begin{equation}
		\label{eq:2.33}
		\begin{cases}
			\begin{aligned}
				& \frac{\partial }{\partial y_i}\Bigl( k_{ij}^{(0)}\frac{\partial \mathcal{R}_{\alpha_1}}{\partial y_j} \Bigr)= \frac{\partial \hat{k}_{i\alpha_1}}{\partial x_i} - \frac{\partial k_{i\alpha_1}^{(0)}}{\partial x_i} - \frac{\partial }{\partial x_i}\Bigl( k_{ij}^{(0)}\frac{\partial \mathcal{H}_{\alpha_1}}{\partial y_j} \Bigr) - \frac{\partial }{\partial y_i}\Bigl( k_{ij}^{(0)}\frac{\partial \mathcal{H}_{\alpha_1}}{\partial x_j} \Bigr), \quad \mathbf{y} \in Y, \\
				& \mathcal{R}_{\alpha_1}(\mathbf{y}, T^{(0)}) = 0, \quad \mathbf{y} \in \partial Y.
			\end{aligned}
		\end{cases}
	\end{equation}
	\begin{equation}
		\label{eq:2.34}
		\begin{cases}
			\begin{aligned}
				& \frac{\partial }{\partial y_i}\Bigl( k_{ij}^{(0)}\frac{\partial \mathcal{E}_{\alpha_1 \alpha_2}}{\partial y_j} \Bigr) = \frac{\partial }{\partial y_i}\Bigl( \mathcal{H}_{\alpha_1}\mathbf{D}^{(0,1)}k_{i\alpha_2}^{(0)} + \mathcal{H}_{\alpha_1}\mathbf{D}^{(0,1)}k_{ij}^{(0)}\frac{\partial \mathcal{H}_{\alpha_2}}{\partial y_j} \Bigr), \quad \mathbf{y} \in Y, \\
				& \mathcal{E}_{\alpha_1 \alpha_2}(\mathbf{y}, T^{(0)}) = 0, \quad \mathbf{y} \in \partial Y.
			\end{aligned}
		\end{cases}
	\end{equation}
	\begin{equation}
		\label{eq:2.35}
		\begin{cases}
			\begin{aligned}
				& \frac{\partial }{\partial y_i}\Bigl( k_{ij}^{(0)}\frac{\partial \mathbb{Q}}{\partial y_j} \Bigr) = \hat{Q}_{hyd} - Q_{hyd}^{(0)}, \quad \mathbf{y} \in Y, \\
				& \mathbb{Q}(\mathbf{y}, T^{(0)}) = 0, \quad \mathbf{y} \in \partial Y.
			\end{aligned}
		\end{cases}
	\end{equation}
	\begin{equation}
		\label{eq:2.36}
		\begin{cases}
			\begin{aligned}
				& \frac{\partial }{\partial y_i}\Bigl( g_{ij}^{(0)}\frac{\partial \mathcal{J}_{\alpha_1 \alpha_2}}{\partial y_j} \Bigr) = \hat{g}_{\alpha_1\alpha_2} - g_{\alpha_1\alpha_2}^{(0)} - g_{\alpha_1 j}^{(0)}\frac{\partial \mathcal{J}_{\alpha_2}}{\partial y_j} - \frac{\partial }{\partial y_i}\bigl( g_{i\alpha_1}^{(0)} \mathcal{J}_{\alpha_2} \bigr), \quad \mathbf{y} \in Y, \\
				& \mathcal{J}_{\alpha_1 \alpha_2}(\mathbf{y}, \omega^{(0)}) = 0, \quad \mathbf{y} \in \partial Y.
			\end{aligned}
		\end{cases}
	\end{equation}
	\begin{equation}
		\label{eq:2.37}
		\begin{cases}
			\begin{aligned}
				& \frac{\partial }{\partial y_i}\Bigl( g_{ij}^{(0)}\frac{\partial \mathcal{I}_{\alpha_1}}{\partial y_j} \Bigr) = \frac{\partial \hat{g}_{i\alpha_1}}{\partial x_i} - \frac{\partial g_{i\alpha_1}^{(0)}}{\partial x_i} - \frac{\partial }{\partial x_i}\Bigl( g_{ij}^{(0)}\frac{\partial \mathcal{J}_{\alpha_1}}{\partial y_j} \Bigr) - \frac{\partial }{\partial y_i}\Bigl( g_{ij}^{(0)}\frac{\partial \mathcal{J}_{\alpha_1}}{\partial x_j} \Bigr), \quad \mathbf{y} \in Y, \\
				& \mathcal{I}_{\alpha_1}(\mathbf{y}, \omega^{(0)}) = 0, \quad \mathbf{y} \in \partial Y.
			\end{aligned}
		\end{cases}
	\end{equation}
	\begin{equation}
		\label{eq:2.38}
		\begin{cases}
			\begin{aligned}
				& \frac{\partial }{\partial y_i}\Bigl( g_{ij}^{(0)}\frac{\partial \mathcal{F}_{\alpha_1\alpha_2}}{\partial y_j} \Bigr) = \frac{\partial }{\partial y_i}\Bigl( \mathcal{J}_{\alpha_1}\mathbf{D}^{(0,1)}g_{i\alpha_2}^{(0)} + \mathcal{J}_{\alpha_1}\mathbf{D}^{(0,1)}g_{ij}^{(0)}\frac{\partial \mathcal{J}_{\alpha_2}}{\partial y_j} \Bigr), \quad \mathbf{y} \in Y, \\
				& \mathcal{F}_{\alpha_1\alpha_2}(\mathbf{y}, \omega^{(0)}) = 0, \quad \mathbf{y} \in \partial Y.
			\end{aligned}
		\end{cases}
	\end{equation}
	\begin{equation}
		\label{eq:2.39}
		\begin{cases}
			\begin{aligned}
				& \frac{\partial }{\partial y_i}\Bigl( g_{ij}^{(0)}\frac{\partial \mathbb{S}}{\partial y_j} \Bigr) = \hat{S}_{hyd} - S_{hyd}^{(0)}, \quad \mathbf{y} \in Y, \\
				& \mathbb{S}(\mathbf{y}, T^{(0)}) = 0, \quad \mathbf{y} \in \partial Y.
			\end{aligned}
		\end{cases}
	\end{equation}
	\begin{equation}
		\label{eq:2.40}
		\begin{cases}
			\begin{aligned}
				& \frac{\partial }{\partial y_j}\Bigl( C_{ijkl}^{(0)}\frac{\partial \mathcal{X}_{km}^{\alpha_1 \alpha_2}}{\partial y_l} \Bigr) = \hat{C}_{i\alpha_1 m\alpha_2} - C_{i\alpha_1 m\alpha_2}^{(0)} - C_{i\alpha_1 kl}^{(0)}\frac{\partial \mathcal{X}_{km}^{\alpha_2}}{\partial y_l} - \frac{\partial }{\partial y_j}\bigl( C_{ijk\alpha_1}^{(0)} \mathcal{X}_{km}^{\alpha_2} \bigr), \quad \mathbf{y} \in Y, \\
				& \mathcal{X}_{km}^{\alpha_1 \alpha_2}(\mathbf{y}, T^{(0)}) = 0, \quad \mathbf{y} \in \partial Y.
			\end{aligned}
		\end{cases}
	\end{equation}
	\begin{equation}
		\label{eq:2.41}
		\begin{cases}
			\begin{aligned}
				& \frac{\partial }{\partial y_j}\Bigl( C_{ijkl}^{(0)}\frac{\partial \mathcal{Q}_{km}^{\alpha_1}}{\partial y_l} \Bigr) \!=\! \frac{\partial \hat{C}_{ijm\alpha_1}}{\partial x_j} \!-\! \frac{\partial C_{ijm\alpha_1}^{(0)}}{\partial x_j} \!-\! \frac{\partial }{\partial x_j}\Bigl( C_{ijkl}^{(0)}\frac{\partial \mathcal{X}_{km}^{\alpha_1}}{\partial y_l} \Bigr) \!-\! \frac{\partial }{\partial y_j}\Bigl( C_{ijkl}^{(0)}\frac{\partial \mathcal{X}_{km}^{\alpha_1}}{\partial x_l} \Bigr), \mathbf{y} \in Y , \\
				& \mathcal{Q}_{km}^{\alpha_1}(\mathbf{y}, T^{(0)}) = 0, \quad \mathbf{y} \in \partial Y.
			\end{aligned}
		\end{cases}
	\end{equation}
	\begin{equation}
		\label{eq:2.42}
		\begin{cases}
			\begin{aligned}
				& \frac{\partial }{\partial y_j}\Bigl( C_{ijkl}^{(0)}\frac{\partial \mathcal{P}_{km}^{\alpha_1 \alpha_2}}{\partial y_l} \Bigr) = \frac{\partial }{\partial y_j}\Bigl( \mathcal{H}_{\alpha_1}\mathbf{D}^{(0,1)}C_{ijm\alpha_2}^{(0)} + \mathcal{H}_{\alpha_1}\mathbf{D}^{(0,1)}C_{ijkl}^{(0)}\frac{\partial \mathcal{X}_{km}^{\alpha_2}}{\partial y_l} \Bigr), \quad \mathbf{y} \in Y, \\
				& \mathcal{P}_{km}^{\alpha_1 \alpha_2}(\mathbf{y}, T^{(0)}) = 0, \quad \mathbf{y} \in \partial Y.
			\end{aligned}
		\end{cases}
	\end{equation}
	\begin{equation}
		\label{eq:2.43}
		\begin{cases}
			\begin{aligned}
				& \frac{\partial }{\partial y_j}\Bigl( C_{ijkl}^{(0)}\frac{\partial \mathcal{W}_k}{\partial y_l} \Bigr) = \frac{\partial \hat{\alpha}_{ij}}{\partial x_j} - \frac{\partial \alpha_{ij}^{(0)}}{\partial x_j} - \frac{\partial }{\partial x_j}\Bigl( C_{ijkl}^{(0)}\frac{\partial \mathcal{M}_k}{\partial y_l} \Bigr) - \frac{\partial }{\partial y_j}\Bigl( C_{ijkl}^{(0)}\frac{\partial \mathcal{M}_k}{\partial x_l} \Bigr), \quad \mathbf{y} \in Y, \\
				& \mathcal{W}_k(\mathbf{y}, T^{(0)}) = 0, \quad \mathbf{y} \in \partial Y.
			\end{aligned}
		\end{cases}
	\end{equation}
	\begin{equation}
		\label{eq:2.44}
		\begin{cases}
			\begin{aligned}
				& \frac{\partial }{\partial y_j}\Bigl( C_{ijkl}^{(0)}\frac{\partial \mathcal{Z}_k^{\alpha_1}}{\partial y_l} \Bigr) \!=\! \hat{\alpha}_{i\alpha_1} \!-\! \alpha_{i\alpha_1}^{(0)} \!-\! C_{i\alpha_1 kl}^{(0)}\frac{\partial \mathcal{M}_k}{\partial y_l} \!-\! \frac{\partial }{\partial y_j}\bigl( \alpha_{ij}^{(0)} \mathcal{H}_{\alpha_1} \bigr) \!-\! \frac{\partial }{\partial y_j}\bigl( C_{ijk\alpha_1}^{(0)} \mathcal{M}_k \bigr), \mathbf{y} \in  Y, \\
				& \mathcal{Z}_k^{\alpha_1}(\mathbf{y}, T^{(0)}) = 0, \quad \mathbf{y} \in \partial Y.
			\end{aligned}
		\end{cases}
	\end{equation}
	\begin{equation}
		\label{eq:2.45}
		\begin{cases}
			\begin{aligned}
				& \frac{\partial }{\partial y_j}\Bigl( C_{ijkl}^{(0)}\frac{\partial \mathcal{A}_k^{\alpha_1}}{\partial y_l} \Bigr) = \frac{\partial }{\partial y_j}\Bigl( \mathcal{H}_{\alpha_1}\mathbf{D}^{(0,1)}C_{ijkl}^{(0)}\frac{\partial \mathcal{M}_k}{\partial y_l} + \mathcal{H}_{\alpha_1}\mathbf{D}^{(0,1)}\alpha_{ij}^{(0)} \Bigr), \quad \mathbf{y} \in Y, \\
				& \mathcal{A}_k^{\alpha_1}(\mathbf{y}, T^{(0)}) = 0, \quad \mathbf{y} \in \partial Y.
			\end{aligned}
		\end{cases}
	\end{equation}
	\begin{equation}
		\label{eq:2.46}
		\begin{cases}
			\begin{aligned}
				& \frac{\partial }{\partial y_j}\Bigl( C_{ijkl}^{(0)}\frac{\partial \mathcal{V}_k}{\partial y_l} \Bigr) = \frac{\partial \hat{\beta}_{ij}}{\partial x_j} - \frac{\partial \beta_{ij}^{(0)}}{\partial x_j} - \frac{\partial }{\partial x_j}\Bigl( C_{ijkl}^{(0)}\frac{\partial \mathcal{N}_k}{\partial y_l} \Bigr) - \frac{\partial }{\partial y_j}\Bigl( C_{ijkl}^{(0)}\frac{\partial \mathcal{N}_k}{\partial x_l} \Bigr), \quad \mathbf{y} \in Y, \\
				& \mathcal{V}_k(\mathbf{y}, T^{(0)}) = 0, \quad \mathbf{y} \in \partial Y.
			\end{aligned}
		\end{cases}
	\end{equation}
	\begin{equation}
		\label{eq:2.47}
		\begin{cases}
			\begin{aligned}
				& \frac{\partial }{\partial y_j}\Bigl( C_{ijkl}^{(0)}\frac{\partial \mathcal{G}_k^{\alpha_1}}{\partial y_l} \Bigr) = \hat{\beta}_{i\alpha_1} - \beta_{i\alpha_1}^{(0)} - C_{i\alpha_1 kl}^{(0)}\frac{\partial \mathcal{N}_k}{\partial y_l} - \frac{\partial }{\partial y_j}\bigl( \beta_{ij}^{(0)} \mathcal{J}_{\alpha_1} \bigr) - \frac{\partial }{\partial y_j}\bigl( C_{ijk\alpha_1}^{(0)} \mathcal{N}_k \bigr), \mathbf{y} \in Y, \\
				& \mathcal{G}_k^{\alpha_1}(\mathbf{y}, T^{(0)}, \omega^{(0)}) = 0, \quad \mathbf{y} \in \partial Y.
			\end{aligned}
		\end{cases}
	\end{equation}
	\begin{equation}
		\label{eq:2.48}
		\begin{cases}
			\begin{aligned}
				& \frac{\partial }{\partial y_j}\Bigl( C_{ijkl}^{(0)}\frac{\partial \mathcal{B}_k^{\alpha_1}}{\partial y_l} \Bigr) = \frac{\partial }{\partial y_j}\Bigl( \mathcal{H}_{\alpha_1}\mathbf{D}^{(0,1)}C_{ijkl}^{(0)}\frac{\partial \mathcal{N}_k}{\partial y_l} + \mathcal{H}_{\alpha_1}\mathbf{D}^{(0,1)}\beta_{ij}^{(0)} \Bigr), \quad \mathbf{y} \in Y, \\
				& \mathcal{B}_k^{\alpha_1}(\mathbf{y}, T^{(0)}) = 0, \quad \mathbf{y} \in \partial Y.
			\end{aligned}
		\end{cases}
	\end{equation}
	
	\begin{rmk}	
		According to Refs. \citep{R23,R54,R43}, the homogeneous Dirichlet boundary condition may substitute for the classical periodic boundary condition for the first-order cell problems \eqref{eq:2.22}-\eqref{eq:2.26} and second-order cell problems \eqref{eq:2.31}-\eqref{eq:2.48}, when the material property parameters satisfy geometric symmetry and regularity assumptions.
	\end{rmk}
	
	\begin{rmk}
		For any fixed macroscopic temperature $T^{(0)}$ or any fixed macroscopic moisture $\omega^{(0)}$, the existence and uniqueness of solutions to the auxiliary cell problems \eqref{eq:2.22}-\eqref{eq:2.26} and \eqref{eq:2.31}-\eqref{eq:2.48} can be obtained from the Lax-Milgram theorem together with assumptions (A)-(B).
	\end{rmk}
	
	\begin{rmk}
		From the auxiliary cell problems \eqref{eq:2.22}-\eqref{eq:2.26} and \eqref{eq:2.31}-\eqref{eq:2.48}, one can prove that all cell functions are continuous with respect to $T^{(0)}$ and $\omega^{(0)}$.
	\end{rmk}
	
	In conclusion, we establish the LOMS solutions for the nonlinear dynamic hygro-thermo-mechanical problems \eqref{eq:2.1} as below.
	\begin{equation}
		\label{eq:2.49}
		T^{(1, \epsilon)}(\mathbf{x}, t) = T^{(0)}(\mathbf{x}, t) + \epsilon \mathcal{H}_{\alpha_1}(\mathbf{y}, T^{(0)}) \frac{\partial T^{(0)}(\mathbf{x}, t)}{\partial x_{\alpha_1}}.
	\end{equation}
	\begin{equation}
		\label{eq:2.50}
		\omega^{(1,\epsilon)}(\mathbf{x}, t) = \omega^{(0)}(\mathbf{x}, t) + \epsilon \mathcal{J}_{\alpha_1}(\mathbf{y}, \omega^{(0)}) \frac{\partial \omega^{(0)}(\mathbf{x}, t)}{\partial x_{\alpha_1}}.
	\end{equation}
	\begin{equation}
		\label{eq:2.51}
		\begin{aligned}
			u_i^{(1,\epsilon)}(\mathbf{x}, t) & = u_i^{(0)}(\mathbf{x}, t) + \epsilon \Bigl[  \mathcal{X}_{ih}^{\alpha_1}(\mathbf{y}, T^{(0)}) \frac{\partial u_h^{(0)}(\mathbf{x}, t)}{\partial x_{\alpha_1}}\\
			& - \mathcal{M}_i(\mathbf{y}, T^{(0)}) \bigl( T^{(0)}(\mathbf{x}, t) - \tilde{T} \bigr) - \mathcal{N}_i(\mathbf{y}, T^{(0)}) \bigl( \omega^{(0)}(\mathbf{x}, t) - \tilde{\omega} \bigr) \Bigr].
		\end{aligned}
	\end{equation}
	
	Furthermore, the HOMS solutions for the multi-scale problem \eqref{eq:2.1} as follows.
	\begin{equation}
		\label{eq:2.52}
		\begin{aligned}
			T^{(2,\epsilon)}(\mathbf{x}, t) & = T^{(0)}(\mathbf{x}, t) + \epsilon \mathcal{H}_{\alpha_1}(\mathbf{y}, T^{(0)}) \frac{\partial T^{(0)}(\mathbf{x}, t)}{\partial x_{\alpha_1}} \\
			& + \epsilon^2 \Bigl[ \mathcal{S}(\mathbf{y}, T^{(0)}) \frac{\partial T^{(0)}(\mathbf{x}, t)}{\partial t} + \mathcal{H}_{\alpha_1 \alpha_2}(\mathbf{y}, T^{(0)}) \frac{\partial^2 T^{(0)}(\mathbf{x}, t)}{\partial x_{\alpha_1} \partial x_{\alpha_2}}\\
			& + \mathcal{R}_{\alpha_1}(\mathbf{y}, T^{(0)}) \frac{\partial T^{(0)}(\mathbf{x}, t)}{\partial x_{\alpha_1}} - \mathcal{E}_{\alpha_1 \alpha_2}(\mathbf{y}, T^{(0)}) \frac{\partial T^{(0)}(\mathbf{x}, t)}{\partial x_{\alpha_1}} \frac{\partial T^{(0)}(\mathbf{x}, t)}{\partial x_{\alpha_2}} + \mathbb{Q}(\mathbf{y}, T^{(0)}) \Bigr].
		\end{aligned}
	\end{equation}
	\begin{equation}
		\label{eq:2.53}
		\begin{aligned}
			\omega^{(2,\epsilon)}(\mathbf{x}, t) & = \omega^{(0)}(\mathbf{x}, t) + \epsilon \mathcal{J}_{\alpha_1}(\mathbf{y}, \omega^{(0)}) \frac{\partial \omega^{(0)}(\mathbf{x}, t)}{\partial x_{\alpha_1}} + \epsilon^2 \Bigl[ \mathcal{J}_{\alpha_1 \alpha_2}(\mathbf{y}, \omega^{(0)}) \frac{\partial^2 \omega^{(0)}(\mathbf{x}, t)}{\partial x_{\alpha_1} \partial x_{\alpha_2}}\\
			&  + \mathcal{I}_{\alpha_1}(\mathbf{y}, \omega^{(0)}) \frac{\partial \omega^{(0)}(\mathbf{x}, t)}{\partial x_{\alpha_1}} - \mathcal{F}_{\alpha_1\alpha_2}(\mathbf{y}, \omega^{(0)}) \frac{\partial \omega^{(0)}(\mathbf{x}, t)}{\partial x_{\alpha_1}} \frac{\partial \omega^{(0)}(\mathbf{x}, t)}{\partial x_{\alpha_2}} - \mathbb{S}(\mathbf{y}, T^{(0)}) \Bigr].
		\end{aligned}
	\end{equation}
	\begin{equation}
		\label{eq:2.54}
		\begin{aligned}
			u_i^{(2,\epsilon)}(\mathbf{x}, t) &= u_i^{(0)}(\mathbf{x}, t) + \epsilon \Bigl[ \mathcal{X}_{ih}^{\alpha_1}(\mathbf{y}, T^{(0)}) \frac{\partial u_h^{(0)}(\mathbf{x}, t)}{\partial x_{\alpha_1}} - \mathcal{M}_i(\mathbf{y}, T^{(0)}) ( T^{(0)}(\mathbf{x}, t) - \tilde{T} ) \\
			& - \mathcal{N}_i(\mathbf{y}, T^{(0)}) ( \omega^{(0)}(\mathbf{x}, t) - \tilde{\omega} ) \Bigr] + \epsilon^2 \Bigl[ \mathcal{X}_{ih}^{\alpha_1 \alpha_2}(\mathbf{y}, T^{(0)}) \frac{\partial^2 u_h^{(0)}(\mathbf{x}, t)}{\partial x_{\alpha_1} \partial x_{\alpha_2}}\\
			& + \mathcal{Q}_{ih}^{\alpha_1}(\mathbf{y}, T^{(0)}) \frac{\partial u_h^{(0)}(\mathbf{x}, t)}{\partial x_{\alpha_1}} - \mathcal{P}_{ih}^{\alpha_1 \alpha_2}(\mathbf{y}, T^{(0)}) \frac{\partial T^{(0)}(\mathbf{x}, t)}{\partial x_{\alpha_1}} \frac{\partial u_h^{(0)}(\mathbf{x}, t)}{\partial x_{\alpha_2}} \\
			& - \mathcal{W}_i(\mathbf{y}, T^{(0)}) ( T^{(0)}(\mathbf{x}, t) - \tilde{T} ) - \mathcal{Z}_i^{\alpha_1}(\mathbf{y}, T^{(0)}) \frac{\partial T^{(0)}(\mathbf{x}, t)}{\partial x_{\alpha_1}} \\
			& + \mathcal{A}_i^{\alpha_1}(\mathbf{y}, T^{(0)}) \frac{\partial T^{(0)}(\mathbf{x}, t)}{\partial x_{\alpha_1}} ( T^{(0)}(\mathbf{x}, t) - \tilde{T} ) - \mathcal{V}_i(\mathbf{y}, T^{(0)}) ( \omega^{(0)}(\mathbf{x}, t) - \tilde{\omega} )\\
			& - \mathcal{G}_i^{\alpha_1}(\mathbf{y}, T^{(0)}, \omega^{(0)}) \frac{\partial \omega^{(0)}(\mathbf{x}, t)}{\partial x_{\alpha_1}} + \mathcal{B}_i^{\alpha_1}(\mathbf{y}, T^{(0)}) \frac{\partial T^{(0)}(\mathbf{x}, t)}{\partial x_{\alpha_1}} ( \omega^{(0)}(\mathbf{x}, t) - \tilde{\omega} ) \Bigr].
		\end{aligned}
	\end{equation}

\section{The error analyses of multi-scale asymptotic solutions}
\label{sec:3}
	This section conducts rigorous point-wise and integral error analyses pertaining to multi-scale asymptotic solutions. Before giving the detailed error analyses, we define the residual functions $T_\Delta^{(1,\epsilon)}$, $\omega_\Delta^{(1,\epsilon)}$ and $u_{\Delta i}^{(1,\epsilon)}$ for LOMS solutions as below.
	
	\begin{equation}
		\label{eq:3.1}
		T_\Delta^{(1, \epsilon)}(\mathbf{x}, t) = T^\epsilon - T^{(1,\epsilon)}, \quad \omega_\Delta^{(1, \epsilon)}(\mathbf{x}, t) = \omega^\epsilon - \omega^{(1, \epsilon)}, \quad u_{\Delta i}^{(1, \epsilon)}(\mathbf{x}, t) = u_i^\epsilon - u_i^{(1, \epsilon)}.
	\end{equation}
	
	Additionally, the residual functions $T_\Delta^{(2,\epsilon)}$, $\omega_\Delta^{(2,\epsilon)}$ and $u_{\Delta i}^{(2,\epsilon)}$ are denoted for HOMS solutions as below.
	\begin{equation}
		\label{eq:3.2}
		T_\Delta^{(2, \epsilon)}(\mathbf{x}, t) = T^\epsilon - T^{(2, \epsilon)}, \quad \omega_\Delta^{(2, \epsilon)}(\mathbf{x}, t) = \omega^\epsilon - \omega^{(2, \epsilon)}, \quad u_{\Delta i}^{(2, \epsilon)}(\mathbf{x}, t) = u_i^\epsilon - u_i^{(2, \epsilon)}.
	\end{equation}
	
	\subsection{The error analysis in the point-wise sense}
	\label{sec:31}
	Firstly, substituting the residual functions \eqref{eq:3.1} into the multi-scale equations \eqref{eq:2.1} yields the residual equations for LOMS solutions presented below.
	\begin{equation}
		\label{eq:3.3}
		\begin{cases}
			\begin{aligned}
				& \rho^\epsilon c^\epsilon \frac{\partial T_\Delta^{(1,\epsilon)}}{\partial t} - \frac{\partial }{\partial x_i}\Bigl( k_{ij}^\epsilon \frac{\partial T_\Delta^{(1,\epsilon)}}{\partial x_j} \Bigr) = \Phi_0(\mathbf{x}, \mathbf{y}, t) + \epsilon \Phi_1(\mathbf{x}, \mathbf{y}, t), \text{in } \Omega \times (0, T^*), \\
				& \frac{\partial \omega_\Delta^{(1,\epsilon)}}{\partial t} - \frac{\partial }{\partial x_i}\Bigl( g_{ij}^\epsilon \frac{\partial \omega_\Delta^{(1,\epsilon)}}{\partial x_j} \Bigr)  = \Psi_0(\mathbf{x}, \mathbf{y}, t) + \epsilon \Psi_1(\mathbf{x}, \mathbf{y}, t), \text{in } \Omega \times (0, T^*), \\
				& - \frac{\partial }{\partial x_j}\Bigl( C_{ijkl}^\epsilon \frac{\partial u_{\Delta k}^{(1,\epsilon)}}{\partial x_l} - \alpha_{ij}^\epsilon\! T_\Delta^{(1,\epsilon)} - \beta_{ij}^\epsilon \omega_\Delta^{(1,\epsilon)} \Bigr) = \Theta_{0i}(\mathbf{x}, \mathbf{y}, t) + \epsilon \Theta_{1i}(\mathbf{x}, \mathbf{y}, t), \text{in } \Omega \times (0, T^*).
			\end{aligned}
		\end{cases}
	\end{equation}
	Secondly, by substituting the residual functions \eqref{eq:3.2} into multi-scale equations \eqref{eq:2.1}, we derive the residual equations for HOMS solutions as below.
	\begin{equation}
		\label{eq:3.4}
		\begin{cases}
			\begin{aligned}
				& \rho^\epsilon c^\epsilon \frac{\partial T_\Delta^{(2,\epsilon)}}{\partial t} - \frac{\partial }{\partial x_i}\Bigl( k_{ij}^\epsilon \frac{\partial T_\Delta^{(2,\epsilon)}}{\partial x_j} \Bigr) = \epsilon \varphi(\mathbf{x}, \mathbf{y}, t), \text{in } \Omega \times (0, T^* ), \\
				& \frac{\partial \omega_\Delta^{(2,\epsilon)}}{\partial t} - \frac{\partial }{\partial x_i}\Bigl( g_{ij}^\epsilon \frac{\partial \omega_\Delta^{(2,\epsilon)}}{\partial x_j} \Bigr) = \epsilon \psi(\mathbf{x}, \mathbf{y}, t), \text{in } \Omega \times (0, T^* ), \\
				& -\frac{\partial }{\partial x_j}\Bigl( C_{ijkl}^\epsilon \frac{\partial u_{\Delta k}^{(2,\epsilon)}}{\partial x_l} - \alpha_{ij}^\epsilon T_\Delta^{(2,\epsilon)} - \beta_{ij}^\epsilon \omega_\Delta^{(2,\epsilon)} \Bigr) = \epsilon \theta_i(\mathbf{x}, \mathbf{y}, t), \text{in } \Omega \times (0, T^*).
			\end{aligned}
		\end{cases}
	\end{equation}
	
	In residual equations \eqref{eq:3.3} and \eqref{eq:3.4}, the specific expressions of functions $\Phi_0$, $\Phi_1$, $\Psi_0$, $\Psi_1$, $\Theta_{0i}$, $\Theta_{1i}$, $\varphi$, $\psi$ and $\theta_i$ are exhibited in Appendix~\ref{app:A} of the present study because of their lengthy forms.
	
	The point-wise error analysis leads to the following key findings. For the LOMS solutions, the residuals are of order $O(1)$ as seen from \eqref{eq:3.3}, because the terms $\Phi_0$, $\Psi_0$ and $\Theta_{0i}$ do not tend to zero with $\epsilon$. Hence, the LOMS solutions fail to maintain the local physical balance of the original multi-scale equations. By contrast, the HOMS solutions, owing to the introduction of higher-order correction terms, yield residuals of order $O(\epsilon)$ from \eqref{eq:3.4}. They therefore rigorously maintain local physical balance while achieving point-wise convergence of order $O(\epsilon)$. Thus, even for small but finite $\epsilon$, the HOMS solutions provide sufficient accuracy for engineering practice and correctly capture microscopic oscillations in heterogeneous structures. This advantage motivates the development of the HOMS method presented in this work.
	
	\subsection{The error analysis in the integral sense}
	\label{sec:32}
	In order to obtain the optimal error estimations in the integral sense, we postulate three fundamental assumptions regarding the multi-scale problem \eqref{eq:2.1} as follows:
	\begin{enumerate}
		\item[(i)] Assume that $\Omega \subset \mathbb{R}^n$ is a bounded domain and the union of entire periodic cells, i.e. $\bar{\Omega}=\cup_{\mathbf{z}\in I_\epsilon}\epsilon(\mathbf{z}+\bar{Y})$ , where the index set  $I_{\epsilon}=\{\mathbf{z}=(z_{1},\ldots,z_{n})\in Z^{n},\epsilon(\mathbf{z}+\bar{Y})\subset\bar{\Omega}\}$. Besides, let $E_z=\epsilon(\mathbf{z}+\bar{Y})$ be translational unit cell and $\partial E_z$ be the boundary of $E_z$.
		\item[(ii)] $\rho(\mathbf{y}, T^\epsilon)$, $c(\mathbf{y}, T^\epsilon)$, $k_{ij}(\mathbf{y}, T^\epsilon)$, $g_{ij}(\mathbf{y}, \omega^\epsilon)$, $C_{ijkl}(\mathbf{y}, T^\epsilon)$, $\alpha_{ij}(\mathbf{y}, T^\epsilon)$, $\beta_{ij}(\mathbf{y}, T^\epsilon)$, $Q_{hyd}(\mathbf{y}, T^\epsilon)$ and  $S_{hyd}(\mathbf{y}, T^\epsilon)$ are piecewise constant functions on $Y$, taking constant values on the matrix region $Y_a$ and the inclusion region $Y_b$ , respectively.
		\item[(iii)] Let $\Delta_1,\cdots,\Delta_n(n=2,3)$ denote the middle hyperplanes of the PUC $Y$. We assume reference cell $Y$ are symmetric with respect to $\Delta_1,\cdots,\Delta_n$.
	\end{enumerate}
	
	Based on the above simplifications and hypotheses, we further derive the following initial-boundary conditions for the residual equations \eqref{eq:3.4} of the HOMS solutions when applied to nonlinear dynamic hygro-thermo-mechanical coupling problems \eqref{eq:2.1} with pure Dirichlet boundary conditions, which will be employed for global error estimation.
	\begin{equation}
		\label{eq:3.5}
		\begin{cases}
			\begin{aligned}
				& T_\Delta^{(2,\epsilon)}(\mathbf{x}, t) =- \epsilon \mathcal{H}_{\alpha_1} \frac{\partial T^{(0)}}{\partial x_{\alpha_1}} - \epsilon^2 \Bigl( \mathcal{S} \frac{\partial T^{(0)}}{\partial t} + \mathcal{H}_{\alpha_1 \alpha_2} \frac{\partial^2 T^{(0)}}{\partial x_{\alpha_1} \partial x_{\alpha_2}} + \mathcal{R}_{\alpha_1} \frac{\partial T^{(0)}}{\partial x_{\alpha_1}}\\
				& - \mathcal{E}_{\alpha_1 \alpha_2} \frac{\partial T^{(0)}}{\partial x_{\alpha_1}} \frac{\partial T^{(0)}}{\partial x_{\alpha_2}} + \mathbb{Q} \Bigr) = 0,\quad \text{on } \partial\Omega_T \times (0, T^*), \\
				& \omega_\Delta^{(2,\epsilon)}(\mathbf{x}, t) = - \epsilon \mathcal{J}_{\alpha_1} \frac{\partial \omega^{(0)}}{\partial x_{\alpha_1}} - \epsilon^2 \Bigl( \mathcal{J}_{\alpha_1 \alpha_2} \frac{\partial^2 \omega^{(0)}}{\partial x_{\alpha_1} \partial x_{\alpha_2}} + \mathcal{I}_{\alpha_1} \frac{\partial \omega^{(0)}}{\partial x_{\alpha_1}} \\
				& - \mathcal{F}_{\alpha_1\alpha_2} \frac{\partial \omega^{(0)}}{\partial x_{\alpha_1}} \frac{\partial \omega^{(0)}}{\partial x_{\alpha_2}} - \mathbb{S} \Bigr) = 0, \quad \text{on } \partial\Omega_\omega \times (0, T^*), \\
				& u_{\Delta i}^{(2,\epsilon)}(\mathbf{x}, t) = - \epsilon \Bigl( \mathcal{X}_{ih}^{\alpha_1} \frac{\partial u_h^{(0)}}{\partial x_{\alpha_1}} - \mathcal{M}_i ( T^{(0)} - \tilde{T} ) - \mathcal{N}_i ( \omega^{(0)} - \tilde{\omega} ) \Bigr) \\
				& - \epsilon^2 \Bigl( \mathcal{X}_{ih}^{\alpha_1 \alpha_2} \frac{\partial^2 u_h^{(0)}}{\partial x_{\alpha_1} \partial x_{\alpha_2}} + \mathcal{Q}_{ih}^{\alpha_1} \frac{\partial u_h^{(0)}}{\partial x_{\alpha_1}} - \mathcal{P}_{ih}^{\alpha_1 \alpha_2} \frac{\partial T^{(0)}}{\partial x_{\alpha_1}} \frac{\partial u_h^{(0)}}{\partial x_{\alpha_2}} \\
				& - \mathcal{W}_i ( T^{(0)} - \tilde{T} ) - \mathcal{Z}_i^{\alpha_1} \frac{\partial T^{(0)}}{\partial x_{\alpha_1}} + \mathcal{A}_i^{\alpha_1} \frac{\partial T^{(0)}}{\partial x_{\alpha_1}} ( T^{(0)} - \tilde{T} ) - \mathcal{V}_i ( \omega^{(0)} - \tilde{\omega} )\\
				& - \mathcal{G}_i^{\alpha_1} \frac{\partial \omega^{(0)}}{\partial x_{\alpha_1}} + \mathcal{B}_i^{\alpha_1} \frac{\partial T^{(0)}}{\partial x_{\alpha_1}} ( \omega^{(0)} - \tilde{\omega} ) \Bigr) = 0, \quad \text{on } \partial\Omega_u \times (0, T^*), \\
				& T_\Delta^{(2,\epsilon)}(\mathbf{x}, 0) =: \epsilon \hat{\varphi}(\mathbf{x}), \omega_\Delta^{(2,\epsilon)}(\mathbf{x}, 0) =: \epsilon \hat{\psi}(\mathbf{x}), u_{\Delta i}^{(2,\epsilon)}(\mathbf{x}, 0) =: \epsilon \hat{\theta}_i(\mathbf{x}), \quad \text{in } \Omega.
			\end{aligned}
		\end{cases}
	\end{equation}
	\begin{lemma}
		\label{lem:3.1}
		Defining three differential operators $\displaystyle\sigma_{TY}(\chi)=n_i k_{ij}(\mathbf{y}, T^\epsilon)\frac{\partial \chi}{\partial y_j}$,  $\displaystyle\sigma_{\omega Y}(\chi)=n_i g_{ij}(\mathbf{y}, \omega^\epsilon)\frac{\partial \chi}{\partial y_j}$ and $\displaystyle\sigma_{iY}(\bm{\phi})=n_j C_{ijkl}(\mathbf{y}, T^\epsilon)\frac{\partial \phi_{k}}{\partial x_{l}}$, and given assumptions (A)-(B) and (ii)-(iii), the normal derivatives $\sigma_{TY}(\mathcal{H}_{\alpha_1})$, $\sigma_{TY}(\mathcal{S})$, $\sigma_{TY}(\mathcal{H}_{\alpha_1 \alpha_2})$, $\sigma_{TY}(\mathcal{R}_{\alpha_1})$, $\sigma_{TY}(\mathcal{E}_{\alpha_1 \alpha_2})$,
		$\sigma_{TY}(\mathbb{Q})$, $\sigma_{\omega Y}(\mathcal{J}_{\alpha_1})$, $\sigma_{\omega Y}(\mathcal{J}_{\alpha_1 \alpha_2})$, $\sigma_{\omega Y}(\mathcal{I}_{\alpha_1})$, $\sigma_{\omega Y}(\mathcal{F}_{\alpha_1\alpha_2})$, $\sigma_{\omega Y}(\mathbb{S})$ and $\sigma_{iY}(\mathbcal{X}_h^{\alpha_1})$, $\sigma_{iY}(\bm{\mathbcal{M}})$, $\sigma_{iY}(\mathbcal{N})$, $\sigma_{iY}(\mathbcal{X}_h^{\alpha_1\alpha_2})$, $\sigma_{iY}(\mathbcal{Q}_h^{\alpha_1})$, $\sigma_{iY}(\mathbcal{P}_h^{\alpha_1\alpha_2})$, $\sigma_{iY}(\mathbcal{W})$, $\sigma_{iY}(\mathbcal{Z}^{\alpha_1})$, $\sigma_{iY}(\mathbcal{A}^{\alpha_1})$, $\sigma_{iY}(\mathbcal{V})$, $\sigma_{iY}(\mathbcal{G}^{\alpha_1})$, $\sigma_{iY}(\mathbcal{B}^{\alpha_1})$
		 can be proved to be continuous on the boundary of the PUC $Y$ by using the same method as in Refs. \citep{R22,R43,R44}. This method comprises three main parts. First, the symmetry or anti-symmetry of the components of the microscopic cell functions about the middle hyperplanes is verified. Second, the absolute and uniform convergence of the Fourier series of the transformed functions on $\bar Y$ is established for any fixed $T^{(\epsilon)} \in [T_{\min}, T_{\max} + C_*]$ and $\omega^{(\epsilon)} \in [\omega_{\min}, \omega_{\max} + C^*]$. Third, several essential equalities are proved to ensure the completeness of the function family, and the continuity of the cell functions on the boundary of PUC $Y$ is obtained.
	\end{lemma}
	
	\begin{theorem}
		\label{thm:1}
		Suppose that $\Omega \subset \mathbb{R}^n$ is the union of entire periodic cells, i.e. $\bar{\Omega}=\cup_{\mathbf{z}\in I_\epsilon}\epsilon(\mathbf{z}+\bar{Y})$ , where the index set  $I_{\epsilon}=\{\mathbf{z}=(z_{1},\ldots,z_{n})\in Z^{n},\epsilon(\mathbf{z}+\bar{Y})\subset\bar{\Omega}\}$. Let $T^{\epsilon}(\mathbf{x}, t)$, $\omega^{\epsilon}(\mathbf{x}, t)$, $\bm{u}^{\epsilon}(\mathbf{x}, t)$ and $T^{(0)}(\mathbf{x}, t)$, $\omega^{(0)}(\mathbf{x}, t)$, $\bm{u}^{(0)}(\mathbf{x}, t)$ be the weak solutions of multi-scale problem \eqref{eq:2.1} and corresponding homogenized problem \eqref{eq:2.27}, respectively. The specific forms of HOMS solutions are defined in \eqref{eq:2.52}-\eqref{eq:2.54}. Under the assumptions (A)-(C), (i)-(iii), and Lemma \ref{lem:3.1}, if $T^{(0)} \in L^\infty(0,T^*; H^4(\Omega))$, $\frac{\partial T^{(0)}}{\partial t} \in L^\infty(0,T^*; H^2(\Omega))$, $\omega^{(0)} \in L^\infty(0,T^*; H^4(\Omega))$, $\frac{\partial \omega^{(0)}}{\partial t} \in L^\infty(0,T^*; H^2(\Omega))$ and  $\bm{u}^{(0)} \in L^\infty(0,T^*; (H^4(\Omega))^n)$, then we derive the following error estimation of HOMS solutions.
		\begin{equation}
			\label{eq:3.6a}
			\bigl\| T_\Delta^{(2, \epsilon)} \bigr\|_{L^\infty(0,T^*; L^2(\Omega))} + \bigl\| T_\Delta^{(2, \epsilon)} \bigr\|_{L^2(0,T^*; H_0^1(\Omega))} \le C(\Omega, T^*)\epsilon,
		\end{equation}
		\begin{equation}
			\label{eq:3.6b}
			\bigl\| \omega_\Delta^{(2, \epsilon)} \bigr\|_{L^\infty(0,T^*; L^2(\Omega))} + \bigl\| \omega_\Delta^{(2, \epsilon)} \bigr\|_{L^2(0,T^*; H_0^1(\Omega))} \le C(\Omega, T^*)\epsilon,
		\end{equation}
		\begin{equation}
			\label{eq:3.6c}
			\bigl\| \bm{u}_\Delta^{(2, \epsilon)} \bigr\|_{L^\infty(0,T^*; (H_0^1(\Omega))^n)} \le C(\Omega, T^*)\epsilon.
		\end{equation}
		where $C(\Omega, T^*)$ is a positive constant independent of $\epsilon$, but dependent of $\Omega$ and $T^*$.
	\end{theorem}
	\begin{proof}
		To facilitate the error estimation, $\sigma_{TY}(T^{(2, \epsilon)})$, $\sigma_{\omega Y}(\omega^{(2, \epsilon)})$, and $\sigma_{iY}(\bm{u}^{(2, \epsilon)})$ are first derived from the HOMS solutions \eqref{eq:2.52}-\eqref{eq:2.54}. Their explicit expressions are provided in the Appendix~\ref{app:B}.
		
		Next, we perform the error estimation using the residual equations \eqref{eq:3.4} and associated initial-boundary conditions \eqref{eq:3.5}. Multiplying each equation in \eqref{eq:3.4} by the corresponding error variable $T_\Delta^{(2, \epsilon)}$, $\omega_\Delta^{(2, \epsilon)}$, and $u_{\Delta i}^{(2, \epsilon)}$, respectively, integrating over $\Omega$, applying Green's formula, and substituting the corresponding boundary condition in \eqref{eq:3.5} yields the following equations.
		\begin{equation}
			\label{eq:3.11}
			\begin{cases}
				\begin{aligned}
					& \int_\Omega \rho^\epsilon c^\epsilon \frac{\partial T_\Delta^{(2, \epsilon)}}{\partial t} T_\Delta^{(2, \epsilon)} \, d\Omega + \int_\Omega k_{ij}^\epsilon \frac{\partial T_\Delta^{(2, \epsilon)}}{\partial x_j} \frac{\partial T_\Delta^{(2, \epsilon)}}{\partial x_i} \, d\Omega \\
					& = \int_\Omega \epsilon \varphi(\mathbf{x}, \mathbf{y}, t) T_\Delta^{(2, \epsilon)} \, d\Omega + \sum_{\mathbf{z} \in I_\epsilon} \int_{\partial E_z} \sigma_{TY}(T_\Delta^{(2, \epsilon)}) T_\Delta^{(2, \epsilon)} \, d\Gamma_{\mathbf{y}},\\
					& \int_\Omega \frac{\partial \omega_\Delta^{(2, \epsilon)}}{\partial t} \omega_\Delta^{(2, \epsilon)} \, d\Omega + \int_\Omega g_{ij}^\epsilon \frac{\partial \omega_\Delta^{(2, \epsilon)}}{\partial x_j} \frac{\partial \omega_\Delta^{(2, \epsilon)}}{\partial x_i} \, d\Omega \\
					& = \int_\Omega \epsilon \psi(\mathbf{x}, \mathbf{y}, t) \omega_\Delta^{(2, \epsilon)} \, d\Omega + \sum_{\mathbf{z} \in I_\epsilon} \int_{\partial E_z} \sigma_{\omega Y}(\omega_\Delta^{(2, \epsilon)}) \omega_\Delta^{(2, \epsilon)} \, d\Gamma_{\mathbf{y}},\\
					& \int_\Omega C_{ijkl}^\epsilon \frac{\partial u_{\Delta k}^{(2, \epsilon)}}{\partial x_l} \frac{\partial u_{\Delta i}^{(2, \epsilon)}}{\partial x_j} \, d\Omega - \int_\Omega \alpha_{ij}^\epsilon T_\Delta^{(2, \epsilon)} \frac{\partial u_{\Delta i}^{(2, \epsilon)}}{\partial x_j} \, d\Omega - \int_\Omega \beta_{ij}^\epsilon \omega_\Delta^{(2, \epsilon)} \frac{\partial u_{\Delta i}^{(2, \epsilon)}}{\partial x_j} \, d\Omega \\
					& = \int_\Omega \epsilon \theta_i(\mathbf{x}, \mathbf{y}, t) u_{\Delta i}^{(2, \epsilon)} \, d\Omega + \sum_{\mathbf{z} \in I_\epsilon} \int_{\partial E_z} \sigma_{iY}(\bm{u}_{\Delta}^{(2, \epsilon)}) u_{\Delta i}^{(2, \epsilon)} \, d\Gamma_{\mathbf{y}}.
				\end{aligned}
			\end{cases}
		\end{equation}
		where $\sigma_{TY}(T_\Delta^{(2, \epsilon)})$, $\sigma_{\omega Y}(\omega_\Delta^{(2, \epsilon)})$ and $\sigma_{iY}(\bm{u}_{\Delta}^{(2, \epsilon)})$ are obtained by applying Green's formula over $\partial E_z$.
		
		Recalling the expressions for $\sigma_{TY}(T^{(2, \epsilon)})$, $\sigma_{\omega Y}(\omega^{(2, \epsilon)})$, and $\sigma_{iY}(\bm{u}^{(2, \epsilon)})$ given in the Appendix~\ref{app:B}, together with Lemma \ref{lem:3.1}, we can obtain the following:
		\begin{equation}
			\label{eq:3.12}
			\begin{cases}
				\begin{aligned}
					& \sum_{\mathbf{z} \in I_\epsilon} \int_{\partial E_z} \sigma_{TY}(T_\Delta^{(2, \epsilon)}) T_\Delta^{(2, \epsilon)} \, d\Gamma_{\mathbf{y}} = \sum_{\mathbf{z} \in I_\epsilon} \int_{\partial E_z} \sigma_{TY}(T^\epsilon - T^{(2, \epsilon)}) T_\Delta^{(2, \epsilon)} \, d\Gamma_{\mathbf{y}} \\
					& = -\sum_{\mathbf{z} \in I_\epsilon} \int_{\partial E_z} \sigma_{TY}(T^{(2, \epsilon)}) T_\Delta^{(2, \epsilon)} \, d\Gamma_{\mathbf{y}} = 0,\\
					& \sum_{\mathbf{z} \in I_\epsilon} \int_{\partial E_z} \sigma_{\omega Y}(\omega_\Delta^{(2, \epsilon)}) \omega_\Delta^{(2, \epsilon)} \, d\Gamma_{\mathbf{y}} = \sum_{\mathbf{z} \in I_\epsilon} \int_{\partial E_z} \sigma_{\omega Y}(\omega^\epsilon - \omega^{(2, \epsilon)}) \omega_\Delta^{(2, \epsilon)} \, d\Gamma_{\mathbf{y}} \\
					& = -\sum_{\mathbf{z} \in I_\epsilon} \int_{\partial E_z} \sigma_{\omega Y}(\omega^{(2, \epsilon)}) \omega_\Delta^{(2, \epsilon)} \, d\Gamma_{\mathbf{y}} = 0,\\
					& \sum_{\mathbf{z} \in I_\epsilon} \int_{\partial E_z} \sigma_{iY}(\bm{u}_{\Delta}^{(2, \epsilon)}) u_{\Delta i}^{(2, \epsilon)} \, d\Gamma_{\mathbf{y}} = \sum_{\mathbf{z} \in I_\epsilon} \int_{\partial E_z} \sigma_{iY}(\bm{u}^\epsilon - \bm{u}^{(2, \epsilon)}) u_{\Delta i}^{(2, \epsilon)} \, d\Gamma_{\mathbf{y}} \\
					& = -\sum_{\mathbf{z} \in I_\epsilon} \int_{\partial E_z} \sigma_{iY}(\bm{u}^{(2, \epsilon)}) u_{\Delta i}^{(2, \epsilon)} \, d\Gamma_{\mathbf{y}} = 0.
				\end{aligned}
			\end{cases}
		\end{equation}
		
		Then, substituting \eqref{eq:3.12} into \eqref{eq:3.11} and following similar calculations as in \citep{R47, R48}, we readily obtain the following three equations.
		\begin{equation}
			\label{eq:3.13}
			\begin{aligned}
				& \frac{1}{2} \frac{\partial}{\partial t} \int_\Omega \rho^\epsilon(\mathbf{x}, T^\epsilon) c^\epsilon(\mathbf{x}, T^\epsilon) \bigl( T_\Delta^{(2, \epsilon)} \bigr)^2 \, d\Omega + \int_\Omega k_{ij}^\epsilon(\mathbf{x}, T^\epsilon) \frac{\partial T_\Delta^{(2, \epsilon)}}{\partial x_j} \frac{\partial T_\Delta^{(2, \epsilon)}}{\partial x_i} \, d\Omega\\
				& - \frac{1}{2} \int_\Omega \frac{\partial \rho^\epsilon(\mathbf{x}, T^\epsilon)}{\partial t} c^\epsilon(\mathbf{x}, T^\epsilon) \bigl( T_\Delta^{(2\epsilon)} \bigr)^2 d\Omega - \frac{1}{2} \int_\Omega \rho^\epsilon(\mathbf{x}, T^\epsilon) \frac{\partial c^\epsilon(\mathbf{x}, T^\epsilon)}{\partial t} \bigl( T_\Delta^{(2\epsilon)} \bigr)^2 d\Omega\\
				& = \int_\Omega \epsilon \varphi(\mathbf{x}, \mathbf{y}, t) T_\Delta^{(2, \epsilon)} \, d\Omega.
			\end{aligned}
		\end{equation}
		\begin{equation}
			\label{eq:3.14}
			\frac{1}{2} \frac{\partial}{\partial t} \int_\Omega \bigl( \omega_\Delta^{(2, \epsilon)} \bigr)^2 \, d\Omega + \int_\Omega g_{ij}^\epsilon(\mathbf{x}, \omega^\epsilon) \frac{\partial \omega_\Delta^{(2, \epsilon)}}{\partial x_j} \frac{\partial \omega_\Delta^{(2, \epsilon)}}{\partial x_i} \, d\Omega
			= \int_\Omega \epsilon \psi(\mathbf{x}, \mathbf{y}, t) \omega_\Delta^{(2, \epsilon)} \, d\Omega.
		\end{equation}
		\begin{equation}
			\label{eq:3.15}
			\begin{aligned}
				& \int_\Omega C_{ijkl}^\epsilon(\mathbf{x}, T^\epsilon) \frac{\partial u_{\Delta k}^{(2, \epsilon)}}{\partial x_l} \frac{\partial u_{\Delta i}^{(2, \epsilon)}}{\partial x_j} \, d\Omega - \int_\Omega \alpha_{ij}^\epsilon(\mathbf{x}, T^\epsilon) T_\Delta^{(2, \epsilon)} \frac{\partial u_{\Delta i}^{(2, \epsilon)}}{\partial x_j} \, d\Omega \\
				& - \int_\Omega \beta_{ij}^\epsilon(\mathbf{x}, T^\epsilon) \omega_\Delta^{(2, \epsilon)} \frac{\partial u_{\Delta i}^{(2, \epsilon)}}{\partial x_j} \, d\Omega = \int_\Omega \epsilon \theta_i(\mathbf{x}, \mathbf{y}, t) u_{\Delta i}^{(2, \epsilon)} \, d\Omega.
			\end{aligned}
		\end{equation}
		
		We first derive the convergence of the temperature and moisture fields from \eqref{eq:3.13} and \eqref{eq:3.14}. Then, using these results, we derive the convergence of the displacement field.
		
		To begin, we integrate both sides of \eqref{eq:3.13} and \eqref{eq:3.14} over the time interval $[0, t]$ $(0 < t \le T^*)$, substitute the initial conditions from \eqref{eq:3.5}, and rearrange to obtain the following equations.
		\begin{equation}
			\label{eq:3.18}
			\begin{aligned}
				& \int_\Omega \rho^\epsilon c^\epsilon \bigl( T_\Delta^{(2, \epsilon)}(\mathbf{x}, t) \bigr)^2 \, d\Omega + \int_0^t \int_\Omega 2 k_{ij}^\epsilon \frac{\partial T_\Delta^{(2, \epsilon)}(\mathbf{x}, \tau)}{\partial x_j} \frac{\partial T_\Delta^{(2, \epsilon)}(\mathbf{x}, \tau)}{\partial x_i} \, d\Omega \, d\tau \\
				& = \int_0^t \int_\Omega 2 \epsilon \varphi(\mathbf{x}, \mathbf{y}, \tau) T_\Delta^{(2, \epsilon)}(\mathbf{x}, \tau) \, d\Omega \, d\tau + \int_\Omega \rho^\epsilon c^\epsilon \bigl( \epsilon \hat{\varphi}(\mathbf{x}) \bigr)^2 \, d\Omega \\
				& + \int_0^t \int_\Omega \frac{\partial \rho^\epsilon}{\partial \tau} c^\epsilon \bigl( T_\Delta^{(2\epsilon)}(\mathbf{x}, \tau) \bigr)^2 d\Omega d\tau + \int_0^t \int_\Omega \rho^\epsilon \frac{\partial c^\epsilon}{\partial \tau} \bigl( T_\Delta^{(2\epsilon)}(\mathbf{x}, \tau) \bigr)^2 d\Omega d\tau.
			\end{aligned}
		\end{equation}
		\begin{equation}
			\label{eq:3.19}
			\begin{aligned}
				& \int_\Omega \bigl( \omega_\Delta^{(2, \epsilon)}(\mathbf{x}, t) \bigr)^2 \, d\Omega + \int_0^t \int_\Omega 2 g_{ij}^\epsilon \frac{\partial \omega_\Delta^{(2, \epsilon)}(\mathbf{x}, \tau)}{\partial x_j} \frac{\partial \omega_\Delta^{(2, \epsilon)}(\mathbf{x}, \tau)}{\partial x_i} \, d\Omega \, d\tau \\
				& = \int_0^t \int_\Omega 2 \epsilon \psi(\mathbf{x}, \mathbf{y}, \tau) \omega_\Delta^{(2, \epsilon)}(\mathbf{x}, \tau) \, d\Omega \, d\tau + \int_\Omega \bigl( \epsilon \hat{\psi}(\mathbf{x}) \bigr)^2 \, d\Omega.
			\end{aligned}
		\end{equation}
		
		By virtue of assumptions (A) and (B) and applying the Poincar$\rm{\acute{e}}$-Friedrichs inequality, the following inequalities are readily obtained after transforming the left-hand sides of \eqref{eq:3.18} and \eqref{eq:3.19}.
		\begin{equation}
			\label{eq:3.20}
			\begin{aligned}
				& \int_\Omega \rho^\epsilon c^\epsilon \bigl( T_\Delta^{(2, \epsilon)}(\mathbf{x}, t) \bigr)^2 \, d\Omega + \int_0^t \int_\Omega 2 k_{ij}^\epsilon \frac{\partial T_\Delta^{(2, \epsilon)}(\mathbf{x}, \tau)}{\partial x_j} \frac{\partial T_\Delta^{(2, \epsilon)}(\mathbf{x}, \tau)}{\partial x_i} \, d\Omega \, d\tau \\
				& \ge \rho^0 c^0 \bigl\| T_\Delta^{(2, \epsilon)} \bigr\|_{L^2(\Omega)}^2 + C_1(\Omega) \int_0^t \bigl\| T_\Delta^{(2, \epsilon)} \bigr\|_{H_0^1(\Omega)}^2 \, d\tau \\
				& \ge \lambda_1(\Omega) \Bigl( \bigl\| T_\Delta^{(2, \epsilon)} \bigr\|_{L^2(\Omega)}^2 + \int_0^t \bigl\| T_\Delta^{(2, \epsilon)} \bigr\|_{H_0^1(\Omega)}^2 \, d\tau \Bigr).
			\end{aligned}
		\end{equation}
		\begin{equation}
			\label{eq:3.21}
			\begin{aligned}
				& \int_\Omega \bigl( \omega_\Delta^{(2, \epsilon)}(\mathbf{x}, t) \bigr)^2 \, d\Omega + \int_0^t \int_\Omega 2 g_{ij}^\epsilon \frac{\partial \omega_\Delta^{(2, \epsilon)}(\mathbf{x}, \tau)}{\partial x_j} \frac{\partial \omega_\Delta^{(2, \epsilon)}(\mathbf{x}, \tau)}{\partial x_i} \, d\Omega \, d\tau \\
				& \ge \bigl\| \omega_\Delta^{(2, \epsilon)} \bigr\|_{L^2(\Omega)}^2 + C_2(\Omega) \int_0^t \bigl\| \omega_\Delta^{(2, \epsilon)} \bigr\|_{H_0^1(\Omega)}^2 \, d\tau \\
				& \ge \lambda_2(\Omega) \Bigl( \bigl\| \omega_\Delta^{(2, \epsilon)} \bigr\|_{L^2(\Omega)}^2 + \int_0^t \bigl\| \omega_\Delta^{(2, \epsilon)} \bigr\|_{H_0^1(\Omega)}^2 \, d\tau \Bigr).
			\end{aligned}
		\end{equation}
		where $\lambda_1 = \min \left\{ \rho^0 c^0, C_1 \right\}$, $\lambda_2 = \min \left\{ 1, C_2 \right\}$.
		
		Subsequently, applying Young's inequality $ab \le \frac{1}{2} \left( \lambda a^2 + \frac{1}{\lambda} b^2 \right), \forall \lambda \in \mathbb{R}^+$, and transforming the right-hand sides of \eqref{eq:3.18} and \eqref{eq:3.19} yields the following inequalities:
		\begin{equation}
			\label{eq:3.22}
			\begin{aligned}
				& \int_0^t \int_\Omega 2 \epsilon \varphi(\mathbf{x}, \mathbf{y}, \tau) T_\Delta^{(2, \epsilon)}(\mathbf{x}, \tau) \, d\Omega \, d\tau + \int_\Omega \rho^\epsilon c^\epsilon \bigl( \epsilon \hat{\varphi}(\mathbf{x}) \bigr)^2 \, d\Omega \\
				& + \int_0^t \int_\Omega \frac{\partial \rho^\epsilon}{\partial \tau} c^\epsilon \bigl( T_\Delta^{(2\epsilon)}(\mathbf{x}, \tau) \bigr)^2 d\Omega d\tau + \int_0^t \int_\Omega \rho^\epsilon \frac{\partial c^\epsilon}{\partial \tau} \bigl( T_\Delta^{(2\epsilon)}(\mathbf{x}, \tau) \bigr)^2 d\Omega d\tau \\
				& \le 2 \int_0^t \int_\Omega \Bigl( \frac{(\epsilon \varphi)^2}{2} + \frac{\bigl( T_\Delta^{(2, \epsilon)} \bigr)^2}{2} \Bigr) \, d\Omega \, d\tau + C_3(\Omega) \epsilon^2 + C_4(T^*) \int_0^t \bigl\| T_\Delta^{(2\epsilon)} \bigr\|_{L^2(\Omega)}^2 d\tau \\
				& \le C_5(\Omega,T^*) \Bigl( \frac{1}{2} \epsilon^2 + \frac{1}{2} \int_0^t \bigl\| T_\Delta^{(2, \epsilon)} \bigr\|_{L^2(\Omega)}^2 \, d\tau \Bigr) + C_3(\Omega) \epsilon^2 + C_4(T^*) \int_0^t \bigl\| T_\Delta^{(2\epsilon)} \bigr\|_{L^2(\Omega)}^2 d\tau  \\
				& \le C_T(\Omega,T^*) \epsilon^2 \!+\! C_T(\Omega,T^*) \Bigl( \frac{1}{2} \int_0^t \bigl\| T_\Delta^{(2, \epsilon)}(\mathbf{x}, \tau) \bigr\|_{L^2(\Omega)}^2 d\tau \!+\! \frac{1}{2} \int_0^t \int_0^\tau \bigl\| T_\Delta^{(2, \epsilon)}(\mathbf{x}, s) \bigr\|_{H_0^1(\Omega)}^2 ds d\tau \Bigr).
			\end{aligned}
		\end{equation}
		\begin{equation}
			\label{eq:3.23}
			\begin{aligned}
				& \int_0^t \int_\Omega 2 \epsilon \psi(\mathbf{x}, \mathbf{y}, \tau) \omega_\Delta^{(2, \epsilon)}(\mathbf{x}, \tau) \, d\Omega \, d\tau + \int_\Omega \bigl( \epsilon \hat{\psi}(\mathbf{x}) \bigr)^2 \, d\Omega \\
				& \le 2 \int_0^t \int_\Omega \Bigl( \frac{(\epsilon \psi)^2}{2} + \frac{\bigl( \omega_\Delta^{(2, \epsilon)} \bigr)^2}{2} \Bigr) \, d\Omega \, d\tau + C_6(\Omega) \epsilon^2 \\
				& \le C_7(\Omega,T^*) \Bigl( \frac{1}{2} \epsilon^2 + \frac{1}{2} \int_0^t \bigl\| \omega_\Delta^{(2, \epsilon)} \bigr\|_{L^2(\Omega)}^2 \, d\tau \Bigr) + C_6(\Omega) \epsilon^2 \\
				& \le C_{\omega}(\Omega,T^*) \epsilon^2 \!+\! C_{\omega}(\Omega,T^*) \!\Bigl(\! \frac{1}{2} \int_0^t \bigl\| \omega_\Delta^{(2, \epsilon)}(\mathbf{x}, \tau) \bigr\|_{L^2(\Omega)}^2 d\tau \!+\! \frac{1}{2} \int_0^t \int_0^\tau \bigl\| \omega_\Delta^{(2, \epsilon)}(\mathbf{x}, s) \bigr\|_{H_0^1(\Omega)}^2 ds d\tau \!\Bigr)\!.
			\end{aligned}
		\end{equation}
		where $C_T=\max(C_3,C_4,C_5)$, $C_{\omega}=\max(C_6,C_7)$.
		
		Combining \eqref{eq:3.20} with \eqref{eq:3.22} and setting $C = \frac{C_T}{\lambda_1}$ without loss of generality, we define $\Theta_1(t)$ as $\Theta_1(t) = \bigl\| T_\Delta^{(2, \epsilon)}(\mathbf{x}, t) \bigr\|_{L^2(\Omega)}^2 + \int_0^t \bigl\| T_\Delta^{(2, \epsilon)}(\mathbf{x}, \tau) \bigr\|_{H_0^1(\Omega)}^2 \, d\tau$, and we obtain inequality \eqref{eq:3.26}. Likewise, combining \eqref{eq:3.21} with \eqref{eq:3.23} and setting $C = \frac{C_{\omega}}{\lambda_2}$, we define $\Theta_2(t)$ as $\Theta_2(t) = \bigl\| \omega_\Delta^{(2, \epsilon)}(\mathbf{x}, t) \bigr\|_{L^2(\Omega)}^2 + \int_0^t \bigl\| \omega_\Delta^{(2, \epsilon)}(\mathbf{x}, \tau) \bigr\|_{H_0^1(\Omega)}^2 \, d\tau$, which yields inequality \eqref{eq:3.27}.
		\begin{equation}
			\label{eq:3.26}
			\Theta_1(t) \le C(\Omega,T^*) \Bigl( \epsilon^2 + \int_0^t \Theta_1(\tau) \, d\tau \Bigr).
		\end{equation}
		\begin{equation}
			\label{eq:3.27}
			\Theta_2(t) \le C(\Omega,T^*) \Bigl( \epsilon^2 + \int_0^t \Theta_2(\tau) \, d\tau \Bigl).
		\end{equation}
		
		By the Gronwall inequality, we have $\Theta_1(t) \le C(\Omega,T^*) \epsilon^2 e^{C(\Omega,T^*) T^*} \le C(\Omega, T^*) \epsilon^2$ and $\Theta_2(t) \le C(\Omega,T^*) \epsilon^2 e^{C(\Omega,T^*) T^*} \le C(\Omega, T^*) \epsilon^2$, which yield the following inequalities.
		\begin{equation}
			\label{eq:3.28}
			\bigl\| T_\Delta^{(2, \epsilon)}(\mathbf{x}, t) \bigr\|_{L^2(\Omega)}^2 + \int_0^t \bigl\| T_\Delta^{(2, \epsilon)}(\mathbf{x}, \tau) \bigr\|_{H_0^1(\Omega)}^2 \, d\tau \le C(\Omega, T^*) \epsilon^2.
		\end{equation}
		\begin{equation}
			\label{eq:3.29}
			\bigl\| \omega_\Delta^{(2, \epsilon)}(\mathbf{x}, t) \bigr\|_{L^2(\Omega)}^2 + \int_0^t \bigl\| \omega_\Delta^{(2, \epsilon)}(\mathbf{x}, \tau) \bigr\|_{H_0^1(\Omega)}^2 \, d\tau \le C(\Omega, T^*) \epsilon^2.
		\end{equation}
		
		Then, applying the inequality between the arithmetic mean and the quadratic mean, i.e., $\frac{a + b}{2} \le \sqrt{\frac{a^2 + b^2}{2}}$,  to the left-hand sides of \eqref{eq:3.28}-\eqref{eq:3.29} and taking the square root of both sides yields the following inequalities.
		\begin{equation}
			\label{eq:3.30}
			\bigl\| T_\Delta^{(2, \epsilon)} \bigr\|_{L^2(\Omega)} + \bigl\| T_\Delta^{(2, \epsilon)} \bigr\|_{L^2(0, t; H_0^1(\Omega))} \le C(\Omega, T^*) \epsilon.
		\end{equation}
		\begin{equation}
			\label{eq:3.31}
			\bigl\| \omega_\Delta^{(2, \epsilon)} \bigr\|_{L^2(\Omega)} + \bigl\| \omega_\Delta^{(2, \epsilon)} \bigr\|_{L^2(0, t; H_0^1(\Omega))} \le C(\Omega, T^*) \epsilon.
		\end{equation}
		
		Taking advantage of the arbitrariness of the time variable $t$ in \eqref{eq:3.30} and \eqref{eq:3.31}, we obtain the estimates \eqref{eq:3.6a} and \eqref{eq:3.6b} for the temperature and moisture fields, respectively.
		
		Next, we prove \eqref{eq:3.6c}. To this end, rewrite \eqref{eq:3.15} as the following equation:
		\begin{equation}
			\label{eq:3.34}
			\int_\Omega C_{ijkl}^\epsilon \frac{\partial u_{\Delta k}^{(2, \epsilon)}}{\partial x_l} \frac{\partial u_{\Delta i}^{(2, \epsilon)}}{\partial x_j} \, d\Omega = \int_\Omega \bigl( \alpha_{ij}^\epsilon T_\Delta^{(2, \epsilon)} + \beta_{ij}^\epsilon \omega_\Delta^{(2, \epsilon)} \bigr) \frac{\partial u_{\Delta i}^{(2, \epsilon)}}{\partial x_j} \, d\Omega + \epsilon \int_\Omega \theta_i u_{\Delta i}^{(2, \epsilon)} \, d\Omega.
		\end{equation}
		
		By virtue of assumption (A) and applying the Poincar$\rm{\acute{e}}$-Friedrichs inequality, a straightforward transformation of the left-hand side of \eqref{eq:3.34} readily yields the following inequality:
		\begin{equation}
			\label{eq:3.35}
			\int_\Omega C_{ijkl}^\epsilon \frac{\partial u_{\Delta k}^{(2, \epsilon)}}{\partial x_l} \frac{\partial u_{\Delta i}^{(2, \epsilon)}}{\partial x_j} \, d\Omega \ge \lambda_3(\Omega) \bigl\| u_{\Delta i}^{(2, \epsilon)} \bigr\|_{H_0^1(\Omega)}^2.
		\end{equation}
		
		Subsequently, by applying the Schwarz inequality, assumption (B), the convergence estimate \eqref{eq:3.30} for the temperature field, and the convergence estimate \eqref{eq:3.31} for the moisture field, and by transforming the right-hand side of \eqref{eq:3.34}, we obtain the following inequality:
		\begin{equation}
			\label{eq:3.36}
			\begin{aligned}
				& \int_\Omega \bigl( \alpha_{ij}^\epsilon T_\Delta^{(2, \epsilon)} + \beta_{ij}^\epsilon \omega_\Delta^{(2, \epsilon)} \bigr) \frac{\partial u_{\Delta i}^{(2, \epsilon)}}{\partial x_j} \, d\Omega + \epsilon \int_\Omega \theta_i u_{\Delta i}^{(2, \epsilon)} \, d\Omega \\
				& \le C_8 \Bigl( \bigl\| T_\Delta^{(2, \epsilon)} \bigr\|_{L^2(\Omega)} + \bigl\| \omega_\Delta^{(2, \epsilon)} \bigr\|_{L^2(\Omega)} \Bigr) \bigl\| u_{\Delta i}^{(2, \epsilon)} \bigr\|_{H_0^1(\Omega)} + \epsilon \bigl\| \theta_i \bigr\|_{L^2(\Omega)} \bigl\| u_{\Delta i}^{(2, \epsilon)} \bigr\|_{L^2(\Omega)} \\
				& \le C_9(\Omega, T^*) \epsilon \bigl\| u_{\Delta i}^{(2, \epsilon)} \bigr\|_{H_0^1(\Omega)} + C_{10}(\Omega) \epsilon \bigl\| u_{\Delta i}^{(2, \epsilon)} \bigr\|_{H_0^1(\Omega)} \\
				& \le C_u(\Omega, T^*) \epsilon \bigl\| u_{\Delta i}^{(2, \epsilon)} \bigr\|_{H_0^1(\Omega)}.
			\end{aligned}
		\end{equation}
		where $C_u(\Omega, T^*)=\max(C_9,C_{10})$.
		
		Finally, combining \eqref{eq:3.35} and \eqref{eq:3.36} and using the arbitrariness of the time variable $t$, we obtain the estimate \eqref{eq:3.6c} for the displacement field.
	\end{proof}
	
\section{Multi-scale numerical algorithm}
\label{sec:4}
	In this section, we detail the multi-scale algorithm for the nonlinear dynamic hygro-thermo-mechanical coupling problems \eqref{eq:2.1}. The proposed framework consists of microscopic cell models, a macroscopic homogenized model, and HOMS solutions, forming a closed system. Since all microscopic cell functions in \eqref{eq:2.22}-\eqref{eq:2.26} and \eqref{eq:2.31}-\eqref{eq:2.48} depend on the macroscopic temperature $T^{(0)}$ or moisture $\omega^{(0)}$, their continuity can be established using an idea similar to that in \citep{R51}. By this continuity, we need only evaluate these functions at a few representative macroscopic temperature and moisture values, rather than at all possible temperature and moisture points, and then use interpolation to obtain the required auxiliary cell functions during simulation \citep{R31, R39, R32}. In the following, we present a two-stage numerical algorithm that comprises off-line and on-line stages for the efficient simulation of the nonlinear dynamic hygro-thermo-mechanical coupling problems \eqref{eq:2.1} of heterogeneous structures, as depicted in Fig.~\ref{f1:algorithm}.
	\begin{figure}[!htb]
		\centering
		\includegraphics[width=0.9\textwidth]{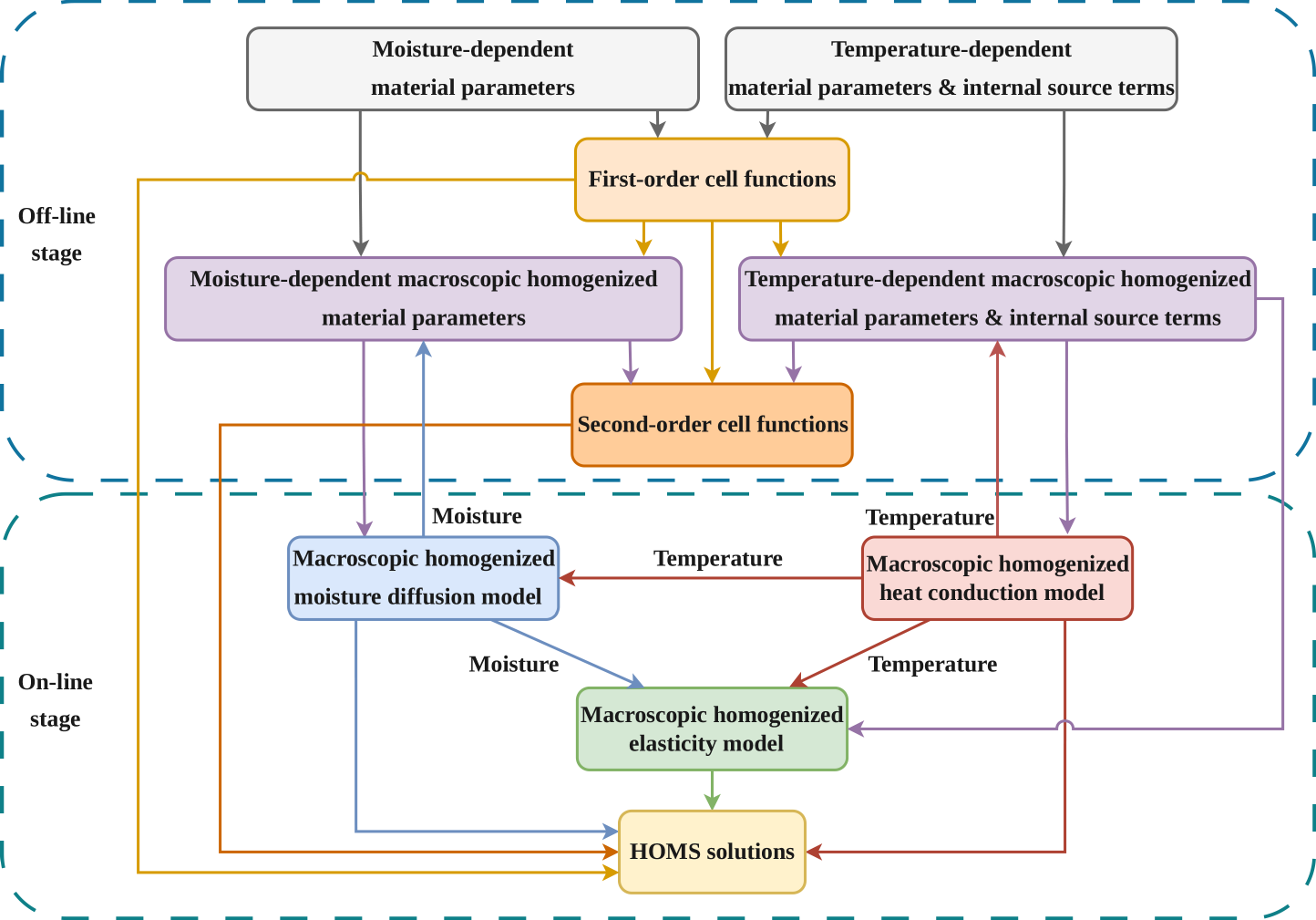}\\
		\caption{Flowchart of the two-stage multi-scale algorithm.}\label{f1:algorithm}
	\end{figure}
	
	\subsection{Off-line stage: solve the microscopic cell problems.}
	\label{sec:41}
	\begin{enumerate}
		\item[(1)] Determine the geometric configuration of PUC $Y=[0,1]^n$ in $\mathbb{R}^n (n=2,3)$, and generate a family of triangular $(n=2)$ or tetrahedral $(n=3)$ finite element meshes $J^{h_0}(Y) = \{K\}$ of $Y$, where $h_0=\max_K \{ h_K\}$. Then denote the linear conforming finite element space $V_{h_0}(Y)=\{\nu\in C^{0}(\bar{Y}):\nu|_{\partial Y}=0,\nu|_{K}\in P_{1}(K)\}\subset H_{0}^{1}(Y)$ for auxiliary cell problems.
		\item[(2)] Define computational temperature range $[T_{\min}, T_{\max}]$ and moisture range $[\omega_{\min}, \omega_{\max}]$. Then, choose a set of representative macroscopic temperatures $\bar{T}_{s_1}$ within the temperature range, and representative macroscopic moisture values $\bar{\omega}_{s_2}$ within the moisture range. Next, employ FEM to solve the first-order cell problems defined by \eqref{eq:2.22}-\eqref{eq:2.26} on $V_{h_0}(Y)$ corresponding to distinct representative macroscopic temperatures $\bar{T}_{s_1}$ or moisture values $\bar{\omega}_{s_2}$.
		\item[(3)] The macroscopic homogenized material parameters $\hat{S}(T^{(0)})$, $\hat{k}_{ij}(T^{(0)})$, $\hat{g}_{ij}(\omega^{(0)})$, $\hat{C}_{ijkl}(T^{(0)})$, $\hat{\alpha}_{ij}(T^{(0)})$, $\hat{\beta}_{ij}(T^{(0)})$, and homogenized internal source terms $\hat{Q}_{hyd}(T^{(0)})$ and $\hat{S}_{hyd}(T^{(0)})$ are evaluated by formula \eqref{eq:2.28} associated with distinct macroscopic temperatures $\bar{T}_{s_1}$ and moisture values $\bar{\omega}_{s_2}$, respectively.
		\item[(4)] Employing the same mesh as first-order cell problems, the second-order cell problems defined by \eqref{eq:2.31}-\eqref{eq:2.48}, which correspond to distinct macroscopic representative macroscopic temperatures $\bar{T}_{s_1}$ or moisture values $\bar{\omega}_{s_2}$, are evaluated on $V_{h_0}(Y)$ using FEM.
	\end{enumerate}
	
	\subsection{On-line stage: solve the macroscopic homogenized problem and compute the HOMS solutions.}
	\label{sec:42}
	\begin{enumerate}
		\item[(1)] Let $J^{h_1}(\Omega)= \{e\}$ be a triangular or tetrahedral  finite element mesh of the macroscopic homogenized region $\Omega$, where $h_1=\max_e \{ h_e \}$. Then define the linear conforming finite element spaces $V^T_{h_1}(\Omega)=\{\nu\in C^0(\bar{\Omega}):\nu|_{\partial\Omega_{T}}=0,\nu|_e\in P_1(e)\}\subset H^1(\Omega)$, $V^\omega_{h_1}(\Omega)=\{\nu\in C^0(\bar{\Omega}):\nu|_{\partial\Omega_{\omega}}=0,\nu|_e\in P_1(e)\}\subset H^1(\Omega)$ and $V^u_{h_1}(\Omega)=\{\nu\in C^0(\bar{\Omega}):\nu|_{\partial\Omega_u}=0,\nu|_e\in P_1(e)\}\subset H^1(\Omega)$ for temperature, moisture and displacement fields, respectively. Moreover, the macroscopic homogenized material parameters and homogenized internal source terms can be calculated by an interpolation approach on each node $\mathbf{x}$ of $V^T_{h_1}(\Omega)$, $V^\omega_{h_1}(\Omega)$ and $V^u_{h_1}(\Omega)$.
		\item[(2)] Solve the macroscopic homogenized equations \eqref{eq:2.27} without oscillatory coefficients by mixed FDM-FEM proposed in reference \citep{R52} on a coarse mesh and with a larger time step on the computational domain $\Omega \times (0, T^*)$, which means FEM is employed in spatial discretization and FDM is used to discretize time-domain. Using the equidistant time step $\Delta t = \frac{T^*}{M}$ to discretize time-domain $(0, T^*)$ as $0 = t_0 < t_1 < \cdots < t_M = T^*$ and $t_N = N \Delta t \, (N = 0, \cdots, M)$, we denote the value of any function $F(\mathbf{x}, t)$ at time $t_N$ by $F^{N} = F(\mathbf{x}, t_N)$. Below, we describe in detail the FDM-FEM scheme for the macroscopic homogenized problem \eqref{eq:2.27}.
		\begin{equation}
			\label{eq:4.1}
			\begin{cases}
				\begin{aligned}
					& \int_\Omega \hat{S}(T^{(0), N+1}) \frac{T^{(0), N+1} - T^{(0), N}}{\Delta t} \varphi^{h_1} d\Omega + \int_\Omega \hat{k}_{ij}(T^{(0), N+1}) \frac{\partial T^{(0), N+1}}{\partial x_j} \frac{\partial \varphi^{h_1}}{\partial x_i} d\Omega \\
					& = \int_\Omega \bigl[ h^{N+1} + \hat{Q}_{hyd}(T^{(0), N+1}) \bigr] \varphi^{h_1} \, d\Omega + \int_{\partial \Omega_q} \bar{q}^{N+1} \varphi^{h_1} \, ds,\, \forall \varphi^{h_1} \in V_{h_1}^T(\Omega), \\
					&T^{(0), N+1} = \hat{T}(\mathbf{x}, t_{N+1}), \quad \text{on } \partial \Omega_T.
				\end{aligned}
			\end{cases}
		\end{equation}
		\begin{equation}
			\label{eq:4.2}
			\begin{cases}
				\begin{aligned}
					& \int_\Omega \frac{\omega^{(0), N+1} - \omega^{(0), N}}{\Delta t} \psi^{h_1} \, d\Omega + \int_\Omega \hat{g}_{ij}(\omega^{(0), N+1}) \frac{\partial \omega^{(0), N+1}}{\partial x_j} \frac{\partial \psi^{h_1}}{\partial x_i} \, d\Omega \\
					& = \int_\Omega \bigl[ m^{N+1} - \hat{S}_{hyd}(T^{(0), N+1}) \bigr] \psi^{h_1} d\Omega + \int_{\partial \Omega_d} \bar{d}^{N+1} \psi^{h_1} ds,\, \forall \psi^{h_1} \in V_{h_1}^\omega(\Omega), \\
					& \omega^{(0), N+1} = \hat{\omega}(\mathbf{x}, t_{N+1}), \quad \text{on } \partial \Omega_\omega.
				\end{aligned}
			\end{cases}
		\end{equation}
		\begin{equation}
			\label{eq:4.3}
			\begin{cases}
				\begin{aligned}
					& \int_\Omega \hat{C}_{ijkl}(T^{(0), N+1}) \frac{\partial u_k^{(0), N+1}}{\partial x_l} \frac{\partial v_i^{h_1}}{\partial x_j} \, d\Omega \\
					& - \int_\Omega \bigl[ \hat{\alpha}_{ij}(T^{(0), N+1})(T^{(0), N+1} - \tilde{T}) + \hat{\beta}_{ij}(T^{(0), N+1})(\omega^{(0), N+1} - \tilde{\omega}) \bigr] \frac{\partial v_i^{h_1}}{\partial x_j} d\Omega \\
					& = \int_\Omega f_i^{N+1} v_i^{h_1} \, d\Omega + \int_{\partial \Omega_\sigma} \bar{\sigma}_i^{N+1} v_i^{h_1} \, ds, \quad \forall \bm{v}^{h_1} \in (V_{h_1}^u(\Omega))^n, \\
					& \bm{u}^{(0), N+1} = \hat{\bm{u}}(\mathbf{x}, t_{N+1}), \quad \text{on } \partial \Omega_u.
				\end{aligned}
			\end{cases}
		\end{equation}
		
		Next, we present a direct iteration method to simulate the nonlinear systems \eqref{eq:4.1}-\eqref{eq:4.3}.
		\begin{enumerate}
			\item[Step 1: ] Initialize the macroscopic temperature and moisture as $\breve{T}^{(0)}(\mathbf{x})$ and $\breve{\omega}^{(0)}(\mathbf{x})$. Denote by $\breve{T}_\lambda(\mathbf{x})$, $\breve{\omega}_\lambda(\mathbf{x})$ and $\breve{\bm{u}}_\lambda(\mathbf{x})$ the solutions at the $\lambda$-th iteration ($\lambda \ge 1$). Set the iteration thresholds $E_{tol}^T$, $E_{tol}^\omega$ and $E_{tol}^u$ for temperature, moisture and displacement, respectively, and start the iteration.			
			\item[Step 2: ] In the $\lambda$-th iteration, employ $\breve{T}_{\lambda-1}(\mathbf{x})$ and $\breve{\omega}_{\lambda-1}(\mathbf{x})$ to linearize the nonlinear systems \eqref{eq:4.1}-\eqref{eq:4.3} as follows. First, compute the macroscopic homogenized material parameters $\hat{S}$, $\hat{k}_{ij}$ and homogenized internal source term $\hat{Q}_{hyd}$ using $\breve{T}_{\lambda-1}$ and solve the linearized temperature equation:
			\begin{equation}
				\label{eq:4.4}
				\begin{cases}
					\begin{aligned}
						& \int_\Omega \hat{S}(\breve{T}_{\lambda-1}) \frac{\breve{T}_\lambda - T^{(0), N}}{\Delta t} \varphi^{h_1} \, d\Omega + \int_\Omega \hat{k}_{ij}(\breve{T}_{\lambda-1}) \frac{\partial \breve{T}_\lambda}{\partial x_j} \frac{\partial \varphi^{h_1}}{\partial x_i} \, d\Omega \\
						& = \int_\Omega \bigl[ h^{N+1} + \hat{Q}_{hyd}(\breve{T}_{\lambda-1}) \bigr] \varphi^{h_1} d\Omega + \int_{\partial \Omega_q} \bar{q}^{N+1} \varphi^{h_1} ds, \forall \varphi^{h_1} \in V_{h_1}^T(\Omega), \\
						& \breve{T}_\lambda = \hat{T}(\mathbf{x}, t_{N+1}), \quad \text{on } \partial \Omega_T.
					\end{aligned}
				\end{cases}
			\end{equation}
			Then, using $\breve{T}_{\lambda}$ and $\breve{\omega}_{\lambda-1}$ compute the macroscopic homogenized material parameter $\hat{g}_{ij}$ and homogenized internal source term $\hat{S}_{hyd}$, and solve the linearized moisture equation:
			\begin{equation}
				\label{eq:4.5}
				\begin{cases}
					\begin{aligned}
						& \int_\Omega \frac{\breve{\omega}_\lambda - \omega^{(0), N}}{\Delta t} \psi^{h_1} \, d\Omega + \int_\Omega \hat{g}_{ij}(\breve{\omega}_{\lambda-1}) \frac{\partial \breve{\omega}_\lambda}{\partial x_j} \frac{\partial \psi^{h_1}}{\partial x_i} \, d\Omega \\
						& = \int_\Omega \bigl[ m^{N+1} - \hat{S}_{hyd}(\breve{T}_\lambda) \bigr] \psi^{h_1} \, d\Omega + \int_{\partial \Omega_d} \bar{d}^{N+1} \psi^{h_1} \, ds, \forall \psi^{h_1} \in V_{h_1}^\omega(\Omega), \\
						& \breve{\omega}_\lambda = \hat{\omega}(\mathbf{x}, t_{N+1}), \quad \text{on } \partial \Omega_\omega.
					\end{aligned}
				\end{cases}
			\end{equation}
			Finally, using $\breve{T}_{\lambda}$ and $\breve{\omega}_{\lambda}$ compute the macroscopic homogenized material parameters $\hat{C}_{ijkl}$, $\hat{\alpha}_{ij}$ and $\hat{\beta}_{ij}$, and solve the linearized displacement equation:
			\begin{equation}
				\label{eq:4.6}
				\begin{cases}
					\begin{aligned}
						& \int_\Omega \hat{C}_{ijkl}(\breve{T}_\lambda) \frac{\partial \breve{u}_{\lambda, k}}{\partial x_l} \frac{\partial v_i^{h_1}}{\partial x_j} d\Omega - \int_\Omega \bigl[ \hat{\alpha}_{ij}(\breve{T}_\lambda)(\breve{T}_\lambda - \tilde{T}) + \hat{\beta}_{ij}(\breve{T}_\lambda)(\breve{\omega}_\lambda - \tilde{\omega})\bigr] \frac{\partial v_i^{h_1}}{\partial x_j} d\Omega \\
						& = \int_\Omega f_i^{N+1} v_i^{h_1} \, d\Omega + \int_{\partial \Omega_\sigma} \bar{\sigma}_i^{N+1} v_i^{h_1} \, ds, \quad \forall v^{h_1} \in (V_{h_1}^u(\Omega))^n, \\
						& \breve{\bm{u}}_\lambda = \bm{\hat{u}}(\mathbf{x}, t_{N+1}), \quad \text{on } \partial \Omega_u.
					\end{aligned}
				\end{cases}
			\end{equation}
			\item[Step 3: ] if $\left\| \breve{T}_\lambda - \breve{T}_{\lambda-1} \right\|_{L^\infty(\Omega)} \le E_{tol}^T$, $\left\| \breve{\omega}_\lambda - \breve{\omega}_{\lambda-1} \right\|_{L^\infty(\Omega)} \le E_{tol}^\omega$ and $\left\| \breve{\bm{u}}_\lambda - \breve{\bm{u}}_{\lambda-1} \right\|_{L^\infty(\Omega)} \le E_{tol}^u$, terminate the iteration; otherwise set $\lambda=\lambda+1$, and return to Step 2.
			\item[Step 4: ] Set $T^{(0), N+1} = T_{sat}$, $\omega^{(0), N+1} = \omega_{sat}$ and $\bm{u}^{(0), N+1} = \bm{u}_{sat}$, where $T_{sat}$, $\omega_{sat}$ and $\bm{u}_{sat}$ are the solutions of linear systems \eqref{eq:4.4}-\eqref{eq:4.6} that satisfy the respective iteration thresholds $E_{tol}^T$, $E_{tol}^\omega$ and $E_{tol}^u$.
		\end{enumerate}
		\item[(3)] At any point $(\mathbf{x}, t) \in \Omega \times (0, T^*)$, we compute the first‑order cell functions, second‑order cell functions and macroscopic homogenized solutions by interpolation.
		\item[(4)] In the HOMS formulas \eqref{eq:2.52}-\eqref{eq:2.54}, the spatial derivatives $\frac{\partial T^{(0)}}{\partial x_{\alpha_1}}$, $\frac{\partial^2 T^{(0)}}{\partial x_{\alpha_1} \partial x_{\alpha_2}}$, $\frac{\partial \omega^{(0)}}{\partial x_{\alpha_1}}$, $\frac{\partial^2 \omega^{(0)}}{\partial x_{\alpha_1} \partial x_{\alpha_2}}$, $\frac{\partial u_h^{(0)}}{\partial x_{\alpha_1}}$ and $\frac{\partial^2 u_h^{(0)}}{\partial x_{\alpha_1} \partial x_{\alpha_2}}$ are evaluated by the average technique on relative elements \citep{R43, R44, R54}, and the temporal derivative $\frac{\partial T^{(0)}}{\partial t}$ is evaluated using difference schemes at each time step.
		\item[(5)] Finally, the temperature field $T^{(2, \epsilon)}(\mathbf{x}, t)$, moisture field $\omega^{(2, \epsilon)}(\mathbf{x}, t)$ and displacement field $u_i^{(2, \epsilon)}(\mathbf{x}, t)$ are obtained from formulas \eqref{eq:2.52}-\eqref{eq:2.54}, respectively. Moreover, high‑accuracy HOMS solutions can be achieved by further employing higher‑order interpolation and post‑processing techniques \citep{R54, R55}.
	\end{enumerate}
	
\section{Error estimate for the multi-scale numerical algorithm}
\label{sec:4.4}
	The total error of the proposed multi-scale algorithm comprises not only the error arising from multi-scale modeling, but also two additional components: the numerical error from solving auxiliary cell problems \eqref{eq:2.22}-\eqref{eq:2.26} and \eqref{eq:2.31}-\eqref{eq:2.48}, and that from solving the macroscopic homogenized problem \eqref{eq:2.27}, using FEM. Before giving the detailed error estimation, we prepare some lemmas in advance.
	\begin{lemma}
		\label{lem:4.1}
		Let $\mathcal{H}_{\alpha_1}^{h_0}$, $\mathcal{J}_{\alpha_1}^{h_0}$, $\mathcal{X}_{ih}^{\alpha_1, h_0}$, $\mathcal{M}_{i}^{h_0}$, $\mathcal{N}_{i}^{h_0}$,$\mathcal{S}^{h_0}$,  $\mathcal{H}_{\alpha_1\alpha_2}^{h_0}$, $\mathcal{R}_{\alpha_1}^{h_0}$, $\mathcal{E}_{\alpha_1\alpha_2}^{h_0}$, $\mathbb{Q}^{h_0}$, $\mathcal{J}_{\alpha_1\alpha_2}^{h_0}$,  $\mathcal{I}_{\alpha_1}^{h_0}$, $\mathcal{F}_{\alpha_1\alpha_2}^{h_0}$, $\mathbb{S}^{h_0}$,  $\mathcal{X}_{ih}^{\alpha_1\alpha_2, h_0}$, $\mathcal{Q}_{ih}^{\alpha_1, h_0}$, $\mathcal{P}_{ih}^{\alpha_1\alpha_2, h_0}$, $\mathcal{W}_{i}^{h_0}$,  $\mathcal{Z}_i^{\alpha_1, h_0}$, $\mathcal{A}_i^{\alpha_1, h_0}$, $\mathcal{V}_{i}^{h_0}$, $\mathcal{G}_i^{\alpha_1, h_0}$ and $\mathcal{B}_i^{\alpha_1,h_0}$ denote the finite element solutions of the first-order and second-order cell functions, respectively. If, for any fixed macroscopic temperature $T^{(0)}$ and moisture $\omega^{(0)}$, all microscopic cell functions belong to $H^2({Y})$, then the following inequality holds:
		\begin{equation}
			\label{eq:5.1}
			\begin{aligned}
				&\bigl\|\mathcal{J}_{\alpha_1}^{h_0}(\mathbf{y},\omega^{(0)})-\mathcal{J}_{\alpha_1}(\mathbf{y},\omega^{(0)})\bigr\|_{H^m(Y)}\leq Ch_0^{2-m}\bigl\|\mathcal{J}_{\alpha_1}(\mathbf{y},\omega^{(0)})\bigr\|_{H^2(Y)},
			\end{aligned}
		\end{equation}
		where $m=0,1$ and $C$ denotes the finite element estimate constant independent of $h_0$ and dependent on $Y$. Moreover, other microscopic cell functions have the similar error estimates to the $\mathcal{J}_{\alpha_1}^{h_0}$.
	\end{lemma}
	\begin{proof}
		By employing the classical finite element theory, the above inequalities are easily obtained.
	\end{proof}
	
	\begin{lemma}
		\label{lem:4.2}
		Let $\hat k_{ij}^{h_0}(T^{(0)})$, $\hat g_{ij}^{h_0}(\omega^{(0)})$, $\hat C_{ijkl}^{h_0}(T^{(0)})$, $\hat \alpha_{ij}^{h_0}(T^{(0)})$ and $\hat \beta_{ij}^{h_0}(T^{(0)})$ be the FE approximations of the corresponding macroscopic homogenized parameters, the following results hold.
		\begin{equation}
			\label{eq:5.2}
			\begin{aligned}
				& \bigl|\hat k_{ij}^{h_0}(T^{(0)})-\hat k_{ij}(T^{(0)})\bigr|\leq Ch_0^2\bigl\| \mathcal{H}_{i}(\mathbf{y}, T^{(0)})\bigr\|_{H^2(Y)}\bigl\| \mathcal{H}_{j}(\mathbf{y}, T^{(0)})\bigr\|_{H^2(Y)},\\
				& \underline{\kappa} | \bm{\xi} |^2 \leq \hat k_{ij}^{h_0}(T^{(0)}) \xi_i \xi_j \leq \overline{\kappa} | \bm{\xi} |^2,
			\end{aligned}
		\end{equation}
		\begin{equation}
			\label{eq:5.3}
			\begin{aligned}
				& \bigl|\hat g_{ij}^{h_0}(\omega^{(0)})-\hat g_{ij}(\omega^{(0)})\bigr|\leq Ch_0^2\bigl\| \mathcal{J}_{i}(\mathbf{y}, \omega^{(0)})\bigr\|_{H^2(Y)}\bigl\| \mathcal{J}_{j}(\mathbf{y}, \omega^{(0)})\bigr\|_{H^2(Y)},\\
				& \underline{\kappa} | \bm{\xi} |^2 \leq \hat g_{ij}^{h_0}(\omega^{(0)}) \xi_i \xi_j \leq \overline{\kappa} | \bm{\xi} |^2,
			\end{aligned}
		\end{equation}
		\begin{equation}
			\label{eq:5.4}
			\begin{aligned}
				& \bigl|\hat C_{ijkl}^{h_0}(T^{(0)})-\hat C_{ijkl}(T^{(0)})\bigr|\leq Ch_0^2\bigl\| \mathbcal{X}_{i}^{j}(\mathbf{y}, T^{(0)})\bigr\|_{(H^2(Y))^n}\bigl\| \mathbcal{X}_{k}^{l}(\mathbf{y}, T^{(0)})\bigr\|_{(H^2(Y))^n},\\
				& \underline{\kappa} \eta_{ij} \eta_{ij}  \leq \hat C_{ijkl}^{h_0}(T^{(0)}) \eta_{ij} \eta_{kl} \leq \overline{\kappa} \eta_{ij} \eta_{ij},
			\end{aligned}
		\end{equation}
		\begin{equation}
			\label{eq:5.5}
			\begin{aligned}
				& \bigl|\hat \alpha_{ij}^{h_0}(T^{(0)})-\hat \alpha_{ij}(T^{(0)})\bigr|\leq Ch_0^2\bigl\| \mathbcal{M}(\mathbf{y}, T^{(0)})\bigr\|_{(H^2(Y))^n} \bigl\| \mathbcal{X}_{i}^{j}(\mathbf{y}, T^{(0)})\bigr\|_{(H^2(Y))^n},\\
				& \underline{\kappa} | \bm{\xi} |^2 \leq \hat \alpha_{ij}^{h_0}(T^{(0)})  \xi_i \xi_j \leq \overline{\kappa} | \bm{\xi} |^2,
			\end{aligned}
		\end{equation}
		\begin{equation}
			\label{eq:5.6}
			\begin{aligned}
				& \bigl|\hat \beta_{ij}^{h_0}(T^{(0)})-\hat \beta_{ij}(T^{(0)})\bigr|\leq Ch_0^2\bigl\| \mathbcal{N}(\mathbf{y}, T^{(0)})\bigr\|_{(H^2(Y))^n} \bigl\| \mathbcal{X}_{i}^{j}(\mathbf{y}, T^{(0)})\bigr\|_{(H^2(Y))^n},\\
				& \underline{\kappa} | \bm{\xi} |^2 \leq \hat \beta_{ij}^{h_0}(T^{(0)}) \xi_i \xi_j \leq \overline{\kappa} | \bm{\xi} |^2,
			\end{aligned}
		\end{equation}
		where $C$ is a constant independent of $h_0$.
	\end{lemma}
	\begin{proof}
		Using the definitions of the macroscopic homogenized material parameters in \eqref{eq:2.28}, together with assumption (A) and Lemma \ref{lem:4.1}, we obtain
		\begin{equation}
			\label{eq:5.7}
			\begin{aligned}
				&\bigl|\hat g_{ij}^{h_0}(\omega^{(0)})-\hat g_{ij}(\omega^{(0)})\bigr|\\
				&=\biggl|\frac{1}{|{Y}|}{\int_{{Y}}}\big(g_{ij}^{(0)} + g_{ik}^{(0)} \frac{\partial \mathcal{J}_{j}^{h_0}}{\partial y_{k}}\big)d{Y}-\frac{1}{|{Y}|}{\int_{{Y}}}\big(g_{ij}^{(0)} + g_{ik}^{(0)} \frac{\partial \mathcal{J}_{j}}{\partial y_{k}}\big)d{Y}\biggr|\\
				&=\biggl|\frac{1}{|{Y}|}{\int_{{Y}}}{g_{ik}^{(0)}{\frac{\partial \big(\mathcal{J}_j^{h_0}-\mathcal{J}_j\big)}{\partial y_{k}}}}d{Y}\biggr|\\
				&=\frac{1}{|{Y}|}\biggl|-{\int_{{Y}}}\frac{\partial \mathcal{J}_i}{\partial y_{\alpha_1}}g_{\alpha_1\alpha_2}^{(0)}{\frac{\partial }{\partial y_{\alpha_2}}\big(\mathcal{J}_j^{h_0}-\mathcal{J}_j\big)}d{Y}\biggr|\\
				&=\frac{1}{|{Y}|}\biggl|{\int_{{Y}}}\frac{\partial \mathcal{J}_i^{h_0}}{\partial y_{\alpha_1}}g_{\alpha_1\alpha_2}^{(0)}{\frac{\partial }{\partial y_{\alpha_2}}\big(\mathcal{J}_j^{h_0}-\mathcal{J}_j\big)}d{Y}-{\int_{{Y}}}\frac{\partial \mathcal{J}_i}{\partial y_{\alpha_1}}g_{\alpha_1\alpha_2}^{(0)}{\frac{\partial }{\partial y_{\alpha_2}}\big(\mathcal{J}_j^{h_0}-\mathcal{J}_j\big)}d{Y}\biggr|\\
				&=\frac{1}{|{Y}|}\biggl|{\int_{{Y}}}\frac{\partial}{\partial y_{\alpha_1}}\big(\mathcal{J}_i^{h_0}-\mathcal{J}_i\big)g_{\alpha_1\alpha_2}^{(0)}{\frac{\partial }{\partial y_{\alpha_2}}\big(\mathcal{J}_j^{h_0}-\mathcal{J}_j\big)}d{Y}\biggr|\\
				&\leq C\bigl\|\mathcal{J}_i^{h_0}-\mathcal{J}_i\bigr\|_{H^1({Y})}\bigl\|\mathcal{J}_j^{h_0}-\mathcal{J}_j\bigr\|_{H^1({Y})}\leq Ch_0^2\bigl\|\mathcal{J}_i\bigr\|_{H^2({Y})}\bigl\|\mathcal{J}_j\bigr\|_{H^2({Y})}.
			\end{aligned}
		\end{equation}
		Furthermore, choosing a sufficiently small $h_0>0$ satisfies
		\begin{equation}
			\label{eq:5.8}
			Ch_0^2\bigl\|\mathcal{J}_i(\mathbf{y},\omega^{(0)})\bigr\|_{H^2({Y})}\bigl\| \mathcal{J}_j(\mathbf{y},\omega^{(0)})\bigr\|_{H^2({Y})}\leq\underline{\varsigma}/2.
		\end{equation}
		Hence, it can be verified that the lower bound in \eqref{eq:5.3} holds
		\begin{equation}
			\label{eq:5.9}
			\hat g_{ij}^{h_0}(\omega^{(0)})\xi_i\xi_j=\hat g_{ij}(\omega^{(0)})\xi_i\xi_j+\bigl(\hat g_{ij}^{h_0}(\omega^{(0)})-\hat g_{ij}(\omega^{(0)})\bigr)\xi_i\xi_j\geq(\underline{\varsigma}-\underline{\varsigma}/2)\xi_i\xi_i=\underline{\kappa}|\bm{\xi}|^2,
		\end{equation}
		where $\underline{\kappa}=\underline{\varsigma}/2$ is a constant independent of $h_0$. Moreover, the upper bound in \eqref{eq:5.3} can be readily obtained by taking $\overline{\kappa}=\overline{\varsigma}+\underline{\varsigma}/2$. Finally, following the same argument, we arrive at the results \eqref{eq:5.2} and \eqref{eq:5.4}-\eqref{eq:5.6}.
	\end{proof}
	
	According to Lemmas \ref{lem:4.1} and \ref{lem:4.2}, the values of the macroscopic homogenized material parameters $\hat k_{ij}^{h_0}$, $\hat g_{ij}^{h_0}$, $\hat C_{ijkl}^{h_0}$, $\hat \alpha_{ij}^{h_0}$ and $\hat \beta_{ij}^{h_0}$ are determined by the finite element computation of the first-order cell functions $\mathcal{H}_{\alpha_1}^{h_0}$, $\mathcal{J}_{\alpha_1}^{h_0}$, $\mathcal{X}_{ih}^{\alpha_1, h_0}$, $\mathcal{M}_{i}^{h_0}$ and $\mathcal{N}_{i}^{h_0}$. Therefore, in practice, we need to numerically solve the modified homogenized equations as below.
	\begin{equation}
		\label{eq:5.10}
		\begin{cases}
			\begin{aligned}
				& \hat S(T^{(0,h_0)}) \frac{\partial T^{(0,h_0)}}{\partial t} -\frac{\partial}{\partial x_i}\Bigl(\hat{k}_{ij}^{h_0}(T^{(0,h_0)}) \frac{\partial T^{(0, h_0)}}{\partial x_j}\Bigr) = h +\hat Q_{hyd}(T^{(0,h_0)}), \text{in } \Omega \times (0, T^* ),\\
				& \frac{\partial \omega^{(0,h_0)}}{\partial t}-\frac{\partial}{\partial x_i}\Bigl(\hat{g}_{ij}^{h_0}(\omega^{(0,h_0)}) \frac{\partial \omega^{(0, h_0)}}{\partial x_j}\Bigr) = m -\hat S_{hyd}(T^{(0,h_0)}),\text{in } \Omega \times (0, T^* ),\\
				& - \frac{\partial}{\partial x_j}\Bigl(\hat{C}_{ijkl}^{h_0}(T^{(0,h_0)}) \frac{\partial u_{k}^{(0, h_0)}}{\partial x_l}
				- \hat{\alpha}_{ij}^{h_0}( T^{(0,h_0)}) ( T^{(0, h_0)}-\tilde{T})
				- \hat{\beta}_{ij}^{h_0}( T^{(0,h_0)}) (\omega^{(0, h_0)}-\tilde{\omega})\Bigr) \\
				&\quad = f_i,\text{in } \Omega \times (0, T^*),\\
				& T^{(0, h_0)}(\mathbf{x},t) = \hat{T}(\mathbf{x},t),\;\text{on }\;\partial \Omega_{T} \times (0,T^*),\\
				& \hat{k}_{ij}^{h_0}(T^{(0, h_0)}) \frac{\partial T^{(0, h_0)}}{\partial x_{j}}n_i=\overline{q}(\mathbf{x},t),\;\text{on }\;\partial \Omega_{q} \times (0,T^*),\\
				& \omega^{(0, h_0)}(\mathbf{x},t) = \hat{\omega}(\mathbf{x},t),\;\text{on }\;\partial \Omega_{
					\omega} \times (0,T^*),\\
				& \hat{g}_{ij}^{h_0}(\omega^{(0, h_0)}) \frac{\partial \omega^{(0, h_0)}}{\partial x_{j}}n_i=\overline{d}(\mathbf{x},t),\;\text{on }\;\partial \Omega_{d} \times (0,T^*),\\
				& \bm{u}^{(0, h_0)}(\mathbf{x},t) = \hat{\bm{u}}(\mathbf{x},t),\;\text{on }\;\partial \Omega_{u} \times (0,T^*),\\
				& \Bigl[\hat{C}_{ijkl}^{h_0}( T^{(0, h_0)}) \frac{\partial u_{k}^{(0, h_0)}}{\partial x_l}
				- \hat{\alpha}_{ij}^{h_0}( T^{(0, h_0)}) ( T^{(0, h_0)}-\tilde{T})
				- \hat{\beta}_{ij}^{h_0}( T^{(0, h_0)}) (\omega^{(0, h_0)}-\tilde{\omega})\Bigr] n_{j}\\
				&\quad =\overline{\sigma}_{i}(\mathbf{x},t),\;\text{on }\;\partial \Omega_{\sigma} \times (0,T^*),\\
				& T^{(0, h_0)}(\mathbf{x},0) = \tilde T,\quad
				\omega^{(0, h_0)}(\mathbf{x},0) = \tilde \omega,\quad
				\bm{u}^{(0, h_0)}(\mathbf{x},0) = \tilde{\bm{u}},\quad \text{in } \Omega.
			\end{aligned}
		\end{cases}
	\end{equation}
	\begin{lemma}
		\label{lem:4.3}
		Let $T^{(0, h_0)}$, $\omega^{(0, h_0)}$ and $u_{i}^{(0, h_0)}$ represent the exact solutions of the revised macroscopic homogenized equations \eqref{eq:5.10}, the following estimates hold
		\begin{equation}
			\label{eq:5.11}
			\begin{aligned}
				\| T^{(0,h_0)}-T^{(0)} \|_{L^\infty (0,T^*;L^2(\Omega))} + \| T^{(0,h_0)}-T^{(0)} \|_{L^2 (0,T^*;H^1(\Omega))} \leq Ch_0^2,
			\end{aligned}
		\end{equation}
		\begin{equation}
			\label{eq:5.12}
			\begin{aligned}
				\| \omega^{(0,h_0)}-\omega^{(0)} \|_{L^\infty (0,T^*;L^2(\Omega))} + \| \omega^{(0,h_0)}-\omega^{(0)} \|_{L^2 (0,T^*;H^1(\Omega))} \leq Ch_0^2,
			\end{aligned}
		\end{equation}
		\begin{equation}
			\label{eq:5.13}
			\|\bm{u}^{(0,h_0)}-\bm{u}^{(0)}\|_{L^\infty (0,T^*;(H^1(\Omega))^n)}\leq Ch_0^2,
		\end{equation}
		where $C$ is a constant independent of $h_0$.
	\end{lemma}
	\begin{proof}
		Through subtracting macroscopic homogenized equations in \eqref{eq:2.27} from corresponding equations in \eqref{eq:5.10}, one can directly check that
		\begin{equation}
			\begin{aligned}
				\label{eq:5.14}
				& \hat S(T^{(0,h_0)}) \frac{\partial (T^{(0,h_0)} - T^{(0)})}{\partial t} - \frac{\partial}{\partial x_i} \Bigl( \hat k_{ij}^{h_0}(T^{(0,h_0)})
				\frac{\partial (T^{(0,h_0)} - T^{(0)})}{\partial x_j} \Bigr) \\
				&= \bigl(\hat S(T^{(0)}) - \hat S(T^{(0,h_0)}) \bigr)
				\frac{\partial T^{(0)}}{\partial t}  -\frac{\partial}{\partial x_i}\Bigl[ \bigl( \hat k_{ij}(T^{(0)}) - \hat k_{ij}^{h_0}(T^{(0)}) \bigr) \frac{\partial T^{(0)}}{\partial x_j} \Bigr] \\
				& -\frac{\partial}{\partial x_i}\Bigl[ \bigl( \hat k_{ij}^{h_0}(T^{(0)}) - \hat k_{ij}^{h_0}(T^{(0,h_0)}) \bigr) \frac{\partial T^{(0)}}{\partial x_j} \Bigr] + \hat Q_{hyd}(T^{(0,h_0)}) - \hat Q_{hyd}(T^{(0)}).
			\end{aligned}
		\end{equation}
		\begin{equation}
			\begin{aligned}
				\label{eq:5.15}
				& \frac{\partial ( \omega^{(0,h_0)} - \omega^{(0)} )}{\partial t} - \frac{\partial}{\partial x_i} \Bigl( \hat{g}_{ij}^{h_0}(\omega^{(0,h_0)}) \frac{\partial ( \omega^{(0,h_0)} - \omega^{(0)} )}{\partial x_j} \Bigr)\\
				& = - \frac{\partial}{\partial x_i} \Bigl[ \bigl( \hat{g}_{ij}(\omega^{(0)}) - \hat{g}_{ij}^{h_0}(\omega^{(0)}) \bigr) \frac{\partial \omega^{(0)}}{\partial x_j}\Bigr] - \frac{\partial}{\partial x_i} \Bigl[ \bigl( \hat{g}_{ij}^{h_0}(\omega^{(0)}) - \hat{g}_{ij}^{h_0}(\omega^{(0,h_0)}) \bigr) \frac{\partial \omega^{(0)}}{\partial x_j} \Bigr] \\
				& + \hat{S}_{hyd}(T^{(0)}) - \hat{S}_{hyd}(T^{(0,h_0)}).
			\end{aligned}
		\end{equation}
		\begin{equation}
			\begin{aligned}
				\label{eq:5.16}
				& -\frac{\partial}{\partial x_j} \biggl( \hat{C}_{ijkl}^{h_0}(T^{(0,h_0)}) \frac{\partial ( u_k^{(0,h_0)} - u_k^{(0)} )}{\partial x_l} \biggr)\\
				&= - \frac{\partial}{\partial x_j} \Bigl[ \hat{\alpha}_{ij}^{h_0}(T^{(0,h_0)}) \bigl( T^{(0,h_0)} - T^{(0)} \bigr) \Bigr] - \frac{\partial}{\partial x_j} \Bigl[ \hat{\beta}_{ij}^{h_0}(T^{(0,h_0)}) \bigl( \omega^{(0,h_0)} - \omega^{(0)} \bigr) \Bigr] \\
				& -\frac{\partial}{\partial x_j} \Bigl[ \bigl( \hat{C}_{ijkl}(T^{(0)}) - \hat{C}_{ijkl}^{h_0}(T^{(0)}) \bigr) \frac{\partial u_k^{(0)}}{\partial x_l} \Bigr] - \frac{\partial}{\partial x_j} \Bigl[ \bigl( \hat{C}_{ijkl}^{h_0}(T^{(0)}) - \hat{C}_{ijkl}^{h_0}(T^{(0,h_0)}) \bigr) \frac{\partial u_k^{(0)}}{\partial x_l} \Bigr] \\
				& + \frac{\partial}{\partial x_j} \Bigl[ \bigl( \hat{\alpha}_{ij}(T^{(0)}) - \hat{\alpha}_{ij}^{h_0}(T^{(0)}) \bigr) \bigl( T^{(0)} - \tilde{T} \bigr) \Bigr] + \frac{\partial}{\partial x_j} \Bigl[ \bigl( \hat{\alpha}_{ij}^{h_0}(T^{(0)}) - \hat{\alpha}_{ij}^{h_0}(T^{(0,h_0)}) \bigr) \bigl( T^{(0)} - \tilde{T} \bigr) \Bigr]\\
				& + \frac{\partial}{\partial x_j} \Bigl[ \bigl( \hat{\beta}_{ij}(T^{(0)}) - \hat{\beta}_{ij}^{h_0}(T^{(0)}) \bigr) \bigl( \omega^{(0)} - \tilde{\omega} \bigr) \Bigr] + \frac{\partial}{\partial x_j} \Bigl[ \bigl( \hat{\beta}_{ij}^{h_0}(T^{(0)}) - \hat{\beta}_{ij}^{h_0}(T^{(0,h_0)}) \bigr) \bigl( \omega^{(0)} - \tilde{\omega} \bigr) \Bigr].
			\end{aligned}
		\end{equation}
		
		Furthermore, multiplying on both sides of equalities \eqref{eq:5.14}, \eqref{eq:5.15}, and \eqref{eq:5.16} by $T^{(0,h_0)}-T^{(0)}$, $\omega^{(0,h_0)}-\omega^{(0)}$ and $u_i^{(0,h_0)}-u_i^{(0)}$, respectively, and integrating on $\Omega$, it follows that
		\begin{equation}
			\begin{aligned}
				\label{eq:5.17}				
				&\frac{1}{2}\frac{\partial}{\partial t}\int_\Omega \hat S( T^{(0,h_0)}) \bigl( T^{(0,h_0)} - T^{(0)} \bigr)^2 d\Omega
				+ \int_\Omega \hat k_{ij}^{h_0}( T^{(0,h_0)})
				\frac{\partial \bigl( T^{(0,h_0)} - T^{(0)} \bigr)}{\partial x_j}
				\frac{\partial \bigl( T^{(0,h_0)} - T^{(0)} \bigr)}{\partial x_i} d\Omega \\
				& = \frac{1}{2}\int_\Omega \frac{\partial \hat S( T^{(0,h_0)})}{\partial t}
				\bigl( T^{(0,h_0)} - T^{(0)} \bigr)^2 d\Omega
				+ \int_\Omega \bigl(\hat S( T^{(0)}) - \hat S( T^{(0,h_0)}) \bigr)
				\frac{\partial T^{(0)}}{\partial t} \bigl( T^{(0,h_0)} - T^{(0)} \bigr) d\Omega \\
				& + \int_\Omega \bigl( \hat k_{ij}(T^{(0)}) - \hat k_{ij}^{h_0}(T^{(0)}) \bigr)
				\frac{\partial T^{(0)}}{\partial x_j}
				\frac{\partial (T^{(0,h_0)} - T^{(0)})}{\partial x_i} d\Omega \\
				& + \int_\Omega \bigl( \hat k_{ij}^{h_0}(T^{(0)}) - \hat k_{ij}^{h_0}(T^{(0,h_0)}) \bigr)
				\frac{\partial T^{(0)}}{\partial x_j}
				\frac{\partial (T^{(0,h_0)} - T^{(0)})}{\partial x_i} d\Omega \\
				& + \int_\Omega \bigl( \hat Q_{hyd}(T^{(0,h_0)}) - \hat Q_{hyd}(T^{(0)}) \bigr)
				\bigl( T^{(0,h_0)} - T^{(0)} \bigr) d\Omega.
			\end{aligned}
		\end{equation}
		\begin{equation}
			\begin{aligned}
				\label{eq:5.18}
				& \frac{1}{2}\frac{\partial}{\partial t}\int_\Omega \bigl( \omega^{(0,h_0)} - \omega^{(0)} \bigr)^2 d\Omega + \int_\Omega  \hat{g}_{ij}^{h_0}(\omega^{(0,h_0)}) \frac{\partial \bigl( \omega^{(0,h_0)} - \omega^{(0)} \bigr)}{\partial x_j} \frac{\partial \bigl( \omega^{(0,h_0)} - \omega^{(0)} \bigr)}{\partial x_i} d\Omega \\
				& = \int_\Omega \bigl( \hat{g}_{ij}(\omega^{(0)}) - \hat{g}_{ij}^{h_0}(\omega^{(0)}) \bigr) \frac{\partial \omega^{(0)}}{\partial x_j} \frac{\partial \bigl( \omega^{(0,h_0)} - \omega^{(0)} \bigr)}{\partial x_i} d\Omega \\
				& + \int_\Omega \bigl( \hat{g}_{ij}^{h_0}(\omega^{(0)}) - \hat{g}_{ij}^{h_0}(\omega^{(0,h_0)}) \bigr) \frac{\partial \omega^{(0)}}{\partial x_j} \frac{\partial \bigl( \omega^{(0,h_0)} - \omega^{(0)} \bigr)}{\partial x_i} d\Omega\\
				& + \int_\Omega \bigl( \hat{S}_{hyd}(T^{(0)}) - \hat{S}_{hyd}(T^{(0,h_0)}) \bigr) \bigl( \omega^{(0,h_0)} - \omega^{(0)} \bigr) d\Omega.
			\end{aligned}
		\end{equation}
		\begin{equation}
			\begin{aligned}
				\label{eq:5.19}
				& \int_\Omega \hat{C}_{ijkl}^{h_0}(T^{(0,h_0)}) \frac{\partial \bigl( u_k^{(0,h_0)} - u_k^{(0)} \bigr)}{\partial x_l} \frac{\partial \bigl( u_i^{(0,h_0)} - u_i^{(0)} \bigr)}{\partial x_j} d\Omega \\
				& = \int_\Omega \hat{\alpha}_{ij}^{h_0}(T^{(0,h_0)}) \bigl( T^{(0,h_0)} - T^{(0)} \bigr) \frac{\partial \bigl( u_i^{(0,h_0)} - u_i^{(0)} \bigr)}{\partial x_j} d\Omega \\
				& + \int_\Omega \hat{\beta}_{ij}^{h_0}(T^{(0,h_0)}) \bigl( \omega^{(0,h_0)} - \omega^{(0)} \bigr) \frac{\partial \bigl( u_i^{(0,h_0)} - u_i^{(0)} \bigr)}{\partial x_j} d\Omega \\
				& + \int_\Omega \bigl( \hat{C}_{ijkl}(T^{(0)}) - \hat{C}_{ijkl}^{h_0}(T^{(0)}) \bigr) \frac{\partial u_k^{(0)}}{\partial x_l} \frac{\partial \bigl( u_i^{(0,h_0)} - u_i^{(0)} \bigr)}{\partial x_j} d\Omega \\
				& + \int_\Omega \bigl( \hat{C}_{ijkl}^{h_0}(T^{(0)}) - \hat{C}_{ijkl}^{h_0}(T^{(0,h_0)}) \bigr) \frac{\partial u_k^{(0)}}{\partial x_l} \frac{\partial \bigl( u_i^{(0,h_0)} - u_i^{(0)} \bigr)}{\partial x_j} d\Omega \\
				& - \int_\Omega \bigl( \hat{\alpha}_{ij}(T^{(0)}) - \hat{\alpha}_{ij}^{h_0}(T^{(0)}) \bigr) \bigl( T^{(0)} - \tilde{T} \bigr) \frac{\partial \bigl( u_i^{(0,h_0)} - u_i^{(0)} \bigr)}{\partial x_j} d\Omega \\
				& - \int_\Omega \bigl( \hat{\alpha}_{ij}^{h_0}(T^{(0)}) - \hat{\alpha}_{ij}^{h_0}(T^{(0,h_0)}) \bigr) \bigl( T^{(0)} - \tilde{T} \bigr) \frac{\partial \bigl( u_i^{(0,h_0)} - u_i^{(0)} \bigr)}{\partial x_j} d\Omega \\
				& - \int_\Omega \bigl( \hat{\beta}_{ij}(T^{(0)}) - \hat{\beta}_{ij}^{h_0}(T^{(0)}) \bigr) \bigl( \omega^{(0)} - \tilde{\omega} \bigr) \frac{\partial \bigl( u_i^{(0,h_0)} - u_i^{(0)} \bigr)}{\partial x_j} d\Omega \\
				& - \int_\Omega \bigl( \hat{\beta}_{ij}^{h_0}(T^{(0)}) - \hat{\beta}_{ij}^{h_0}(T^{(0,h_0)}) \bigr) \bigl( \omega^{(0)} - \tilde{\omega} \bigr) \frac{\partial \bigl( u_i^{(0,h_0)} - u_i^{(0)} \bigr)}{\partial x_j} d\Omega.
			\end{aligned}
		\end{equation}
		
		Relying on \eqref{eq:5.2}, the Cauchy–Schwarz inequality and Young inequality, we directly obtain from \eqref{eq:5.17} the inequality below, assuming $| \hat S(T^{(0)}) - \hat S(T^{(0,h_0)}) | \le C| T^{(0)} - T^{(0,h_0)}|$, $| \hat k_{ij}^{h_0}(T^{(0)}) - \hat k_{ij}^{h_0}(T^{(0,h_0)})| \le C| T^{(0)} - T^{(0,h_0)}|$ and $| \hat Q_{hyd}(T^{(0)}) - \hat Q_{hyd}(T^{(0,h_0)}) | \le C | T^{(0)} - T^{(0,h_0)} |$.
		\begin{equation}
			\begin{aligned}
				\label{eq:5.20}
				&\frac{\partial}{\partial t}\Bigl( C \| T^{(0,h_0)} - T^{(0)} \|_{L^2(\Omega)}^2 \Bigr)
				+ C \| T^{(0,h_0)} - T^{(0)} \|_{H^1(\Omega)}^2\\
				& \le C \| T^{(0,h_0)} - T^{(0)} \|_{L^2(\Omega)}^2 + C h_0^4.
			\end{aligned}
		\end{equation}
		Then, integrating both sides of \eqref{eq:5.20} from $0$ to $t$ $(0<t \le T^*)$ yields the following inequality.
		\begin{equation}
			\begin{aligned}
				\label{eq:5.21}
				& C \bigl\| T^{(0, h_0)} - T^{(0)} \bigr\|_{L^2(\Omega)}^2 + C \int_0^t \bigl\| T^{(0, h_0)} - T^{(0)} \bigr\|_{H^1(\Omega)}^2 d\tau \\
				& \le C \int_0^t \bigl\| T^{(0, h_0)} - T^{(0)} \bigr\|_{L^2(\Omega)}^2 d\tau + \int_0^t C h_0^4 d\tau.
			\end{aligned}
		\end{equation}
		When setting $\Upsilon (t) = \bigl\| T^{(0, h_0)} - T^{(0)} \bigr\|_{L^2(\Omega)}^2 + \int_0^t \bigl\| T^{(0, h_0)} - T^{(0)} \bigr\|_{H^1(\Omega)}^2 d\tau$, then we can derive $\Upsilon(t) \le C h_0^4 + C \int_0^t \Upsilon(t) d\tau$ from \eqref{eq:5.21}.
		Consequently, applying Gronwall inequality and exploiting the arbitrariness of $t$, we obtain the estimate \eqref{eq:5.11}.
		
		Relying on the inequalities in \eqref{eq:5.3} and \eqref{eq:5.11}, together with the Cauchy–Schwarz and Young inequalities, we directly obtain from \eqref{eq:5.18} the following inequality, assuming $|\hat g_{ij}^{{h_0}}({\omega ^{(0)}}) - \hat g_{ij}^{{h_0}}({\omega ^{(0,{h_0})}})| \le C|{\omega ^{(0)}} - {\omega ^{(0,{h_0})}}|$ and $|{\hat S_{hyd}}({T^{(0)}}) - {\hat S_{hyd}}({T^{(0,{h_0})}})| \le C|{T^{(0)}} - {T^{(0,{h_0})}}|$.
		\begin{equation}
			\begin{aligned}
				\label{eq:5.22}
				& \frac{\partial}{\partial t} \Bigl( C \bigl\| \omega^{(0,h_0)} - \omega^{(0)} \bigr\|_{L^2(\Omega)}^2 \Bigr) + C \bigl\| \omega^{(0,h_0)} - \omega^{(0)} \bigr\|_{H^1(\Omega)}^2 \\
				& \leq C \bigl\| \omega^{(0,h_0)} - \omega^{(0)} \bigr\|_{L^2(\Omega)}^2 + C h_0^4.
			\end{aligned}
		\end{equation}
		Similarly, we obtain that inequality \eqref{eq:5.12} holds.
		
		From the inequalities \eqref{eq:5.4}-\eqref{eq:5.6}, \eqref{eq:5.11}-\eqref{eq:5.12} and the Cauchy–Schwarz and Young inequalities, we directly obtain the following inequality from equality \eqref{eq:5.19}, assuming $|\hat C_{ijkl}^{{h_0}}({T^{(0)}}) - \hat C_{ijkl}^{{h_0}}({T^{(0,{h_0})}})| \le C|{T^{(0)}} - {T^{(0,{h_0})}}|$, $|\hat \alpha _{ij}^{{h_0}}({T^{(0)}}) - \hat \alpha _{ij}^{{h_0}}({T^{(0,{h_0})}})| \le C|{T^{(0)}} - {T^{(0,{h_0})}}|$ and $|\hat \beta _{ij}^{{h_0}}({T^{(0)}}) - \hat \beta _{ij}^{{h_0}}({T^{(0,{h_0})}})| \le C|{T^{(0)}} - {T^{(0,{h_0})}}|$. Similarly, we obtain that inequality \eqref{eq:5.13} holds.
	\end{proof}
	
	\begin{assumption}
		\label{ass:4.1}
		Let $T^{(0, h_0)}$, $\omega^{(0, h_0)}$ and $u_{i}^{(0, h_0)}$ represent the exact solutions of the revised macroscopic homogenized equations \eqref{eq:5.10}, and $T^{(0, h_0, h_1)}$, $\omega^{(0, h_0, h_1)}$ and $u_{i}^{(0, h_0, h_1)}$ denote the corresponding finite element solutions of the revised macroscopic homogenized equations \eqref{eq:5.10}. 
		There exists a constant $C$ independent of $h_0$ and $\Delta t$ such that
		\begin{equation}
			\label{eq:5.23}
			\max_{1 \leq k \leq M} \| T^{(0,h_0,h_1)}(\mathbf{x}, t_k) - T^{(0,h_0)}(\mathbf{x}, t_k) \|_{L^2(\Omega)} \leq C \Delta t + C h_1^2,
		\end{equation}
		\begin{equation}
			\label{eq:5.24}
			\max_{1 \leq k \leq M} \| \omega^{(0,h_0,h_1)}(\mathbf{x}, t_k) - \omega^{(0,h_0)}(\mathbf{x}, t_k) \|_{L^2(\Omega)} \leq C \Delta t + C h_1^2,
		\end{equation}
		\begin{equation}
			\label{eq:5.25}
			\max_{1 \leq k \leq M} \|\bm{u}^{(0,h_0,h_1)}(\mathbf{x}, t_k)-\bm{u}^{(0,h_0)}(\mathbf{x}, t_k)\|_{(H^1(\Omega))^n}\leq C h_1^2.
		\end{equation}
		Estimates \eqref{eq:5.23} and \eqref{eq:5.24} are the optimal error bounds for the backward Euler--Galerkin discretization of nonlinear parabolic equations, while \eqref{eq:5.25} follows from the standard finite element error estimate for linear elliptic problems. 
		Detailed proofs of analogous estimates for closely related coupled systems can be found in, e.g., \cite{R52,R58,R59}.
	\end{assumption}
	
	\begin{theorem}
		\label{thm:2}
		Denote by $T^{(0, h_0, h_1)}$, $\omega^{(0, h_0, h_1)}$ and $u_{i}^{(0, h_0, h_1)}$ the finite element solutions of the modified homogenized equations \eqref{eq:5.10}, and by $T^{(0)}$, $\omega^{(0)}$ and $u_{i}^{(0)}$ the exact solutions of the homogenized equations \eqref{eq:2.27}. Then the following estimate holds
		\begin{equation}
			\label{eq:5.26}
			\max_{1 \leq k \leq M} \bigl\| T^{(0,h_0,h_1)}(\mathbf{x}, t_k) - T^{(0)}(\mathbf{x}, t_k) \bigr\|_{L^2(\Omega)} \leq C \Delta t + C h_1^2 + C h_0^2,
		\end{equation}
		\begin{equation}
			\label{eq:5.27}
			\max_{1 \leq k \leq M} \bigl\| \omega^{(0,h_0,h_1)}(\mathbf{x}, t_k) - \omega^{(0)}(\mathbf{x}, t_k) \bigr\|_{L^2(\Omega)} \leq C \Delta t + C h_1^2 + C h_0^2,
		\end{equation}
		\begin{equation}
			\label{eq:5.28}
			\max_{1 \leq k \leq M} \bigl\| \bm{u}^{(0,h_0,h_1)}(\mathbf{x}, t_k) - \bm{u}^{(0)}(\mathbf{x}, t_k) \bigr\|_{(H^1(\Omega))^n} \leq C h_1^2 + C h_0^2,
		\end{equation}
		where $C$ is a constant independent of $h_0$, $h_1$ and $\Delta t$.
	\end{theorem}
	
	\begin{proof}
		Using the triangle inequality, we obtain the following inequalities:
		\begin{equation}
			\label{eq:5.29}
			\begin{aligned}
				& \max_{1 \leq k \leq M} \bigl\| T^{(0,h_0,h_1)}(\mathbf{x}, t_k) - T^{(0)}(\mathbf{x}, t_k) \bigr\|_{L^2(\Omega)} \\
				& \leq \max_{1 \leq k \leq M} \bigl\| T^{(0,h_0,h_1)}(\mathbf{x}, t_k) - T^{(0, h_0)}(\mathbf{x}, t_k) \bigr\|_{L^2(\Omega)} + \max_{1 \leq k \leq M} \bigl\| T^{(0,h_0)}(\mathbf{x}, t_k) - T^{(0)}(\mathbf{x}, t_k) \bigr\|_{L^2(\Omega)}.
			\end{aligned}
		\end{equation}
		\begin{equation}
			\label{eq:5.30}
			\begin{aligned}
				& \max_{1 \leq k \leq M} \bigl\| \omega^{(0,h_0,h_1)}(\mathbf{x}, t_k) - \omega^{(0)}(\mathbf{x}, t_k) \bigr\|_{L^2(\Omega)}\\
				& \leq \max_{1 \leq k \leq M} \bigl\| \omega^{(0,h_0,h_1)}(\mathbf{x}, t_k) - \omega^{(0, h_0)}(\mathbf{x}, t_k) \bigr\|_{L^2(\Omega)} + \max_{1 \leq k \leq M} \bigl\| \omega^{(0,h_0)}(\mathbf{x}, t_k) - \omega^{(0)}(\mathbf{x}, t_k) \bigr\|_{L^2(\Omega)}.
			\end{aligned}
		\end{equation}
		\begin{equation}
			\label{eq:5.31}
			\begin{aligned}
				& \max_{1 \leq k \leq M} \bigl\| \bm{u}^{(0,h_0,h_1)}(\mathbf{x}, t_k) - \bm{u}^{(0)}(\mathbf{x}, t_k) \bigr\|_{(H^1(\Omega))^n}\\
				& \leq \max_{1 \leq k \leq M} \bigl\| \bm{u}^{(0,h_0,h_1)}(\mathbf{x}, t_k) - \bm{u}^{(0, h_0)}(\mathbf{x}, t_k) \bigr\|_{(H^1(\Omega))^n} + \max_{1 \leq k \leq M} \bigl\| \bm{u}^{(0,h_0)}(\mathbf{x}, t_k) - \bm{u}^{(0)}(\mathbf{x}, t_k) \bigr\|_{(H^1(\Omega))^n}.
			\end{aligned}
		\end{equation}
		
		By applying Lemma~\ref{lem:4.3} and Assumption~\ref{ass:4.1}, we substitute \eqref{eq:5.11} and \eqref{eq:5.23} into \eqref{eq:5.29}, \eqref{eq:5.12} and \eqref{eq:5.24} into \eqref{eq:5.30}, and \eqref{eq:5.13} and \eqref{eq:5.25} into \eqref{eq:5.31}, respectively. This directly yields the estimates \eqref{eq:5.26}-\eqref{eq:5.28}, which completes the proof.
	\end{proof}
	
	To summarize, the theoretical analysis presented above rigorously guarantees convergence of the two-stage numerical algorithm at both the microscopic and macroscopic levels.

\section{Numerical examples and results}
\label{sec:5}
	This section presents several numerical examples to validate the proposed HOMS computational model along with its numerical algorithm. All numerical experiments are conducted on the same computer equipped with an Intel Core i7-10700 processor (2.90 GHz) and 16.0 GB RAM, and all numerical simulations are executed using the Freefem++ software.
	
	Given the difficulty in obtaining exact solutions for the multi-scale problem \eqref{eq:2.1}, we substitute $T^{\epsilon}(\mathbf{x},t), \omega^{\epsilon}(\mathbf{x},t)$ and $\bm{u}^{\epsilon}(\mathbf{x},t)$ with corresponding high-resolution FEM solutions $T_e(\mathbf{x},t), \omega_e(\mathbf{x},t)$ and $\bm{u}_e(\mathbf{x},t)$, which serve as reference solutions. Moreover, let $\|\cdot\|_{L^{2}}$ and $\mid\cdot\mid _{H^{1}}$ denote the $L^2$ norm and $H^1$ semi-norm, respectively. The $L^2$ norm reflects the global error in the solution and supports validation of the macroscopic response, while the $H^1$ semi-norm measures gradient errors and is essential for assessing the resolution of microscopic oscillations. Then, the relative errors in the $L^2$ norm for $T^{(0)}$, $T^{(1,\epsilon)}$, $T^{(2,\epsilon)}$, $\omega^{(0)}$, $\omega^{(1,\epsilon)}$, $\omega^{(2,\epsilon)}$, $\bm{u}^{(0)}$, $\bm{u}^{(1,\epsilon)}$ and $\bm{u}^{(2,\epsilon)}$ are defined as $e_{T,L^2}^0(t)$, $e_{T,L^2}^1(t)$, $e_{T,L^2}^2(t)$, $e_{\omega,L^2}^0(t)$, $e_{\omega,L^2}^1(t)$, $e_{\omega,L^2}^2(t)$, $e_{\bm{u},L^2}^0(t)$, $e_{\bm{u},L^2}^1(t)$ and $e_{\bm{u},L^2}^2(t)$, respectively. The relative errors in the $H^1$ semi-norm are defined as $e_{T,H^1}^0(t)$, $e_{T,H^1}^1(t)$, $e_{T,H^1}^2(t)$, $e_{\omega,H^1}^0(t)$, $e_{\omega,H^1}^1(t)$, $e_{\omega,H^1}^2(t)$, $e_{\bm{u},H^1}^0(t)$, $e_{\bm{u},H^1}^1(t)$ and $e_{\bm{u},H^1}^2(t)$, respectively.
	
	\subsection{Example 1: nonlinear hygro-thermo-mechanical coupling simulation of 2D heterogeneous structure}
	\label{sec:51}
	In this example, the nonlinear dynamic hygro-thermo-mechanical coupling behavior of 2D heterogeneous structure is simulated, which is modeled as a periodic array of microscopic cells, each comprising of matrix and inclusion constituents. Here, the investigated heterogeneous structure $\Omega$ is defined as $\Omega= (x_1,x_2)= [0,1]^2 \mathrm{cm}^2$ and small periodic parameter $\epsilon=1/10$.(see Fig.~\ref{f1:2D})
	\begin{figure}[!htb]
		\centering
		\begin{minipage}[c]{0.32\textwidth}
			\centering
			\includegraphics[width=40mm]{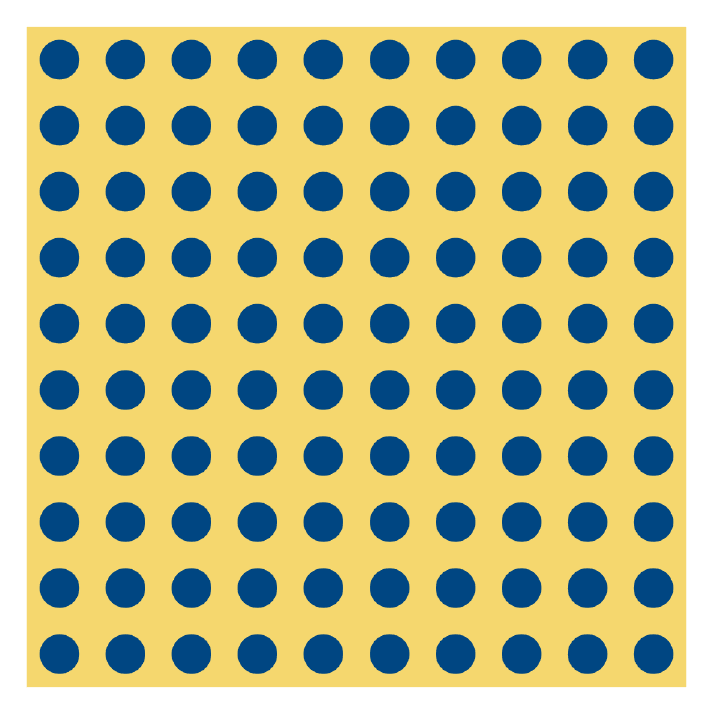}\\
			(a)
		\end{minipage}
		\begin{minipage}[c]{0.3\textwidth}
			\centering
			\includegraphics[width=40mm]{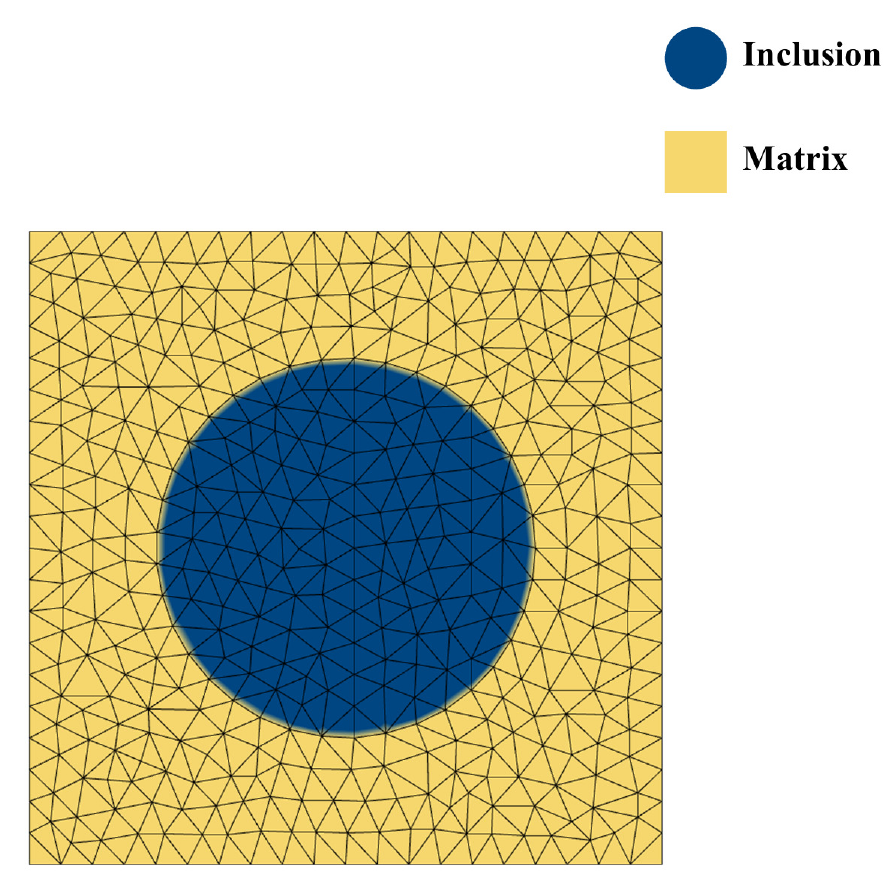}\\
			(b)
		\end{minipage}
		\begin{minipage}[c]{0.3\textwidth}
			\centering
			\includegraphics[width=40mm]{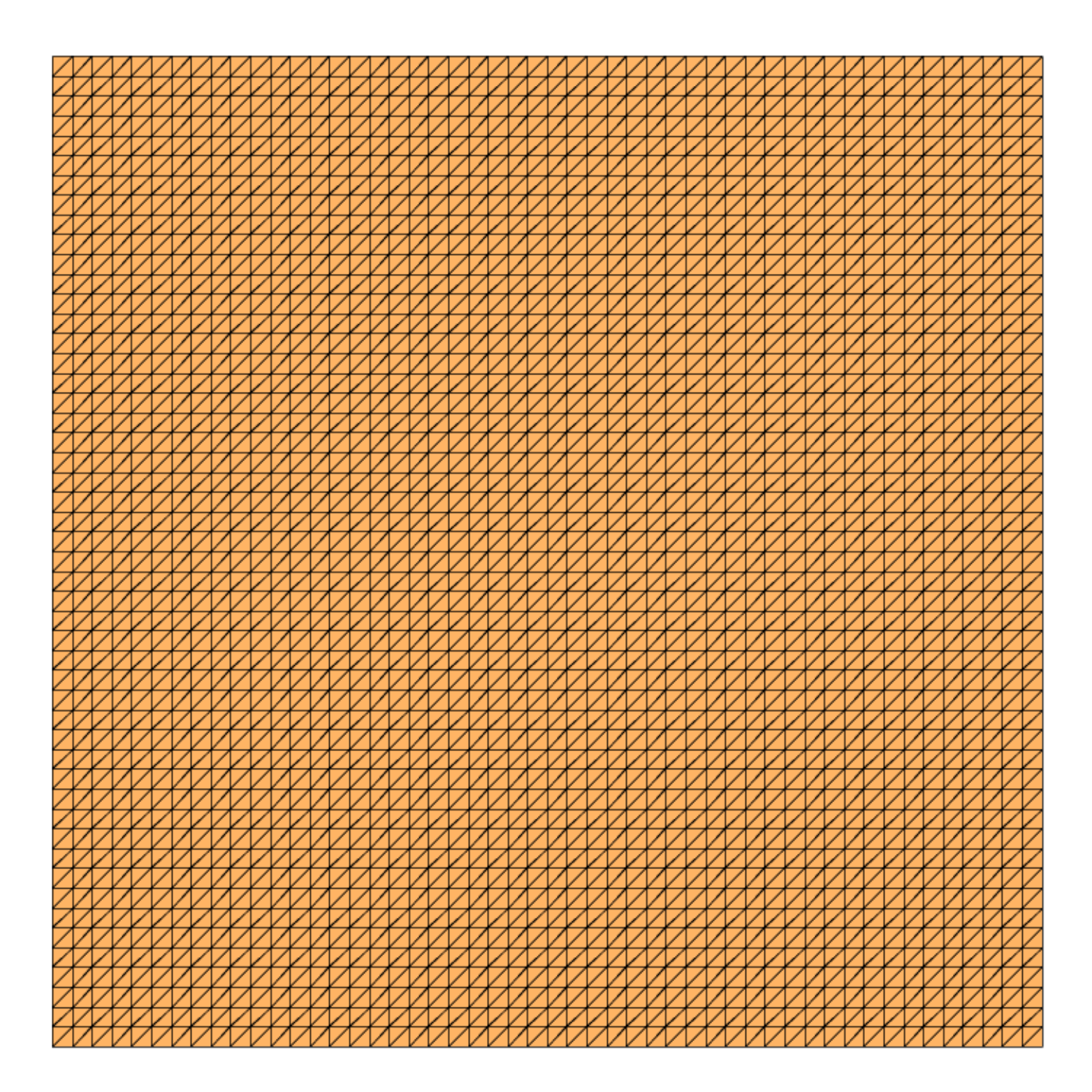}\\
			(c)
		\end{minipage}
		\caption{(a) The 2D heterogeneous structure $\Omega$; (b) PUC $Y$; (c) homogenized structure $\Omega$.}\label{f1:2D}
	\end{figure}
	
	The material parameters of heterogeneous structure are given in Table~\ref{t1}.
	\begin{table}[!t]
		\caption{Material property parameters ($T$ represents temperature and $\omega$ represents moisture).\label{t1}}
		\centering
		\small
		\begin{tabular*}{\columnwidth}{@{\extracolsep\fill}lcc@{\extracolsep\fill}}
			\toprule
			Property & Matrix & Inclusion \\
			\midrule
			Mass density $\rho^{\epsilon}$ (kg/m$^3$) & 7000.0 & 5000.0 \\
			Specific heat $c^{\epsilon}$ (J/(kg$\cdot$K)) & $900.0+1.5T+0.015T^2$ & $620.0+0.7T+0.007T^2$ \\
			Young's modulus $E^{\epsilon}$ (GPa) & $300.0-0.05T-5.0\times10^{-7}T^2$ & $1.5-5.0\times10^{-4}T-5.0\times10^{-9}T^2$ \\
			Poisson's ratio $\nu^{\epsilon}$ & 0.30 & 0.25 \\
			Thermal conductivity $k^{\epsilon}_{ij}$ (W/(m$\cdot$K)) & $1500.0+0.5T+5.0\times10^{-5}T^2$ & $15+5.0\times10^{-3}T+5.0\times10^{-7}T^2$ \\
			Moisture diffusion $g^{\epsilon}_{ij}$ (m$^2$/s) & $1.5\times10^{-5}+5.0\times10^{-9}\omega+5.0\times10^{-13}\omega^2$ & $1.5\times10^{-7}+5.0\times10^{-11}\omega+5.0\times10^{-15}\omega^2$ \\
			Thermal stress coefficient $\alpha^{\epsilon}_{ij}$ (MPa/K) & $0.5-1.0\times10^{-4}T-1.0\times10^{-8}T^2$ & $0.05-1.0\times10^{-5}T-1.0\times10^{-9}T^2$ \\
			Moisture stress coefficient $\beta_{ij}^{\epsilon}$ (MPa) & $0.05-1.0\times10^{-5}T-1.0\times10^{-9}T^2$ & $5.0\times10^{-3}-1.0\times10^{-6}T-1.0\times10^{-10}T^2$ \\
			\bottomrule
		\end{tabular*}
	\end{table}
	
	In addition, the source terms, boundary conditions, and initial conditions for the multi-scale nonlinear problem \eqref{eq:2.1} in this example are specified below.
	\begin{equation}
		\label{eq:6.1}
		\begin{aligned}
			& h(\mathbf{x},t)=2000.0\ \mathrm{J/(cm^{3}\cdot s)}, \ m(\mathbf{x},t)=0.03\ \mathrm{s^{-1}}, \\
			& f_1(\mathbf{x},t)=f_2(\mathbf{x},t)=-5000\ \mathrm{N/cm^3}, \ \text{in } \Omega \times (0,T^*),\\
			& Q_{hyd}^\epsilon(\mathbf{x},T^\epsilon) =
			\begin{cases}
				2000.0 + 0.001T + 1.0\times10^{-8}T^2 \; \mathrm{J/(cm^3\cdot s)}, & \text{in Matrix},\\
				0.0, & \text{in Inclusion},
			\end{cases} \\
			& S_{hyd}^\epsilon(\mathbf{x},T^\epsilon) =
			\begin{cases}
				0.08+1.0\times10^{-7}T+1.0\times10^{-12}T^2 \; \mathrm{s^{-1}}, & \text{in Matrix},\\
				0.0, & \text{in Inclusion},
			\end{cases}\\
			& \hat{T}(\mathbf{x},t)=293.15\ \mathrm{K},\ \hat{\omega}(\mathbf{x},t)=0.8,\  \hat{\bm{u}}(\mathbf{x},t)=0.0\ \mathrm{cm}, \ \text{on } \partial\Omega \times (0,T^*),\\
			& \tilde{T}=293.15\ \mathrm{K},\ \tilde{\omega}=0.8,\  \tilde{\bm{u}}=0.0\ \mathrm{cm}, \ \text{in } \Omega.
		\end{aligned}
	\end{equation}
	
	In Table~\ref{t2}, triangular meshes are generated for the multi-scale problem \eqref{eq:2.1}, the auxiliary cell problems, and the corresponding homogenized problem \eqref{eq:2.27}. The computational cost of the precise FEM and the HOMS method is detailed, including the numbers of elements and nodes, as well as the CPU times for the precise finite element simulation and the multi-scale simulation.
	\begin{table}[!t]
		\caption{Comparison of computational cost ($\Delta t=0.01\,$s, $t\in[0,1.0]\,$s).\label{t2}}
		\centering
		\begin{tabular*}{\columnwidth}{@{\extracolsep\fill}cccc@{\extracolsep\fill}}
			\toprule
			& Cell equations & Homogenized equations & Multi-scale equations \\
			\midrule
			FEM nodes & 461 & 2601 & 35761 \\
			FEM elements & 840 & 5000 & 70800 \\
			\midrule
			& \multicolumn{2}{c}{HOMS method} & precise FEM \\
			\midrule
			Computational time & \multicolumn{2}{c}{817.613\,s} & 2295.457\,s \\
			\bottomrule
		\end{tabular*}
	\end{table}
	
	As shown in Table~\ref{t2}, the computational cost of the HOMS method is significantly lower than that of the precise FEM. The superiority of the proposed HOMS method over the precise FEM is obvious since a highly fine mesh is demanded to capture the microscopic oscillatory behaviors in this heterogeneous structure. In comparison, the proposed HOMS approach significantly accelerates the simulation of the multi-scale nonlinear problem \eqref{eq:2.1}, saving about $64.38\%$ of the computational time. Moreover, the time savings achieved by the HOMS method become increasingly significant as the simulation duration grows.
	
	For this example, $10$ equidistant interpolation points for macroscopic temperature and another $10$ for macroscopic moisture are prescribed within a single unit cell. It is worth noting that the auxiliary cell problems are solved off‑line, prior to the on‑line multi‑scale computation, and the resulting solutions remain applicable to different heterogeneous structures. The nonlinear hygro‑thermo‑mechanical response of the 2D heterogeneous structure is simulated over $t \in [0, 1.0] \mathrm{s}$. With a time step $\Delta t = 0.01 \mathrm{s}$, we solve the macroscopic homogenized equations \eqref{eq:2.27} and the multi‑scale nonlinear equations \eqref{eq:2.1} in the on‑line stage. The final temperature, moisture and displacement fields at $t = 1.0 \mathrm{s}$ are presented in Figs.~\ref{f2}-\ref{f4}, respectively. Furthermore, Fig.~\ref{f6} shows the evolution of the relative errors of these fields in the $L^2$ norm and $H^1$ semi-norm.
	\begin{figure}[!htb]
		\centering
		\begin{minipage}[c]{0.24\textwidth}
			\centering
			\includegraphics[width=\linewidth]{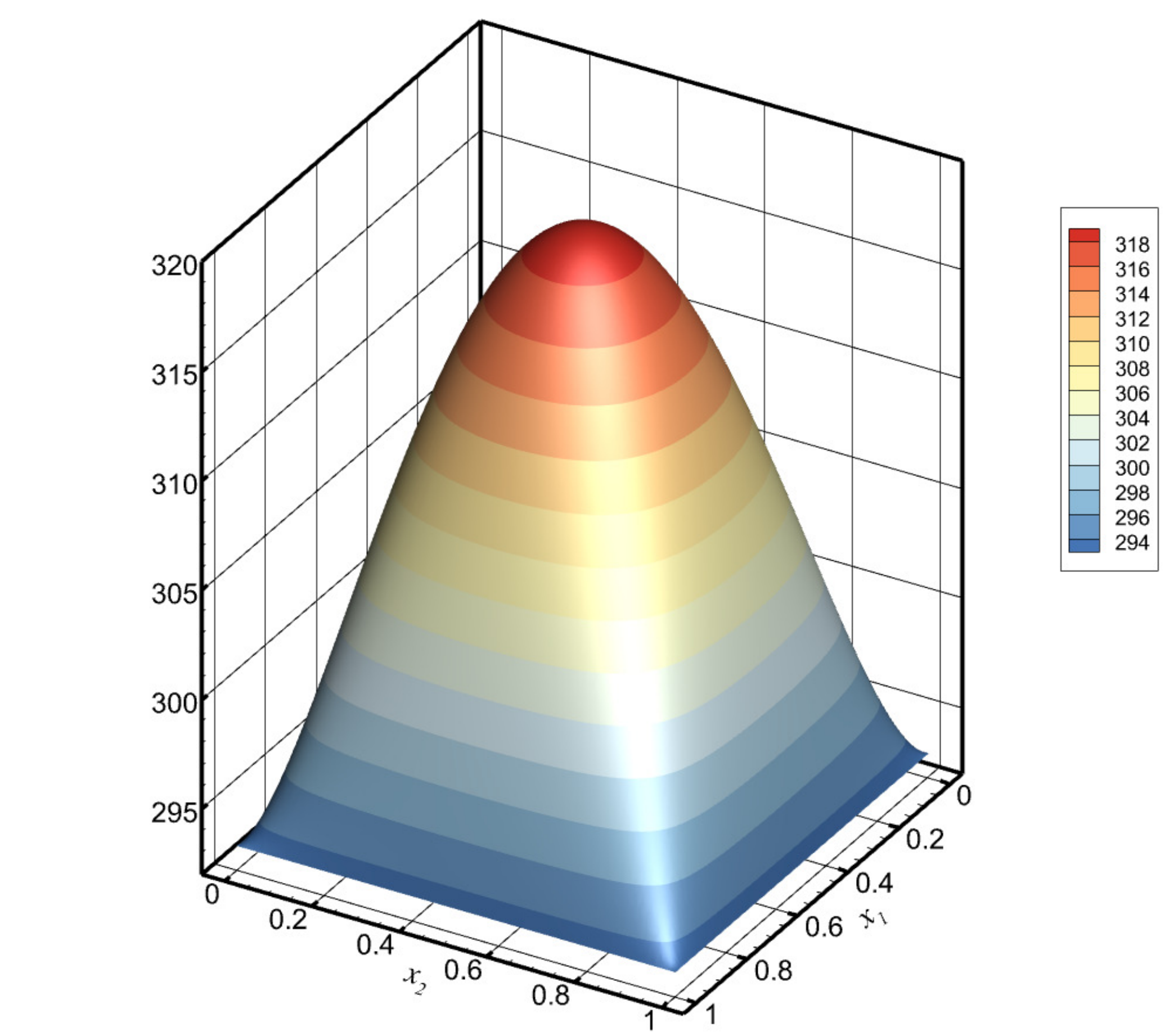}
			(a)
		\end{minipage}
		\hfill
		\begin{minipage}[c]{0.24\textwidth}
			\centering
			\includegraphics[width=\linewidth]{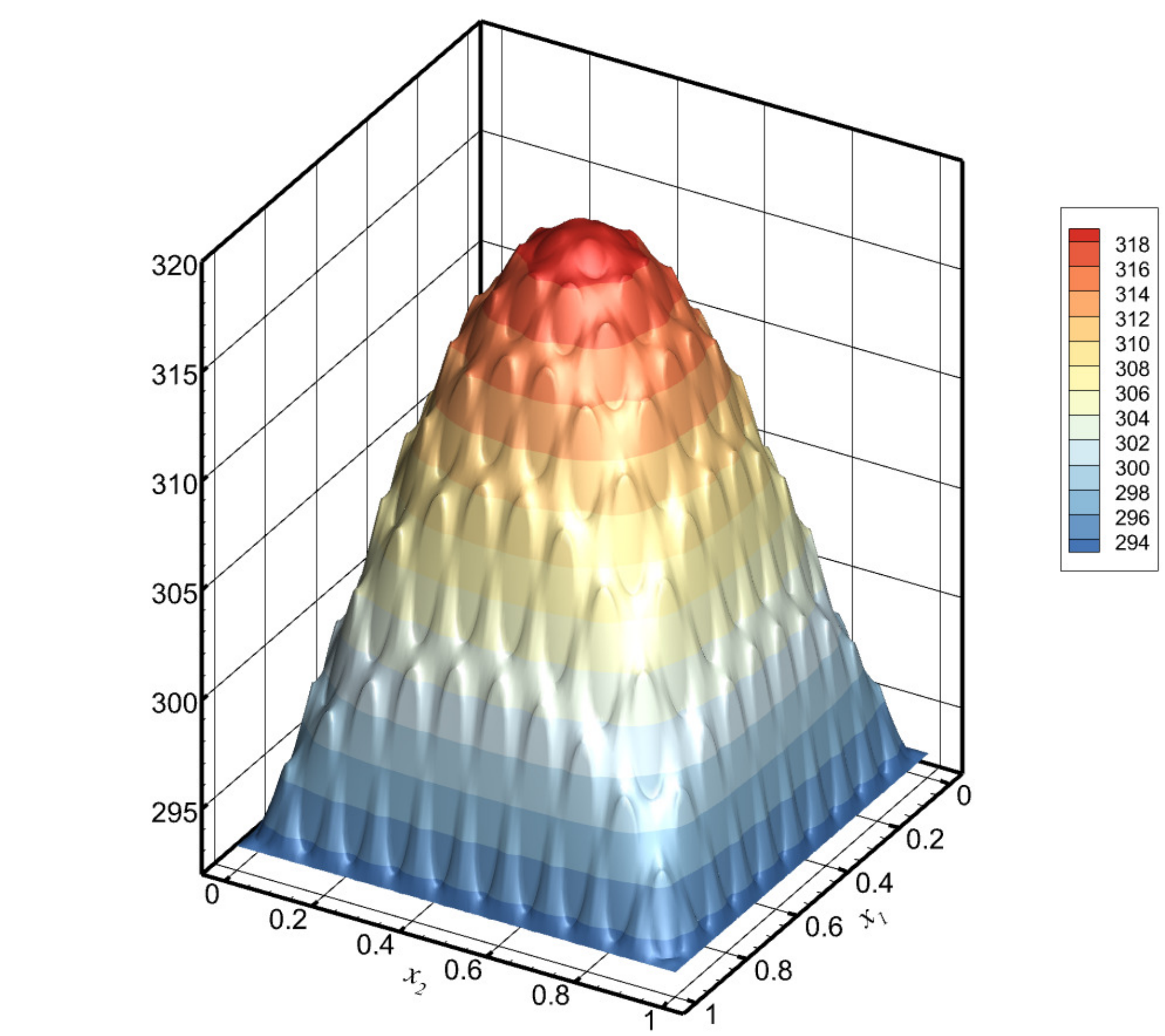}
			(b)
		\end{minipage}
		\hfill
		\begin{minipage}[c]{0.24\textwidth}
			\centering
			\includegraphics[width=\linewidth]{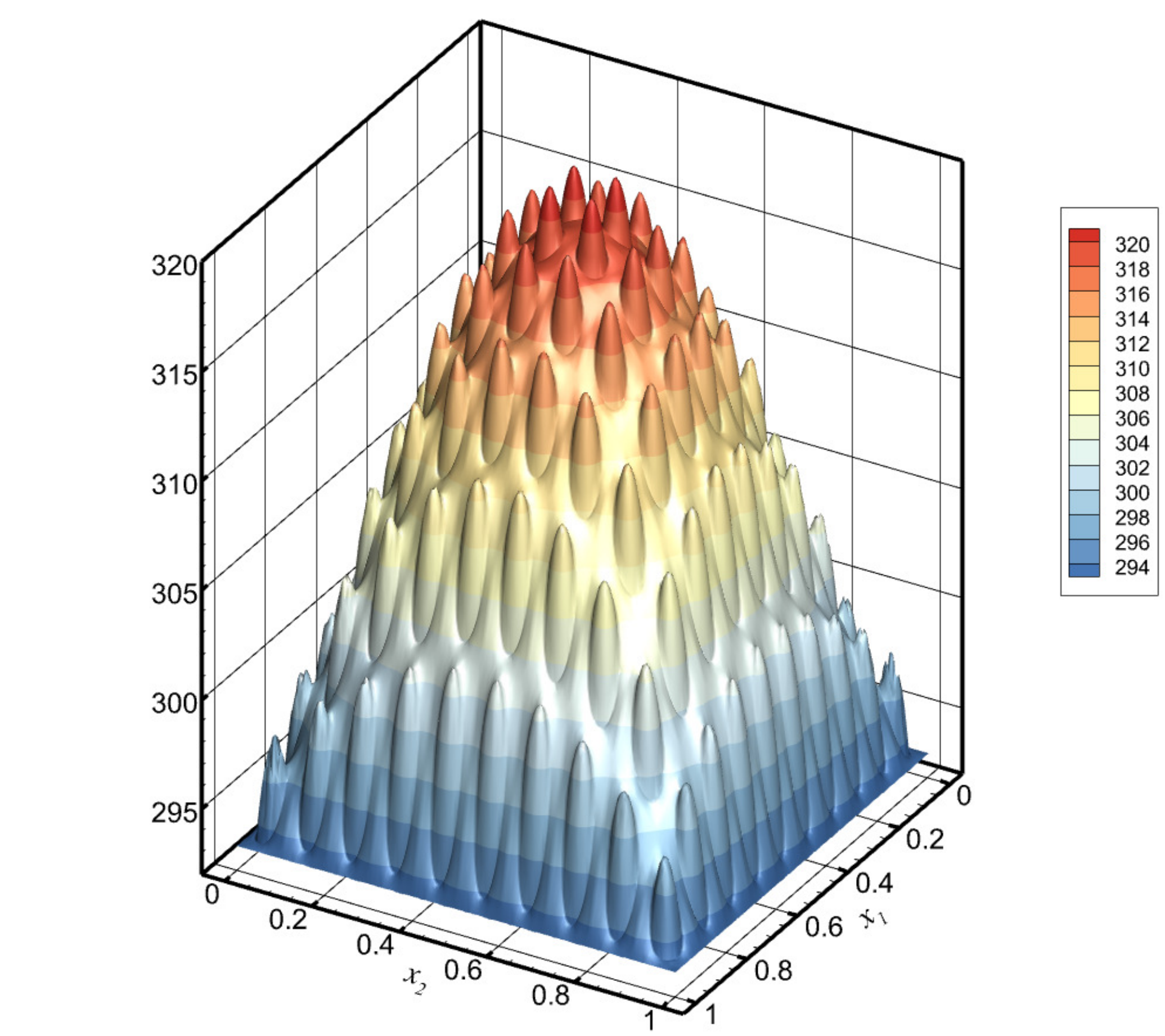}
			(c)
		\end{minipage}
		\hfill
		\begin{minipage}[c]{0.24\textwidth}
			\centering
			\includegraphics[width=\linewidth]{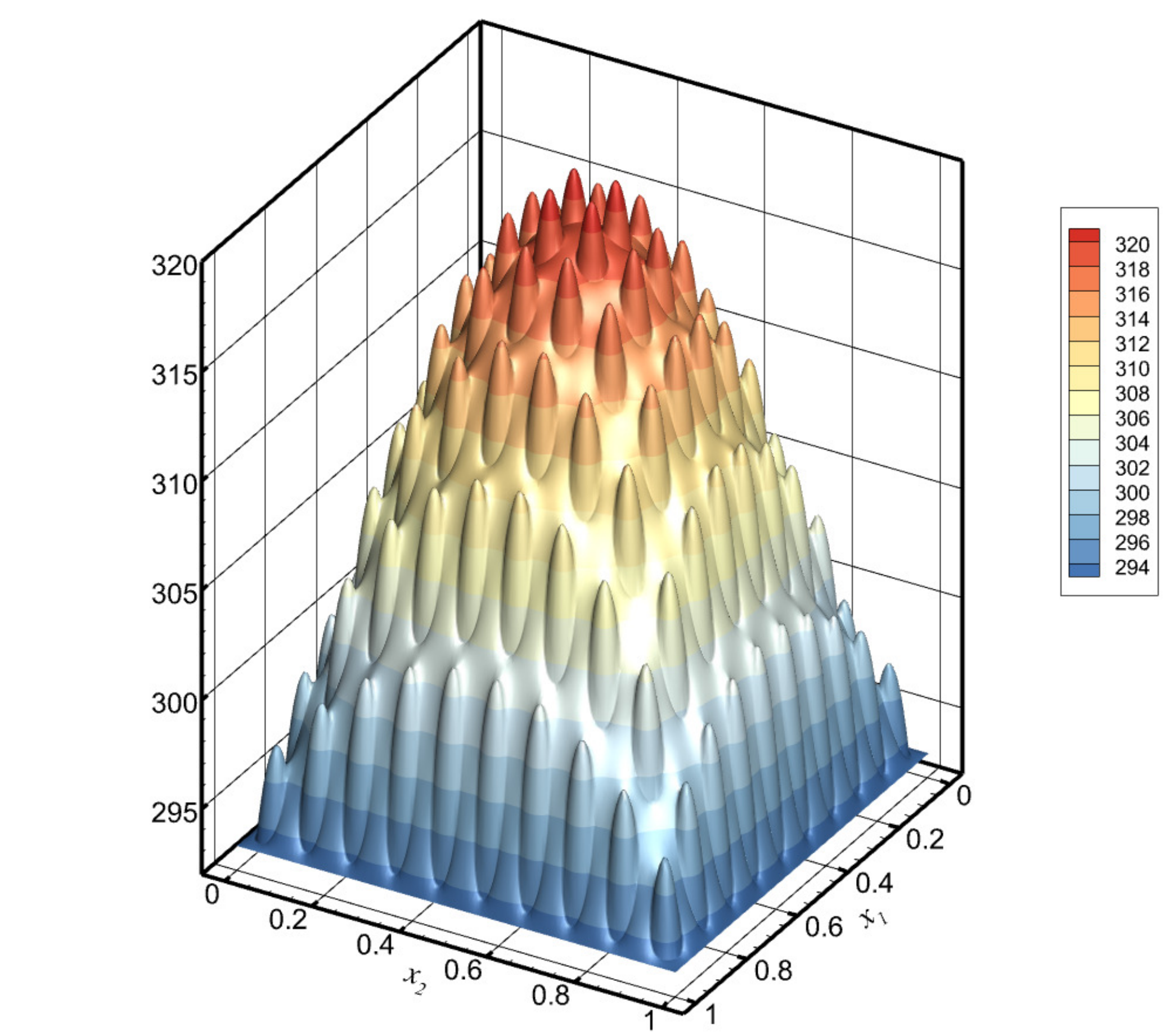}
			(d)
		\end{minipage}
		\caption{The temperature field at time $t=1.0\mathrm{s}$: (a) $T^{(0)}$; (b) $T^{(1,\epsilon)}$; (c) $T^{(2,\epsilon)}$; (d) $T_e$.}\label{f2}
	\end{figure}
	\begin{figure}[!htb]
		\centering
		\begin{minipage}[c]{0.24\textwidth}
			\centering
			\includegraphics[width=\linewidth]{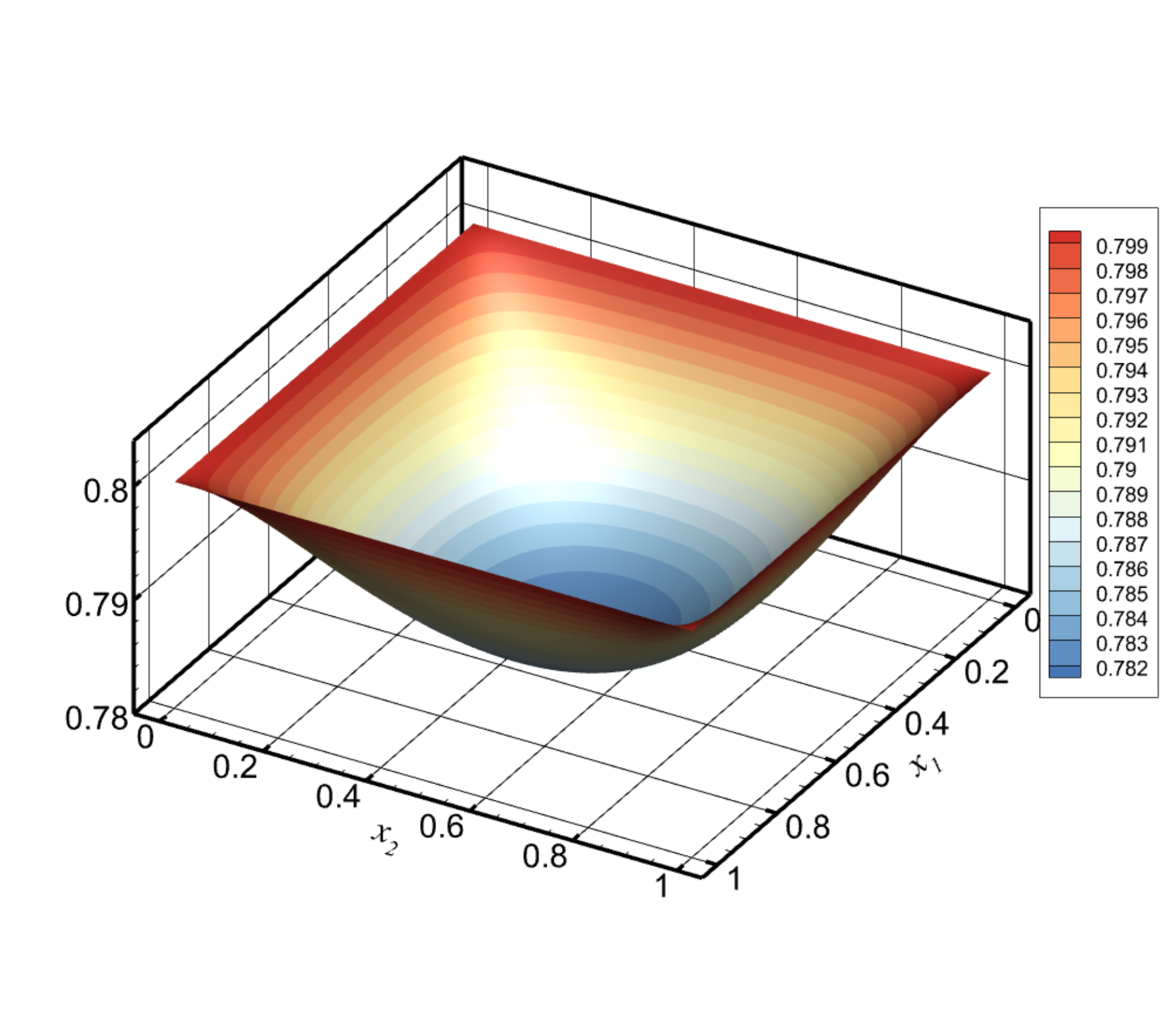}
			(a)
		\end{minipage}
		\hfill
		\begin{minipage}[c]{0.24\textwidth}
			\centering
			\includegraphics[width=\linewidth]{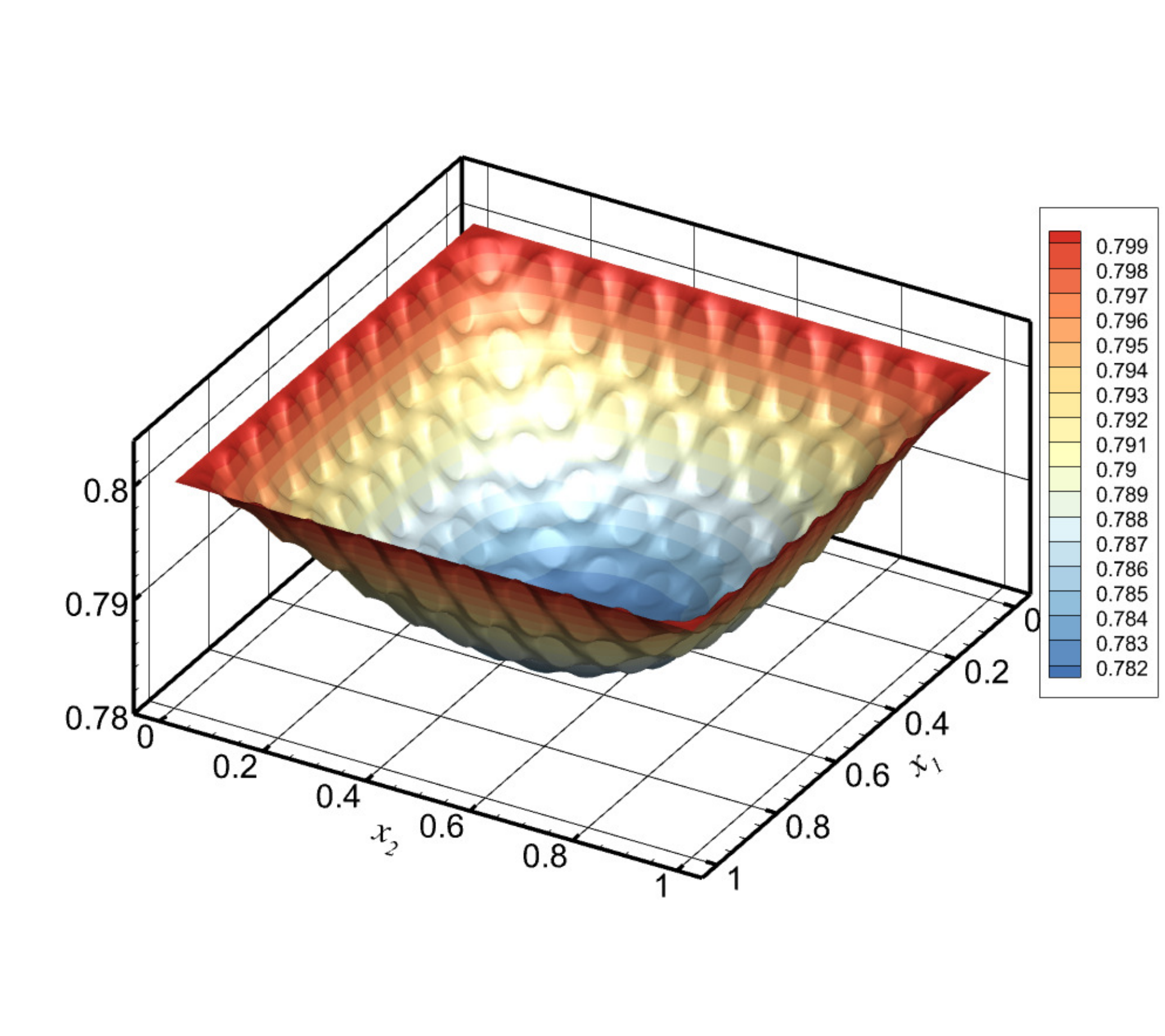}
			(b)
		\end{minipage}
		\hfill
		\begin{minipage}[c]{0.24\textwidth}
			\centering
			\includegraphics[width=\linewidth]{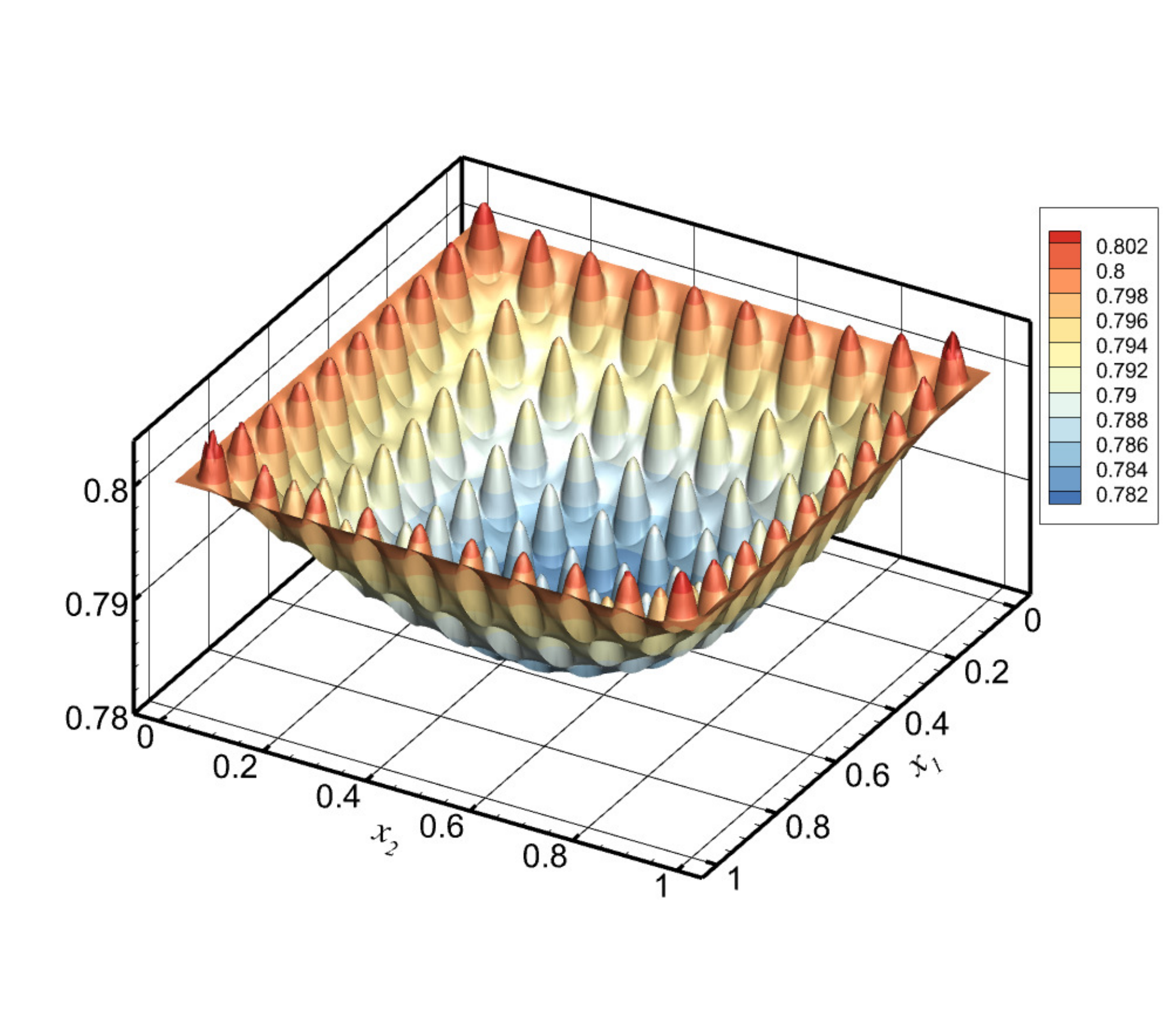}
			(c)
		\end{minipage}
		\hfill
		\begin{minipage}[c]{0.24\textwidth}
			\centering
			\includegraphics[width=\linewidth]{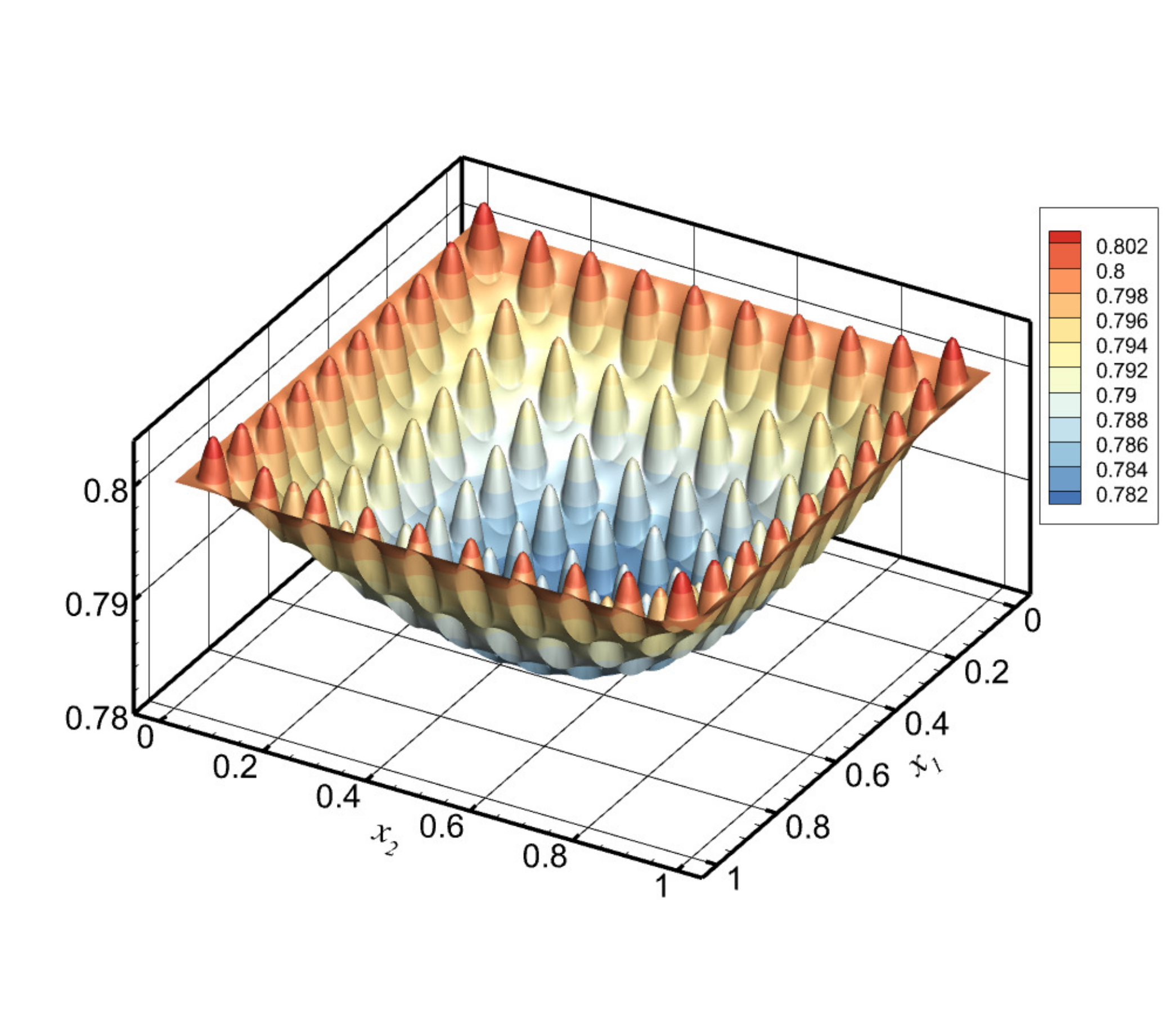}
			(d)
		\end{minipage}
		\caption{The moisture field at time $t=1.0\mathrm{s}$: (a) $\omega^{(0)}$; (b) $\omega^{(1,\epsilon)}$; (c) $\omega^{(2,\epsilon)}$; (d) $\omega_e$.}\label{f3}
	\end{figure}
	\begin{figure}[!htb]
		\centering
		\begin{minipage}[c]{0.24\textwidth}
			\centering
			\includegraphics[width=\linewidth]{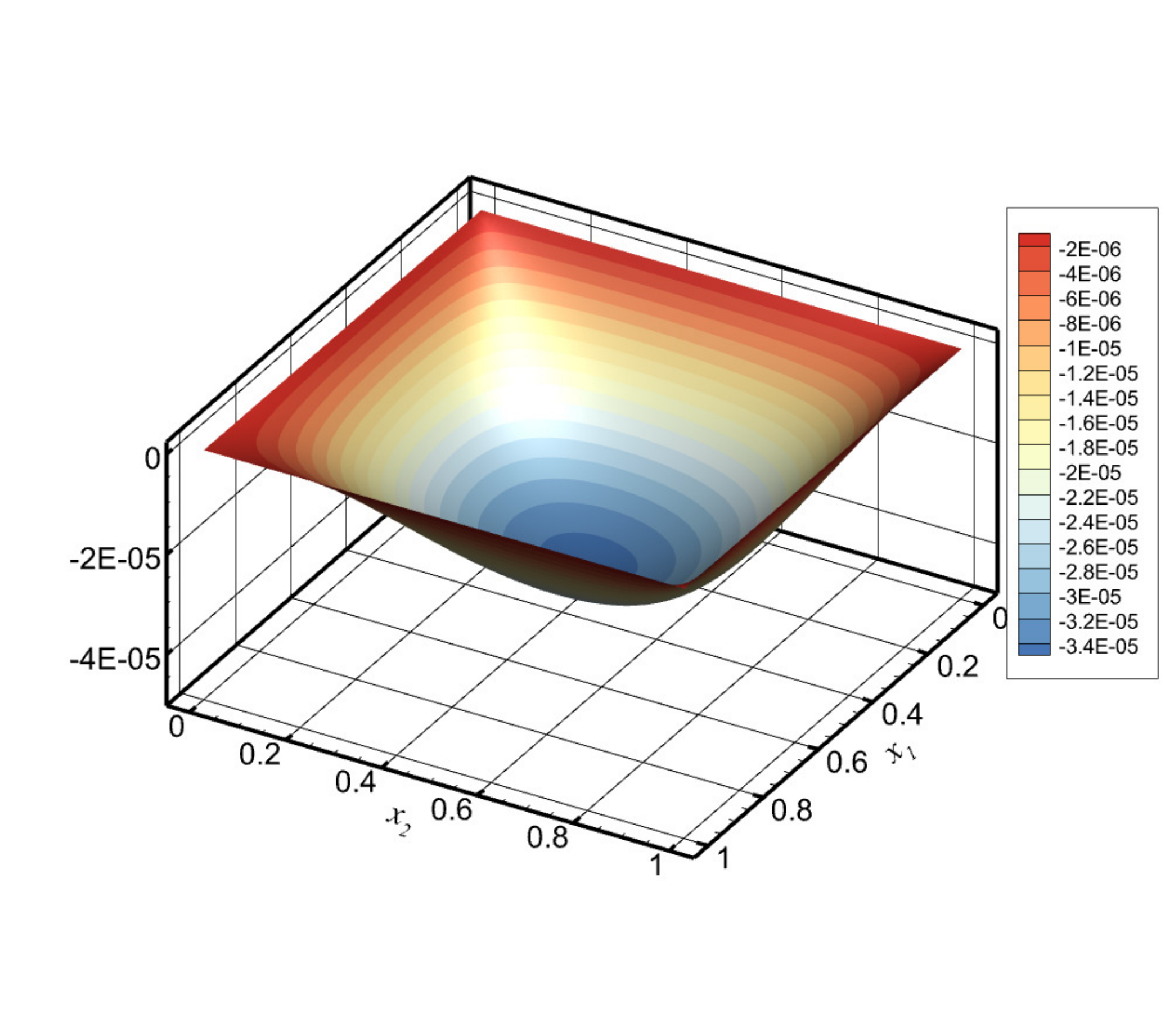}
			(a)
		\end{minipage}
		\hfill
		\begin{minipage}[c]{0.24\textwidth}
			\centering
			\includegraphics[width=\linewidth]{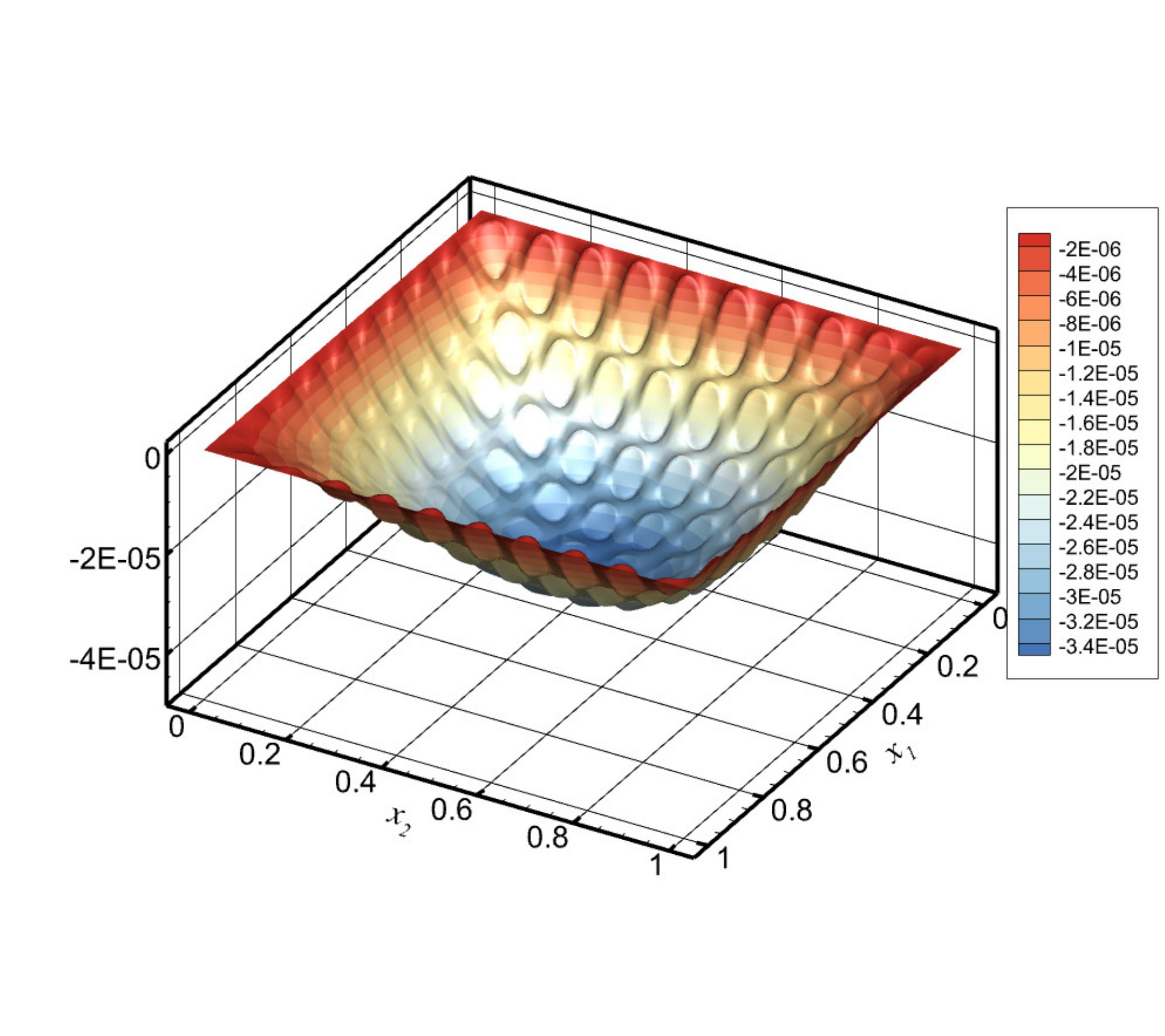}
			(b)
		\end{minipage}
		\hfill
		\begin{minipage}[c]{0.24\textwidth}
			\centering
			\includegraphics[width=\linewidth]{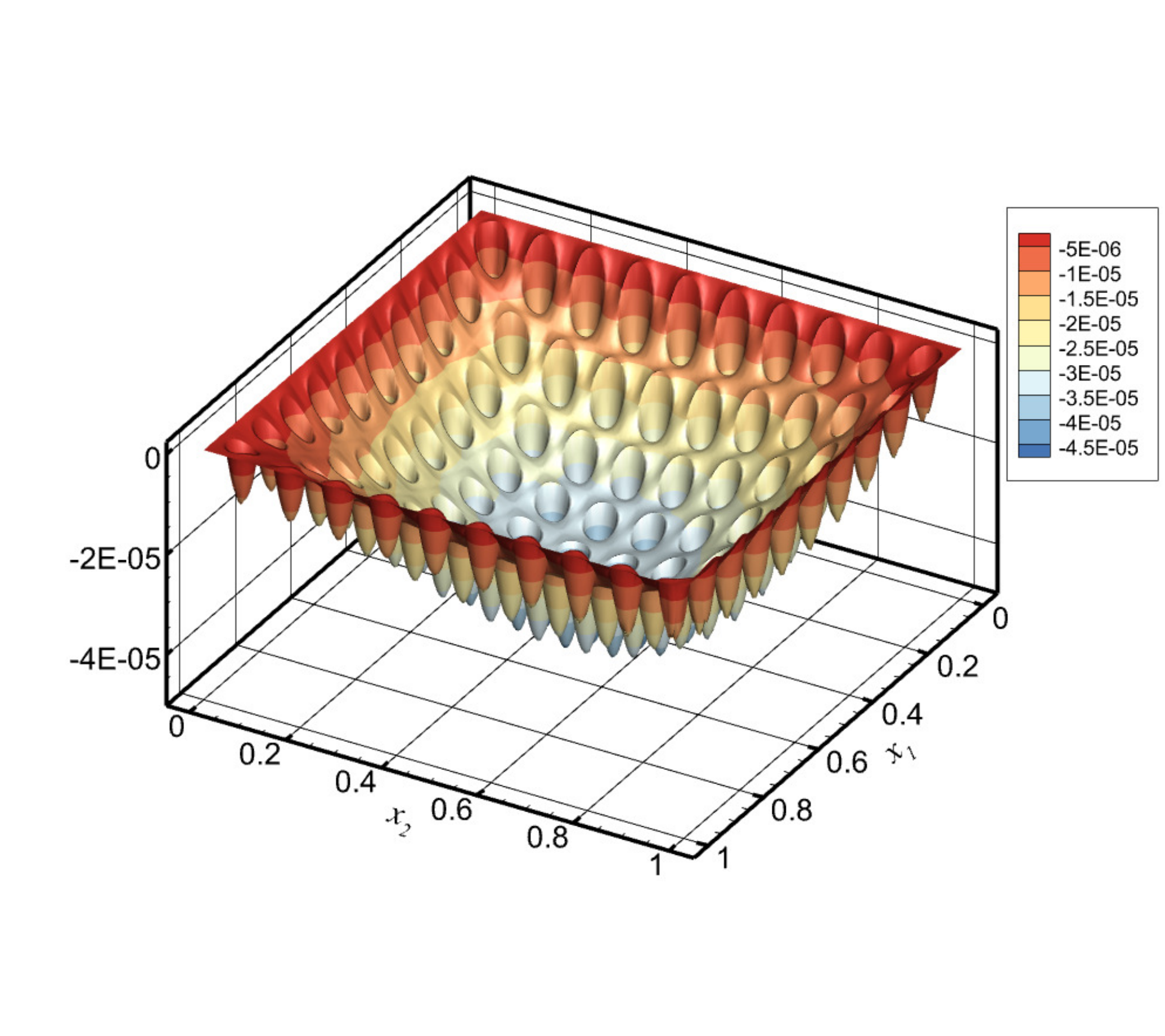}
			(c)
		\end{minipage}
		\hfill
		\begin{minipage}[c]{0.24\textwidth}
			\centering
			\includegraphics[width=\linewidth]{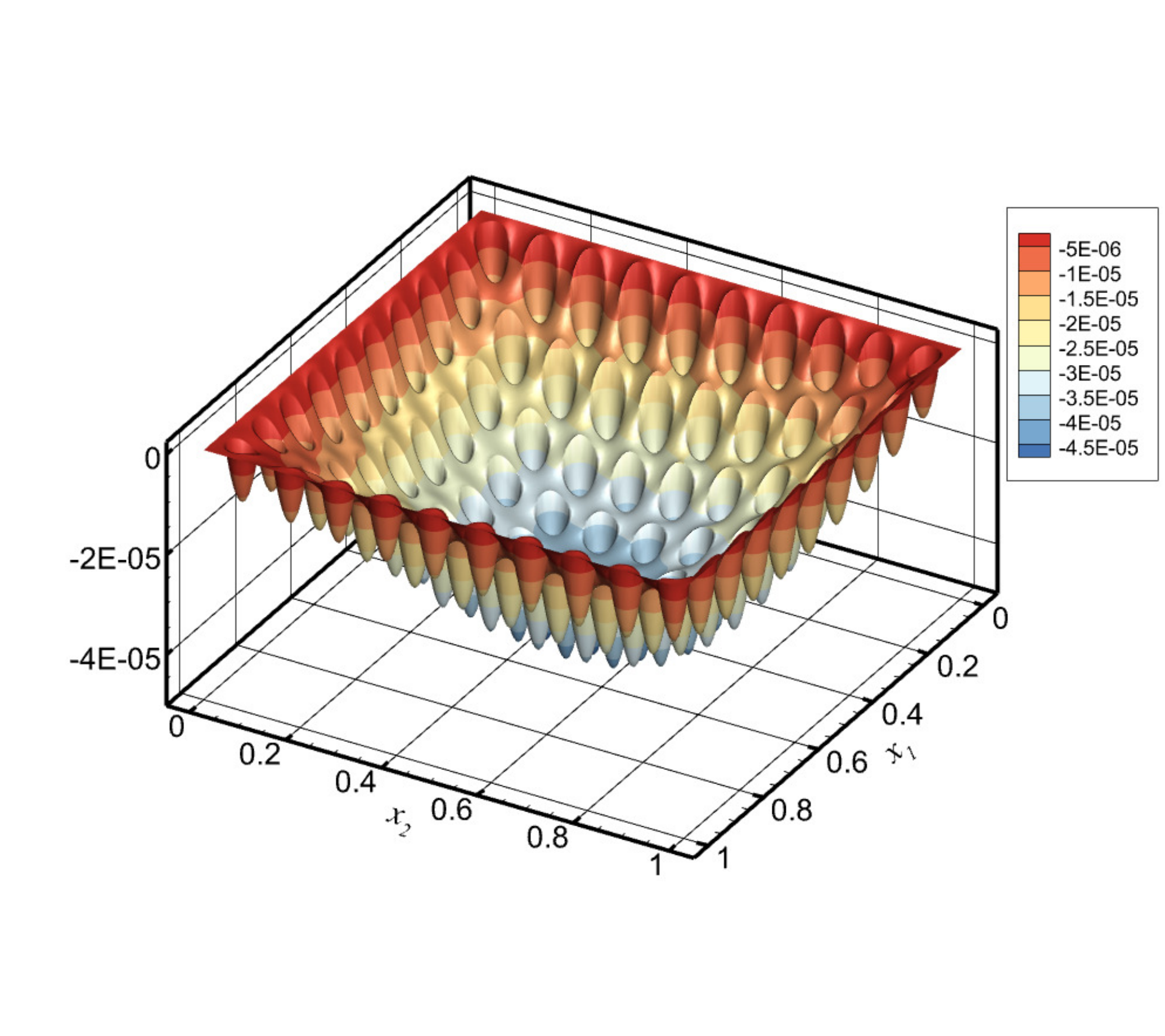}
			(d)
		\end{minipage}
		\caption{The first component of the displacement field at time $t=1.0\mathrm{s}$: (a) $u_1^{(0)}$; (b) $u_1^{(1,\epsilon)}$; (c) $u_1^{(2,\epsilon)}$; (d) $u_{1e}$.}\label{f4}
	\end{figure}
	\begin{figure}[!htb]
		\centering
		\begin{minipage}[c]{0.3\textwidth}
			\centering
			\includegraphics[width=\linewidth]{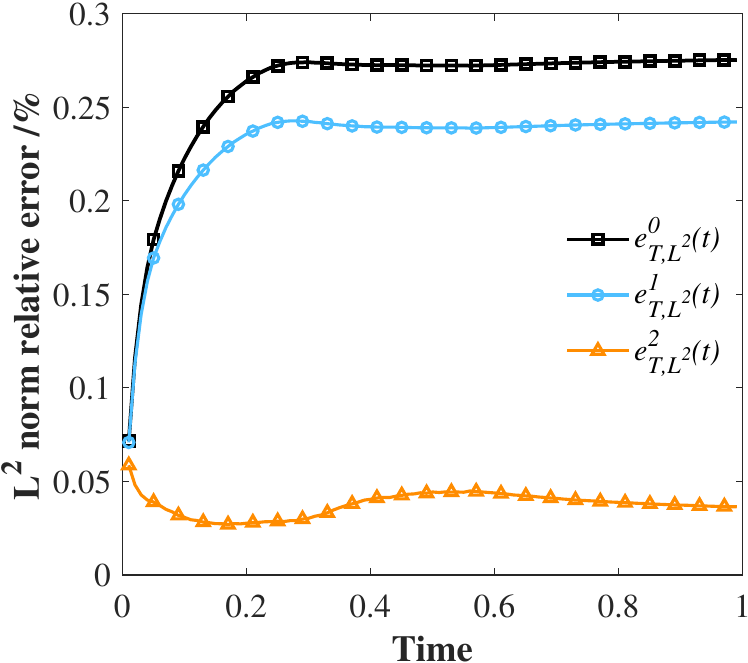}
			(a)
		\end{minipage}
		\hfill
		\begin{minipage}[c]{0.3\textwidth}
			\centering
			\includegraphics[width=\linewidth]{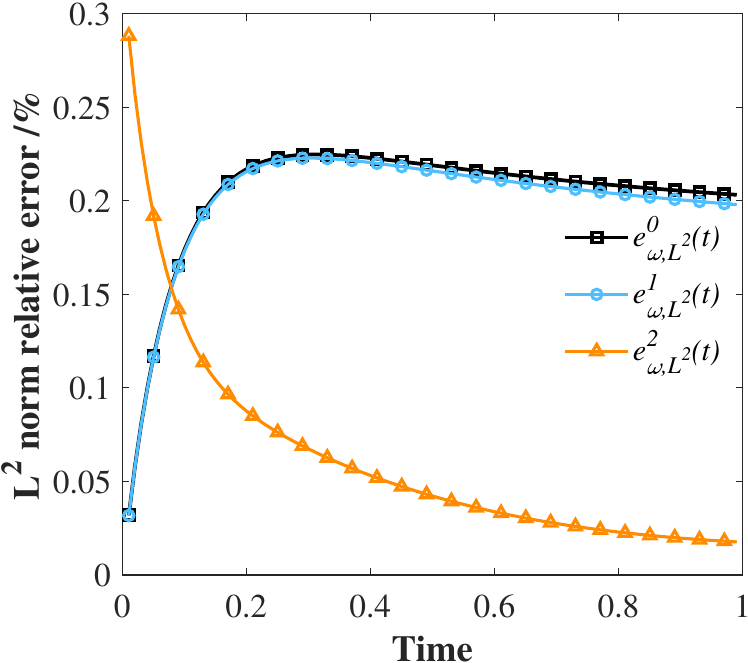}
			(b)
		\end{minipage}
		\hfill
		\begin{minipage}[c]{0.3\textwidth}
			\centering
			\includegraphics[width=\linewidth]{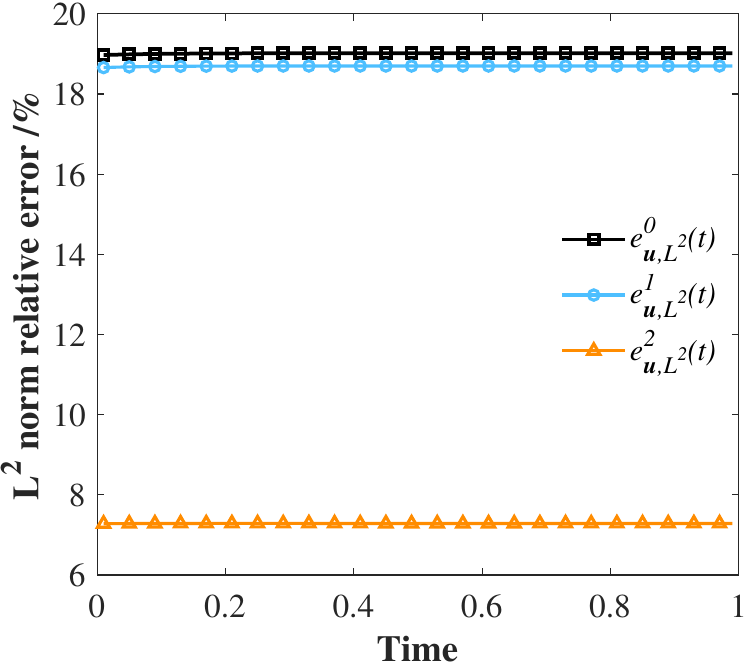}
			(c)
		\end{minipage}
		\hfill
		\begin{minipage}[c]{0.3\textwidth}
			\centering
			\includegraphics[width=\linewidth]{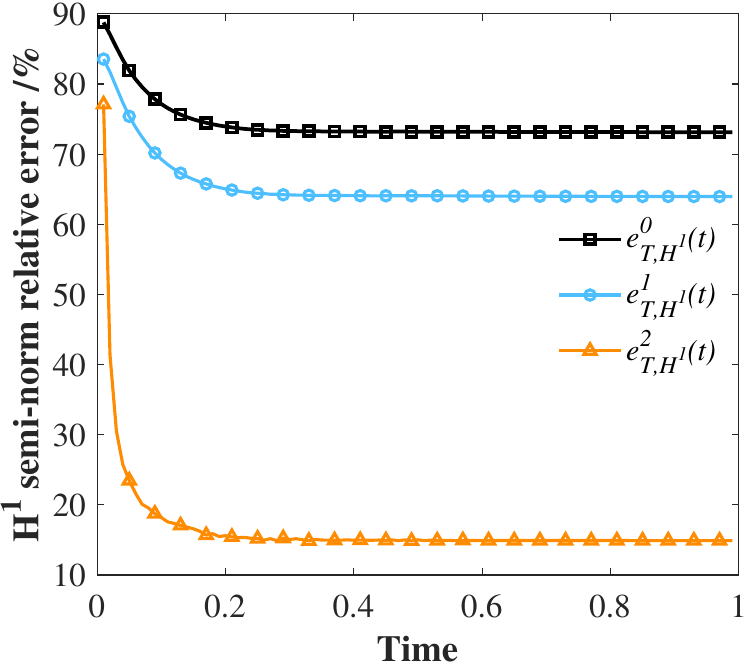}
			(d)
		\end{minipage}
		\hfill
		\begin{minipage}[c]{0.3\textwidth}
			\centering
			\includegraphics[width=\linewidth]{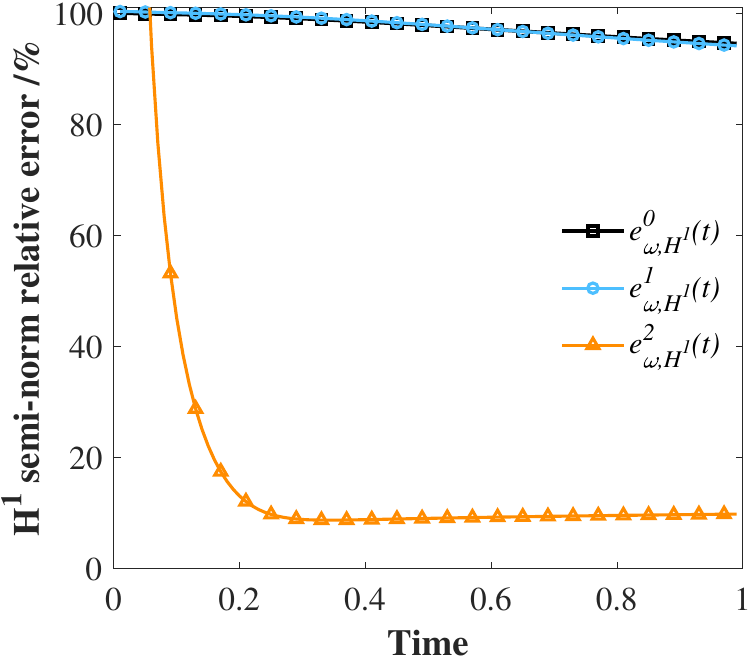}
			(e)
		\end{minipage}
		\hfill
		\begin{minipage}[c]{0.3\textwidth}
			\centering
			\includegraphics[width=\linewidth]{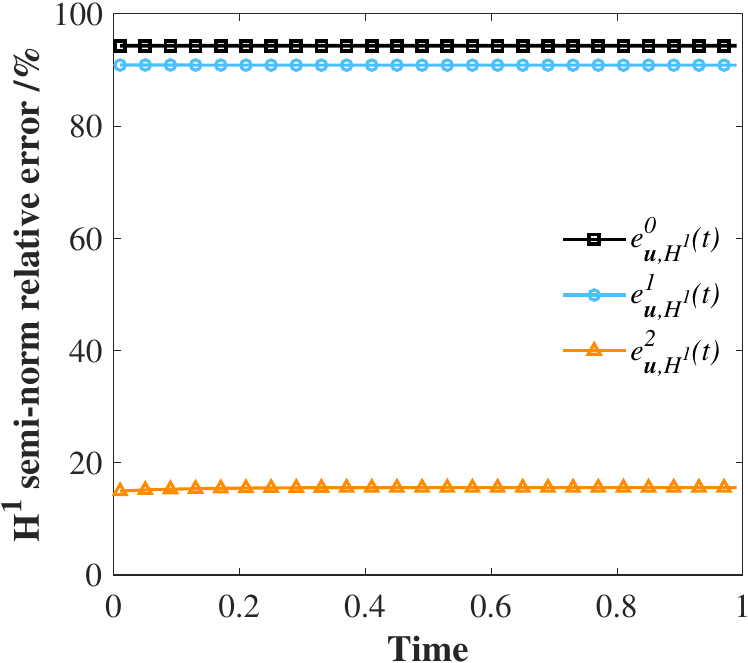}
			(f)
		\end{minipage}
		\caption{The evolutive relative errors of temperature, moisture and displacement fields: (a) $e_{T,L^2}(t)$; (b) $e_{\omega,L^2}(t)$; (c) $e_{\bm{u},L^2}(t)$; (d) $e_{T,H^1}(t)$; (E) $e_{\omega,H^1}(t)$; (F) $e_{\bm{u},H^1}(t)$.}\label{f6}
	\end{figure}
	
	As illustrated in Figs.~\ref{f2}-\ref{f4}, the HOMS approximate solutions agree well with the precise FEM solutions, and the numerical accuracy of the HOMS method is considerably superior to that of the homogenized and LOMS solutions. Moreover, the HOMS solutions can accurately capture the microscopic oscillatory behavior of the material. In contrast, the homogenized approach only captures the macroscopic response, while the LOMS method can only capture limited microscopic responses. It is noteworthy that Fig.~\ref{f6} reveals that the HOMS solutions for temperature, moisture and displacement produce substantially smaller relative errors with respect to the precise FEM solutions than those of the homogenized and LOMS solutions, both in the $L^2$ norm and the $H^1$ semi-norm. In addition, the proposed two-stage algorithm maintains stability and efficiency even over long simulation times, exhibiting no blow-up or numerical deterioration up to the final time $t=1.0 \mathrm{s}$. In practical engineering applications, where direct FEM simulations become computationally prohibitive for large-scale heterogeneous structures containing a large number of unit cells, the HOMS method offers an efficient and robust alternative with low computational cost, making it well-suited for dynamic hygro-thermo-mechanical coupling simulations.
	
	\subsection{Example 2: nonlinear hygro-thermo-mechanical coupling simulation of 3D heterogeneous structure}
	\label{sec:52}
	This example examines the nonlinear dynamic hygro‑thermo‑mechanical response of a 3D heterogeneous structure consisting of periodic unit cells with $\epsilon=1/5$. The whole domain $\Omega= (x_1,x_2,x_3) = [0,1]^3 \mathrm{cm}^3$ and PUC $Y$ are displayed in Fig.~\ref{f1:3D}.
	\begin{figure}[!htb]
		\centering
		\begin{minipage}[c]{0.32\textwidth}
			\centering
			\includegraphics[width=40mm]{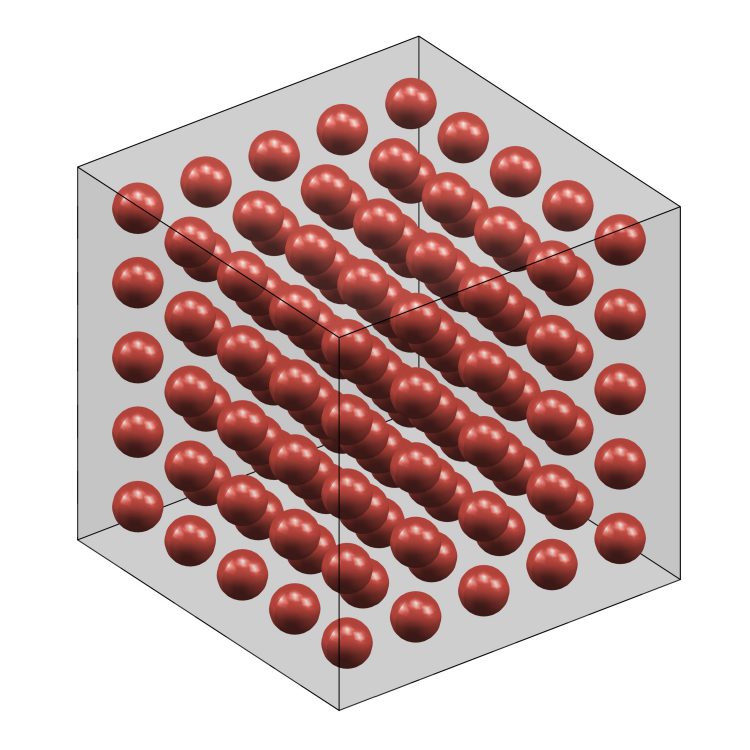}\\
			(a)
		\end{minipage}
		\begin{minipage}[c]{0.3\textwidth}
			\centering
			\includegraphics[width=40mm]{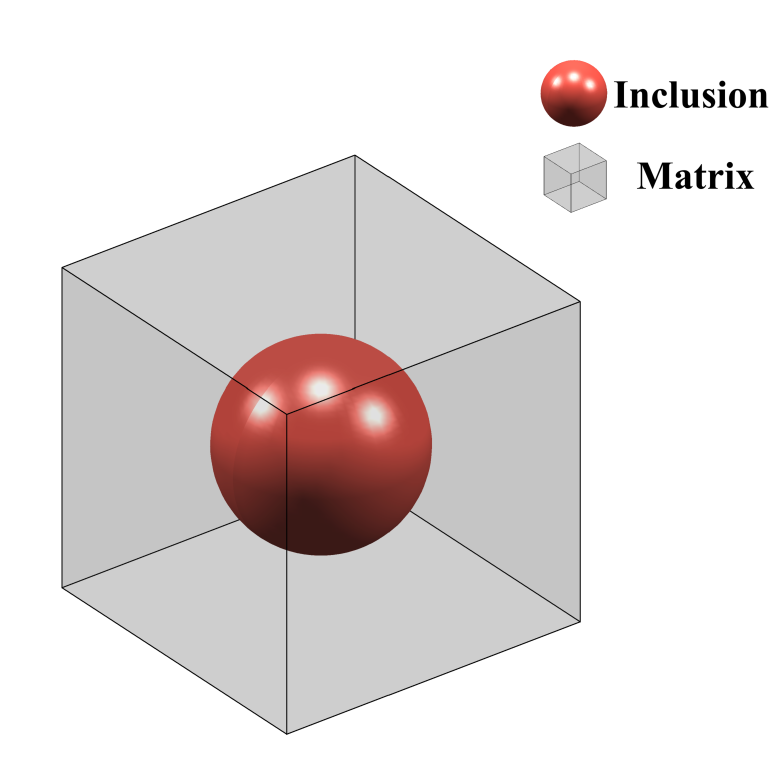}\\
			(b)
		\end{minipage}
		\begin{minipage}[c]{0.3\textwidth}
			\centering
			\includegraphics[width=40mm]{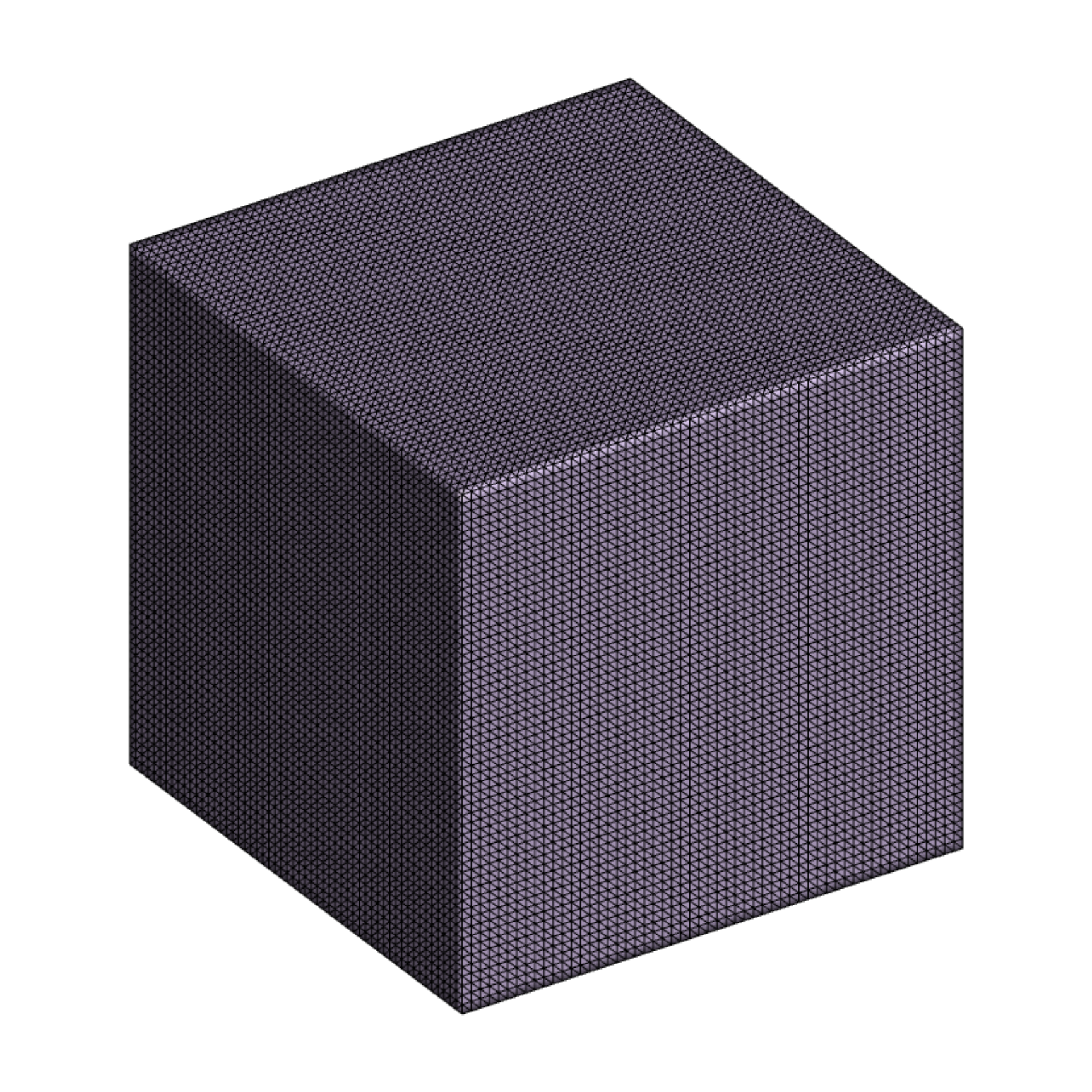}\\
			(c)
		\end{minipage}
		\caption{(a) The 3D heterogeneous structure $\Omega$; (b) PUC $Y$; (c) homogenized structure $\Omega$.}\label{f1:3D}
	\end{figure}
	
	For this example, the material parameters of the studied 3D periodic heterogeneous structure are given in Table~\ref{t3}.
	\begin{table}[!t]
		\caption{Material property parameters ($T$ represents temperature and $\omega$ represents moisture).\label{t3}}
		\centering
		\small
		\begin{tabular*}{\columnwidth}{@{\extracolsep\fill}lcc@{\extracolsep\fill}}
			\toprule
			Property & Matrix & Inclusion \\
			\midrule
			Mass density $\rho^{\epsilon}$ (kg/m$^3$) & 7500.0 & 2500.0 \\
			Specific heat $c^{\epsilon}$ (J/(kg$\cdot$K)) & $3600.0+6.0T-6.0\times10^{-4}T^2$ & $1080.0+1.8T-1.8\times10^{-4}T^2$ \\
			Young's modulus $E^{\epsilon}$ (GPa) & $1050.0-0.030435T-1.0\times10^{-12}T^2$ & $220.0-1.1018\times10^{-4}T-1.0\times10^{-14}T^2$ \\
			Poisson's ratio $\nu^{\epsilon}$ & 0.25 & 0.20 \\
			Thermal conductivity $k^{\epsilon}_{ij}$ (W/(m$\cdot$K)) & $800.0+0.5T+0.0025$ & $50.0+0.005T+2.5\times10^{-5}T^2$ \\
			Moisture diffusion $g^{\epsilon}_{ij}$ (m$^2$/s) & $1.5\times10^{-5}+5.0\times10^{-10}\omega+2.5\times10^{-14}\omega^2$ & $1.5\times10^{-6}+5.0\times10^{-11}\omega+2.5\times10^{-15}\omega^2$ \\
			Thermal stress coefficient $\alpha^{\epsilon}_{ij}$ (MPa/K) & $1.0-1.0\times10^{-4}T-1.0\times10^{-7}T^2$ & $0.05-1.0\times10^{-5}T-1.0\times10^{-8}T^2$ \\
			Moisture stress coefficient $\beta_{ij}^{\epsilon}$ (MPa) & $0.1-1.0\times10^{-5}T-1.0\times10^{-8}T^2$ & $0.005-1.0\times10^{-6}T-1.0\times10^{-9}T^2$ \\
			\bottomrule
		\end{tabular*}
	\end{table}
	
	Moreover, except that the heat source and body forces are given by $h=1500.0 \ \mathrm{J/(cm^{3}\cdot s)}$ and $(f_1,f_2,f_3)=(0,0,-30000) \ \mathrm{N/cm^3}$, all other source items, boundary conditions and initial conditions are the same as in Section~\ref{sec:51}.
	
	To avoid the prohibitive computational cost of directly resolving the fine-scale features, separate tetrahedral meshes are constructed for the multi-scale nonlinear problem \eqref{eq:2.1}, the auxiliary cell problems, and the corresponding homogenized equations \eqref{eq:2.27}. Table~\ref{t4} provides a detailed breakdown of the mesh resolution and execution time.
	\begin{table}[!t]
		\caption{Comparison of computational cost ($\Delta t=0.01\,$s, $t\in[0,1.0]\,$s).\label{t4}}
		\centering
		\begin{tabular*}{\columnwidth}{@{\extracolsep\fill}cccc@{\extracolsep\fill}}
			\toprule
			& Cell equations & Homogenized equations & Multi-scale equations \\
			\midrule
			FEM nodes & 5978 & 175616 & 759403 \\
			FEM elements & 33492 & 998250 & 4762528 \\
			\midrule
			& \multicolumn{2}{c}{HOMS method} & precise FEM \\
			\midrule
			Computational time & \multicolumn{2}{c}{186566.598\,s} & 603471.170\,s \\
			\bottomrule
		\end{tabular*}
	\end{table}
	
	According to Table~\ref{t4}, compared with the precise FEM approach, the proposed HOMS method greatly reduces memory usage and computational time, cutting runtime by about $69.08\%$. Furthermore, the computational time saved by the HOMS method increases with the duration of the multi-scale simulation.
	
	For this example, $10$ equidistant interpolation points each for macroscopic temperature and moisture are prescribed within a single unit cell. The nonlinear dynamic hygro‑thermo‑mechanical behavior of the 3D heterogeneous structure is simulated over the time interval $t \in [0, 1.0] \mathrm{s}$. With a time step $\Delta t = 0.01 \mathrm{s}$, the macroscopic homogenized equations \eqref{eq:2.27} and the multi‑scale nonlinear equations \eqref{eq:2.1} are solved on‑line, respectively. The final temperature, moisture and displacement fields at $t = 1.0 \mathrm{s}$ are presented in Figs.~\ref{f8}-\ref{f12}. Moreover, Fig.~\ref{f13} illustrates the evolution of relative errors of temperature, moisture and displacement fields.
	\begin{figure}[!htb]
		\centering
		\begin{minipage}[c]{0.24\textwidth}
			\centering
			\includegraphics[width=\linewidth]{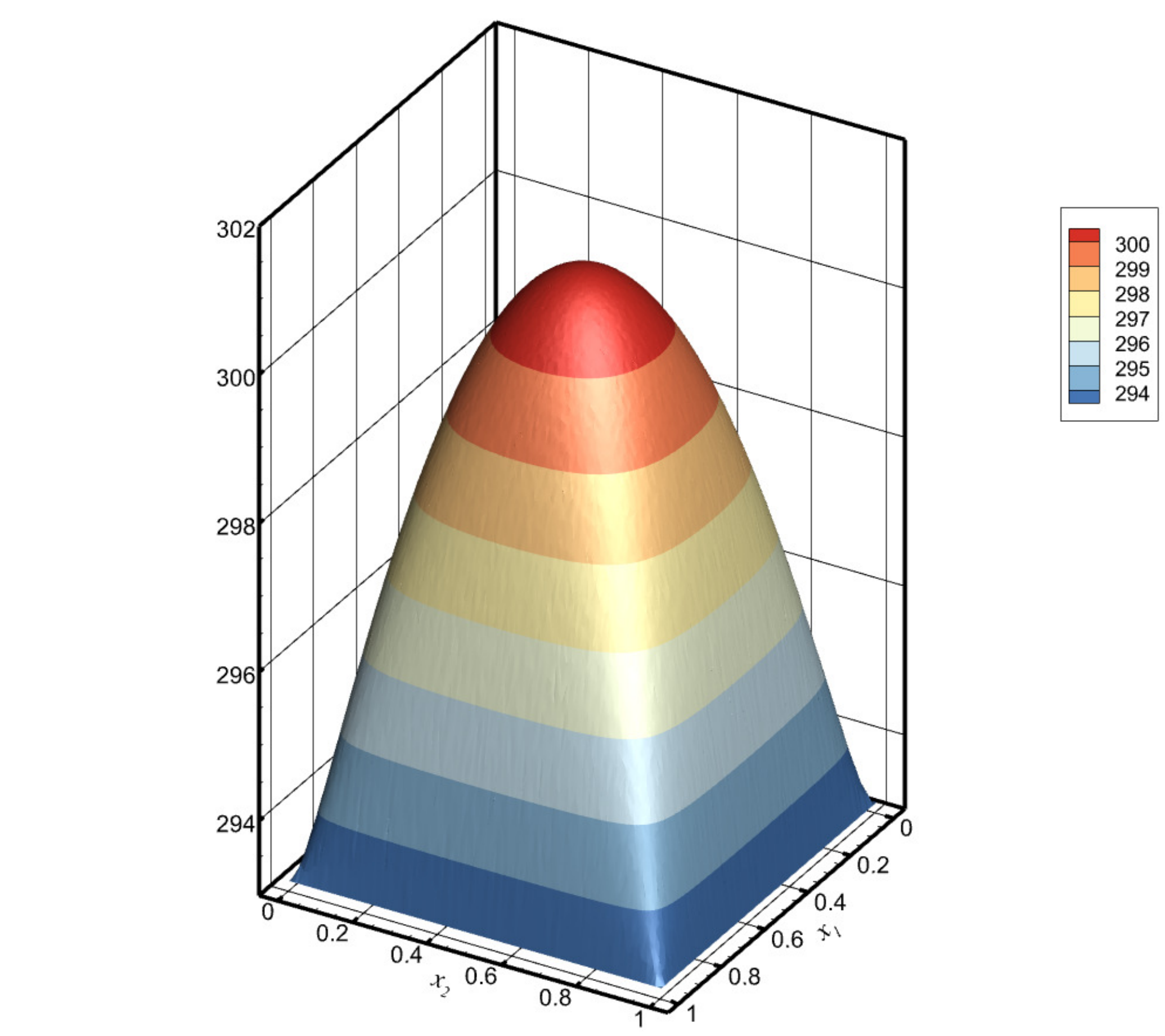}
			(a)
		\end{minipage}
		\hfill
		\begin{minipage}[c]{0.24\textwidth}
			\centering
			\includegraphics[width=\linewidth]{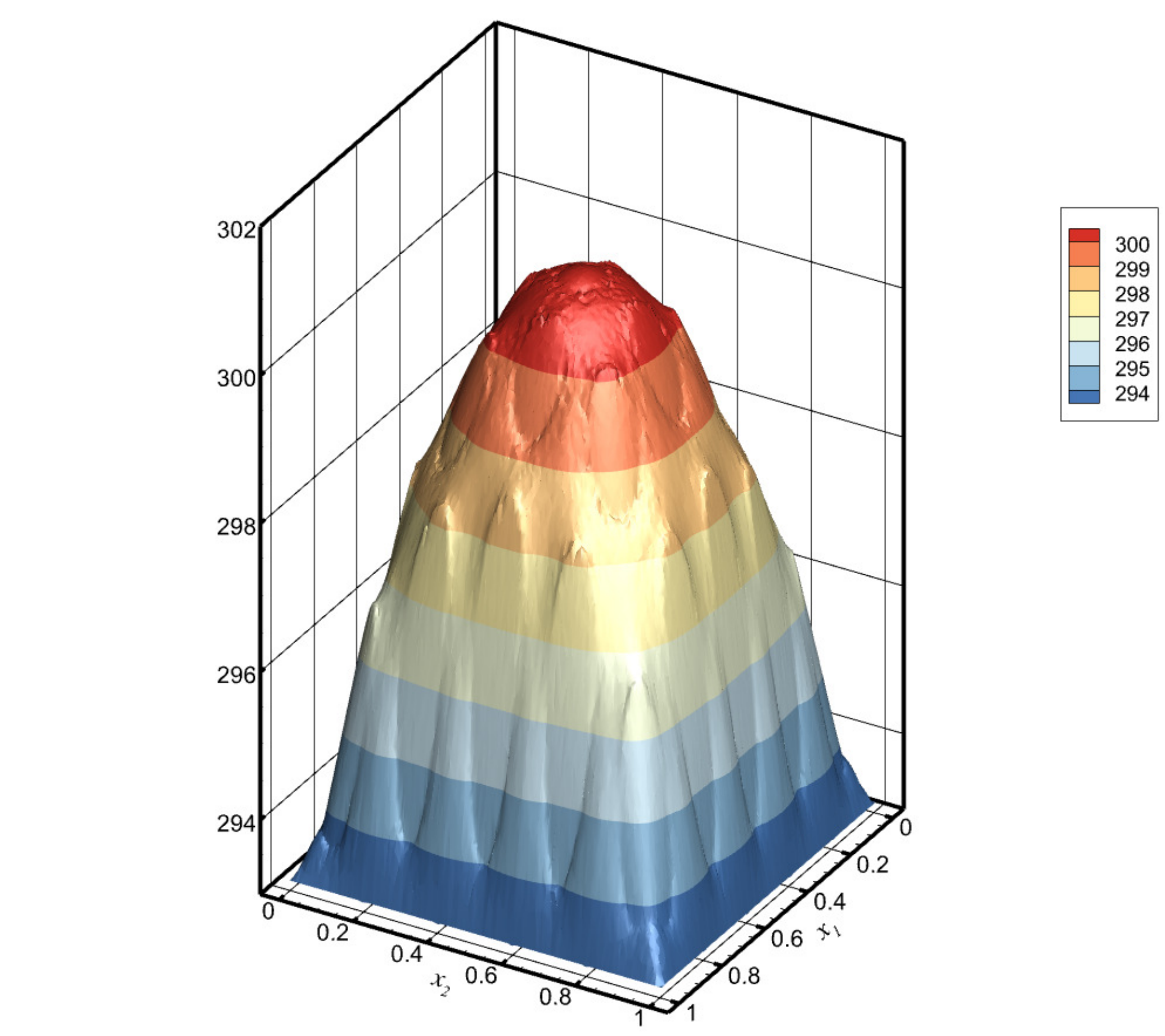}
			(b)
		\end{minipage}
		\hfill
		\begin{minipage}[c]{0.24\textwidth}
			\centering
			\includegraphics[width=\linewidth]{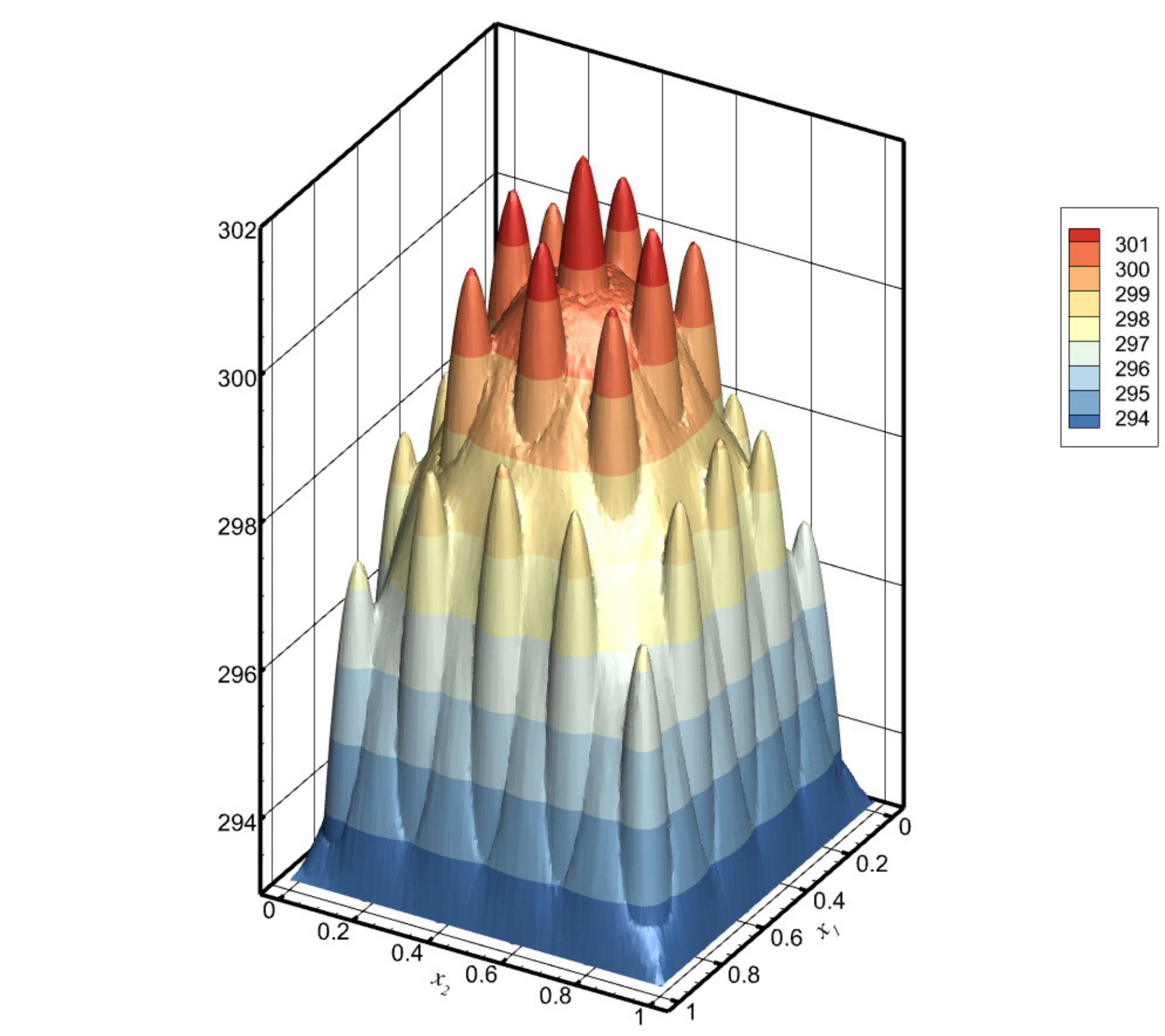}
			(c)
		\end{minipage}
		\hfill
		\begin{minipage}[c]{0.24\textwidth}
			\centering
			\includegraphics[width=\linewidth]{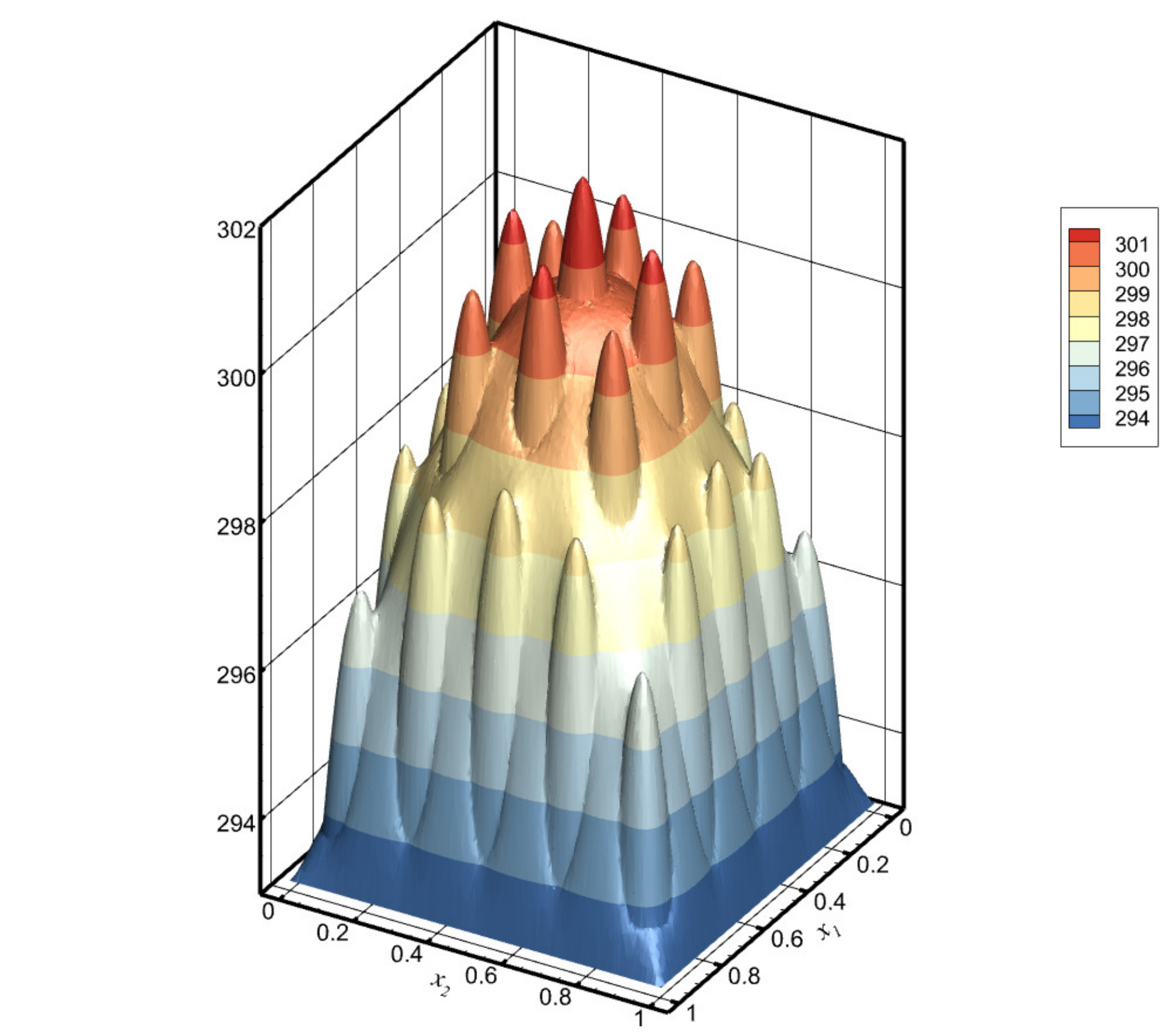}
			(d)
		\end{minipage}
		\caption{The temperature field in cross section $x_3=0.1 \mathrm{cm}$ at time $t=1.0\mathrm{s}$: (a) $T^{(0)}$; (b) $T^{(1,\epsilon)}$; (c) $T^{(2,\epsilon)}$; (d) $T_e$.}\label{f8}
	\end{figure}
	\begin{figure}[!htb]
		\centering
		\begin{minipage}[c]{0.24\textwidth}
			\centering
			\includegraphics[width=\linewidth]{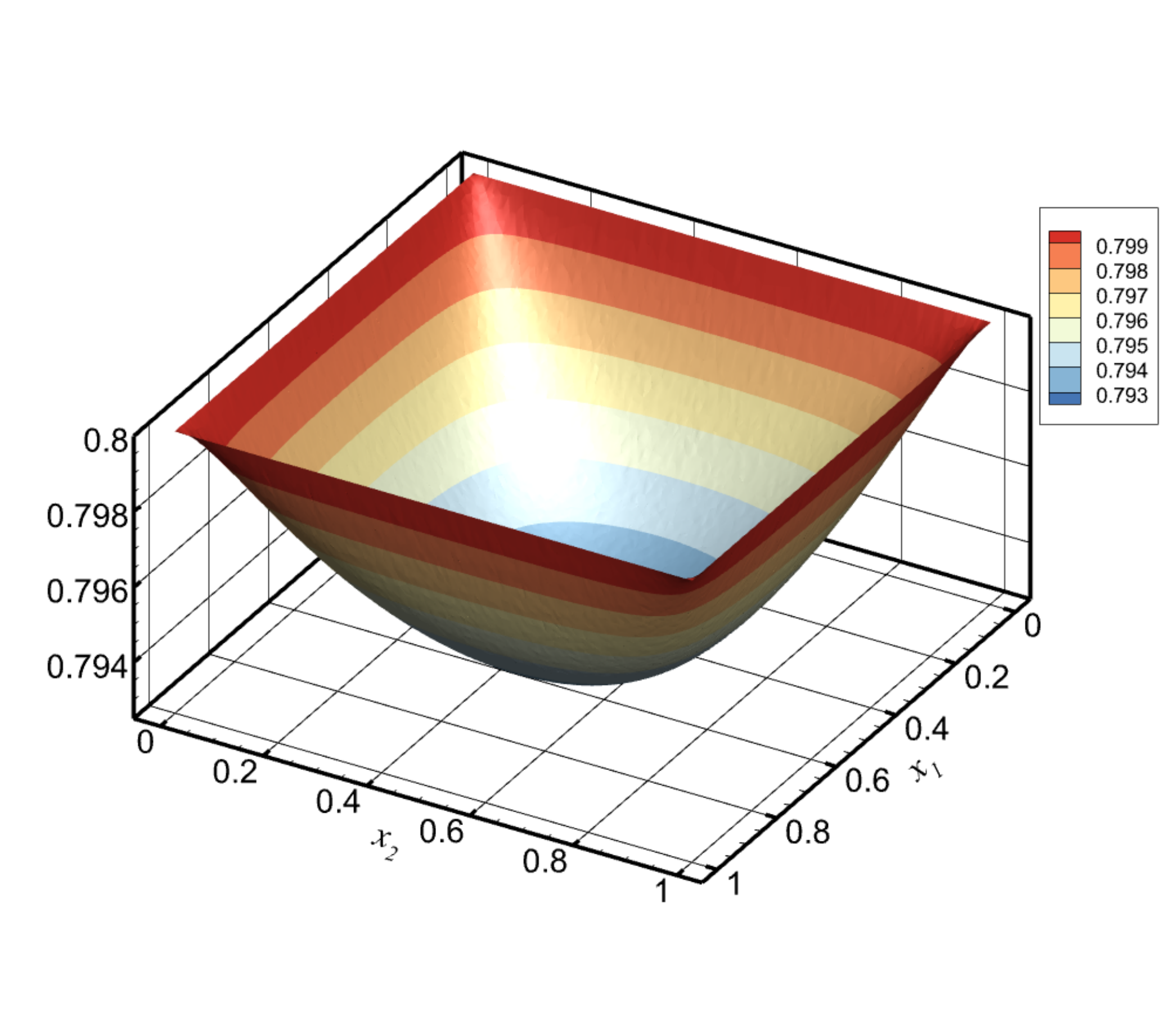}
			(a)
		\end{minipage}
		\hfill
		\begin{minipage}[c]{0.24\textwidth}
			\centering
			\includegraphics[width=\linewidth]{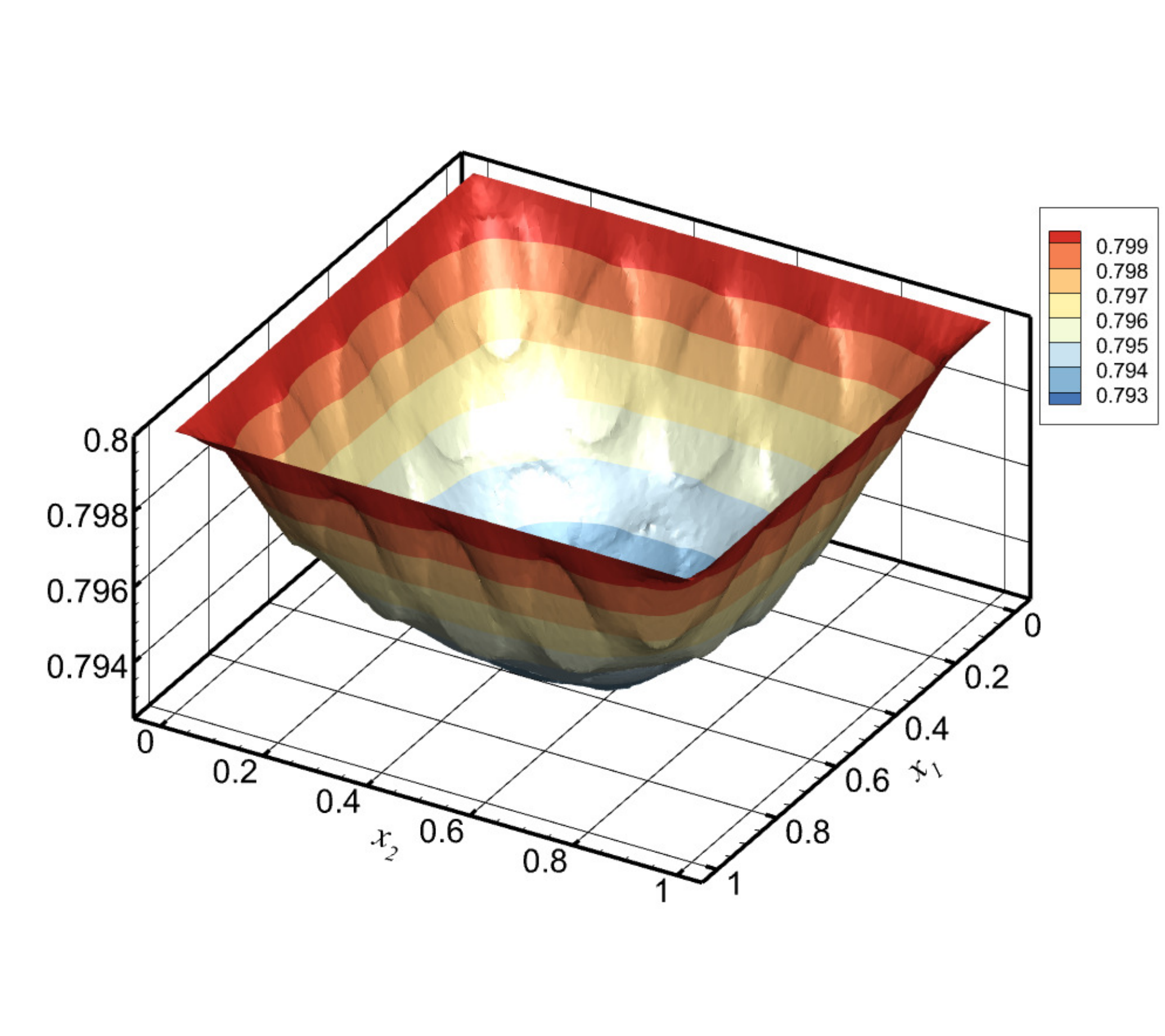}
			(b)
		\end{minipage}
		\hfill
		\begin{minipage}[c]{0.24\textwidth}
			\centering
			\includegraphics[width=\linewidth]{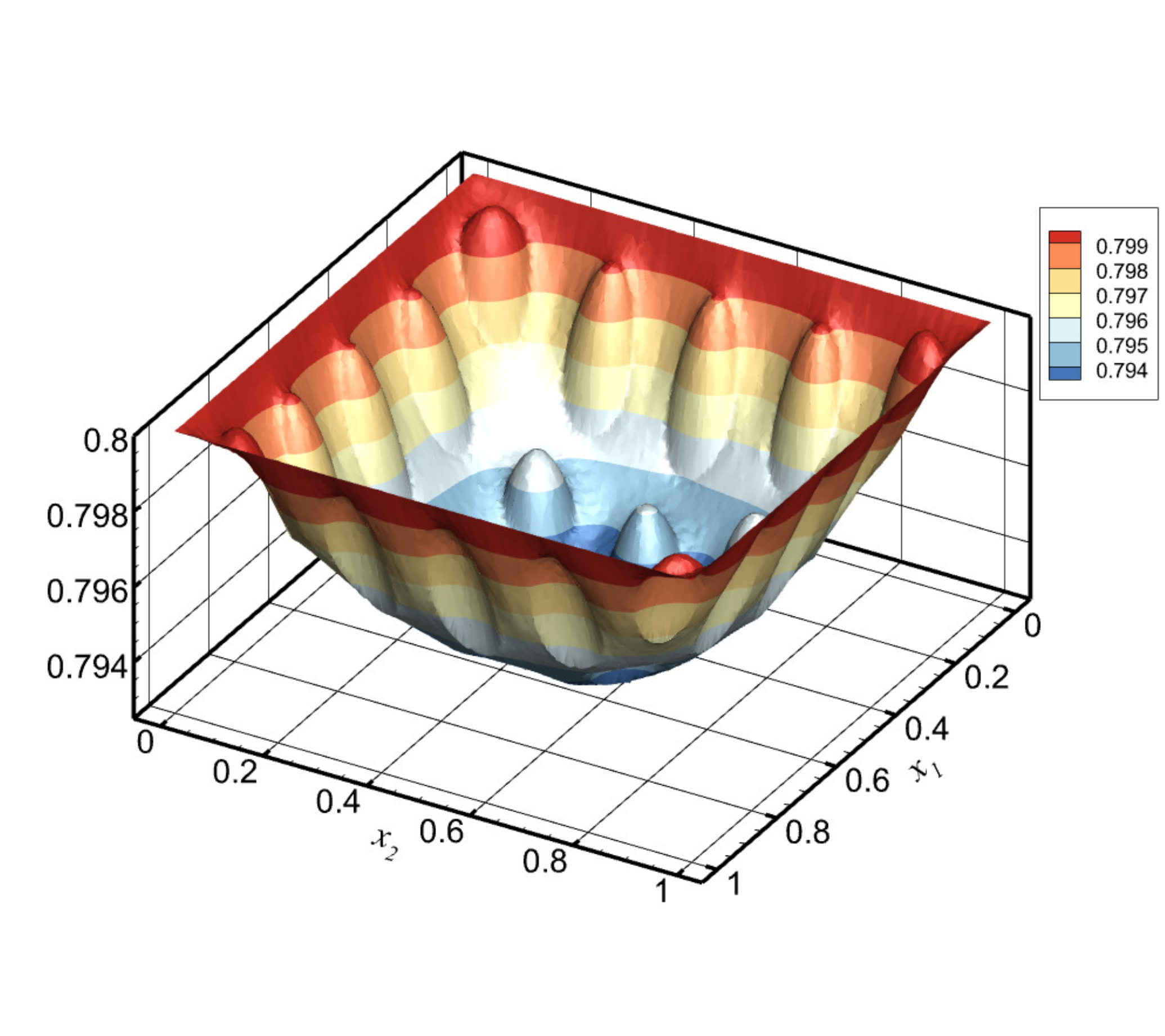}
			(c)
		\end{minipage}
		\hfill
		\begin{minipage}[c]{0.24\textwidth}
			\centering
			\includegraphics[width=\linewidth]{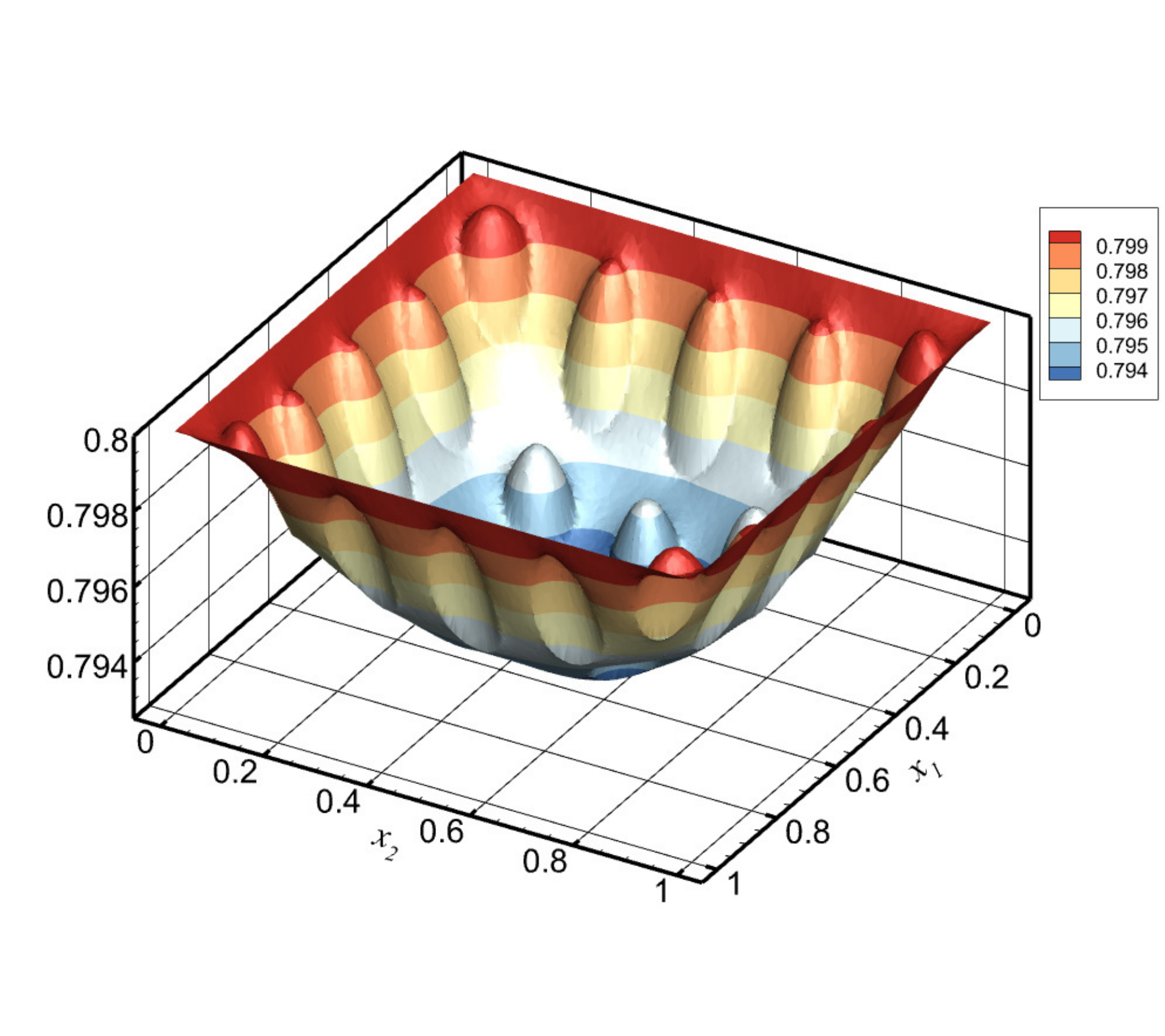}
			(d)
		\end{minipage}
		\caption{The moisture field in cross section $x_3=0.1 \mathrm{cm}$ at time $t=1.0\mathrm{s}$: (a) $\omega^{(0)}$; (b) $\omega^{(1,\epsilon)}$; (c) $\omega^{(2,\epsilon)}$; (d) $\omega_e$.}\label{f9}
	\end{figure}
	\begin{figure}[!htb]
		\centering
		\begin{minipage}[c]{0.24\textwidth}
			\centering
			\includegraphics[width=\linewidth]{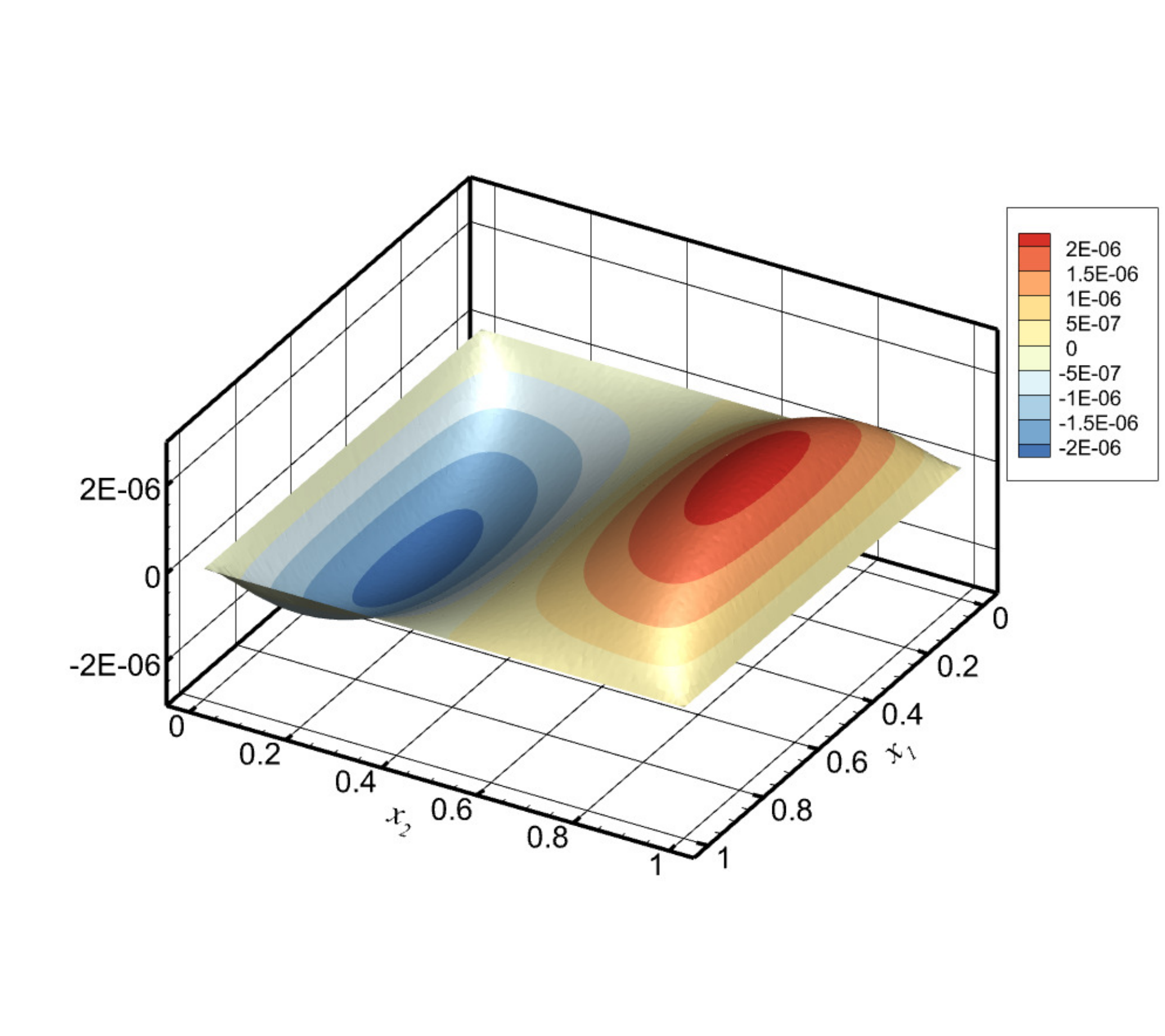}
			(a)
		\end{minipage}
		\hfill
		\begin{minipage}[c]{0.24\textwidth}
			\centering
			\includegraphics[width=\linewidth]{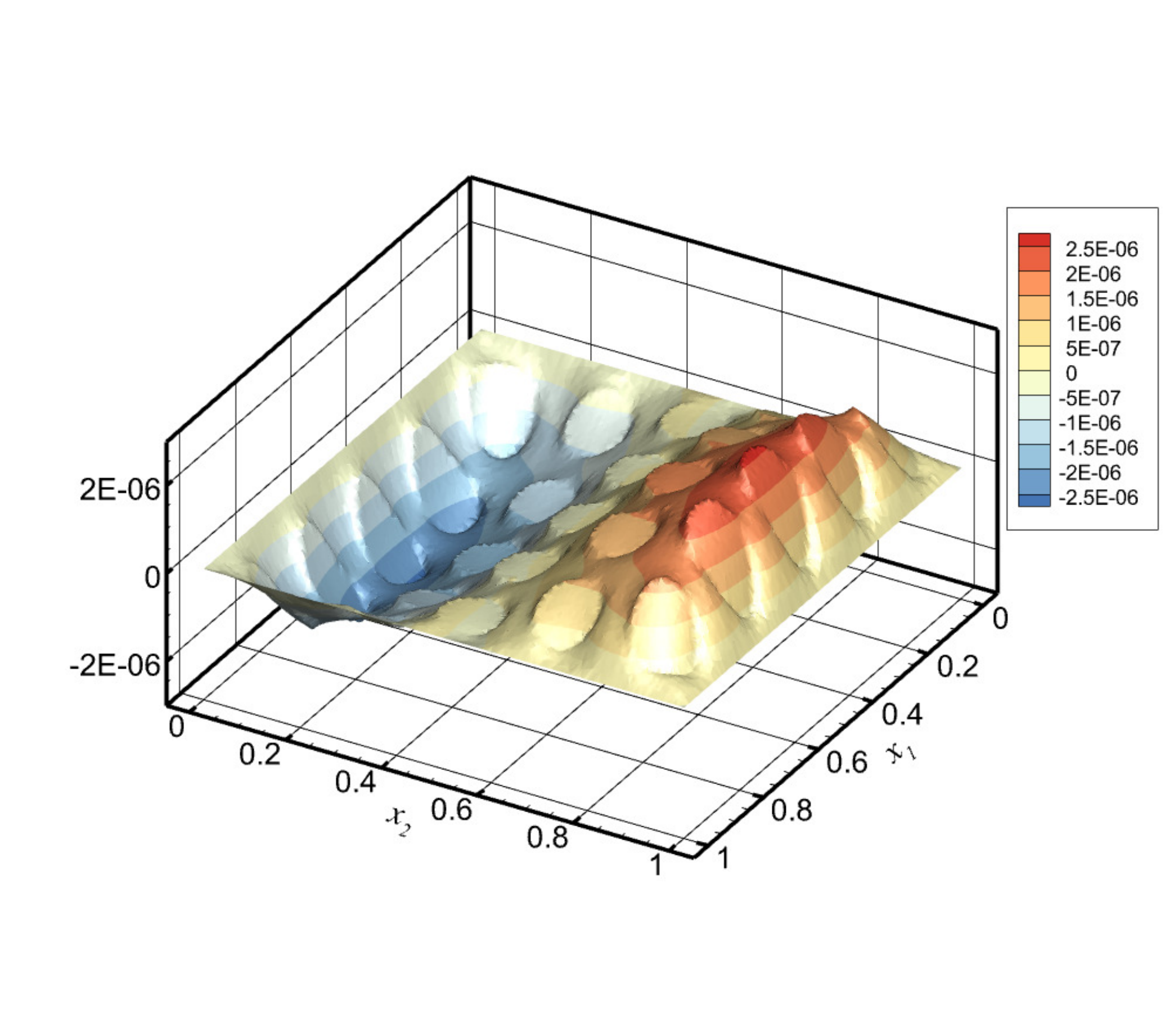}
			(b)
		\end{minipage}
		\hfill
		\begin{minipage}[c]{0.24\textwidth}
			\centering
			\includegraphics[width=\linewidth]{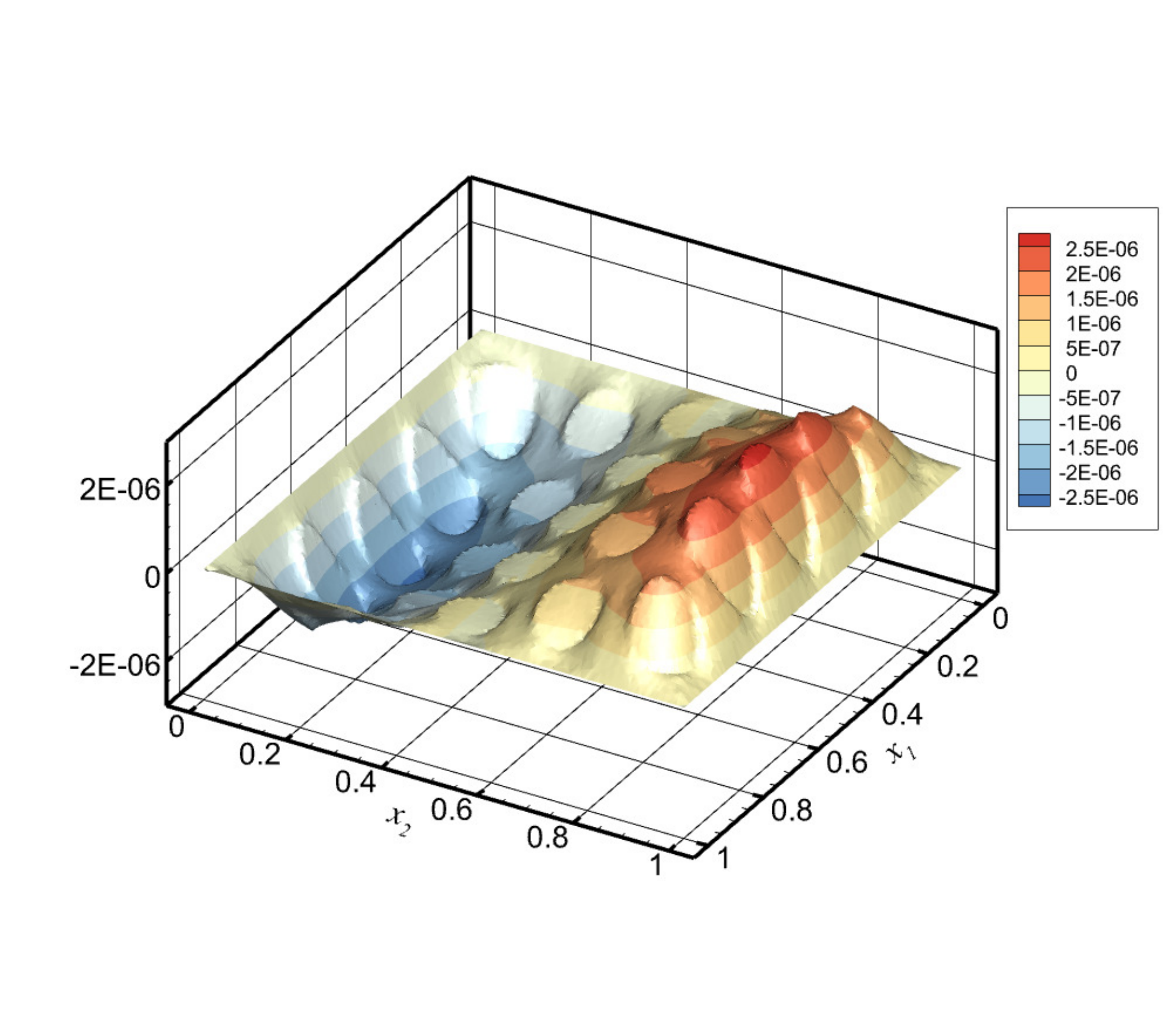}
			(c)
		\end{minipage}
		\hfill
		\begin{minipage}[c]{0.24\textwidth}
			\centering
			\includegraphics[width=\linewidth]{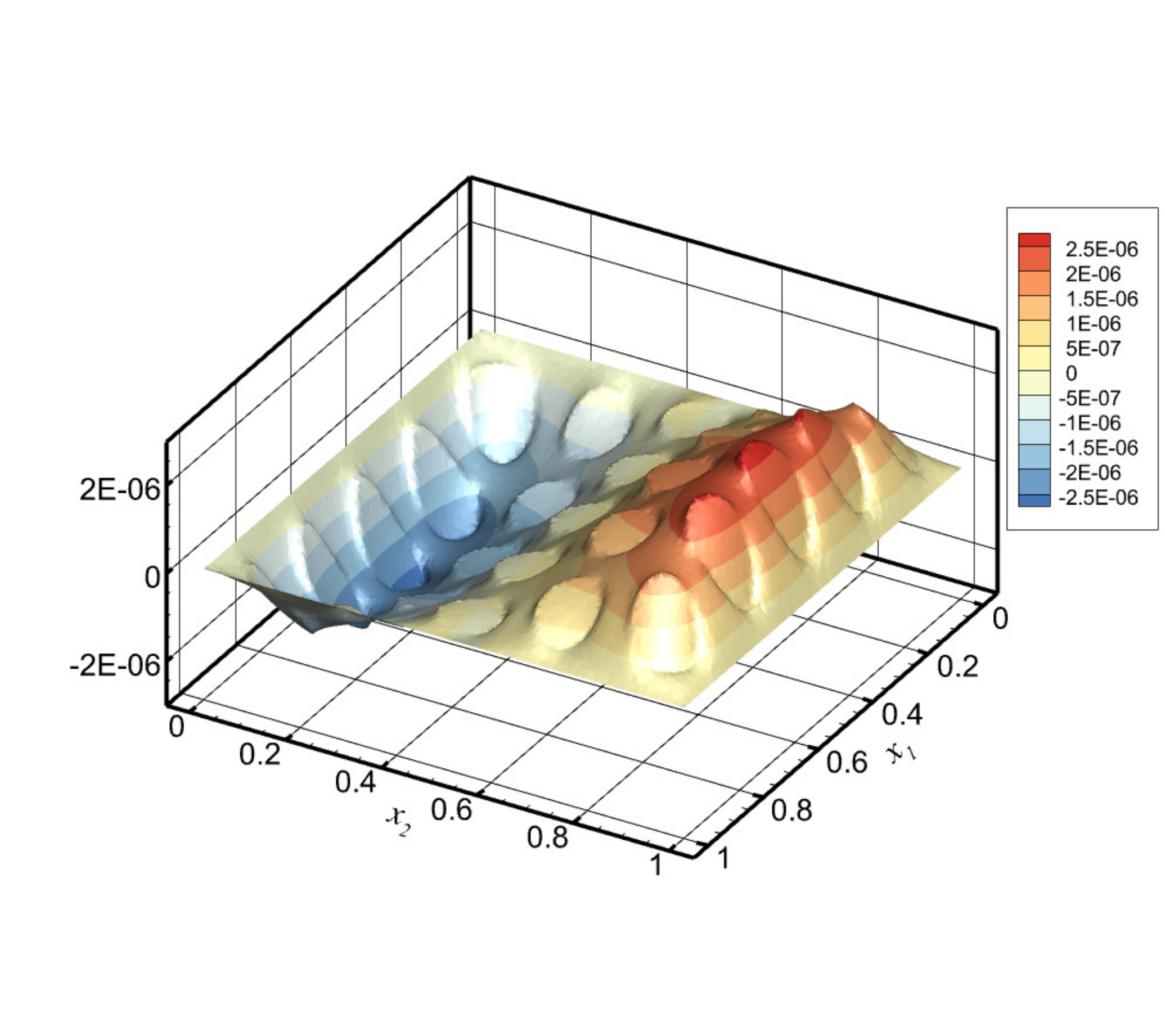}
			(d)
		\end{minipage}
		\caption{The second component of the displacement field in cross section $x_3=0.1 \mathrm{cm}$ at time $t=1.0\mathrm{s}$: (a) $u_2^{(0)}$; (b) $u_2^{(1,\epsilon)}$; (c) $u_2^{(2,\epsilon)}$; (d) $u_{2e}$.}\label{f11}
	\end{figure}
	\begin{figure}[!htb]
		\centering
		\begin{minipage}[c]{0.24\textwidth}
			\centering
			\includegraphics[width=\linewidth]{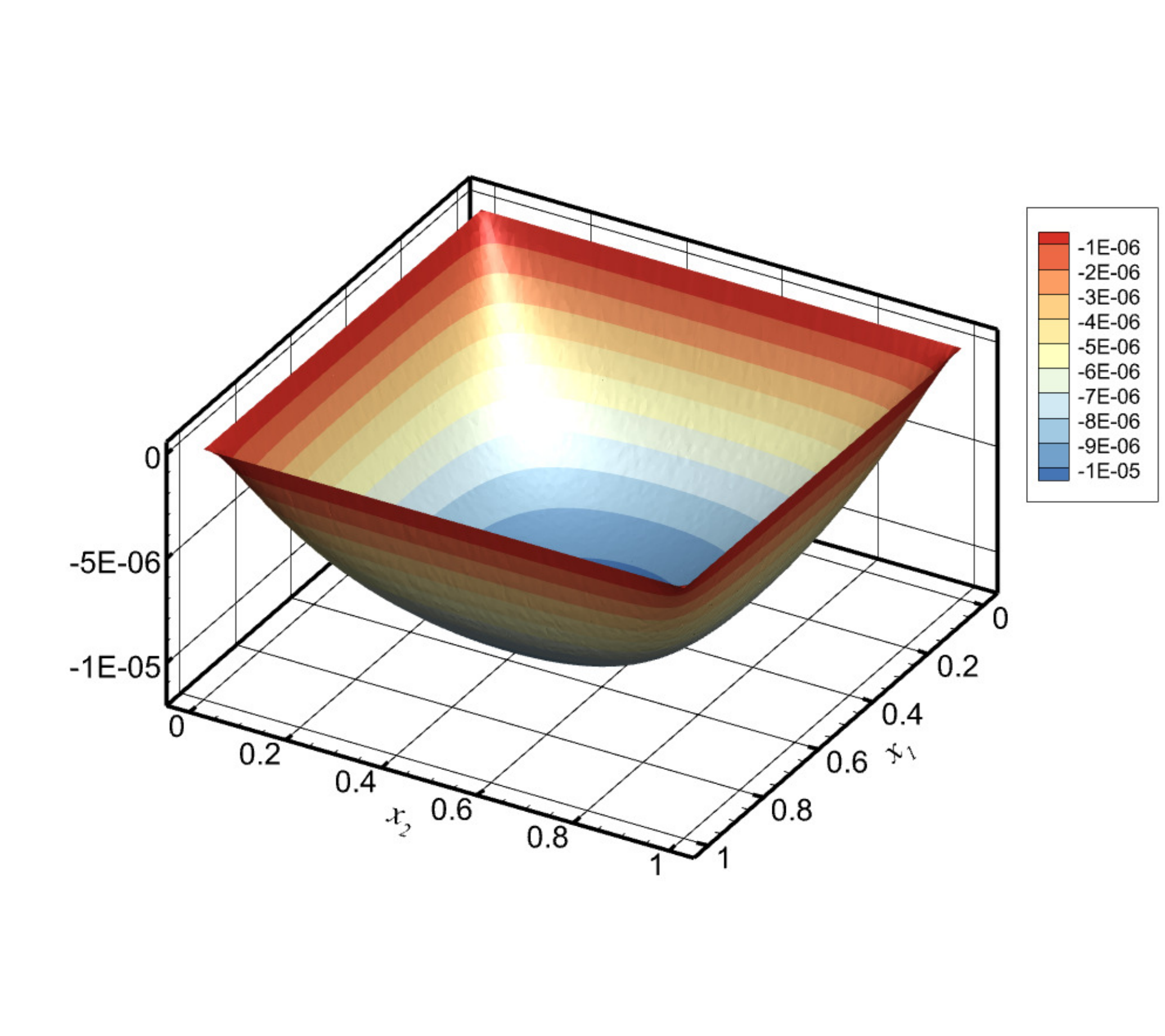}
			(a)
		\end{minipage}
		\hfill
		\begin{minipage}[c]{0.24\textwidth}
			\centering
			\includegraphics[width=\linewidth]{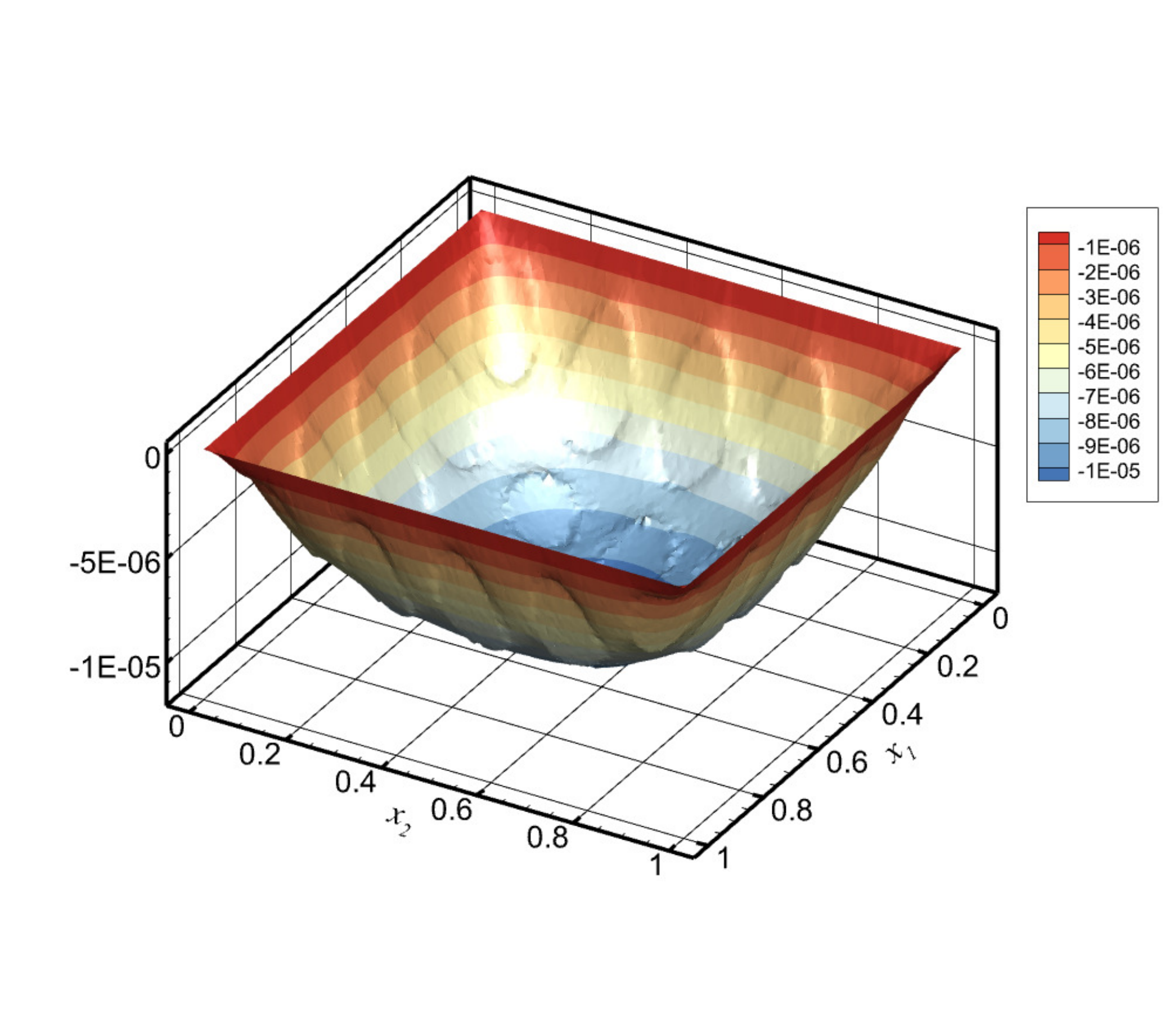}
			(b)
		\end{minipage}
		\hfill
		\begin{minipage}[c]{0.24\textwidth}
			\centering
			\includegraphics[width=\linewidth]{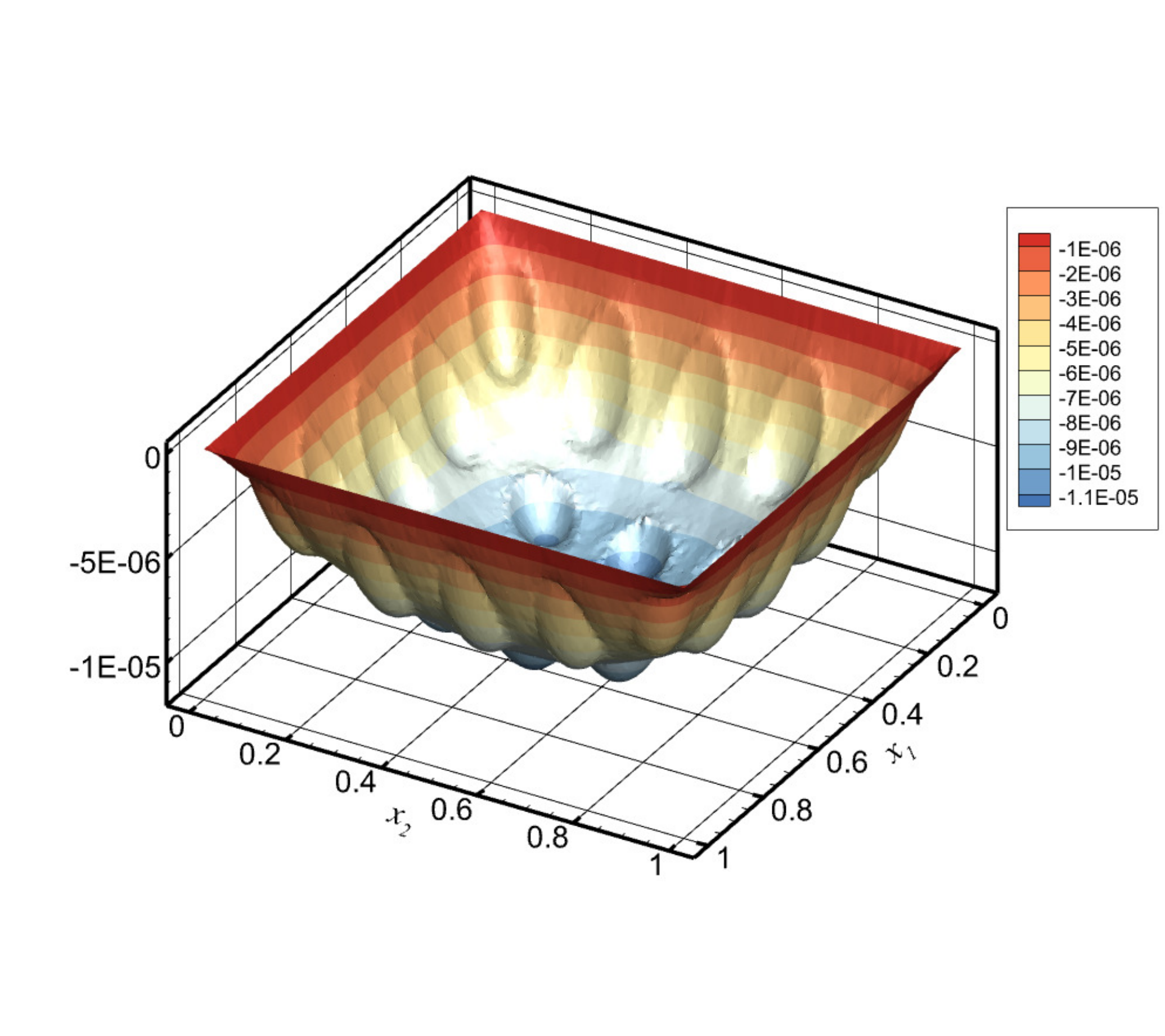}
			(c)
		\end{minipage}
		\hfill
		\begin{minipage}[c]{0.24\textwidth}
			\centering
			\includegraphics[width=\linewidth]{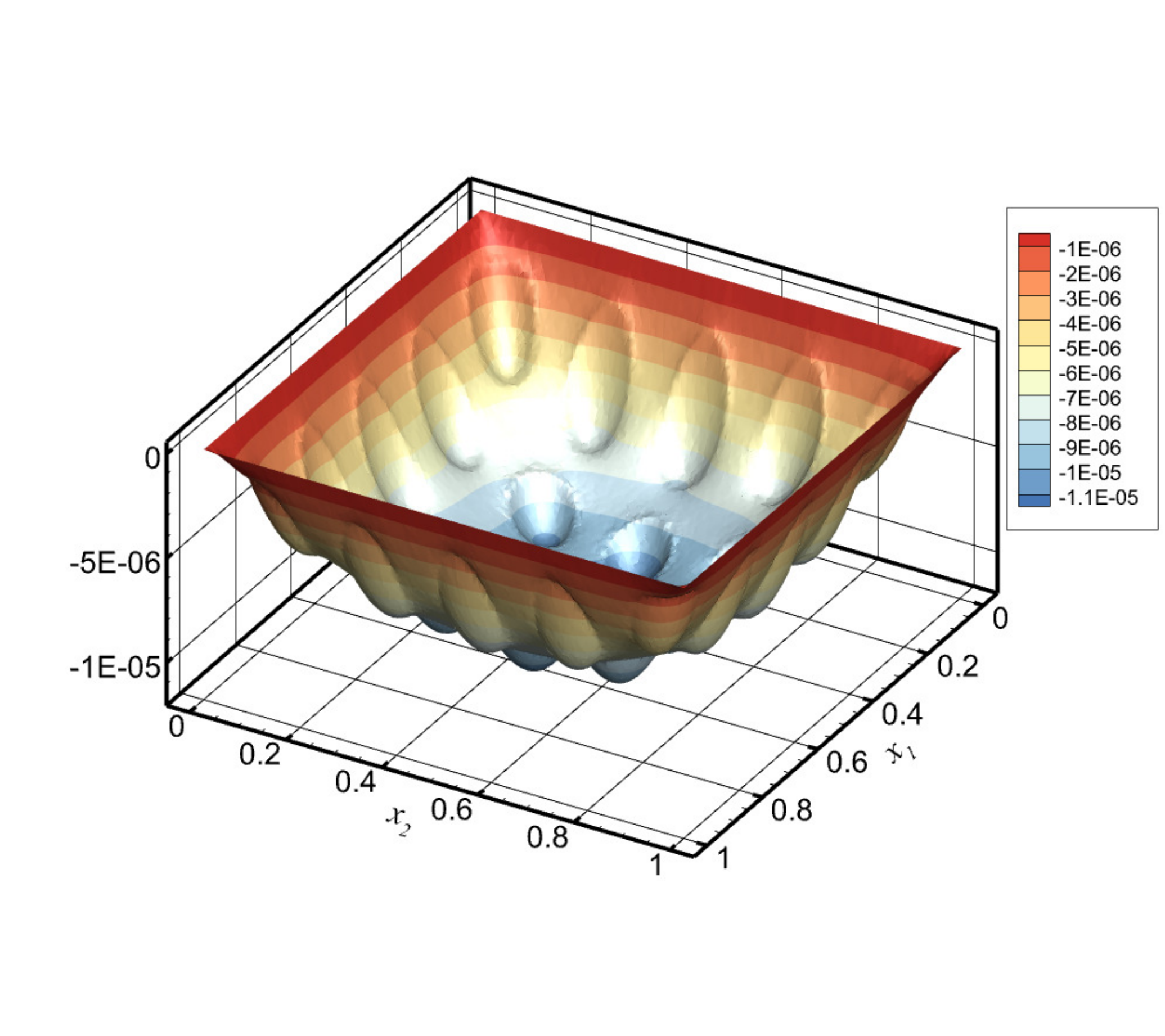}
			(d)
		\end{minipage}
		\caption{The third component of the displacement field in cross section $x_3=0.1 \mathrm{cm}$ at time $t=1.0\mathrm{s}$: (a) $u_3^{(0)}$; (b) $u_3^{(1,\epsilon)}$; (c) $u_3^{(2,\epsilon)}$; (d) $u_{3e}$.}\label{f12}
	\end{figure}
	\begin{figure}[!htb]
		\centering
		\begin{minipage}[c]{0.3\textwidth}
			\centering
			\includegraphics[width=\linewidth]{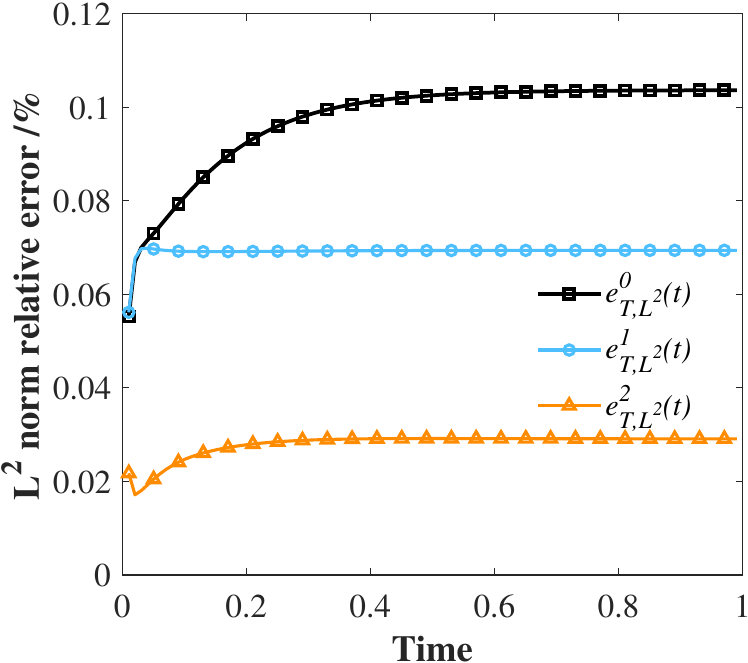}
			(a)
		\end{minipage}
		\hfill
		\begin{minipage}[c]{0.3\textwidth}
			\centering
			\includegraphics[width=\linewidth]{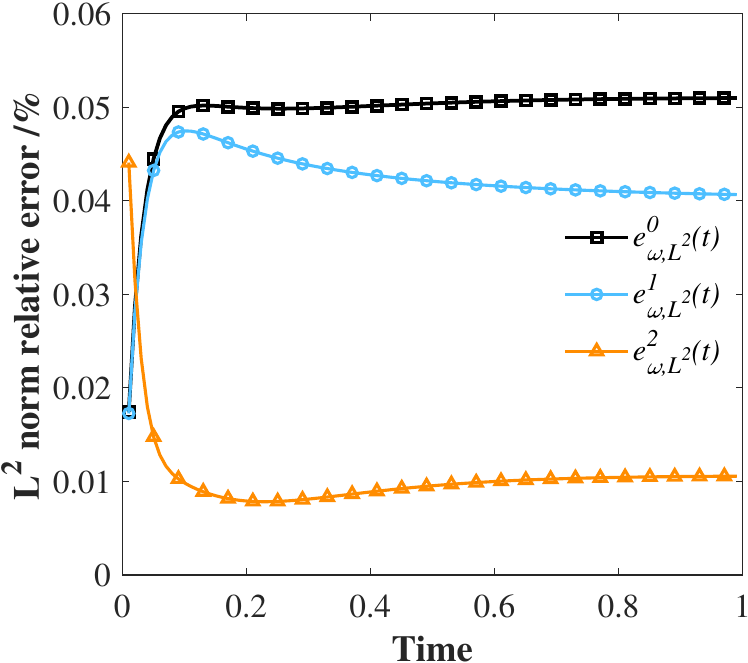}
			(b)
		\end{minipage}
		\hfill
		\begin{minipage}[c]{0.3\textwidth}
			\centering
			\includegraphics[width=\linewidth]{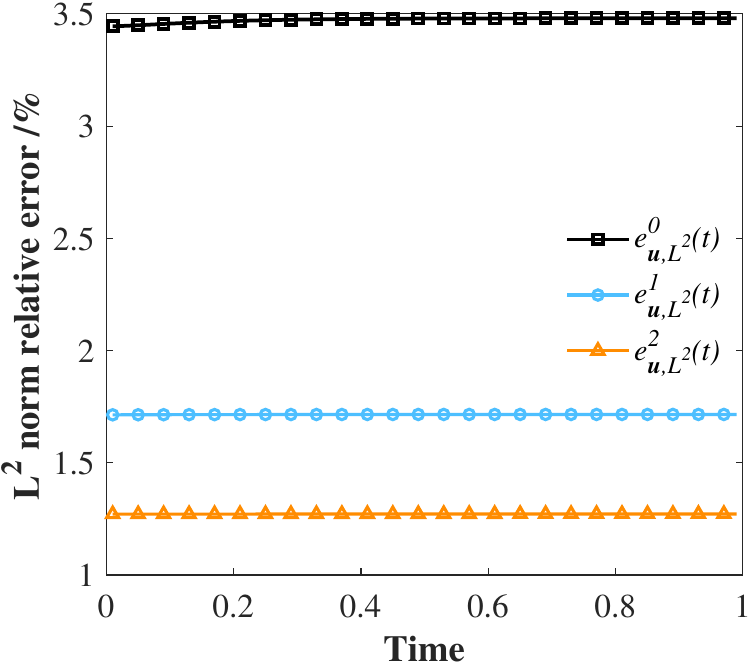}
			(c)
		\end{minipage}
		\hfill
		\begin{minipage}[c]{0.3\textwidth}
			\centering
			\includegraphics[width=\linewidth]{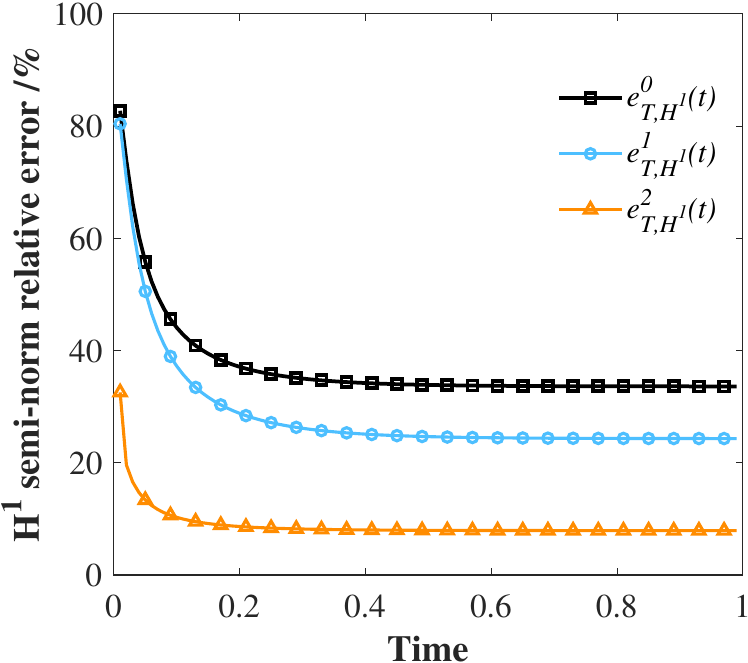}
			(d)
		\end{minipage}
		\hfill
		\begin{minipage}[c]{0.3\textwidth}
			\centering
			\includegraphics[width=\linewidth]{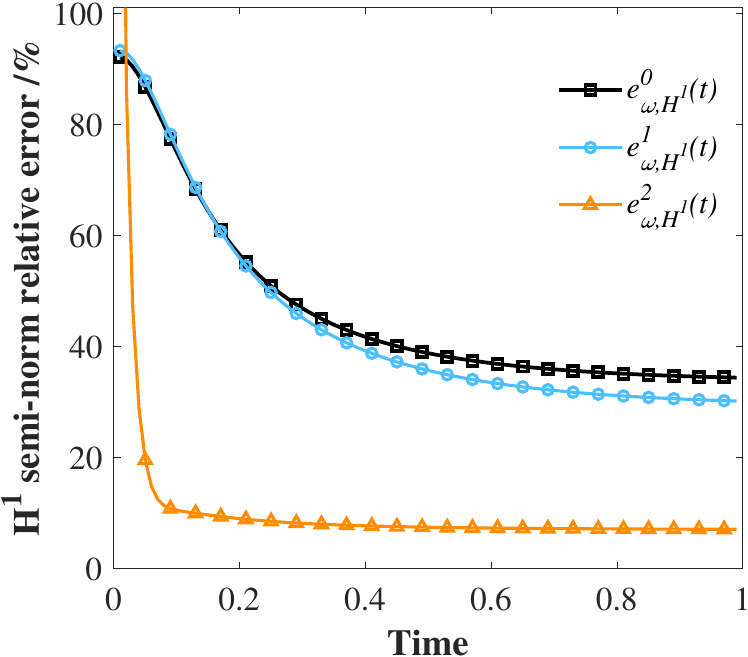}
			(e)
		\end{minipage}
		\hfill
		\begin{minipage}[c]{0.3\textwidth}
			\centering
			\includegraphics[width=\linewidth]{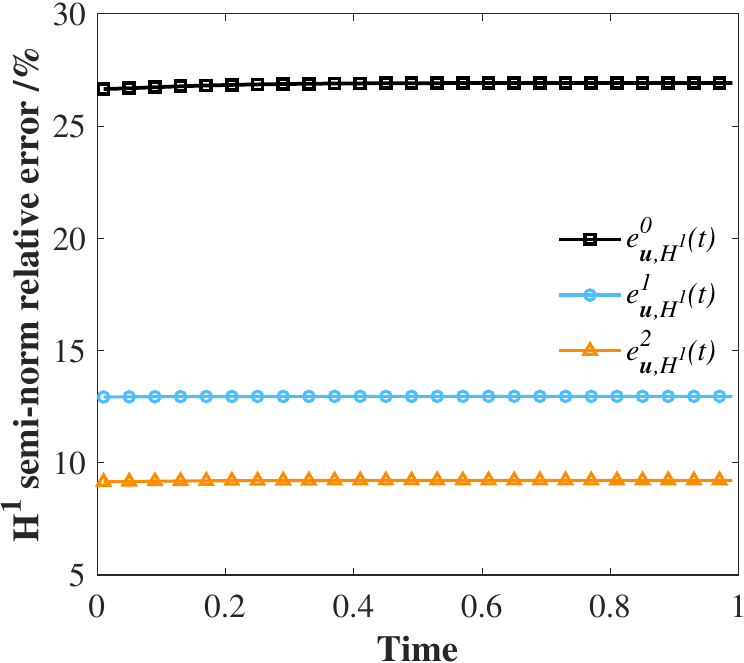}
			(f)
		\end{minipage}
		\caption{The evolutive relative errors of temperature, moisture and displacement fields: (a) $e_{T,L^2}(t)$; (b) $e_{\omega,L^2}(t)$; (c) $e_{\bm{u},L^2}(t)$; (d) $e_{T,H^1}(t)$; (E) $e_{\omega,H^1}(t)$; (F) $e_{\bm{u},H^1}(t)$.}\label{f13}
	\end{figure}
	
	Figs.~\ref{f8}-\ref{f12} demonstrate that only the HOMS method is capable of accurately reproducing the nonlinear hygro-thermo-mechanical coupling behavior of the 3D heterogeneous structure. The homogenized approach captures solely the macroscopic response, whereas the LOMS method offers only limited microscopic corrections. Notably, as evidenced in Fig.~\ref{f13}, the proposed HOMS approach can retain long-time numerical stability, as the blow-up phenomenon does not appear within the simulation time up to $t=1.0\mathrm{s}$.
	
	\subsection{Example 3: nonlinear hygro-thermo-mechanical coupling simulation of 3D heterogeneous plate with a large number of inclusions}
	\label{sec:53}
	The nonlinear dynamic hygro‑thermo‑mechanical behavior of a 3D heterogeneous plate is investigated. The plate consists of a matrix reinforced with particulate inclusions. Its macroscopic domain $\Omega$ and the corresponding microscopic unit cell $Y$ are illustrated in Fig.~\ref{f1:3Dplate}. As indicated in the figure, $\Omega= (x_1,x_2,x_3) = [0,1] \times [0,1] \times[0,1/4] \mathrm{cm}^3$ and the small periodic parameter is $\epsilon=1/32$. This heterogeneous plate represents a large‑scale engineering structure containing a high volume of spherical inclusions.
	\begin{figure}[!htb]
		\centering
		\begin{minipage}[c]{0.35\textwidth}
			\centering
			\includegraphics[width=47mm]{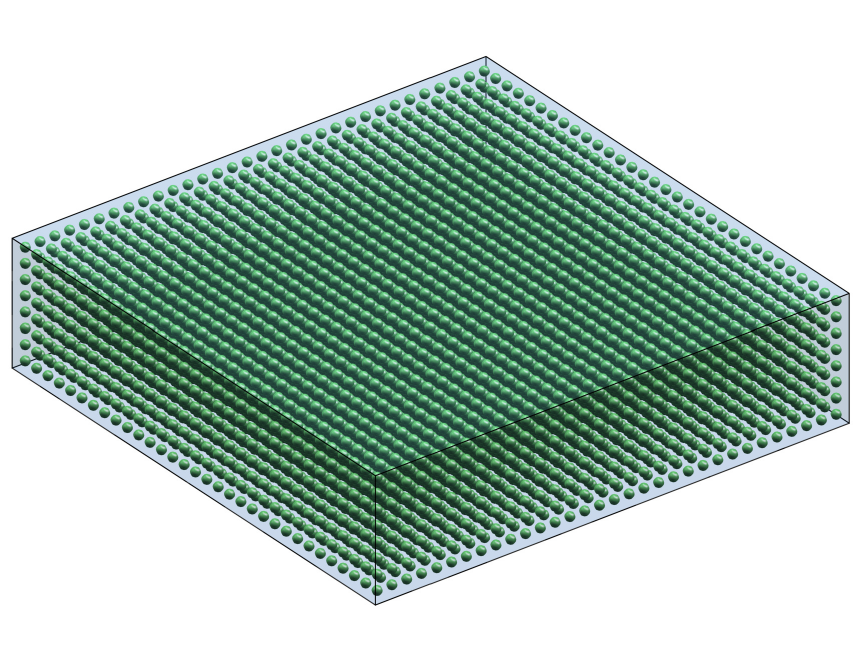}\\
			(a)
		\end{minipage}
		\begin{minipage}[c]{0.3\textwidth}
			\centering
			\includegraphics[width=37mm]{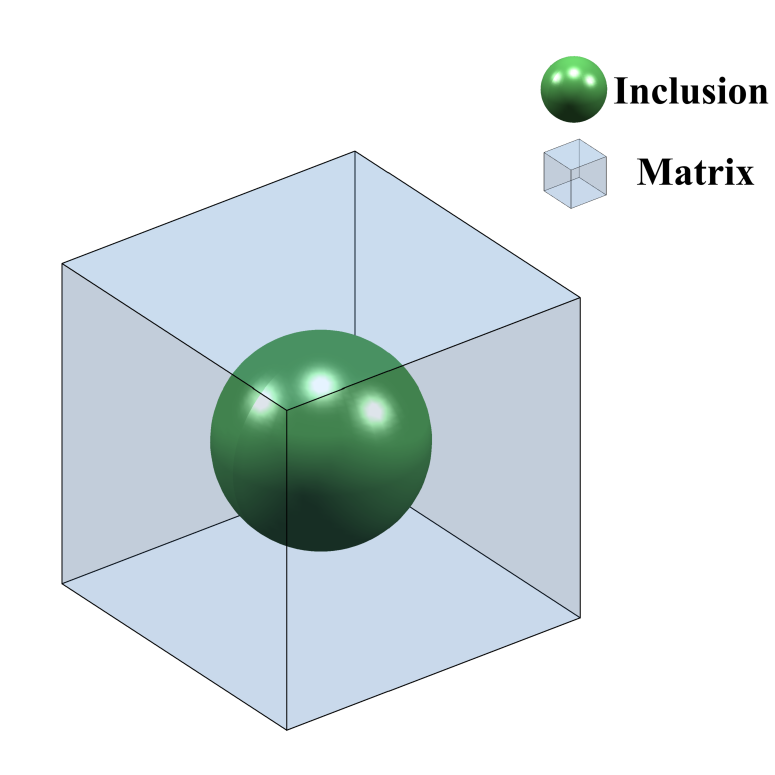}\\
			(b)
		\end{minipage}
		\begin{minipage}[c]{0.3\textwidth}
			\centering
			\includegraphics[width=47mm]{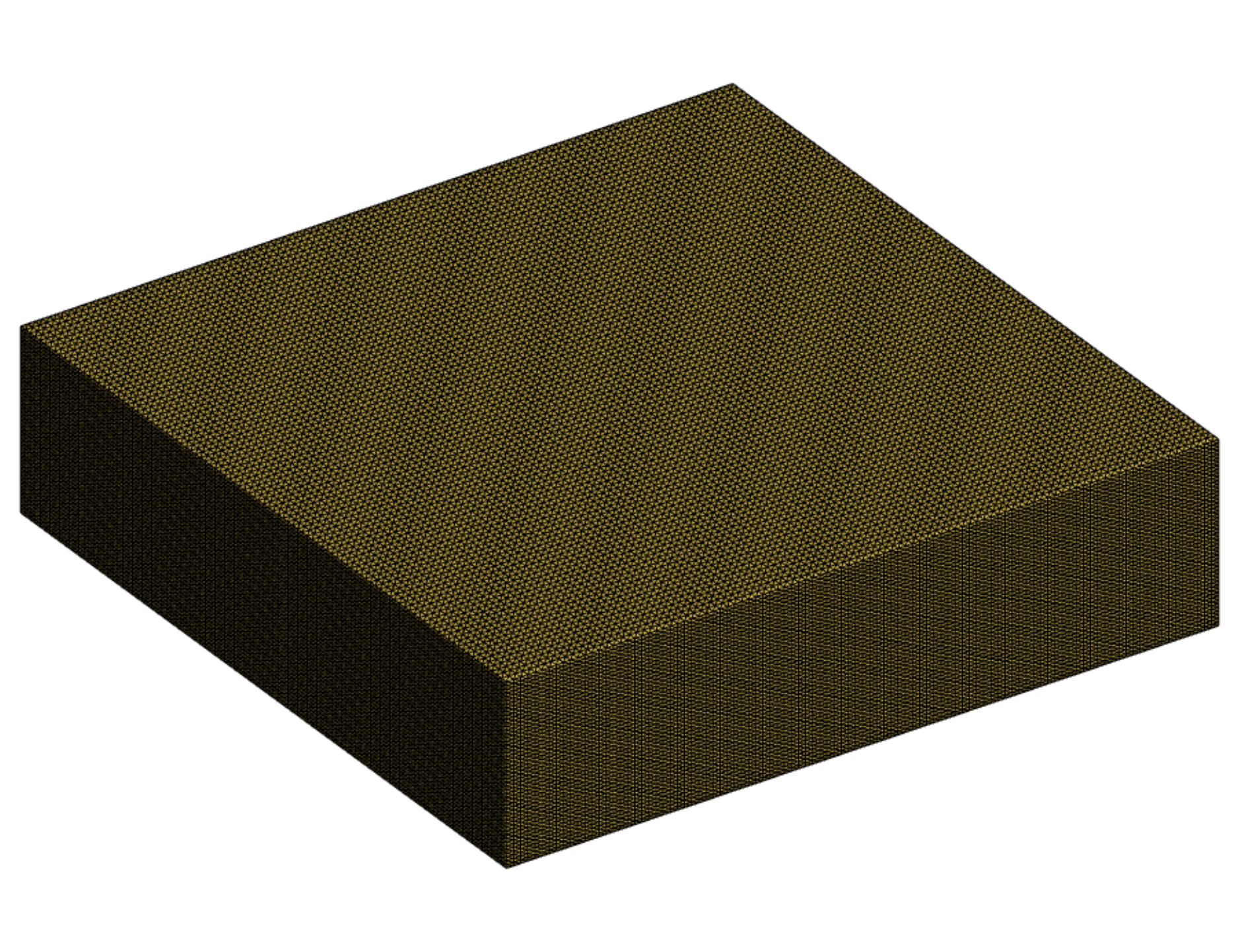}\\
			(c)
		\end{minipage}
		\caption{(a) The 3D composite plate structure $\Omega$; (b) PUC $Y$; (c) homogenized structure $\Omega$.}\label{f1:3Dplate}
	\end{figure}
	
	In this example, all material property parameters are the same as those listed in Table~\ref{t3} of Section~\ref{sec:52}, except that for the Matrix, Young's modulus $E^{\epsilon}$ is given by $1250.0-0.030435T-1.0\times10^{-12}T^2 \ \mathrm{GPa}$ and thermal conductivity $k^{\epsilon}_{ij}$ by $5000.0+0.5T+0.0025 \ \mathrm{W / (m \cdot K})$. In addition, the heat source is changed to $h=10000.0 \ \mathrm{J/(cm^{3}\cdot s)}$. All other source items, boundary conditions, and initial conditions remain the same as in Section~\ref{sec:52}.
	
	For the discretization, tetrahedral meshes are generated for the multi-scale nonlinear problem \eqref{eq:2.1}, the auxiliary cell problems, and the associated homogenized problem \eqref{eq:2.27}. The detailed mesh information, including the numbers of FEM elements and nodes, is documented in Table~\ref{t5}.
	\begin{table}[!t]
		\caption{Comparison of computational cost ($\Delta t=0.01\,$s, $t\in[0,1.0]\,$s).\label{t5}}
		\centering
		\begin{tabular}{cccc}
			\toprule
			& Cell equations & Homogenized equations & Multi-scale equations \\
			\midrule
			FEM nodes & 65420 & 520251 & 535920640 (estimated) \\
			FEM elements & 394701 & 3000000 & 3233390592 (estimated) \\
			\bottomrule
		\end{tabular}
	\end{table}
	
	For this larger-scale heterogeneous plate, we still prescribe 10 equidistant interpolation points each for macroscopic temperature and moisture within a single unit cell. The dynamic nonlinear hygro-thermo-mechanical responses are simulated over $t \in [0, 1.0] \mathrm{s}$ with $\Delta t = 0.01 \mathrm{s}$, during which the homogenized equations \eqref{eq:2.27} and the multi-scale equations \eqref{eq:2.1} are solved on-line. The resulting temperature, moisture and displacement fields at $t = 1.0 \mathrm{s}$ are displayed in Figs.~\ref{f15}-\ref{f19}.
	\begin{figure}[!htb]
		\centering
		\begin{minipage}[c]{0.32\textwidth}
			\centering
			\includegraphics[width=\linewidth]{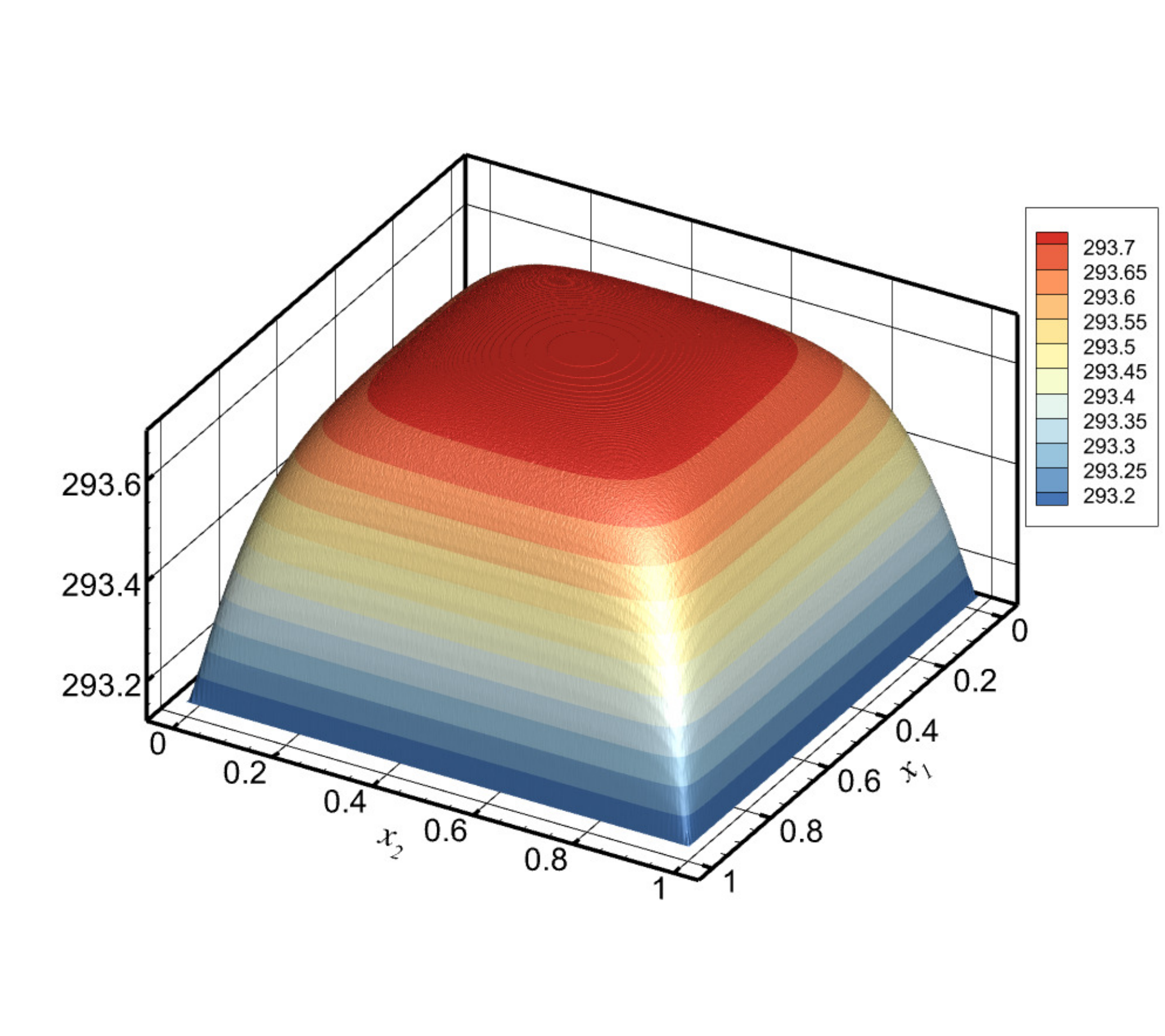}
			(a)
		\end{minipage}
		\hfill
		\begin{minipage}[c]{0.32\textwidth}
			\centering
			\includegraphics[width=\linewidth]{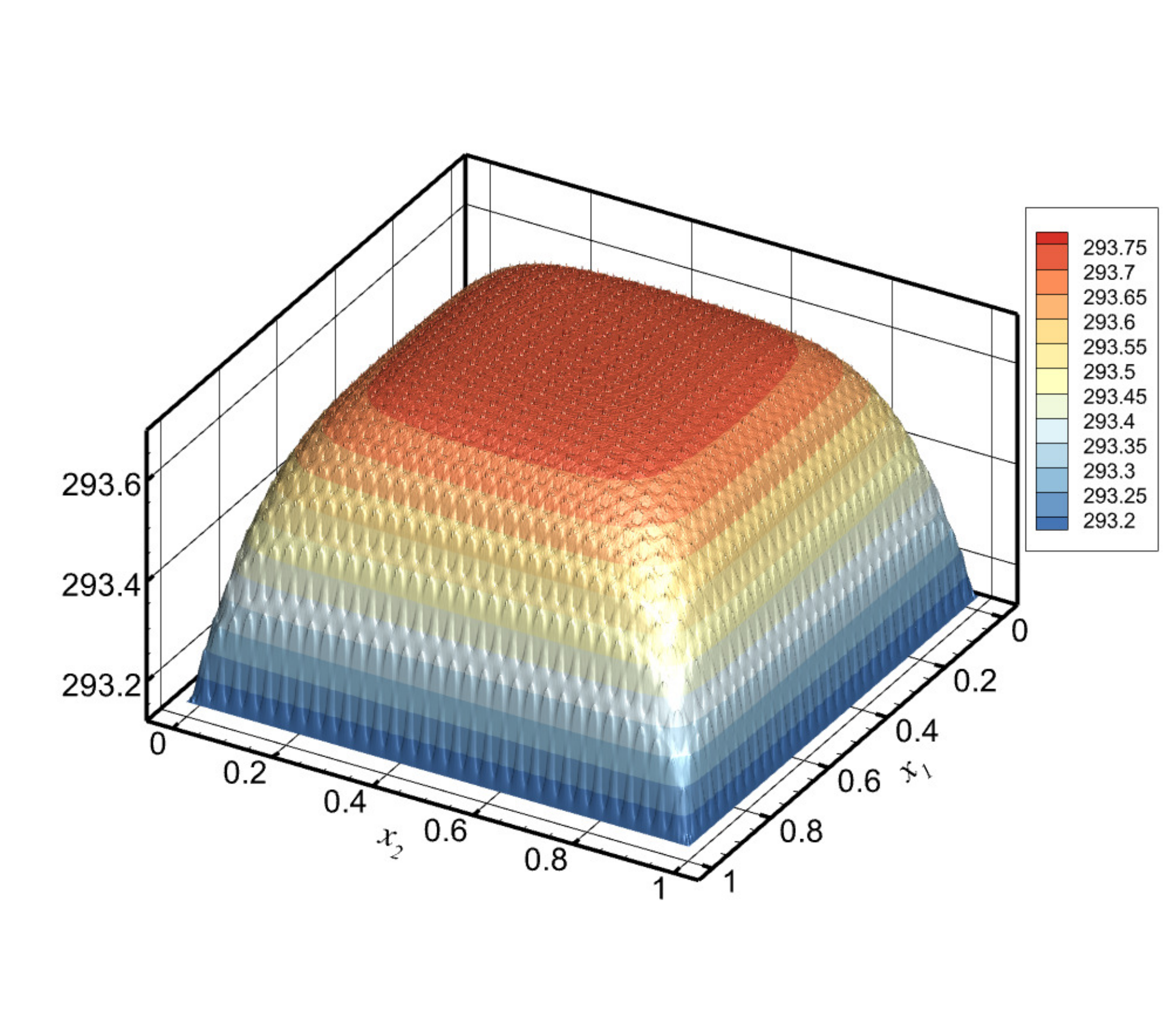}
			(b)
		\end{minipage}
		\hfill
		\begin{minipage}[c]{0.32\textwidth}
			\centering
			\includegraphics[width=\linewidth]{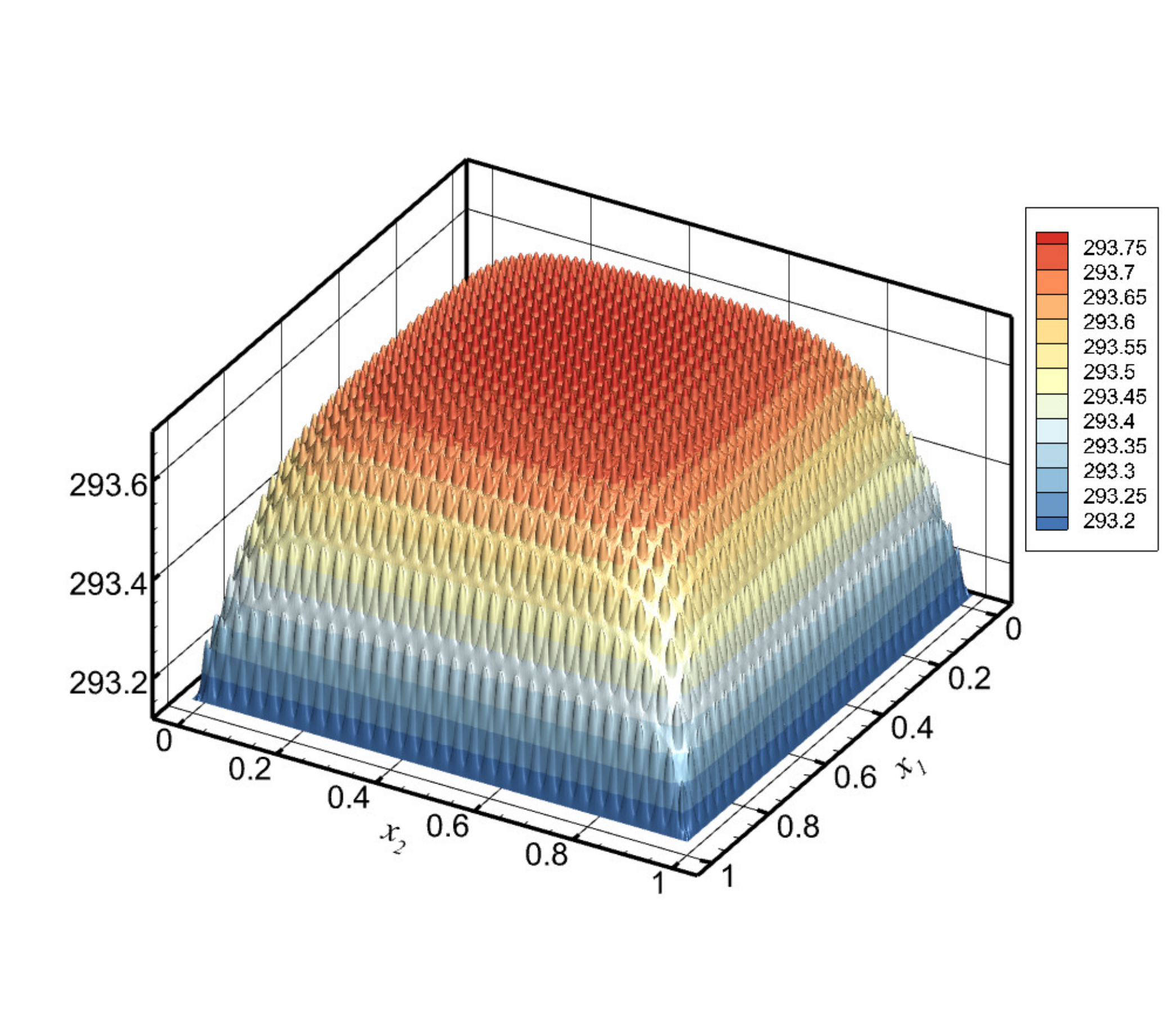}
			(c)
		\end{minipage}
		\caption{The temperature field in cross section $x_3=0.015625 \mathrm{cm}$ at time $t=1.0\mathrm{s}$: (a) $T^{(0)}$; (b) $T^{(1,\epsilon)}$; (c) $T^{(2,\epsilon)}$.}\label{f15}
	\end{figure}
	\begin{figure}[!htb]
		\centering
		\begin{minipage}[c]{0.32\textwidth}
			\centering
			\includegraphics[width=\linewidth]{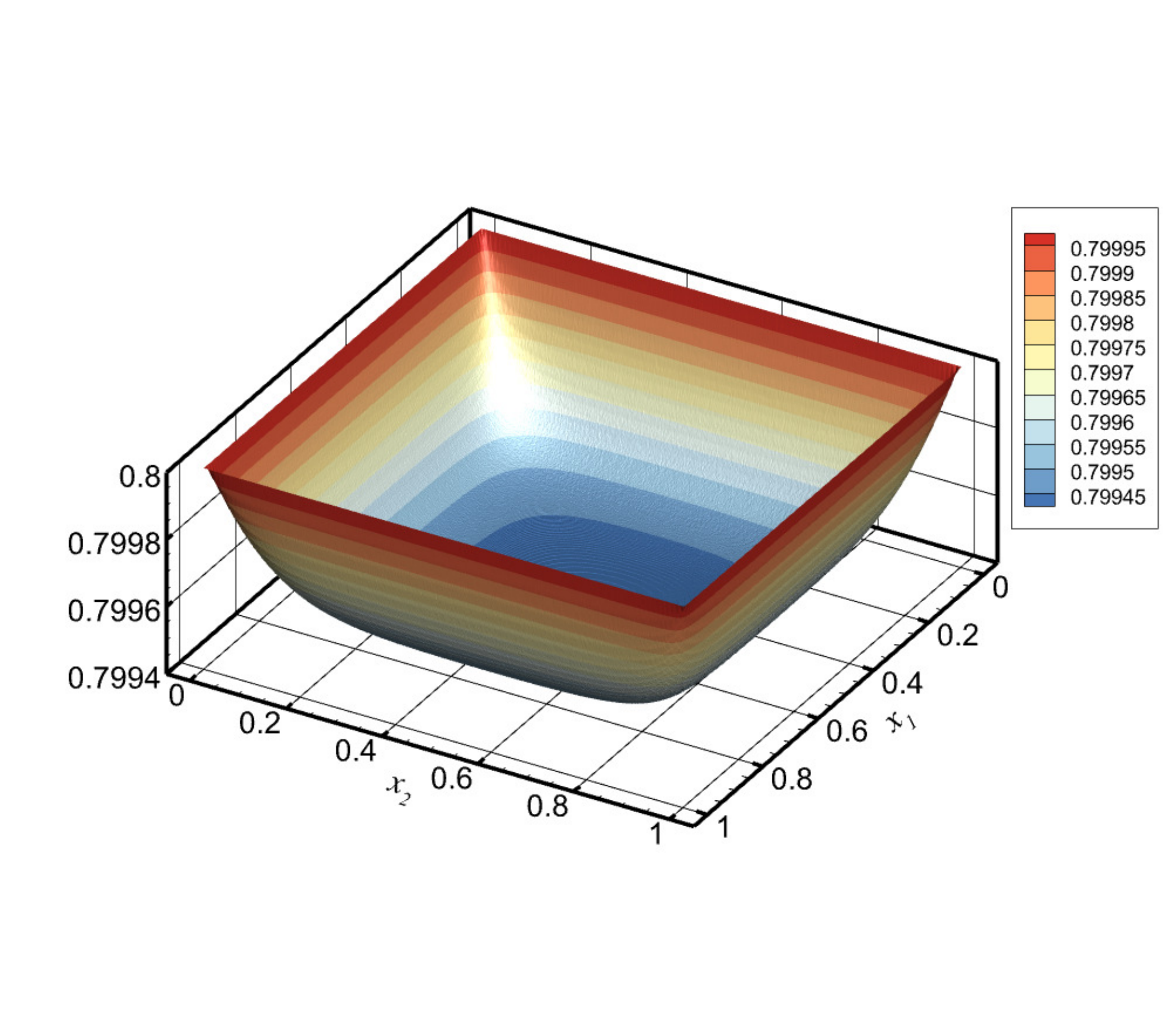}
			(a)
		\end{minipage}
		\hfill
		\begin{minipage}[c]{0.32\textwidth}
			\centering
			\includegraphics[width=\linewidth]{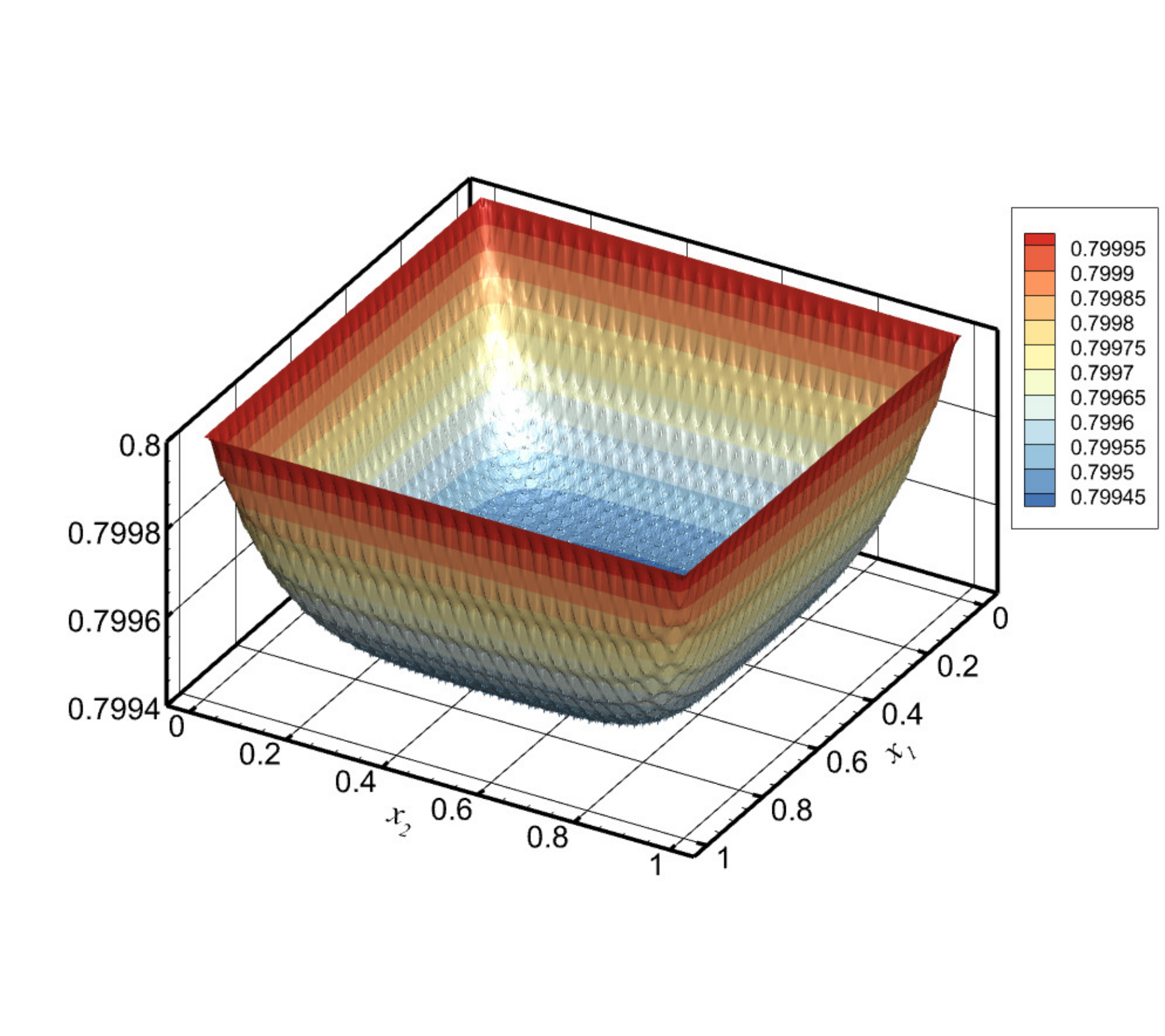}
			(b)
		\end{minipage}
		\hfill
		\begin{minipage}[c]{0.32\textwidth}
			\centering
			\includegraphics[width=\linewidth]{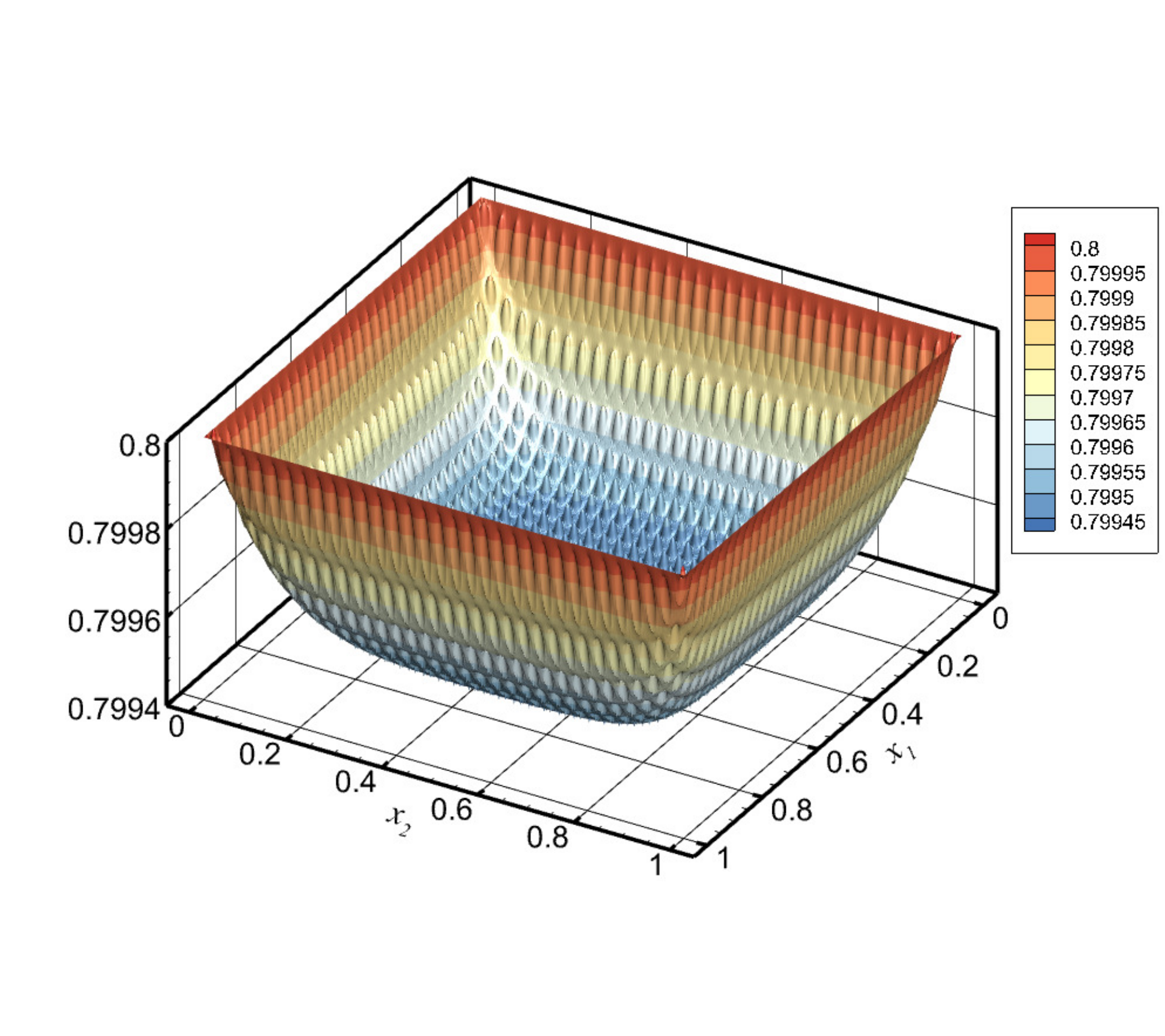}
			(c)
		\end{minipage}
		\caption{The moisture field in $x_3=0.015625 \mathrm{cm}$ at time $t=1.0\mathrm{s}$: (a) $\omega^{(0)}$; (b) $\omega^{(1,\epsilon)}$; (c) $\omega^{(2,\epsilon)}$.}\label{f16}
	\end{figure}
	\begin{figure}[!htb]
		\centering
		\begin{minipage}[c]{0.32\textwidth}
			\centering
			\includegraphics[width=\linewidth]{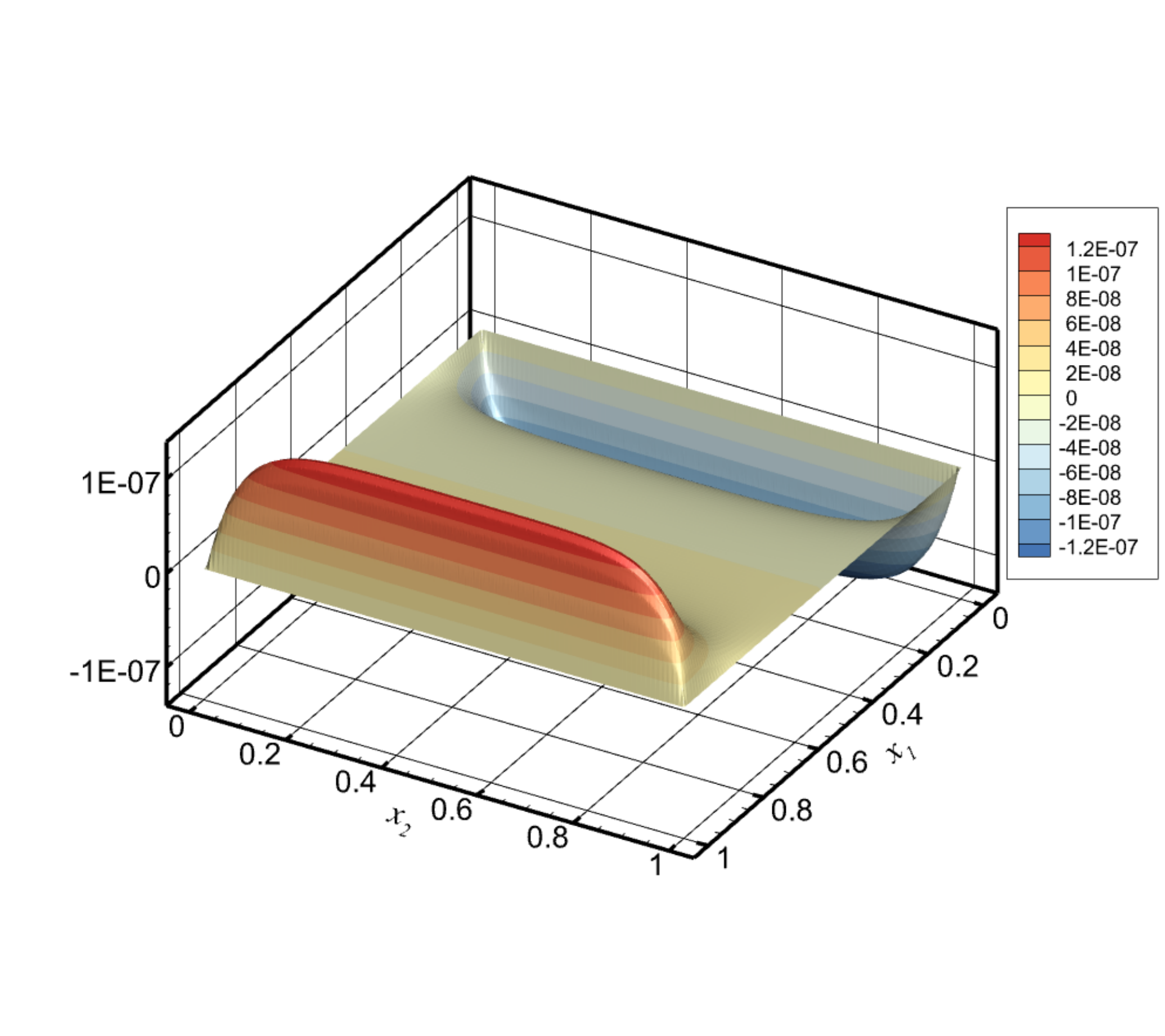}
			(a)
		\end{minipage}
		\hfill
		\begin{minipage}[c]{0.32\textwidth}
			\centering
			\includegraphics[width=\linewidth]{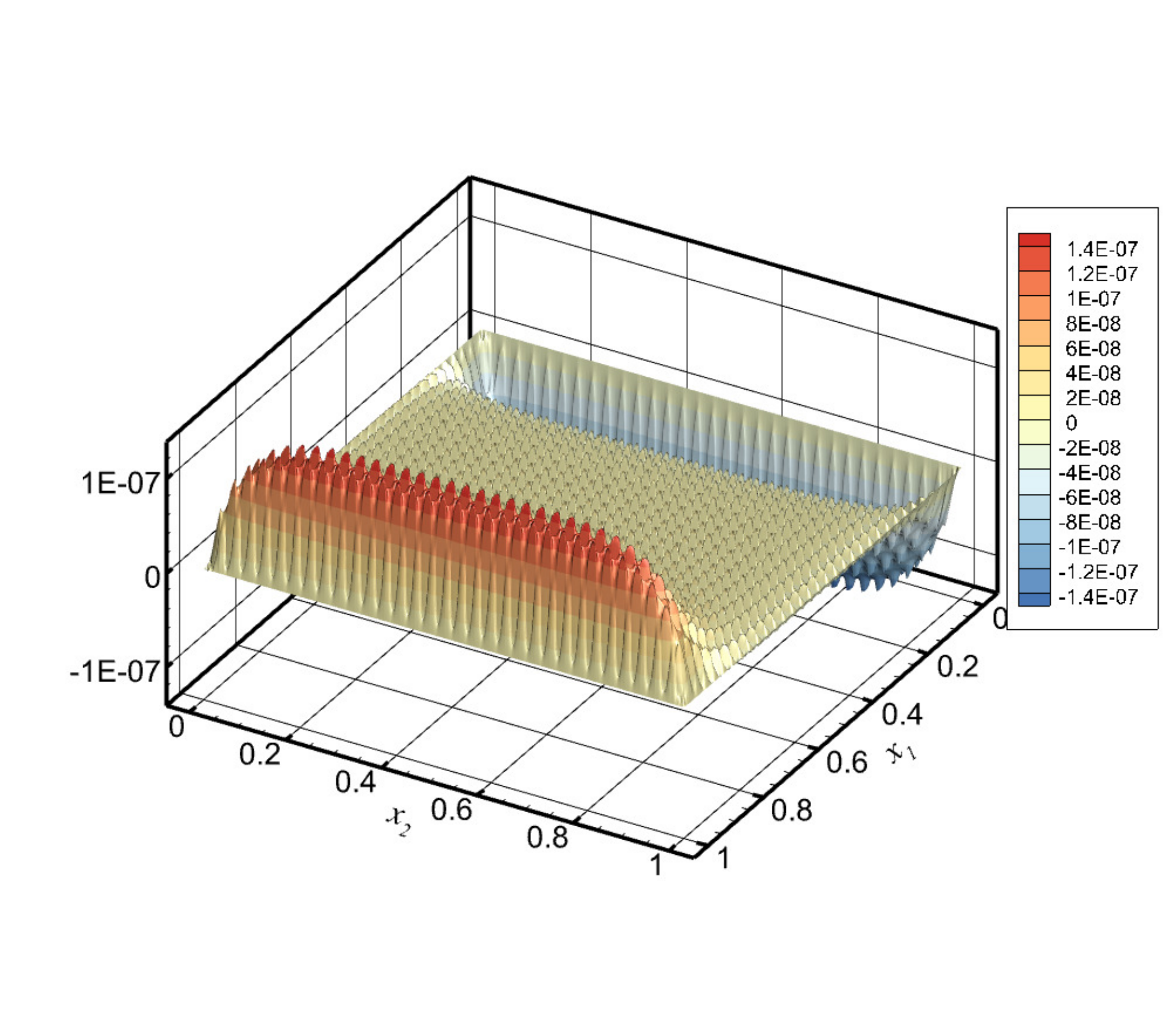}
			(b)
		\end{minipage}
		\hfill
		\begin{minipage}[c]{0.32\textwidth}
			\centering
			\includegraphics[width=\linewidth]{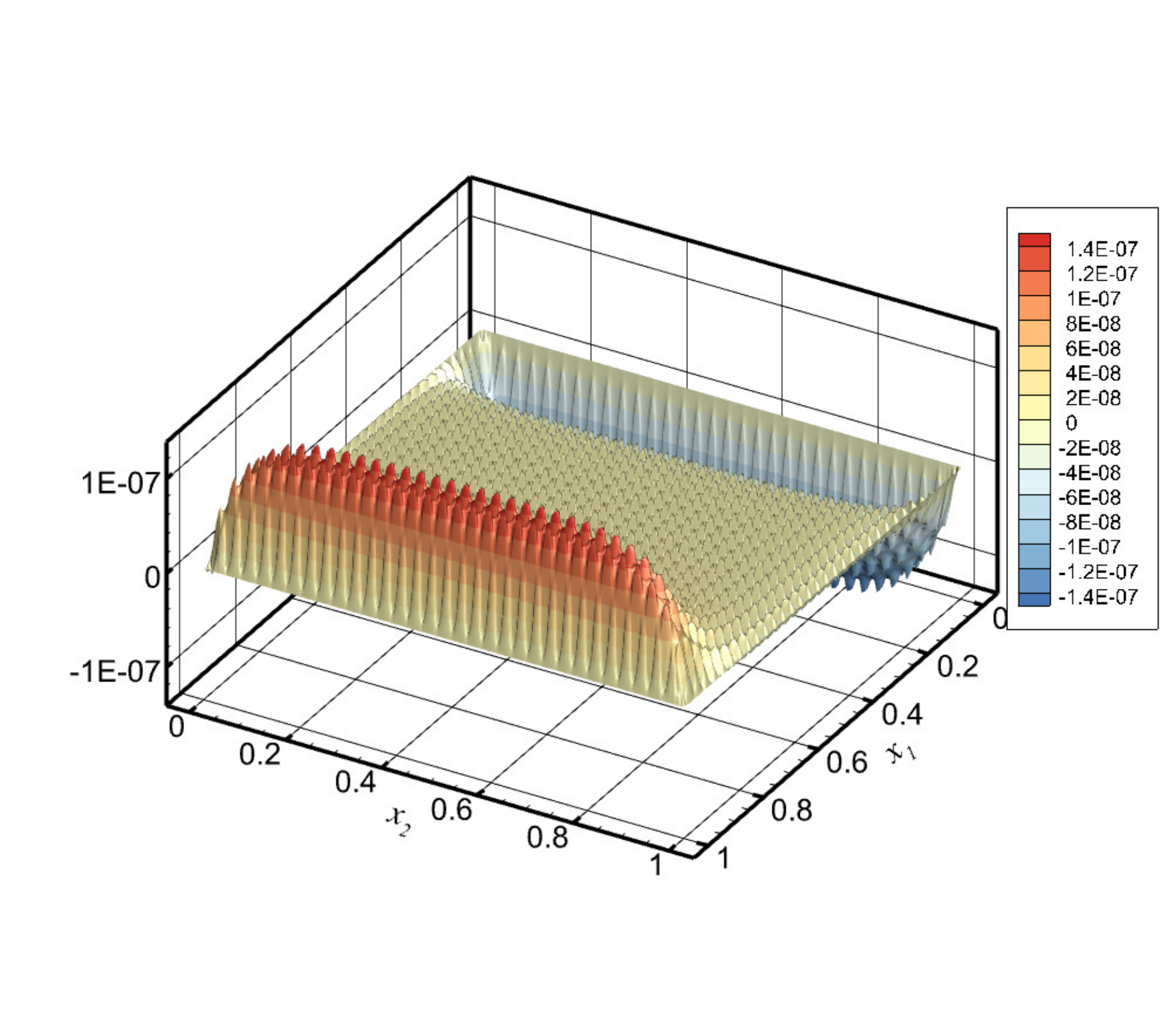}
			(c)
		\end{minipage}
		\caption{The first component of the displacement field in cross section $x_3=0.015625 \mathrm{cm}$ at time $t=1.0\mathrm{s}$: (a) $u_1^{(0)}$; (b) $u_1^{(1,\epsilon)}$; (c) $u_1^{(2,\epsilon)}$.}\label{f17}
	\end{figure}
	\begin{figure}[!htb]
		\centering
		\begin{minipage}[c]{0.32\textwidth}
			\centering
			\includegraphics[width=\linewidth]{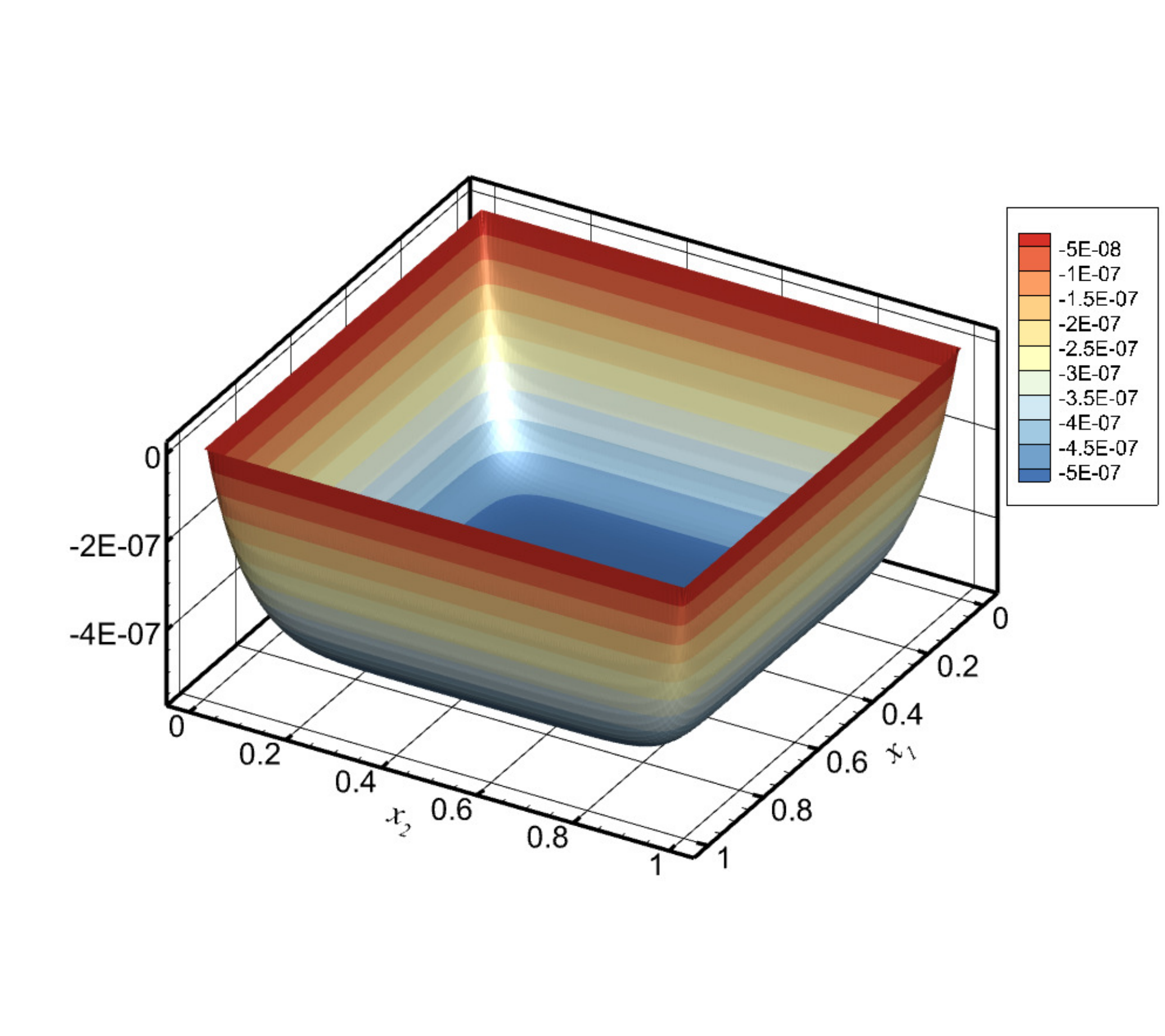}
			(a)
		\end{minipage}
		\hfill
		\begin{minipage}[c]{0.32\textwidth}
			\centering
			\includegraphics[width=\linewidth]{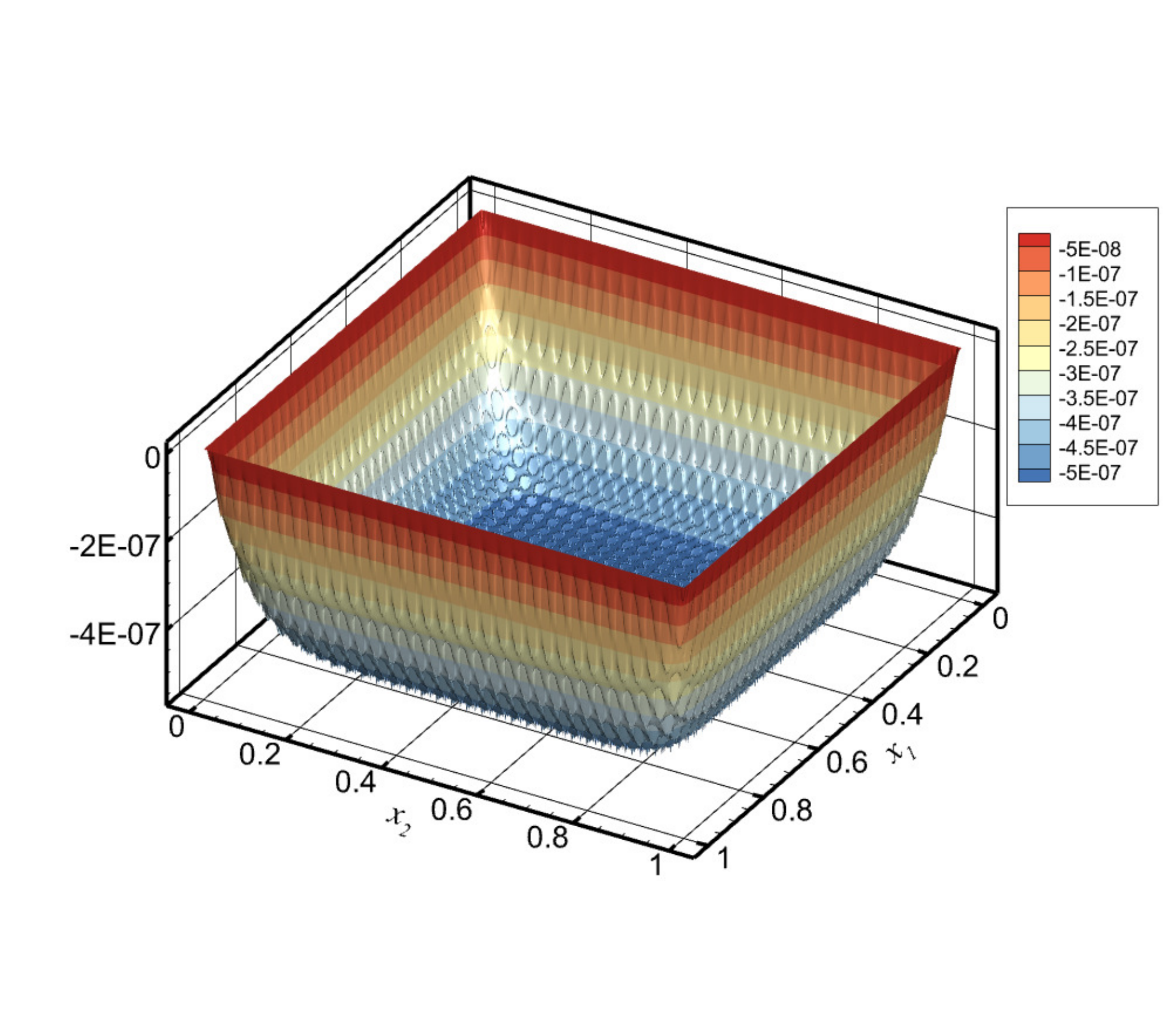}
			(b)
		\end{minipage}
		\hfill
		\begin{minipage}[c]{0.32\textwidth}
			\centering
			\includegraphics[width=\linewidth]{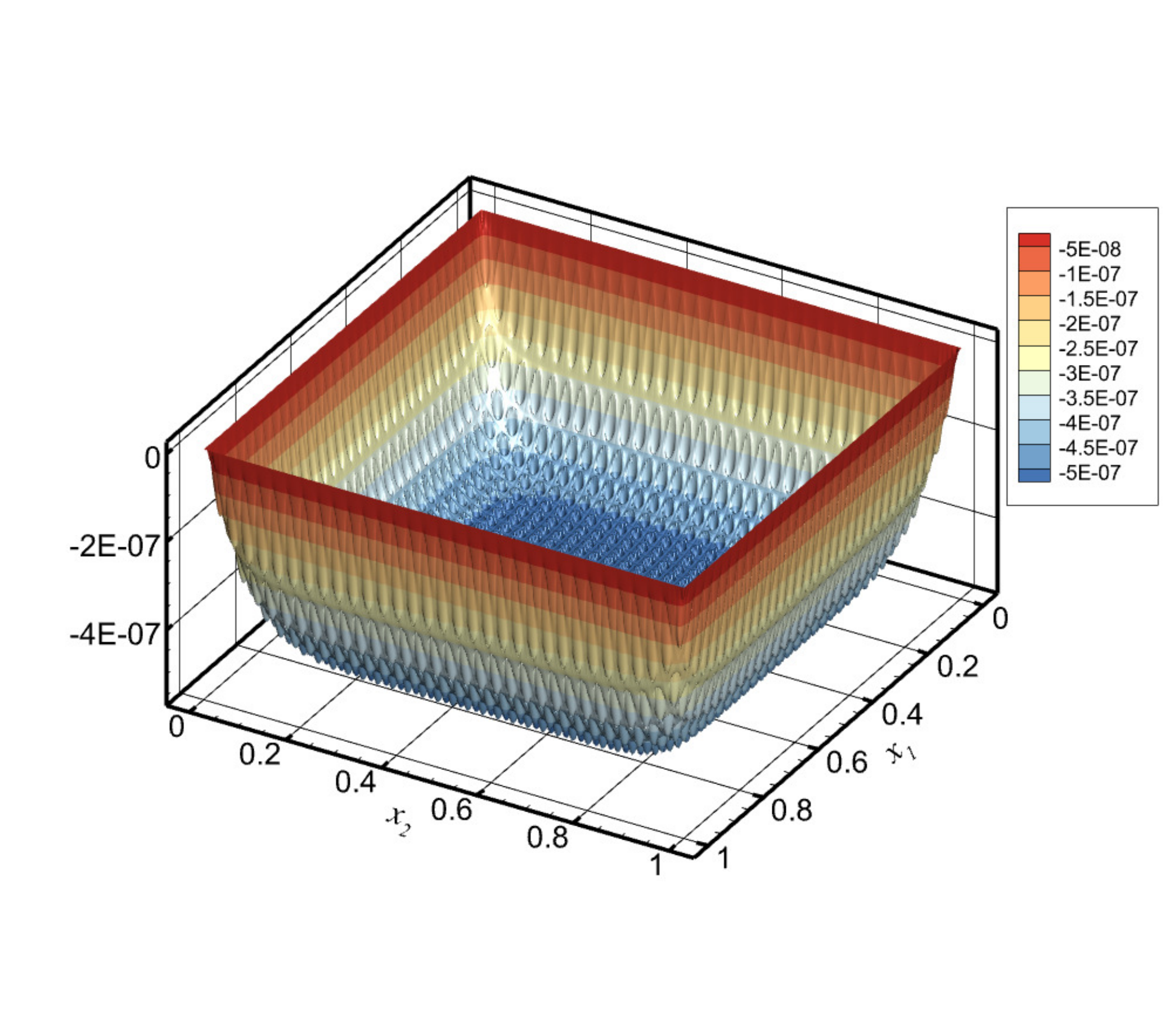}
			(c)
		\end{minipage}
		\caption{The third component of the displacement field in cross section $x_3=0.015625 \mathrm{cm}$ at time $t=1.0\mathrm{s}$: (a) $u_3^{(0)}$; (b) $u_3^{(1,\epsilon)}$; (c) $u_3^{(2,\epsilon)}$.}\label{f19}
	\end{figure}
	
	Due to the enormous mesh size of $O(10^8)$ nodes and $O(10^9)$ elements, as shown in Table~\ref{t5}, direct FEM simulation of this large-scale heterogeneous plate is practically infeasible. As shown in Figs.~\ref{f15}-\ref{f19}, the HOMS solutions effectively capture the steep microscopic fluctuations of the 3D heterogeneous plate with a large number of inclusions, whereas the homogenized solutions fail to capture any oscillations and the LOMS solutions capture only inadequate microscopic responses. Consequently, the proposed HOMS method offers a practical and resource-efficient way to simulate large-scale heterogeneous structures, capturing microscopic fluctuations while avoiding the prohibitive cost of direct FEM, which is of great significance in engineering computations.
	
\section{Conclusions}
\label{sec:6}
	This work presents a HOMS computational approach with low computational cost for the accurate simulation of nonlinear dynamic hygro-thermo-mechanical coupling problems in heterogeneous structures that possess microscopic heterogeneities. The principal contributions are threefold. First, new macro-micro correlation formulations that incorporate higher-order correction terms are established for such problems in heterogeneous structures with periodic microscopic configurations. Second, explicit error estimates are rigorously derived for the HOMS solutions in both local and global senses. Third, an efficient two- stage (off-line and on-line) multi-scale numerical algorithm is designed to overcome the limitation of prohibitive computational cost, and its convergence analysis is also provided. Numerical experiments demonstrate that the proposed HOMS method delivers excellent numerical accuracy and computational efficiency. Notably, for large-scale heterogeneous structures that are difficult to handle with classical finite element methods, the HOMS method offers an effective computational framework that enables high-accuracy, high-efficiency simulations of nonlinear dynamic hygro-thermo-mechanical problems.
	
	Future work will focus on three aspects. First, the computational efficiency of the off-line stage will be further improved through parallel computing and model reduction techniques. Second, the current two-scale framework will be extended to a three-scale (macro-meso-micro) formulation for hierarchical composites whose spatial configurations inherently involve more than two distinct scales, such as woven composites and porous materials. Third, the current deterministic framework can be generalized to account for random heterogeneous composites.

\appendix
\section{Supplemental expressions for Section~\ref{sec:31}}
\label{app:A}
	\begin{footnotesize}
		\begin{equation}
			\begin{aligned}
				& \Phi_0(\mathbf{x}, \mathbf{y}, t) = \bigl( \hat{S} - \rho^{(0)} c^{(0)} \bigr) \frac{\partial T^{(0)}}{\partial t} + \frac{\partial}{\partial y_i} \Bigl[ k_{ij}^{(0)} \frac{\partial}{\partial x_j} \Bigl( \mathcal{H}_{\alpha_1} \frac{\partial T^{(0)}}{\partial x_{\alpha_1}} \Bigr) \Bigr] \\
				& + \frac{\partial}{\partial y_i} \Bigl( k_{ij}^{(0)} \frac{\partial \mathcal{E}_{\alpha_1 \alpha_2}}{\partial y_j} \Bigr) \frac{\partial T^{(0)}}{\partial x_{\alpha_1}} \frac{\partial T^{(0)}}{\partial x_{\alpha_2}} - \frac{\partial}{\partial x_i} \Bigl( \hat{k}_{ij} \frac{\partial T^{(0)}}{\partial x_j} \Bigr) - \hat{Q}_{hyd} + Q_{hyd}^{(0)}.
			\end{aligned}
		\end{equation}
	\end{footnotesize}
	\begin{footnotesize}
		\begin{equation}
			\begin{aligned}
				& \Phi_1(\mathbf{x}, \mathbf{y}, t) = - \bigl( c^{(0)} \mathbf{D}^{(0,1)} \rho^{(0)} + \rho^{(0)} \mathbf{D}^{(0,1)} c^{(0)} \bigr) \mathcal{H}_{\alpha_1} \frac{\partial T^{(0)}}{\partial x_{\alpha_1}} \frac{\partial T^{(0)}}{\partial t} - \rho^{(0)} c^{(0)} \mathcal{H}_{\alpha_1} \frac{\partial}{\partial t} \Bigl( \frac{\partial T^{(0)}}{\partial x_{\alpha_1}} \Bigr) + \mathcal{H}_{\alpha_1} \frac{\partial T^{(0)}}{\partial x_{\alpha_1}} \mathbf{D}^{(0,1)} Q_{hyd}^{(0)} \\
				&+ \frac{\partial}{\partial x_i} \Bigl( k_{ij}^{(0)} \frac{\partial \mathcal{E}_{\alpha_1 \alpha_2}}{\partial y_j} \frac{\partial T^{(0)}}{\partial x_{\alpha_1}} \frac{\partial T^{(0)}}{\partial x_{\alpha_2}} \Bigr) + \frac{\partial}{\partial x_i} \Bigl[ k_{ij}^{(0)} \frac{\partial}{\partial x_j} \Bigl( \mathcal{H}_{\alpha_1} \frac{\partial T^{(0)}}{\partial x_{\alpha_1}} \Bigr) \Bigr] + \frac{\partial}{\partial y_i} \Bigl[ \mathcal{H}_{\alpha_2} \mathbf{D}^{(0,1)} k_{ij}^{(0)} \frac{\partial}{\partial x_j} \Bigl( \mathcal{H}_{\alpha_1} \frac{\partial T^{(0)}}{\partial x_{\alpha_1}} \Bigr) \Bigr] \frac{\partial T^{(0)}}{\partial x_{\alpha_2}}.
			\end{aligned}
		\end{equation}
	\end{footnotesize}
	\begin{footnotesize}
		\begin{equation}
			\begin{aligned}
				& \Psi_0(\mathbf{x}, \mathbf{y}, t) = \frac{\partial}{\partial y_i} \Bigl[ g_{ij}^{(0)} \frac{\partial}{\partial x_j} \Bigl( \mathcal{J}_{\alpha_1} \frac{\partial \omega^{(0)}}{\partial x_{\alpha_1}} \Bigr) \Bigr] + \frac{\partial}{\partial y_i} \Bigl( g_{ij}^{(0)} \frac{\partial \mathcal{F}_{\alpha_1 \alpha_2}}{\partial y_j} \Bigr) \frac{\partial \omega^{(0)}}{\partial x_{\alpha_1}} \frac{\partial \omega^{(0)}}{\partial x_{\alpha_2}} - \frac{\partial}{\partial x_i} \Bigl( \hat{g}_{ij} \frac{\partial \omega^{(0)}}{\partial x_j} \Bigr) + \hat{S}_{hyd} - S_{hyd}^{(0)}.
			\end{aligned}
		\end{equation}
	\end{footnotesize}
	\begin{footnotesize}
		\begin{equation}
			\begin{aligned}
				& \Psi_1(\mathbf{x}, \mathbf{y}, t) = - \frac{\partial}{\partial t} \Bigl( \mathcal{J}_{\alpha_1} \frac{\partial \omega^{(0)}}{\partial x_{\alpha_1}} \Bigr) + \frac{\partial}{\partial x_i} \Bigl( g_{ij}^{(0)} \frac{\partial \mathcal{F}_{\alpha_1 \alpha_2}}{\partial y_j} \frac{\partial \omega^{(0)}}{\partial x_{\alpha_1}} \frac{\partial \omega^{(0)}}{\partial x_{\alpha_2}} \Bigr) - \mathcal{H}_{\alpha_1} \frac{\partial T^{(0)}}{\partial x_{\alpha_1}} \mathbf{D}^{(0,1)} S_{hyd}^{(0)} \\
				& + \frac{\partial}{\partial x_i} \Bigl[ g_{ij}^{(0)} \frac{\partial}{\partial x_j} \Bigl( \mathcal{J}_{\alpha_1} \frac{\partial \omega^{(0)}}{\partial x_{\alpha_1}} \Bigr) \Bigr] + \frac{\partial}{\partial y_i} \Bigl[ \mathcal{J}_{\alpha_2} \mathbf{D}^{(0,1)} g_{ij}^{(0)} \frac{\partial}{\partial x_j} \Bigl( \mathcal{J}_{\alpha_1} \frac{\partial \omega^{(0)}}{\partial x_{\alpha_1}} \Bigr) \Bigr] \frac{\partial \omega^{(0)}}{\partial x_{\alpha_2}}.
			\end{aligned}
		\end{equation}
	\end{footnotesize}
	\begin{footnotesize}
		\begin{equation}
			\begin{aligned}
				& \Theta_{0i}(\mathbf{x}, \mathbf{y}, t) = \frac{\partial}{\partial y_j} \Bigl[ C_{ijkl}^{(0)} \frac{\partial}{\partial x_l} \Bigl( \mathcal{X}_{kh}^{\alpha_1} \frac{\partial u_h^{(0)}}{\partial x_{\alpha_1}} \Bigr) \Bigr] + \frac{\partial}{\partial y_j} \Bigl( C_{ijkl}^{(0)} \frac{\partial \mathcal{P}_{km}^{\alpha_1 \alpha_2}}{\partial y_l} \Bigr) \frac{\partial T^{(0)}}{\partial x_{\alpha_1}} \frac{\partial u_m^{(0)}}{\partial x_{\alpha_2}} - \frac{\partial}{\partial x_j} \Bigl( \hat{C}_{ijkl} \frac{\partial u_k^{(0)}}{\partial x_l} \Bigr) \\
				& - \frac{\partial}{\partial y_j} \Bigl( C_{ijkl}^{(0)} \frac{\partial}{\partial x_l} \bigl( \mathcal{M}_k (T^{(0)} - \tilde{T}) \bigr) \Bigr) - \frac{\partial}{\partial y_j} \Bigl( C_{ijkl}^{(0)} \frac{\partial \mathcal{A}_k^{\alpha_1}}{\partial y_l} \Bigr) \frac{\partial T^{(0)}}{\partial x_{\alpha_1}} (T^{(0)} - \tilde{T})\\
				& - \frac{\partial}{\partial y_j} \bigl( \alpha_{ij}^{(0)} \mathcal{H}_{\alpha_1} \bigr) \frac{\partial T^{(0)}}{\partial x_{\alpha_1}} + \frac{\partial}{\partial x_j} \bigl( \hat{\alpha}_{ij} (T^{(0)} - \tilde{T}) \bigr) - \frac{\partial}{\partial y_j} \Bigl( C_{ijkl}^{(0)} \frac{\partial}{\partial x_l} \bigl( \mathcal{N}_k (\omega^{(0)} - \tilde{\omega}) \bigr) \Bigr) \\
				& - \frac{\partial}{\partial y_j} \Bigl( C_{ijkl}^{(0)} \frac{\partial \mathcal{B}_k^{\alpha_1}}{\partial y_l} \Bigr) \frac{\partial T^{(0)}}{\partial x_{\alpha_1}} (\omega^{(0)} - \tilde{\omega}) - \frac{\partial}{\partial y_j} \bigl( \beta_{ij}^{(0)} \mathcal{J}_{\alpha_1} \bigr) \frac{\partial \omega^{(0)}}{\partial x_{\alpha_1}} + \frac{\partial}{\partial x_j} \bigl( \hat{\beta}_{ij} (\omega^{(0)} - \tilde{\omega}) \bigr).
			\end{aligned}
		\end{equation}
	\end{footnotesize}
	\begin{footnotesize}
		\begin{equation}
			\begin{aligned}
				& \Theta_{1i}(\mathbf{x}, \mathbf{y}, t) = \frac{\partial}{\partial x_j} \Bigl( C_{ijkl}^{(0)} \frac{\partial \mathcal{P}_{km}^{\alpha_1 \alpha_2}}{\partial y_l} \frac{\partial T^{(0)}}{\partial x_{\alpha_1}} \frac{\partial u_m^{(0)}}{\partial x_{\alpha_2}} \Bigr) + \frac{\partial}{\partial y_j} \Bigl[ \mathcal{H}_{\alpha_2} \frac{\partial T^{(0)}}{\partial x_{\alpha_2}} \mathbf{D}^{(0,1)} C_{ijkl}^{(0)} \frac{\partial}{\partial x_l} \Bigl( \mathcal{X}_{kh}^{\alpha_1} \frac{\partial u_h^{(0)}}{\partial x_{\alpha_1}} \Bigr) \Bigr] \\
				& + \frac{\partial}{\partial x_j} \Bigl[ C_{ijkl}^{(0)} \frac{\partial}{\partial x_l} \Bigl( \mathcal{X}_{kh}^{\alpha_1} \frac{\partial u_h^{(0)}}{\partial x_{\alpha_1}} \Bigr) \Bigr] - \frac{\partial}{\partial x_j} \Bigl( C_{ijkl}^{(0)} \frac{\partial}{\partial x_l} \bigl( \mathcal{M}_k (T^{(0)} - \tilde{T}) \bigr) \Bigr) - \frac{\partial}{\partial x_j} \Bigl( C_{ijkl}^{(0)} \frac{\partial \mathcal{A}_k^{\alpha_1}}{\partial y_l} \frac{\partial T^{(0)}}{\partial x_{\alpha_1}} (T^{(0)} - \tilde{T}) \Bigr) \\
				& - \frac{\partial}{\partial y_j} \Bigl( \mathcal{H}_{\alpha_2} \frac{\partial T^{(0)}}{\partial x_{\alpha_2}} \mathbf{D}^{(0,1)} C_{ijkl}^{(0)} \frac{\partial}{\partial x_l} \bigl( \mathcal{M}_k (T^{(0)} - \tilde{T}) \bigr) \Bigr) - \frac{\partial}{\partial x_j} \Bigl( \alpha_{ij}^{(0)} \mathcal{H}_{\alpha_1} \frac{\partial T^{(0)}}{\partial x_{\alpha_1}} \Bigr) \\
				& - \frac{\partial}{\partial y_j} \Bigl( \mathcal{H}_{\alpha_2} \frac{\partial T^{(0)}}{\partial x_{\alpha_2}} \mathbf{D}^{(0,1)} \alpha_{ij}^{(0)} \mathcal{H}_{\alpha_1} \frac{\partial T^{(0)}}{\partial x_{\alpha_1}} \Bigr) - \frac{\partial}{\partial x_j} \Bigl( C_{ijkl}^{(0)} \frac{\partial}{\partial x_l} \bigl( \mathcal{N}_k (\omega^{(0)} - \tilde{\omega}) \bigr) \Bigr) \\
				& - \frac{\partial}{\partial x_j} \Bigl( C_{ijkl}^{(0)} \frac{\partial \mathcal{B}_k^{\alpha_1}}{\partial y_l} \frac{\partial T^{(0)}}{\partial x_{\alpha_1}} (\omega^{(0)} - \tilde{\omega}) \Bigr) - \frac{\partial}{\partial y_j} \Bigl( \mathcal{H}_{\alpha_2} \frac{\partial T^{(0)}}{\partial x_{\alpha_2}} \mathbf{D}^{(0,1)} C_{ijkl}^{(0)} \frac{\partial}{\partial x_l} \bigl( \mathcal{N}_k (\omega^{(0)} - \tilde{\omega}) \bigr) \Bigr) \\
				& - \frac{\partial}{\partial x_j} \Bigl( \beta_{ij}^{(0)} \mathcal{J}_{\alpha_1} \frac{\partial \omega^{(0)}}{\partial x_{\alpha_1}} \Bigr) - \frac{\partial}{\partial y_j} \Bigl( \mathcal{H}_{\alpha_2} \frac{\partial T^{(0)}}{\partial x_{\alpha_2}} \mathbf{D}^{(0,1)} \beta_{ij}^{(0)} \mathcal{J}_{\alpha_1} \frac{\partial \omega^{(0)}}{\partial x_{\alpha_1}} \Bigr).
			\end{aligned}
		\end{equation}
	\end{footnotesize}
	\begin{footnotesize}
		\begin{equation}
			\begin{aligned}
				& \varphi(\mathbf{x}, \mathbf{y}, t) \!=\! \!\Bigl[ \frac{\partial}{\partial x_i} \!\Bigl(\! k_{ij}^{(0)} \frac{\partial \mathcal{S}}{\partial y_j} \!\Bigr)\! \!-\! \!\Bigl[ \rho^{(0)} c^{(0)} \mathcal{H}_{\alpha_1} \!-\! \frac{\partial}{\partial y_i} \bigl( k_{i\alpha_1}^{(0)} \mathcal{S} \bigr) \Bigr]\! \frac{\partial}{\partial t} \!\Bigl(\! \frac{\partial T^{(0)}}{\partial x_{\alpha_1}} \!\Bigr)\! \!+\! \frac{\partial}{\partial y_i} \!\Bigl(\! k_{ij}^{(0)} \frac{\partial \mathcal{S}}{\partial x_j} \!\Bigr)\! \Bigr]\! \frac{\partial T^{(0)}}{\partial t} \!+\! \!\Bigl[ \frac{\partial}{\partial y_i} \bigl( \mathcal{S} \mathbf{D}^{(0,1)} k_{i\alpha_1}^{(0)} \bigr) \\
				& \!+\! \frac{\partial}{\partial y_i} \Bigl( \mathcal{S} \mathbf{D}^{(0,1)} k_{ij}^{(0)} \frac{\partial \mathcal{H}_{\alpha_1}}{\partial y_j} \Bigr) \!-\! \rho^{(0)} \mathcal{H}_{\alpha_1} \mathbf{D}^{(0,1)} c^{(0)} \!-\! c^{(0)} \mathcal{H}_{\alpha_1} \mathbf{D}^{(0,1)} \rho^{(0)} \!+\! \frac{\partial}{\partial y_i} \Bigl( \mathcal{H}_{\alpha_1} \mathbf{D}^{(0,1)} k_{ij}^{(0)} \frac{\partial \mathcal{S}}{\partial y_j} \Bigr) \Bigr] \frac{\partial T^{(0)}}{\partial t} \frac{\partial T^{(0)}}{\partial x_{\alpha_1}} \\
				& \!+\! \Bigl[ \mathcal{H}_{\alpha_1} \mathbf{D}^{(0,1)} Q_{hyd}^{(0)} \!+\! \frac{\partial}{\partial y_i} \bigl( \mathbb{Q} \mathbf{D}^{(0,1)} k_{i\alpha_1}^{(0)} \bigr) \!+\! \frac{\partial}{\partial x_i} \Bigl( k_{ij}^{(0)} \frac{\partial \mathcal{H}_{\alpha_1}}{\partial x_j} \Bigr) \!+\! \frac{\partial}{\partial y_i} \Bigl( \mathbb{Q}  \mathbf{D}^{(0,1)} k_{ij}^{(0)} \frac{\partial \mathcal{H}_{\alpha_1}}{\partial y_j} \Bigr) \!+\! \frac{\partial}{\partial x_i} \Bigl( k_{ij}^{(0)} \frac{\partial \mathcal{R}_{\alpha_1}}{\partial y_j} \Bigr) \\
				& \!+\! \frac{\partial}{\partial y_i} \Bigl( k_{ij}^{(0)} \frac{\partial \mathcal{R}_{\alpha_1}}{\partial x_j} \Bigr) \!+\! \frac{\partial}{\partial y_i} \Bigl( \mathcal{H}_{\alpha_1} \mathbf{D}^{(0,1)} k_{ij}^{(0)} \frac{\partial \mathbb{Q} }{\partial y_j} \Bigr) \Bigr] \frac{\partial T^{(0)}}{\partial x_{\alpha_1}} \!+\! \Bigl[ \frac{\partial}{\partial y_i} \bigl( \mathcal{R}_{\alpha_1} \mathbf{D}^{(0,1)} k_{i\alpha_2}^{(0)} \bigr) \!+\! \frac{\partial}{\partial y_i} \Bigl( \mathcal{H}_{\alpha_2} \mathbf{D}^{(0,1)} k_{ij}^{(0)} \frac{\partial \mathcal{H}_{\alpha_1}}{\partial x_j} \Bigr) \\
				& \!+\! \frac{\partial}{\partial y_i} \Bigl( \mathcal{R}_{\alpha_1} \mathbf{D}^{(0,1)} k_{ij}^{(0)} \frac{\partial \mathcal{H}_{\alpha_2}}{\partial y_j} \Bigr) \!-\! \frac{\partial}{\partial y_i} \Bigl( k_{ij}^{(0)} \frac{\partial \mathcal{E}_{\alpha_1 \alpha_2}}{\partial x_j} \Bigr) \!+\! \frac{\partial}{\partial y_i} \Bigl( \mathcal{H}_{\alpha_1} \mathbf{D}^{(0,1)} k_{ij}^{(0)} \frac{\partial \mathcal{R}_{\alpha_2}}{\partial y_j} \Bigr) \Bigr] \frac{\partial T^{(0)}}{\partial x_{\alpha_1}} \frac{\partial T^{(0)}}{\partial x_{\alpha_2}}\\
				& \!-\! \Bigl[ \frac{\partial}{\partial y_i} \bigl( \mathcal{E}_{\alpha_1 \alpha_2} \mathbf{D}^{(0,1)} k_{i\alpha_3}^{(0)} \bigr) \!+\! \frac{\partial}{\partial y_i} \Bigl( \mathcal{E}_{\alpha_1 \alpha_2} \mathbf{D}^{(0,1)} k_{ij}^{(0)} \frac{\partial \mathcal{H}_{\alpha_3}}{\partial y_j} \Bigr) \!+\! \frac{\partial}{\partial y_i} \Bigl( \mathcal{H}_{\alpha_3} \mathbf{D}^{(0,1)} k_{ij}^{(0)} \frac{\partial \mathcal{E}_{\alpha_1 \alpha_2}}{\partial y_j} \Bigr) \Bigr] \frac{\partial T^{(0)}}{\partial x_{\alpha_1}} \frac{\partial T^{(0)}}{\partial x_{\alpha_2}} \frac{\partial T^{(0)}}{\partial x_{\alpha_3}} \\
				& \!+\! \Bigl[ k_{\alpha_2 j}^{(0)} \frac{\partial \mathcal{H}_{\alpha_1}}{\partial x_j} \!+\! \frac{\partial}{\partial x_i} \bigl( k_{i\alpha_2}^{(0)} \mathcal{H}_{\alpha_1} \bigr) \!+\! \frac{\partial}{\partial x_i} \Bigl( k_{ij}^{(0)} \frac{\partial \mathcal{H}_{\alpha_1 \alpha_2}}{\partial y_j} \Bigr) \!+\! \frac{\partial}{\partial y_i} \Bigl( k_{ij}^{(0)} \frac{\partial \mathcal{H}_{\alpha_1 \alpha_2}}{\partial x_j} \Bigr) \!+\! \frac{\partial}{\partial y_i} \bigl( k_{i\alpha_2}^{(0)} \mathcal{R}_{\alpha_1} \bigr) \Bigr] \frac{\partial^2 T^{(0)}}{\partial x_{\alpha_1} \partial x_{\alpha_2}} \\
				& \!+\! \Bigl[ k_{\alpha_2 \alpha_3}^{(0)} \mathcal{H}_{\alpha_1} \!+\! \frac{\partial}{\partial y_i} \bigl( k_{i\alpha_3}^{(0)} \mathcal{H}_{\alpha_1 \alpha_2} \bigr) \Bigr] \frac{\partial^3 T^{(0)}}{\partial x_{\alpha_1} \partial x_{\alpha_2} \partial x_{\alpha_3}} \!+\! \Bigl[ \frac{\partial}{\partial y_i} \bigl( \mathcal{H}_{\alpha_1 \alpha_2} \mathbf{D}^{(0,1)} k_{i\alpha_3}^{(0)} \bigr)  \!+\! \frac{\partial}{\partial y_i} \bigl( \mathcal{H}_{\alpha_3} \mathbf{D}^{(0,1)} k_{i\alpha_2}^{(0)} \mathcal{H}_{\alpha_1} \bigr)\\
				& \!+\! \frac{\partial}{\partial y_i} \Bigl( \mathcal{H}_{\alpha_1 \alpha_2} \mathbf{D}^{(0,1)} k_{ij}^{(0)} \frac{\partial \mathcal{H}_{\alpha_3}}{\partial y_j} \Bigr) \!-\! \frac{\partial k_{i\alpha_2}^{(0)}}{\partial y_i} \mathcal{E}_{\alpha_1 \alpha_3} \!-\! \frac{\partial k_{i\alpha_1}^{(0)}}{\partial y_i} \mathcal{E}_{\alpha_3 \alpha_2} \!+\! \frac{\partial}{\partial y_i} \Bigl( \mathcal{H}_{\alpha_3} \mathbf{D}^{(0,1)} k_{ij}^{(0)} \frac{\partial \mathcal{H}_{\alpha_1 \alpha_2}}{\partial y_j} \Bigr) \Bigr] \frac{\partial^2 T^{(0)}}{\partial x_{\alpha_1} \partial x_{\alpha_2}} \frac{\partial T^{(0)}}{\partial x_{\alpha_3}} \\
				& \!+\! \frac{1}{2} \Bigl[ \frac{\partial}{\partial y_i} \Bigl( \bigl( \mathcal{H}_{\alpha_1} \bigr)^2 \mathbf{D}^{(0,2)} k_{ij}^{(0)} \Bigr) \!+\! \frac{\partial}{\partial y_i} \Bigl( \bigl( \mathcal{H}_{\alpha_1} \bigr)^2 \mathbf{D}^{(0,2)} k_{ij}^{(0)} \frac{\partial \mathcal{H}_{\alpha_2}}{\partial y_j} \Bigr) \Bigr] \Bigl( \frac{\partial T^{(0)}}{\partial x_{\alpha_1}} \Bigr)^2 \frac{\partial T^{(0)}}{\partial x_{\alpha_2}} \!+\! \frac{\partial}{\partial x_i} \Bigl( k_{ij}^{(0)} \frac{\partial \mathbb{Q} }{\partial y_j} \Bigr) \!+\! \frac{\partial}{\partial y_i} \Bigl( k_{ij}^{(0)} \frac{\partial \mathbb{Q} }{\partial x_j} \Bigr).
			\end{aligned}
		\end{equation}
	\end{footnotesize}
	\begin{footnotesize}
		\begin{equation}
			\begin{aligned}
				& \psi(\mathbf{x}, \mathbf{y}, t) = \Bigl[ \frac{\partial}{\partial x_i} \Bigl( g_{ij}^{(0)} \frac{\partial \mathcal{J}_{\alpha_1}}{\partial x_j} \Bigr) - \frac{\partial}{\partial y_i} \bigl( \mathbb{S} \mathbf{D}^{(0,1)} g_{i\alpha_1}^{(0)} \bigr) - \frac{\partial}{\partial y_i} \Bigl( \mathbb{S} \mathbf{D}^{(0,1)} g_{ij}^{(0)} \frac{\partial \mathcal{J}_{\alpha_1}}{\partial y_j} \Bigr) + \frac{\partial}{\partial x_i} \Bigl( g_{ij}^{(0)} \frac{\partial \mathcal{I}_{\alpha_1}}{\partial y_j} \Bigr) \\
				& + \frac{\partial}{\partial y_i} \Bigl( g_{ij}^{(0)} \frac{\partial \mathcal{I}_{\alpha_1}}{\partial x_j} \Bigr) - \frac{\partial}{\partial y_i} \Bigl( \mathcal{J}_{\alpha_1} \mathbf{D}^{(0,1)} g_{ij}^{(0)} \frac{\partial \mathbb{S}}{\partial y_j} \Bigr) \Bigr] \frac{\partial \omega^{(0)}}{\partial x_{\alpha_1}} - \mathcal{H}_{\alpha_1} \mathbf{D}^{(0,1)} S_{hyd}^{(0)} \frac{\partial T^{(0)}}{\partial x_{\alpha_1}} - \mathcal{J}_{\alpha_1} \frac{\partial}{\partial t} \Bigl( \frac{\partial \omega^{(0)}}{\partial x_{\alpha_1}} \Bigr)\\
				& - \frac{\partial}{\partial x_i} \Bigl( g_{ij}^{(0)} \frac{\partial \mathbb{S}}{\partial y_j} \Bigr) - \frac{\partial}{\partial y_i} \Bigl( g_{ij}^{(0)} \frac{\partial \mathbb{S}}{\partial x_j} \Bigr) + \Bigl[ \frac{\partial}{\partial x_i} \Bigl( g_{ij}^{(0)} \frac{\partial \mathcal{F}_{\alpha_1 \alpha_2}}{\partial y_j} \Bigr) + \frac{\partial}{\partial y_i} \bigl( \mathcal{I}_{\alpha_1} \mathbf{D}^{(0,1)} g_{i\alpha_2}^{(0)} \bigr) + \frac{\partial}{\partial y_i} \Bigl( \mathcal{J}_{\alpha_2} \mathbf{D}^{(0,1)} g_{ij}^{(0)} \frac{\partial \mathcal{J}_{\alpha_1}}{\partial x_j} \Bigr) \\
				& \!+\! \frac{\partial}{\partial y_i} \!\Bigl(\! \mathcal{I}_{\alpha_1} \mathbf{D}^{(0,1)} g_{ij}^{(0)} \frac{\partial \mathcal{J}_{\alpha_2}}{\partial y_j} \!\Bigr)\! \!-\! \frac{\partial}{\partial x_i} \!\Bigl(\! g_{ij}^{(0)} \frac{\partial \mathcal{F}_{\alpha_1 \alpha_2}}{\partial y_j} \!\Bigr)\! \!-\! \frac{\partial}{\partial y_i} \!\Bigl(\! g_{ij}^{(0)} \frac{\partial \mathcal{F}_{\alpha_1 \alpha_2}}{\partial x_j} \!\Bigr)\! \!+\! \frac{\partial}{\partial y_i} \!\Bigl(\! \mathcal{J}_{\alpha_2} \mathbf{D}^{(0,1)} g_{ij}^{(0)} \frac{\partial \mathcal{I}_{\alpha_1}}{\partial y_j} \!\Bigr)\! \Bigr]\! \frac{\partial \omega^{(0)}}{\partial x_{\alpha_1}} \frac{\partial \omega^{(0)}}{\partial x_{\alpha_2}} \\
				& \!+\! \!\Bigl[ g_{\alpha_2 j}^{(0)} \frac{\partial \mathcal{J}_{\alpha_1}}{\partial x_j} \!+\! \frac{\partial}{\partial x_i} \bigl( g_{i\alpha_2}^{(0)} \mathcal{J}_{\alpha_1} \bigr) \!+\! \frac{\partial}{\partial x_i} \!\Bigl(\! g_{ij}^{(0)} \frac{\partial \mathcal{J}_{\alpha_1 \alpha_2}}{\partial y_j} \!\Bigr)\! \!+\! g_{\alpha_2 j}^{(0)} \frac{\partial \mathcal{I}_{\alpha_1}}{\partial y_j} \!+\! \frac{\partial}{\partial y_i} \!\Bigl(\! g_{ij}^{(0)} \frac{\partial \mathcal{J}_{\alpha_1 \alpha_2}}{\partial x_j} \!\Bigr)\! \!+\! \frac{\partial}{\partial y_i} \bigl( g_{i\alpha_2}^{(0)} \mathcal{I}_{\alpha_1} \bigr) \Bigr]\! \frac{\partial^2 \omega^{(0)}}{\partial x_{\alpha_1} \partial x_{\alpha_2}} \\
				& - \Bigl[ \frac{\partial}{\partial y_i} \bigl( \mathcal{F}_{\alpha_1 \alpha_2} \mathbf{D}^{(0,1)} g_{i\alpha_3}^{(0)} \bigr) + \frac{\partial}{\partial y_i} \Bigl( \mathcal{F}_{\alpha_1 \alpha_2} \mathbf{D}^{(0,1)} g_{ij}^{(0)} \frac{\partial \mathcal{J}_{\alpha_3}}{\partial y_j} \Bigr) + \frac{\partial}{\partial y_i} \Bigl( \mathcal{J}_{\alpha_3} \mathbf{D}^{(0,1)} g_{ij}^{(0)} \frac{\partial \mathcal{F}_{\alpha_1 \alpha_2}}{\partial y_j} \Bigr) \Bigr] \frac{\partial \omega^{(0)}}{\partial x_{\alpha_1}} \frac{\partial \omega^{(0)}}{\partial x_{\alpha_2}} \frac{\partial \omega^{(0)}}{\partial x_{\alpha_3}} \\
				& + \Bigl[ \frac{\partial}{\partial y_i} \bigl( \mathcal{J}_{\alpha_1 \alpha_2} \mathbf{D}^{(0,1)} g_{i\alpha_3}^{(0)} \bigr) + \frac{\partial}{\partial y_i} \bigl( \mathcal{J}_{\alpha_3} \mathbf{D}^{(0,1)} g_{i\alpha_2}^{(0)} \mathcal{J}_{\alpha_1} \bigr) + \frac{\partial}{\partial y_i} \Bigl( \mathcal{J}_{\alpha_1 \alpha_2} \mathbf{D}^{(0,1)} g_{ij}^{(0)} \frac{\partial \mathcal{J}_{\alpha_3}}{\partial y_j} \Bigr) - \frac{\partial}{\partial y_i} \bigl( g_{i\alpha_2}^{(0)} \mathcal{F}_{\alpha_1 \alpha_3} \bigr) \\
				& - \frac{\partial}{\partial y_i} \bigl( g_{i\alpha_1}^{(0)} \mathcal{F}_{\alpha_3 \alpha_2} \bigr) + \frac{\partial}{\partial y_i} \Bigl( \mathcal{J}_{\alpha_3} \mathbf{D}^{(0,1)} g_{ij}^{(0)} \frac{\partial \mathcal{J}_{\alpha_1 \alpha_2}}{\partial y_j} \Bigr) \Bigr] \frac{\partial^2 \omega^{(0)}}{\partial x_{\alpha_1} \partial x_{\alpha_2}} \frac{\partial \omega^{(0)}}{\partial x_{\alpha_3}} \\
				& + \Bigl[ g_{\alpha_2 \alpha_3}^{(0)} \mathcal{J}_{\alpha_1} + g_{\alpha_3 j}^{(0)} \frac{\partial \mathcal{J}_{\alpha_1 \alpha_2}}{\partial y_j} + \frac{\partial}{\partial y_i} \bigl( g_{i\alpha_3}^{(0)} \mathcal{J}_{\alpha_1 \alpha_2} \bigr) \Bigr] \frac{\partial^3 \omega^{(0)}}{\partial x_{\alpha_1} \partial x_{\alpha_2} \partial x_{\alpha_3}} \\
				& + \frac{1}{2} \Bigl[ \frac{\partial}{\partial y_i} \Bigl( \bigl( \mathcal{J}_{\alpha_1} \bigr)^2 \mathbf{D}^{(0,2)} g_{i\alpha_2}^{(0)} \Bigr) + \frac{\partial}{\partial y_i} \Bigl( \bigl( \mathcal{J}_{\alpha_1} \bigr)^2 \mathbf{D}^{(0,2)} g_{ij}^{(0)} \frac{\partial \mathcal{J}_{\alpha_3}}{\partial y_j} \Bigr) \Bigr] \Bigl( \frac{\partial \omega^{(0)}}{\partial x_{\alpha_1}} \Bigr)^2 \frac{\partial \omega^{(0)}}{\partial x_{\alpha_2}}.
			\end{aligned}
		\end{equation}
	\end{footnotesize}
	\begin{footnotesize}
		\begin{equation}
			\begin{aligned}
				& \theta_i(\mathbf{x},\mathbf{y},t) = \Bigl[ \frac{\partial}{\partial y_j}\bigl( \mathbb{Q}\mathbf{D}^{(0,1)} C_{ijh\alpha_1}^{(0)} \bigr) + \frac{\partial}{\partial x_j}\Bigl( C_{ijkl}^{(0)} \frac{\partial \mathcal{X}_{kh}^{\alpha_1}}{\partial x_l} \Bigr) + \frac{\partial}{\partial y_j}\Bigl( \mathbb{Q}\mathbf{D}^{(0,1)} C_{ijkl}^{(0)} \frac{\partial \mathcal{X}_{kh}^{\alpha_1}}{\partial y_l} \Bigr) + \frac{\partial}{\partial x_j}\Bigl( C_{ijkl}^{(0)} \frac{\partial \mathcal{Q}_{kh}^{\alpha_1}}{\partial y_l} \Bigr) + \frac{\partial}{\partial y_j}\Bigl( C_{ijkl}^{(0)} \frac{\partial \mathcal{Q}_{kh}^{\alpha_1}}{\partial x_l} \Bigr) \Bigr] \frac{\partial u_h^{(0)}}{\partial x_{\alpha_1}}\\
				& + \Bigl[ C_{i\alpha_2 kl}^{(0)} \frac{\partial \mathcal{X}_{kh}^{\alpha_1}}{\partial x_l} + \frac{\partial}{\partial x_j}\bigl( C_{ijk\alpha_2}^{(0)} \mathcal{X}_{kh}^{\alpha_1} \bigr) + \frac{\partial}{\partial x_j}\Bigl( C_{ijkl}^{(0)} \frac{\partial \mathcal{X}_{kh}^{\alpha_1\alpha_2}}{\partial y_l} \Bigr) + C_{i\alpha_2 kl}^{(0)} \frac{\partial \mathcal{Q}_{kh}^{\alpha_1}}{\partial y_l} + \frac{\partial}{\partial y_j}\Bigl( C_{ijkl}^{(0)} \frac{\partial \mathcal{X}_{kh}^{\alpha_1\alpha_2}}{\partial x_l} \Bigr) + \frac{\partial}{\partial y_j}\bigl( C_{ijk\alpha_2}^{(0)} \mathcal{Q}_{kh}^{\alpha_1} \bigr) \Bigr] \frac{\partial^2 u_h^{(0)}}{\partial x_{\alpha_1}\partial x_{\alpha_2}} \\
				& + \Bigl[ C_{i\alpha_2 k\alpha_3}^{(0)} \mathcal{X}_{kh}^{\alpha_1} + C_{i\alpha_3 kl}^{(0)} \frac{\partial \mathcal{X}_{kh}^{\alpha_1\alpha_2}}{\partial y_l} + \frac{\partial}{\partial y_j}\bigl( C_{ijk\alpha_3}^{(0)} \mathcal{X}_{kh}^{\alpha_1\alpha_2} \bigr) \Bigr] \frac{\partial^3 u_h^{(0)}}{\partial x_{\alpha_1}\partial x_{\alpha_2}\partial x_{\alpha_3}} - \Bigl[ \frac{\partial}{\partial y_j}\Bigl( \mathcal{S}\mathbf{D}^{(0,1)} C_{ijkl}^{(0)} \frac{\partial \mathcal{M}_k}{\partial y_l} \Bigr)\\
				& + \frac{\partial}{\partial y_j}\bigl( \mathcal{S}\mathbf{D}^{(0,1)} \alpha_{ij}^{(0)} \bigr) \Bigr] \frac{\partial T^{(0)}}{\partial t} (T^{(0)} - \tilde T) - \Bigl[ \frac{\partial}{\partial x_j}\Bigl( C_{ijkl}^{(0)} \frac{\partial \mathcal{M}_k}{\partial x_l} \Bigr) + \frac{\partial}{\partial y_j}\Bigl( \mathbb{Q}\mathbf{D}^{(0,1)} C_{ijkl}^{(0)} \frac{\partial \mathcal{M}_k}{\partial y_l} \Bigr) + \frac{\partial}{\partial x_j}\Bigl( C_{ijkl}^{(0)} \frac{\partial \mathcal{W}_k}{\partial y_l} \Bigr) \\
				& + \frac{\partial}{\partial y_j}\Bigl( C_{ijkl}^{(0)} \frac{\partial \mathcal{W}_k}{\partial x_l} \Bigr) + \frac{\partial}{\partial y_j}\bigl( \mathbb{Q}\mathbf{D}^{(0,1)} \alpha_{ij}^{(0)} \bigr) \Bigr] (T^{(0)} - \tilde T) - \frac{\partial}{\partial y_j} \bigl( \alpha_{ij}^{(0)} \mathcal{S} \bigr) \frac{\partial T^{(0)}}{\partial t} - \Bigl[ C_{i\alpha_1 kl}^{(0)} \frac{\partial \mathcal{M}_k}{\partial x_l} + \frac{\partial}{\partial x_j}\bigl( C_{ijk\alpha_1}^{(0)} \mathcal{M}_k \bigr) \\
				& + C_{i\alpha_1 kl}^{(0)} \frac{\partial \mathcal{W}_k}{\partial y_l} + \frac{\partial}{\partial x_j}\Bigl( C_{ijkl}^{(0)} \frac{\partial \mathcal{Z}_k^{\alpha_1}}{\partial y_l} \Bigr) + \frac{\partial}{\partial y_j}\bigl( C_{ijk\alpha_1}^{(0)} \mathcal{W}_k \bigr) + \frac{\partial}{\partial y_j}\Bigl( C_{ijkl}^{(0)} \frac{\partial \mathcal{Z}_k^{\alpha_1}}{\partial x_l} \Bigr) + \frac{\partial}{\partial x_j}\bigl( \alpha_{ij}^{(0)} \mathcal{H}_{\alpha_1} \bigr) + \frac{\partial}{\partial y_j}\bigl( \alpha_{ij}^{(0)} \mathcal{R}_{\alpha_1} \bigr) \Bigr] \frac{\partial T^{(0)}}{\partial x_{\alpha_1}} \\
				& - \Bigl[ \frac{\partial}{\partial y_j}\Bigl( \mathcal{R}_{\alpha_1}\mathbf{D}^{(0,1)} C_{ijkl}^{(0)} \frac{\partial \mathcal{M}_k}{\partial y_l} \Bigr) + 2\frac{\partial}{\partial x_j}\Bigl( C_{ijkl}^{(0)} \frac{\partial \mathcal{A}_k^{\alpha_1}}{\partial y_l} \Bigr) - \frac{\partial}{\partial y_j}\Bigl( C_{ijkl}^{(0)} \frac{\partial \mathcal{A}_k^{\alpha_1}}{\partial x_l} \Bigr) + \frac{\partial}{\partial y_j}\Bigl( \mathbf{D}^{(0,1)} C_{ijkl}^{(0)} \mathcal{H}_{\alpha_1} C_{ijkl}^{(0)} \frac{\partial \mathcal{W}_k}{\partial y_l} \Bigr) \\
				& + \frac{\partial}{\partial y_j}\bigl( \mathcal{R}_{\alpha_1}\mathbf{D}^{(0,1)} \alpha_{ij}^{(0)} \bigr) + \frac{\partial}{\partial y_j}\Bigl( \mathcal{H}_{\alpha_1}\mathbf{D}^{(0,1)} C_{ijkl}^{(0)} \frac{\partial \mathcal{M}_k}{\partial x_l} \Bigr) \Bigr] \frac{\partial T^{(0)}}{\partial x_{\alpha_1}} (T^{(0)} - \tilde T) - \Bigl[ C_{i\alpha_1 k\alpha_2}^{(0)} \mathcal{M}_k + C_{i\alpha_2 kl}^{(0)} \frac{\partial \mathcal{Z}_k^{\alpha_1}}{\partial y_l} + \frac{\partial}{\partial y_j}\bigl( C_{ijk\alpha_2}^{(0)} \mathcal{Z}_k^{\alpha_1} \bigr)\\
				& + \alpha_{i\alpha_2}^{(0)} \mathcal{H}_{\alpha_1} + \frac{\partial}{\partial y_j}\bigl( \alpha_{ij}^{(0)} \mathcal{H}_{\alpha_1 \alpha_2} \bigr) \Bigr] \frac{\partial^2 T^{(0)}}{\partial x_{\alpha_1}\partial x_{\alpha_2}} + \Bigl[ \frac{\partial}{\partial y_j}\bigl( C_{ijk\alpha_2}^{(0)} \mathcal{A}_k^{\alpha_1} \bigr) - \frac{\partial}{\partial y_j}\Bigl( \mathbf{D}^{(0,1)} C_{ijkl}^{(0)} \mathcal{H}_{\alpha_1} C_{ijkl}^{(0)} \frac{\partial \mathcal{Z}_k^{\alpha_2}}{\partial y_l} \Bigr) - \frac{\partial}{\partial y_j}\bigl( \mathcal{H}_{\alpha_1}\mathbf{D}^{(0,1)} \alpha_{ij}^{(0)} \mathcal{H}_{\alpha_2} \bigr) 
			\end{aligned}
		\end{equation}
	\end{footnotesize}
	\begin{footnotesize}
		\begin{equation}
			\begin{aligned}
				& + \frac{\partial}{\partial y_j}\bigl( \alpha_{ij}^{(0)} \mathcal{E}_{\alpha_1 \alpha_2} \bigr) - \frac{\partial}{\partial y_j}\bigl( \mathcal{H}_{\alpha_2}\mathbf{D}^{(0,1)} C_{ijk\alpha_1}^{(0)} \mathcal{M}_k \bigr) \Bigr] \frac{\partial T^{(0)}}{\partial x_{\alpha_1}} \frac{\partial T^{(0)}}{\partial x_{\alpha_2}} - \Bigl[ \frac{\partial}{\partial y_j}\Bigl( \mathcal{H}_{\alpha_1 \alpha_2}\mathbf{D}^{(0,1)} C_{ijkl}^{(0)} \frac{\partial \mathcal{M}_k}{\partial y_l} \Bigr) - \frac{\partial}{\partial y_j}\bigl( C_{ijk\alpha_2}^{(0)} \mathcal{A}_k^{\alpha_1} \bigr)\\
				& + \frac{\partial}{\partial y_j}\bigl( \mathcal{H}_{\alpha_1 \alpha_2}\mathbf{D}^{(0,1)} \alpha_{ij}^{(0)} \bigr) \Bigr] \frac{\partial^2 T^{(0)}}{\partial x_{\alpha_1}\partial x_{\alpha_2}} (T^{(0)} - \tilde T) + \Bigl[ \frac{\partial}{\partial y_j}\Bigl( \mathcal{E}_{\alpha_1 \alpha_2}\mathbf{D}^{(0,1)} C_{ijkl}^{(0)} \frac{\partial \mathcal{M}_k}{\partial y_l} \Bigr) + \frac{\partial}{\partial y_j}\Bigl( \mathbf{D}^{(0,1)} C_{ijkl}^{(0)} \mathcal{H}_{\alpha_1} C_{ijkl}^{(0)} \frac{\partial \mathcal{A}_k^{\alpha_2}}{\partial y_l} \Bigr) \\
				& + \frac{\partial}{\partial y_j}\bigl( \mathcal{E}_{\alpha_1 \alpha_2}\mathbf{D}^{(0,1)} \alpha_{ij}^{(0)} \bigr) \Bigr] \frac{\partial T^{(0)}}{\partial x_{\alpha_1}} \frac{\partial T^{(0)}}{\partial x_{\alpha_2}} (T^{(0)} - \tilde T) - \frac{1}{2} \Bigl[ \frac{\partial}{\partial y_j}\Bigl( \bigl( \mathcal{H}_{\alpha_1} \bigr)^2 \mathbf{D}^{(0,2)} C_{ijkl}^{(0)} \frac{\partial \mathcal{M}_k}{\partial y_l} \Bigr)\\
				& + \frac{\partial}{\partial y_j}\Bigl( \bigl( \mathcal{H}_{\alpha_1} \bigr)^2 \mathbf{D}^{(0,2)} \alpha_{ij}^{(0)} \Bigr) \Bigr] \Bigl( \frac{\partial T^{(0)}}{\partial x_{\alpha_1}} \Bigr)^2 (T^{(0)} - \tilde T) - \Bigl[ \frac{\partial}{\partial y_j}\Bigl( \mathcal{S}\mathbf{D}^{(0,1)} C_{ijkl}^{(0)} \frac{\partial \mathcal{N}_k}{\partial y_l} \Bigr) + \frac{\partial}{\partial y_j}\bigl( \mathcal{S}\mathbf{D}^{(0,1)} \beta_{ij}^{(0)} \bigr) \Bigr] \frac{\partial T^{(0)}}{\partial t} (\omega^{(0)} - \tilde\omega)\\
				& - \Bigl[ \frac{\partial}{\partial y_j}\Bigl( \mathcal{R}_{\alpha_1}\mathbf{D}^{(0,1)} C_{ijkl}^{(0)} \frac{\partial \mathcal{N}_k}{\partial y_l} \Bigr) - \frac{\partial}{\partial y_j}\Bigl( C_{ijkl}^{(0)} \frac{\partial \mathcal{B}_k^{\alpha_1}}{\partial x_l} \Bigr) + \frac{\partial}{\partial y_j}\Bigl( \mathbf{D}^{(0,1)} C_{ijkl}^{(0)} \mathcal{H}_{\alpha_1} C_{ijkl}^{(0)} \frac{\partial \mathcal{V}_k}{\partial y_l} \Bigr) + \frac{\partial}{\partial y_j}\bigl( \mathcal{R}_{\alpha_1}\mathbf{D}^{(0,1)} \beta_{ij}^{(0)} \bigr)\\
				& + \frac{\partial}{\partial y_j}\Bigl( \mathcal{H}_{\alpha_1}\mathbf{D}^{(0,1)} C_{ijkl}^{(0)} \frac{\partial \mathcal{N}_k}{\partial x_l} \Bigr) \Bigr] \frac{\partial T^{(0)}}{\partial x_{\alpha_1}} (\omega^{(0)} - \tilde\omega) + \Bigl[ \frac{\partial}{\partial y_j}\bigl( C_{ijk\alpha_2}^{(0)} \mathcal{B}_k^{\alpha_1} \bigr) - \frac{\partial}{\partial y_j}\Bigl( \mathbf{D}^{(0,1)} C_{ijkl}^{(0)} \mathcal{H}_{\alpha_1} C_{ijkl}^{(0)} \frac{\partial \mathcal{G}_k^{\alpha_2}}{\partial y_l} \Bigr) \\
				& - \frac{\partial}{\partial y_j}\bigl( \mathcal{H}_{\alpha_1}\mathbf{D}^{(0,1)} \beta_{ij}^{(0)} \mathcal{J}_{\alpha_2} \bigr) - \frac{\partial}{\partial y_j}\bigl( \mathcal{H}_{\alpha_1}\mathbf{D}^{(0,1)} C_{ijk\alpha_2}^{(0)} \mathcal{N}_k \bigr) \Bigr] \frac{\partial T^{(0)}}{\partial x_{\alpha_1}} \frac{\partial \omega^{(0)}}{\partial x_{\alpha_2}} - \Bigl[ \frac{\partial}{\partial y_j}\Bigl( \mathcal{H}_{\alpha_1 \alpha_2}\mathbf{D}^{(0,1)} C_{ijkl}^{(0)} \frac{\partial \mathcal{N}_k}{\partial y_l} \Bigr) \\
				& - \frac{\partial}{\partial y_j}\bigl( C_{ijk\alpha_2}^{(0)} \mathcal{B}_k^{\alpha_1} \bigr) + \frac{\partial}{\partial y_j}\bigl( \mathcal{H}_{\alpha_1 \alpha_2}\mathbf{D}^{(0,1)} \beta_{ij}^{(0)} \bigr) \Bigr] \frac{\partial^2 T^{(0)}}{\partial x_{\alpha_1}\partial x_{\alpha_2}} (\omega^{(0)} - \tilde\omega) + \Bigl[ \frac{\partial}{\partial y_j}\Bigl( \mathcal{E}_{\alpha_1 \alpha_2}\mathbf{D}^{(0,1)} C_{ijkl}^{(0)} \frac{\partial \mathcal{N}_k}{\partial y_l} \Bigr)\\
				& + \frac{\partial}{\partial y_j}\Bigl( \mathbf{D}^{(0,1)} C_{ijkl}^{(0)} \mathcal{H}_{\alpha_1} C_{ijkl}^{(0)} \frac{\partial \mathcal{B}_k^{\alpha_2}}{\partial y_l} \Bigr) + \frac{\partial}{\partial y_j}\bigl( \mathcal{E}_{\alpha_1 \alpha_2}\mathbf{D}^{(0,1)} \beta_{ij}^{(0)} \bigr) \Bigr] \frac{\partial T^{(0)}}{\partial x_{\alpha_1}} \frac{\partial T^{(0)}}{\partial x_{\alpha_2}} (\omega^{(0)} - \tilde\omega) + \frac{\partial}{\partial y_j}\bigl( \beta_{ij}^{(0)} \mathcal{F}_{\alpha_1 \alpha_2} \bigr) \frac{\partial \omega^{(0)}}{\partial x_{\alpha_1}} \frac{\partial \omega^{(0)}}{\partial x_{\alpha_2}}\\
				& - \frac{1}{2} \Bigl[ \frac{\partial}{\partial y_j}\Bigl( \bigl( \mathcal{H}_{\alpha_1} \bigr)^2 \mathbf{D}^{(0,2)} C_{ijkl}^{(0)} \frac{\partial \mathcal{N}_k}{\partial y_l} \Bigr) + \frac{\partial}{\partial y_j}\Bigl( \bigl( \mathcal{H}_{\alpha_1} \bigr)^2 \mathbf{D}^{(0,2)} \beta_{ij}^{(0)} \Bigr) \Bigr] \Bigl( \frac{\partial T^{(0)}}{\partial x_{\alpha_1}} \Bigr)^2 (\omega^{(0)} - \tilde\omega) - \Bigl[ \frac{\partial}{\partial x_j}\Bigl( C_{ijkl}^{(0)} \frac{\partial \mathcal{N}_k}{\partial x_l} \Bigr) \\
				& + \frac{\partial}{\partial y_j}\Bigl( \mathbb{Q}\mathbf{D}^{(0,1)} C_{ijkl}^{(0)} \frac{\partial \mathcal{N}_k}{\partial y_l} \Bigr) + \frac{\partial}{\partial x_j}\Bigl( C_{ijkl}^{(0)} \frac{\partial \mathcal{V}_k}{\partial y_l} \Bigr) + \frac{\partial}{\partial y_j}\Bigl( C_{ijkl}^{(0)} \frac{\partial \mathcal{V}_k}{\partial x_l} \Bigr) + \frac{\partial}{\partial y_j}\bigl( \mathbb{Q}\mathbf{D}^{(0,1)} \beta_{ij}^{(0)} \bigr) \Bigr] (\omega^{(0)} - \tilde\omega)\\
				& - \Bigl[ C_{i\alpha_1 kl}^{(0)} \frac{\partial \mathcal{N}_k}{\partial x_l} + \frac{\partial}{\partial x_j}\bigl( C_{ijk\alpha_1}^{(0)} \mathcal{N}_k \bigr) + C_{i\alpha_1 kl}^{(0)} \frac{\partial \mathcal{V}_k}{\partial y_l} + \frac{\partial}{\partial x_j}\Bigl( C_{ijkl}^{(0)} \frac{\partial \mathcal{G}_k^{\alpha_1}}{\partial y_l} \Bigr) + \frac{\partial}{\partial y_j}\bigl( C_{ijk\alpha_1}^{(0)} \mathcal{V}_k \bigr) + \frac{\partial}{\partial y_j}\Bigl( C_{ijkl}^{(0)} \frac{\partial \mathcal{G}_k^{\alpha_1}}{\partial x_l} \Bigr) + \frac{\partial}{\partial x_j}\bigl( \beta_{ij}^{(0)} \mathcal{J}_{\alpha_1} \bigr)\\
				& + \frac{\partial}{\partial y_j}\bigl( \beta_{ij}^{(0)} \mathcal{I}_{\alpha_1} \bigr) \Bigr] \frac{\partial \omega^{(0)}}{\partial x_{\alpha_1}} - \Bigl[ C_{i\alpha_1 k\alpha_2}^{(0)} \mathcal{N}_k + C_{i\alpha_2 kl}^{(0)} \frac{\partial \mathcal{G}_k^{\alpha_1}}{\partial y_l} + \frac{\partial}{\partial y_j}\bigl( C_{ijk\alpha_2}^{(0)} \mathcal{G}_k^{\alpha_1} \bigr) + \beta_{i\alpha_2}^{(0)} \mathcal{J}_{\alpha_1} + \frac{\partial}{\partial y_j}\bigl( \beta_{ij}^{(0)} \mathcal{J}_{\alpha_1 \alpha_2} \bigr) \Bigr] \frac{\partial^2 \omega^{(0)}}{\partial x_{\alpha_1}\partial x_{\alpha_2}}\\
				& + \Bigl[ \frac{\partial}{\partial y_j}\bigl( \mathcal{S}\mathbf{D}^{(0,1)} C_{ijh\alpha_1}^{(0)} \bigr) + \frac{\partial}{\partial y_j}\Bigl( \mathcal{S}\mathbf{D}^{(0,1)} C_{ijkl}^{(0)} \frac{\partial \mathcal{X}_{kh}^{\alpha_1}}{\partial y_l} \Bigr) \Bigr] \frac{\partial T^{(0)}}{\partial t} \frac{\partial u_h^{(0)}}{\partial x_{\alpha_1}} + \Bigl[ \frac{\partial}{\partial y_j}\bigl( \mathcal{R}_{\alpha_1}\mathbf{D}^{(0,1)} C_{ijh\alpha_2}^{(0)} \bigr) + \frac{\partial}{\partial y_j}\Bigl( \mathcal{H}_{\alpha_1}\mathbf{D}^{(0,1)} C_{ijkl}^{(0)} \frac{\partial \mathcal{X}_{kh}^{\alpha_2}}{\partial x_l} \Bigr) \\
				& + \frac{\partial}{\partial y_j}\Bigl( \mathcal{R}_{\alpha_1}\mathbf{D}^{(0,1)} C_{ijkl}^{(0)} \frac{\partial \mathcal{X}_{kh}^{\alpha_2}}{\partial y_l} \Bigr) - \frac{\partial}{\partial y_j}\Bigl( C_{ijkl}^{(0)} \frac{\partial \mathcal{P}_{kh}^{\alpha_1\alpha_2}}{\partial x_l} \Bigr) + \frac{\partial}{\partial y_j}\Bigl( \mathbf{D}^{(0,1)} C_{ijkl}^{(0)} \mathcal{H}_{\alpha_1} C_{ijkl}^{(0)} \frac{\partial \mathcal{Q}_{kh}^{\alpha_2}}{\partial y_l} \Bigr) \Bigr] \frac{\partial T^{(0)}}{\partial x_{\alpha_1}} \frac{\partial u_h^{(0)}}{\partial x_{\alpha_2}} \\
				& + \Bigl[ \frac{\partial}{\partial y_j}\bigl( \mathcal{H}_{\alpha_1}\mathbf{D}^{(0,1)} C_{ijk\alpha_3}^{(0)} \mathcal{X}_{kh}^{\alpha_2} \bigr) - \frac{\partial}{\partial y_j}\bigl( C_{ijk\alpha_3}^{(0)} \mathcal{P}_{kh}^{\alpha_1\alpha_2} \bigr) + \frac{\partial}{\partial y_j}\Bigl( \mathbf{D}^{(0,1)} C_{ijkl}^{(0)} \mathcal{H}_{\alpha_1} C_{ijkl}^{(0)} \frac{\partial \mathcal{X}_{kh}^{\alpha_2\alpha_3}}{\partial y_l} \Bigr) \Bigr] \frac{\partial T^{(0)}}{\partial x_{\alpha_1}} \frac{\partial^2 u_h^{(0)}}{\partial x_{\alpha_2}\partial x_{\alpha_3}} \\
				& + \frac{1}{2} \Bigl[ \frac{\partial}{\partial y_j}\Bigl( \bigl( \mathcal{H}_{\alpha_1} \bigr)^2 \mathbf{D}^{(0,2)} C_{ijh\alpha_2}^{(0)} \Bigr) + \frac{\partial}{\partial y_j}\Bigl( \bigl( \mathcal{H}_{\alpha_1} \bigr)^2 \mathbf{D}^{(0,2)} C_{ijkl}^{(0)} \frac{\partial \mathcal{X}_{kh}^{\alpha_2}}{\partial y_l} \Bigr) \Bigr] \Bigl( \frac{\partial T^{(0)}}{\partial x_{\alpha_1}} \Bigr)^2 \frac{\partial u_h^{(0)}}{\partial x_{\alpha_2}} - \Bigl[ \frac{\partial}{\partial y_j}\bigl( \mathcal{E}_{\alpha_1 \alpha_2}\mathbf{D}^{(0,1)} C_{ijh\alpha_3}^{(0)} \bigr)\\
				& + \frac{\partial}{\partial y_j}\Bigl( \mathcal{E}_{\alpha_1 \alpha_2}\mathbf{D}^{(0,1)} C_{ijkl}^{(0)} \frac{\partial \mathcal{X}_{kh}^{\alpha_3}}{\partial y_l} \Bigr) + \frac{\partial}{\partial y_j}\Bigl( \mathbf{D}^{(0,1)} C_{ijkl}^{(0)} \mathcal{H}_{\alpha_1} C_{ijkl}^{(0)} \frac{\partial \mathcal{P}_{kh}^{\alpha_2\alpha_3}}{\partial y_l} \Bigr) \Bigr] \frac{\partial T^{(0)}}{\partial x_{\alpha_1}} \frac{\partial T^{(0)}}{\partial x_{\alpha_2}} \frac{\partial u_h^{(0)}}{\partial x_{\alpha_3}} - \frac{\partial}{\partial y_j}\bigl( \alpha_{ij}^{(0)}\mathbb{Q} \bigr) + \frac{\partial}{\partial y_j}\bigl( \beta_{ij}^{(0)}\mathbb{S} \bigr) \\
				& + \Bigl[ \frac{\partial}{\partial y_j}\bigl( \mathcal{H}_{\alpha_1 \alpha_2}\mathbf{D}^{(0,1)} C_{ijh\alpha_3}^{(0)} \bigr) + \frac{\partial}{\partial y_j}\Bigl( \mathcal{H}_{\alpha_1 \alpha_2}\mathbf{D}^{(0,1)} C_{ijkl}^{(0)} \frac{\partial \mathcal{X}_{kh}^{\alpha_3}}{\partial y_l} \Bigr) - \frac{\partial}{\partial y_j}\bigl( C_{ijk\alpha_2}^{(0)} \mathcal{P}_{kh}^{\alpha_1\alpha_3} \bigr) \Bigr] \frac{\partial^2 T^{(0)}}{\partial x_{\alpha_1}\partial x_{\alpha_2}} \frac{\partial u_h^{(0)}}{\partial x_{\alpha_3}}.
			\end{aligned}
		\end{equation}
	\end{footnotesize}
	
	\section{Supplemental expressions for Section~\ref{sec:32}}
	\label{app:B}
	\begin{footnotesize}
	\begin{equation}
		\begin{aligned}
			&\sigma_{TY}(T^{(2, \epsilon)}) = n_i k_{ij}(\mathbf{y},T^\epsilon) \frac{\partial T^{(2, \epsilon)}}{\partial x_j} = n_i k_{ij}(\mathbf{y},T^\epsilon) \Bigl( \frac{\partial}{\partial x_j} + \frac{1}{\epsilon} \frac{\partial}{\partial y_j} \Bigr) \Bigl[ T^{(0)} + \epsilon \mathcal{H}_{\alpha_1}(\mathbf{y},T^{(0)}) \frac{\partial T^{(0)}}{\partial x_{\alpha_1}} + \epsilon^2 \mathcal{S}(\mathbf{y},T^{(0)}) \frac{\partial T^{(0)}}{\partial t} \\
			& + \epsilon^2 \mathcal{H}_{\alpha_1 \alpha_2}(\mathbf{y},T^{(0)}) \frac{\partial^2 T^{(0)}}{\partial x_{\alpha_1} \partial x_{\alpha_2}} + \epsilon^2 \mathcal{R}_{\alpha_1}(\mathbf{y},T^{(0)}) \frac{\partial T^{(0)}}{\partial x_{\alpha_1}} - \epsilon^2 \mathcal{E}_{\alpha_1 \alpha_2}(\mathbf{y},T^{(0)}) \frac{\partial T^{(0)}}{\partial x_{\alpha_1}} \frac{\partial T^{(0)}}{\partial x_{\alpha_2}} + \epsilon^2 \mathbb{Q}(\mathbf{y},T^{(0)}) \Bigr] \\
			&= n_i k_{ij}(\mathbf{y},T^\epsilon) \frac{\partial}{\partial x_j} \Bigl[ T^{(0)} + \epsilon \mathcal{H}_{\alpha_1}(\mathbf{y},T^{(0)}) \frac{\partial T^{(0)}}{\partial x_{\alpha_1}} + \epsilon^2 \mathcal{S}(\mathbf{y},T^{(0)}) \frac{\partial T^{(0)}}{\partial t} + \epsilon^2 \mathcal{H}_{\alpha_1 \alpha_2}(\mathbf{y},T^{(0)}) \frac{\partial^2 T^{(0)}}{\partial x_{\alpha_1} \partial x_{\alpha_2}} \\
			& + \epsilon^2 \mathcal{R}_{\alpha_1}(\mathbf{y},T^{(0)}) \frac{\partial T^{(0)}}{\partial x_{\alpha_1}} - \epsilon^2 \mathcal{E}_{\alpha_1 \alpha_2}(\mathbf{y},T^{(0)}) \frac{\partial T^{(0)}}{\partial x_{\alpha_1}} \frac{\partial T^{(0)}}{\partial x_{\alpha_2}} + \epsilon^2 \mathbb{Q}(\mathbf{y},T^{(0)}) \Bigr] + \sigma_{TY}(\mathcal{H}_{\alpha_1}) \frac{\partial T^{(0)}}{\partial x_{\alpha_1}} \\
			& + \epsilon \sigma_{TY}(\mathcal{S}) \frac{\partial T^{(0)}}{\partial t} + \epsilon \sigma_{TY}(\mathcal{H}_{\alpha_1 \alpha_2}) \frac{\partial^2 T^{(0)}}{\partial x_{\alpha_1} \partial x_{\alpha_2}} + \epsilon \sigma_{TY}(\mathcal{R}_{\alpha_1}) \frac{\partial T^{(0)}}{\partial x_{\alpha_1}} - \epsilon \sigma_{TY}(\mathcal{E}_{\alpha_1 \alpha_2}) \frac{\partial T^{(0)}}{\partial x_{\alpha_1}} \frac{\partial T^{(0)}}{\partial x_{\alpha_2}} + \epsilon \sigma_{TY}(\mathbb{Q}).
		\end{aligned}
	\end{equation}
	\end{footnotesize}
	\begin{footnotesize}
	\begin{equation}
		\begin{aligned}
			&\sigma_{\omega Y}(\omega^{(2, \epsilon)}) = n_i g_{ij}(\mathbf{y},\omega^\epsilon) \frac{\partial \omega^{(2, \epsilon)}}{\partial x_j} = n_i g_{ij}(\mathbf{y},\omega^\epsilon) \Bigl( \frac{\partial}{\partial x_j} + \frac{1}{\epsilon} \frac{\partial}{\partial y_j} \Bigr) \Bigl[ \omega^{(0)} + \epsilon \mathcal{J}_{\alpha_1}(\mathbf{y},\omega^{(0)}) \frac{\partial \omega^{(0)}}{\partial x_{\alpha_1}} \\
			& + \epsilon^2 \mathcal{J}_{\alpha_1 \alpha_2}(\mathbf{y},\omega^{(0)}) \frac{\partial^2 \omega^{(0)}}{\partial x_{\alpha_1} \partial x_{\alpha_2}} + \epsilon^2 \mathcal{I}_{\alpha_1}(\mathbf{y},\omega^{(0)}) \frac{\partial \omega^{(0)}}{\partial x_{\alpha_1}} - \epsilon^2 \mathcal{F}_{\alpha_1\alpha_2}(\mathbf{y},\omega^{(0)}) \frac{\partial \omega^{(0)}}{\partial x_{\alpha_1}} \frac{\partial \omega^{(0)}}{\partial x_{\alpha_2}} - \epsilon^2\mathbb{S} (\mathbf{y},T^{(0)}) \Bigr] \\
			&= n_i g_{ij}(\mathbf{y},\omega^\epsilon) \frac{\partial}{\partial x_j} \Bigl[ \omega^{(0)} + \epsilon \mathcal{J}_{\alpha_1}(\mathbf{y},\omega^{(0)}) \frac{\partial \omega^{(0)}}{\partial x_{\alpha_1}} + \epsilon^2 \mathcal{J}_{\alpha_1 \alpha_2}(\mathbf{y},\omega^{(0)}) \frac{\partial^2 \omega^{(0)}}{\partial x_{\alpha_1} \partial x_{\alpha_2}} \\
			& + \epsilon^2 \mathcal{I}_{\alpha_1}(\mathbf{y},\omega^{(0)}) \frac{\partial \omega^{(0)}}{\partial x_{\alpha_1}} - \epsilon^2 \mathcal{F}_{\alpha_1\alpha_2}(\mathbf{y},\omega^{(0)}) \frac{\partial \omega^{(0)}}{\partial x_{\alpha_1}} \frac{\partial \omega^{(0)}}{\partial x_{\alpha_2}} - \epsilon^2 \mathbb{S} (\mathbf{y},T^{(0)}) \Bigr] + \sigma_{\omega Y}(\mathcal{J}_{\alpha_1}) \frac{\partial \omega^{(0)}}{\partial x_{\alpha_1}} \\
			& + \epsilon \sigma_{\omega Y}(\mathcal{J}_{\alpha_1 \alpha_2}) \frac{\partial^2 \omega^{(0)}}{\partial x_{\alpha_1} \partial x_{\alpha_2}} + \epsilon \sigma_{\omega Y}(\mathcal{I}_{\alpha_1}) \frac{\partial \omega^{(0)}}{\partial x_{\alpha_1}} - \epsilon \sigma_{\omega Y}(\mathcal{F}_{\alpha_1\alpha_2}) \frac{\partial \omega^{(0)}}{\partial x_{\alpha_1}} \frac{\partial \omega^{(0)}}{\partial x_{\alpha_2}} - \epsilon \sigma_{\omega Y}(\mathbb{S}).
		\end{aligned}
	\end{equation}
	\end{footnotesize}
	\begin{footnotesize}
	\begin{equation}
		\begin{aligned}
			&\sigma_{iY}(\bm{u}^{(2, \epsilon)}) = n_j C_{ijkl}(\mathbf{y}, T^\epsilon) \frac{\partial u_k^{(2, \epsilon)}}{\partial x_l} = n_j C_{ijkl}(\mathbf{y}, T^\epsilon) \Bigl( \frac{\partial }{\partial x_l} + \frac{1}{\epsilon} \frac{\partial }{\partial y_l} \Bigr) \Bigl[ u_i^{(0)} + \epsilon \mathcal{X}_{ih}^{\alpha_1}(\mathbf{y}, T^{(0)}) \frac{\partial u_h^{(0)}}{\partial x_{\alpha_1}} \\
			& - \epsilon \mathcal{M}_i(\mathbf{y}, T^{(0)}) ( T^{(0)} - \tilde{T} ) - \epsilon \mathcal{N}_i(\mathbf{y}, T^{(0)}) ( \omega^{(0)} - \tilde{\omega} ) + \epsilon^2 \mathcal{X}_{ih}^{\alpha_1 \alpha_2}(\mathbf{y}, T^{(0)}) \frac{\partial^2 u_h^{(0)}}{\partial x_{\alpha_1} \partial x_{\alpha_2}} + \epsilon^2 \mathcal{Q}_{ih}^{\alpha_1}(\mathbf{y}, T^{(0)}) \frac{\partial u_h^{(0)}}{\partial x_{\alpha_1}} \\
			& - \epsilon^2 \mathcal{P}_{ih}^{\alpha_1 \alpha_2}(\mathbf{y}, T^{(0)}) \frac{\partial T^{(0)}}{\partial x_{\alpha_1}} \frac{\partial u_h^{(0)}}{\partial x_{\alpha_2}} - \epsilon^2 \mathcal{W}_i(\mathbf{y}, T^{(0)}) ( T^{(0)} - \tilde{T} ) - \epsilon^2 \mathcal{Z}_i^{\alpha_1}(\mathbf{y}, T^{(0)}) \frac{\partial T^{(0)}}{\partial x_{\alpha_1}} + \epsilon^2 \mathcal{A}_i^{\alpha_1}(\mathbf{y}, T^{(0)}) \frac{\partial T^{(0)}}{\partial x_{\alpha_1}} ( T^{(0)} - \tilde{T} ) \\
			& - \epsilon^2 \mathcal{V}_i(\mathbf{y}, T^{(0)}) ( \omega^{(0)} - \tilde{\omega} ) - \epsilon^2 \mathcal{G}_i^{\alpha_1}(\mathbf{y}, T^{(0)}, \omega^{(0)}) \frac{\partial \omega^{(0)}}{\partial x_{\alpha_1}} + \epsilon^2 \mathcal{B}_i^{\alpha_1}(\mathbf{y}, T^{(0)}) \frac{\partial T^{(0)}}{\partial x_{\alpha_1}} ( \omega^{(0)} - \tilde{\omega} ) \Bigr] \\
			&= n_j C_{ijkl}(\mathbf{y}, T^\epsilon) \frac{\partial }{\partial x_l} \Bigl[ u_i^{(0)} + \epsilon \mathcal{X}_{ih}^{\alpha_1}(\mathbf{y}, T^{(0)}) \frac{\partial u_h^{(0)}}{\partial x_{\alpha_1}} - \epsilon \mathcal{M}_i(\mathbf{y}, T^{(0)}) ( T^{(0)} - \tilde{T} ) - \epsilon \mathcal{N}_i(\mathbf{y}, T^{(0)}) ( \omega^{(0)} - \tilde{\omega} ) \\
			& + \epsilon^2 \mathcal{X}_{ih}^{\alpha_1 \alpha_2}(\mathbf{y}, T^{(0)}) \frac{\partial^2 u_h^{(0)}}{\partial x_{\alpha_1} \partial x_{\alpha_2}} + \epsilon^2 \mathcal{Q}_{ih}^{\alpha_1}(\mathbf{y}, T^{(0)}) \frac{\partial u_h^{(0)}}{\partial x_{\alpha_1}} - \epsilon^2 \mathcal{P}_{ih}^{\alpha_1 \alpha_2}(\mathbf{y}, T^{(0)}) \frac{\partial T^{(0)}}{\partial x_{\alpha_1}} \frac{\partial u_h^{(0)}}{\partial x_{\alpha_2}} - \epsilon^2 \mathcal{W}_i(\mathbf{y}, T^{(0)}) ( T^{(0)} - \tilde{T} ) \\
			& - \epsilon^2 \mathcal{Z}_i^{\alpha_1}(\mathbf{y}, T^{(0)}) \frac{\partial T^{(0)}}{\partial x_{\alpha_1}} + \epsilon^2 \mathcal{A}_i^{\alpha_1}(\mathbf{y}, T^{(0)}) \frac{\partial T^{(0)}}{\partial x_{\alpha_1}} ( T^{(0)} - \tilde{T} ) - \epsilon^2 \mathcal{V}_i(\mathbf{y}, T^{(0)}) ( \omega^{(0)} - \tilde{\omega} ) - \epsilon^2 \mathcal{G}_i^{\alpha_1}(\mathbf{y}, T^{(0)}, \omega^{(0)}) \frac{\partial \omega^{(0)}}{\partial x_{\alpha_1}} \\
			& + \epsilon^2 \mathcal{B}_i^{\alpha_1}(\mathbf{y}, T^{(0)}) \frac{\partial T^{(0)}}{\partial x_{\alpha_1}} ( \omega^{(0)} - \tilde{\omega} ) \Bigr] + \sigma_{iY}(\mathbcal{X}_h^{\alpha_1}) \frac{\partial u_h^{(0)}}{\partial x_{\alpha_1}} - \sigma_{iY}(\mathbcal{M}) ( T^{(0)} - \tilde{T} ) - \sigma_{iY}(\mathbcal{N}) ( \omega^{(0)} - \tilde{\omega} ) \\
			& + \epsilon \sigma_{iY}(\mathbcal{X}_h^{\alpha_1\alpha_2}) \frac{\partial^2 u_h^{(0)}}{\partial x_{\alpha_1} \partial x_{\alpha_2}} + \epsilon \sigma_{iY}(\mathbcal{Q}_h^{\alpha_1}) \frac{\partial u_h^{(0)}}{\partial x_{\alpha_1}} - \epsilon \sigma_{iY}(\mathbcal{P}_h^{\alpha_1\alpha_2}) \frac{\partial T^{(0)}}{\partial x_{\alpha_1}} \frac{\partial u_h^{(0)}}{\partial x_{\alpha_2}} - \epsilon \sigma_{iY}(\mathbcal{W}) ( T^{(0)} - \tilde{T} ) - \epsilon \sigma_{iY}(\mathbcal{Z}^{\alpha_1}) \frac{\partial T^{(0)}}{\partial x_{\alpha_1}} \\
			& + \epsilon \sigma_{iY}(\mathbcal{A}^{\alpha_1}) \frac{\partial T^{(0)}}{\partial x_{\alpha_1}} ( T^{(0)} - \tilde{T} ) - \epsilon \sigma_{iY}(\mathbcal{V}) ( \omega^{(0)} - \tilde{\omega} ) - \epsilon \sigma_{iY}(\mathbcal{G}^{\alpha_1}) \frac{\partial \omega^{(0)}}{\partial x_{\alpha_1}} + \epsilon \sigma_{iY}(\mathbcal{B}^{\alpha_1}) \frac{\partial T^{(0)}}{\partial x_{\alpha_1}} ( \omega^{(0)} - \tilde{\omega}).
		\end{aligned}
	\end{equation}
	\end{footnotesize}

\section*{Acknowledgments}
This research was supported by the National Natural Science Foundation of China (No.~12471387), Xidian University Specially Funded Project for Interdisciplinary Exploration (No.~TZJH2024008), Fundamental Research Funds for the Central Universities (No.~QTZX25082), Innovation Capability Support Program of Shaanxi Province (No.~2024RS-CXTD-88), and also supported by the Center for high performance computing of Xidian University.

\bibliographystyle{abbrvnat}
\bibliography{paper}   

\end{document}